\documentclass[leqno, 11pt]{amsart}
\usepackage{amsmath,amsthm,amssymb,latexsym,epsf,graphics}
\usepackage{comment}
\usepackage{enumitem, kantlipsum}
\usepackage{mathtools}
\usepackage{epsfig}
\usepackage[arrow]{xy}
\usepackage{float}
\restylefloat{figure}
\usepackage{bbold}
\usepackage{stmaryrd}
\usepackage{scalerel}
\usepackage{wasysym}
\usepackage{tikz-cd}
\tikzcdset{arrows={line width=rule_thickness}}
\usepackage{mathrsfs}
\usepackage[font=singlespacing]{caption}
\usepackage{subfig}
\usepackage{setspace}
\usepackage{indentfirst}

\usepackage{titletoc} 
\usepackage[backend=biber,
style=alphabetic,maxbibnames=99]{biblatex}

\usepackage{tikz}
\usetikzlibrary{calc,matrix,arrows,decorations.markings}
\usepackage[left=1in,right=1in,top=1in,bottom=1in, letterpaper]{geometry}

\usepackage{graphicx,color}
\graphicspath{{figures/}{./}}
\usepackage{hyperref}
\usepackage{cleveref}

\newcommand{\C}{\mathbb{C}}
\newcommand{\Z}{\mathbb{Z}}

\newcommand{\mcf}{\mathcal{F}}
\newcommand{\mcg}{\mathcal{G}}
\newcommand{\mch}{\mathcal{H}}
\newcommand{\mck}{\mathcal{K}}

\newcommand{\msf}{\mathscr{F}}

\newcommand{\ww}{\mathfrak{w}}

\newcommand{\wed}{\curlywedge}
\newcommand{\bb}{\mathbb{b}}
\newcommand{\fix}{\text{Fix}}
\newcommand{\mixwed}{\curlywedge}
\newcommand{\slv}{\text{SL}(V)}
\newcommand{\rab}{R_{a, b}(V)}
\newcommand{\rsigma}{R_\sigma(V)}
\newcommand{\multideg}{\text{multideg}}

\newcommand{\zz}{\mathbf{z}}
\newcommand{\xx}{\mathbf{x}}
\newcommand{\yy}{\mathbf{y}}
\newcommand{\mca}{\mathcal{A}}
\newcommand{\str}[1]{\text{Strip}(#1)}
\newcommand{\sgn}{\text{sgn}}
\newcommand{\uu}{\mathbf{u}}
\newcommand{\qwed}[3]{#1\overset{#3}{\wedge}#2}
\newcommand{\qwedd}[3]{#1\overset{#3}{\wedge^*}#2}
\newcommand{\cv}[3]{\langle #1, #2 \rangle_{#3}}
\newcommand{\dcv}[3]{\langle #1, #2 \rangle^{#3}}
\newcommand{\rev}[1]{\overleftarrow{#1}}
\newcommand{\rarrow}{\rightrightarrows}
\newcommand{\rel}[1]{\stackrel{#1}{\longleftrightarrow}}
\newcommand{\mfm}{\mathfrak{M}}
\newcommand{\mfmo}{\mathfrak{M}^{\mathrm{o}}}
\newcommand{\dec}{\mathfrak{d}}
\newcommand{\seed}{\Sigma}

\newcommand{\spec}{\text{Spec}\,}

\newcommand{\inprod}[2]{\langle #1, #2 \rangle}

\newcommand{\K}{K_{\sigma}(V)}
\newcommand{\mcas}{\mathcal{A}_{\sigma}}
\newcommand{\glueseeds}[3]{{#1}\amalg_{#3} {#2}}
\newcommand{\lpush}[3]{#1^{\leqslant #2}\mixwed #3}
\newcommand{\rpush}[3]{#1 \mixwed #2^{\geqslant #3}}
\newcommand{\lrpush}[4]{#1^{\leqslant #2}\mixwed #3 \mixwed #1^{\geqslant #4}}
\newcommand{\mcb}[2]{\mathcal{B}_{#1}^{#2}}
\newcommand{\mcw}[2]{\mathcal{W}_{#1}^{#2}}

\newtheorem{theorem}{Theorem}[section]
\newtheorem{prop}[theorem]{Proposition}
\newtheorem{conj}[theorem]{Conjecture}
\newtheorem{cor}[theorem]{Corollary}
\newtheorem{lemma}[theorem]{Lemma}

\newtheorem*{theorem*}{Theorem}
\newtheorem*{conj*}{Conjecture}
\newtheorem*{lemma*}{Lemma}
\newtheorem*{prop*}{Proposition}
\newtheorem*{cor*}{Corollary}

\theoremstyle{definition}

\newtheorem{remk}[theorem]{Remark}
\newtheorem{example}[theorem]{Example}
\newtheorem{defn}[theorem]{Definition}

\numberwithin{equation}{section}

\numberwithin{theorem}{section}

\begin{document}

\title{Cluster Structures in Mixed Grassmannians}
\author{Zenan Fu}
\address{Department of Mathematics, University of Michigan, Ann Arbor, MI 48109, USA}
\email{zenanfu@umich.edu}
\date{\today}



\begin{abstract}
	Generalizing the results by Fomin-Pylyavskyy and Carde, we construct a family of  natural cluster structures in the coordinate ring of a mixed Grassmannian, the configuration space of tuples of several vectors and covectors in a finite-dimensional complex vector space. We describe and explore these cluster structures using the machinery of weaves introduced by Casals and Zaslow. 
\end{abstract}

\maketitle

\tableofcontents

\section{Introduction}
\subsection{History and background}

{Cluster algebras} were  introduced by Fomin and Zelevinsky \cite{FominZelevinsky1} 25 years ago. Over the years, they have been objects of intense study thanks to their deep connections with many mathematical disciplines, including the theory of integrable systems, total positivity, Lie theory, Poisson geometry, Teichm\"uller theory, and more. 

A cluster algebra is generated inside an ambient field of rational functions in several variables by a set of generators consisting of a finite set of \emph{frozen variables} and a recursively defined set of \emph{cluster variables}. The recursively generated relations between these generators are called \emph{exchange relations}.  Many combinatorial models have been developed in order to understand the structure of various classes of cluster algebras, including but not limited to:
\begin{itemize}[wide, labelwidth=!, labelindent=0pt]
	\item wiring diagrams of Berenstein, Fomin and Zelevinsky \cite{BerensteinFominZelevinsky};
	\item plabic graphs of Postnikov~\cite{Postnikov} 
	and their generalizations (mixed plabic graphs of Carde~\cite{Carde}; generalized plabic graphs of Serhiyenko, Sherman-Bennett and Williams~\cite{SerhiyenkoShermanWilliams}; 3D plabic graphs of Galashin, Lam, Sherman-Bennett and Speyer \cite{GalashinLamSBSpeyer}; and more);
	\item tensor diagrams of Fomin and Pylyavskyy~\cite{FominPylyavskyy} and their generalizations;
	\item string diagrams of Shen and Weng~\cite{ShenWeng};
	\item weaves of Casals and Zaslow~\cite{CasalsZaslow}, cf.\  also~\cite{CasalsWeng, CasalsLeSBWeng, CGGS, CGGLSS};
	\item Le diagrams of Postnikov~\cite{Postnikov}, cf.\ Galashin and Lam~\cite{PavelLam}.
\end{itemize}

The first and perhaps most important example of a cluster algebra is the homogeneous coordinate ring $\C[\widehat{\text{Gr}}_{d, n}]$ of the complex Grassmannian $\text{Gr}_{d,n}$. The standard cluster structure on the Grassmannian was first described by Scott~\cite{Scott}, using the combinatorics of plabic graphs introduced by Postnikov~\cite{Postnikov}, cf.\ also the work of Gekhtman, Shapiro and Vainshtein~\cite{GSV}. Cluster algebra structures on Grassmannians play an important role in several applications of cluster theory.

Adopting the perspective of Fomin and Pylyavskyy~\cite{FominPylyavskyy}, we view the coordinate ring of the Grassmannian in the context of classical invariant theory. 
Let $V$ be a $d$-dimensional complex vector space. The Pl\"ucker ring $\C[\widehat{\text{Gr}}_{d, n}]$ is isomorphic to the ring of $\text{SL}(V)$-invariant polynomial functions of an $n$-tuple of vectors in $V$:
\[
\C[\widehat{\text{Gr}}_{d, n}] \cong \C[V^n]^{\slv}.
\]
We thus seek a cluster algebra structure on (the coordinate ring of) the \emph{configuration space of $n$ vectors in $V$}, i.e., the GIT quotient 
\[
V^n \sslash \slv := \spec\!\big(\C[V^n]^{\slv}\big).
\]

Following~\cite{FominPylyavskyy}, we consider a more general concept of a  \emph{mixed Grassmannian}, the \emph{configuration space of tuples consisting of $a$ vectors in $V$ and $b$ covectors in  $V^*$}. The corresponding coordinate ring will be denoted by $\rab$ and called a \emph{mixed Pl\"ucker ring}. The mixed Pl\"ucker ring $\rab$ is one of the archetypal objects of classical invariant theory, see, e.g.,~\cite[Chapter 2]{Dolgachev},~\cite{Li, Olver},~\cite[\S 9]{VinbergPopov},~\cite[\S 11]{Procesi},~\cite{SturmfelsBernd, Weyl}. It was explicitly
described by Hermann Weyl \cite{Weyl} in terms of generators (henceforth referred to as \emph{Weyl generators}) and relations. 
Fomin and Pylyavskyy conjectured that every mixed Pl\"ucker ring $\rab$ carries a natural cluster structure. In fact, there are typically many such structures, depending on the choice of a cyclic ordering of the vectors and covectors. We call such an ordering a \emph{signature}. Fomin and Pylyavskyy proved this conjecture in the case when $\dim V = 3$ and the signature is non-alternating. 

The machinery used in \cite{FominPylyavskyy} involves combinatorial gadgets called  \emph{webs} (or \emph{tensor diagrams}), introduced by Kuperberg \cite{Kuperberg}. Webs can be used to define $\text{SL}_3$-invariants and in particular the cluster variables appearing in~\cite{FominPylyavskyy}. The \emph{skein relations} between the webs translate into exchange relations in the cluster algebra. 

In Carde's thesis \cite{Carde}, the results of Fomin and Pylyavskyy \cite{FominPylyavskyy} were partially generalized. Carde proved that for the \emph{separated signature} (where vectors, resp., covectors, are grouped together), the mixed Grassmannian with at least $d$ vectors and $d$ covectors carries a natural cluster structure, for any $d  = \dim V \ge 3$. The combinatorial machinery used in their construction utilized \emph{mixed plabic graphs}. Carde's construction presents a mixed Grassmannian as  an ``amalgamation" of an ``open" Grassmannian formed by the vectors and an ``open dual" Grassmannian formed by the covectors. We will use a similar idea in our construction. 

\subsection{Main results}\label{sec: intro main results}
In this manuscript, we generalize the main results of Fomin and Pylyavskyy~\cite{FominPylyavskyy} to all odd dimensions $d=\dim V$ and arbitrary ``$d$-admissible" signatures (cf.\ Definition~\ref{defn: admissible signature}).

\begin{theorem}
	Assume that $d$ is odd and $n= a+ b >d^2$. 
	Then for any choice of a d-admissible signature $\sigma$ of type $(a, b)$, the mixed Pl\"ucker ring $R_{a, b}(V)$ can be endowed with a natural (explicitly described) cluster structure that depends on $\sigma$.
\end{theorem}

The notion of a $d$-admissible signature is a natural generalization of non-alternating signatures. The $d$-admissibility condition is necessary to ensure the validity of our main construction. The condition $n>d^2$ is not essential for the construction of cluster structures in mixed Grassmannians, but is used  for the proof of the main theorem. When the signature is ``nice", the condition $n>d^2$ can be relaxed significantly. For example, if the signature is separated (and has at least $d-1$ vectors and $d-1$ covectors), then it can be relaxed to $n \ge 2d$.  

The requirement that $d$ is odd is quite subtle and deserves a bit of an explanation. Crucially, we observe that a mixed Grassmannian only has cyclic symmetry up to a sign. This is evident in the case of the Grassmannian $\text{Gr}_{2, 4}$: the set of Pl\"ucker variables $P_{ij}$ is not invariant under cyclic shifts of indices. 
 Requiring $d$ to be odd resolves the problem, since in that case, the set of Pl\"ucker variables is invariant under cyclic shifts. 

One potential solution in the case when $d$ is even relies on breaking the cyclic symmetry by toggling the signs. Similar problem arose in the study of cluster structures for \emph{positroid varieties} by Casals, Le, Sherman-Bennett and Weng \cite{CasalsLeSBWeng}; they resolved the issue by introducing  \emph{sign curves}, cf.\ \cite[Appendix A2]{CasalsLeSBWeng}. We expect a similar method to be applicable to our case. Due to time constraints, we were unable to complete the required adaptation.

Our construction of cluster structures in mixed Grassmannians relies on the combinatorial machinery of \emph{weaves} introduced by Casals and Zaslow~\cite{CasalsZaslow} and further developed in~\cite{CasalsWeng, CasalsLeSBWeng, CGGS, CGGLSS}. Particularly in~\cite{CGGLSS}, Casals, E. Gorsky, M. Gorsky, Le, Shen and Simental used a special family of weaves called \emph{Demazure weaves} to construct cluster structures on \emph{braid varieties}. We modify and adapt their construction by amalgamating a pair of Demazure weaves along a common bottom word.  Although the resulting amalgamated weave is not usually equivalent to a Demazure weave, it still possesses the properties needed to define a cluster structure. 

Each of our cluster structures on a mixed Grassmannian $\rab$ is ``natural" in the following sense:
\begin{itemize}[wide, labelwidth=!, labelindent=0pt]
	\item all cluster and frozen variables are multi-homogeneous elements of $\rab$;
	\item the set of cluster and frozen variables contains all Weyl generators;
	\item the construction recovers 
	\begin{itemize}
		\item the standard cluster structure for $\text{Gr}_{d,n}$ (cf.\ \cite{Scott}) when $b = 0$;
		\item cluster structures for mixed Grassmannians built from tensor diagrams, in the case $d = 3$, cf.\ \cite{FominPylyavskyy};
		\item cluster structure for a mixed Grassmannian built from mixed plabic graphs, in the case when the signature is separated and $a, b \ge d$, cf.\ \cite{Carde};
	\end{itemize} 
	\item the construction does not depend on arbitrary choices; in particular, it is invariant under the cyclic shifts of the signature.
\end{itemize}

\subsection*{Acknowledgment}
This is a Ph.D. thesis supervised by Professor Sergey Fomin at the University of Michigan.
We would like to thank Daping Weng, Casals Roger, Linhui Shen, Evgeny Gorsky, Ian Le,  Michael Shapiro, and João Pedro Carvalho, for helpful conversations towards the thesis.

\subsection{Organization of the manuscript}

The manuscript is structured as follows. 
Section \ref{chap: preliminary} covers the background material required for the manuscript.  
In Section~\ref{sec: admissible signature}, we introduce the notion of a signature $\sigma$, which can be viewed as an ordering of vectors and covectors. The notion of a $d$-admissible signature is a natural generalization of non-alternating signatures for $d= 3$. The condition of $d$-admissibility is required for our construction of cluster algebra structures to work. 

Starting with Section \ref{sec: slv invariants}, we assume that $\sigma$ is $d$-admissible. We fix $V$ to be a $d$-dimensional vector space. Section \ref{sec: slv invariants} introduces the main object of study, the mixed Pl\"ucker ring $\rsigma$, equipped with a choice of a signature $\sigma$. The ring $\rsigma$ is a UFD. It is generated by the Weyl generators. 

In Section \ref{sec: cluster algebras}, we review the basics of cluster algebras, confined to the needs of the manuscript. We restrict ourselves to cluster algebras associated with quivers. They are of geometric type and are defined over $\C$. The key result recalled in this section is the ``Starfish Lemma", which is later used to establish that a given ring endowed with certain algebraic and combinatorial data can be viewed as a cluster algebra. The notion of \emph{amalgamation} is introduced at the end of this section. This procedure is subsequently applied in Section~\ref{sec: cluster structure on mixed Grassmannians} to seeds obtained from two Demazure weaves.

Section \ref{sec: combinatorics of weaves} presents the main combinatorial tool used in our construction, namely the Demazure weaves. They are a particular kind of planar graphs properly embedded in a rectangle with horizontal and vertical sides and oriented from top to bottom. We review the notion of \emph{Lusztig cycles} and associate a quiver to any Demazure weave. The vertices of the quiver correspond to the {Lusztig cycles} of a Demazure weave. The number of arrows between two vertices of the quiver is defined by the \emph{intersection pairing} between the corresponding Lusztig cycles. The cycles ending at the bottom boundary are called \emph{frozen cycles}; the corresponding vertices in the quiver are declared to be frozen.  

We introduce the notion of \emph{marked boundary vertices} to allow extra frozen cycles to originate from the top boundary (typically they will not end at the bottom boundary). These will become the frozen vertices for our amalgamated quiver,  as the cycles ending at the bottom boundary are defrosted after the amalgamation process.  

We review the notions of equivalence and mutation equivalence between two Demazure weaves. The quivers associated with Demazure weaves are well-behaved under equivalence and mutation equivalence. The \emph{classification theorem of Demazure weaves}, due to Roger Casals, Eugene Gorsky, Mikhail Gorsky and José Simental \cite{CGGS}, states that any two reduced Demazure weaves with the same top and bottom words are mutation equivalent to each other. This result is of critical importance in our construction as it provides a great amount of freedom to our choices. 

In Section \ref{sec: mixed exterior algebra}, we first review the basics of \emph{exterior algebras} over $V$ and $V^*$ and discuss the duality map $\psi$ between these two algebras. There are two operators defined on the exterior algebra over $V$: the wedge product $\wedge$ and the intersection product~$\cap$. We ``unify" these two operators by defining a new operator $\mixwed$ called the \emph{mixed wedge}. It is either a wedge or an intersection, depending on the extensors being operated on. The mixed wedge operator significantly simplifies our notation for cluster variables and the computations for  exchange relations. Using the duality between the exterior algebras over $V$ and $V^*$, we extend the operator $\mixwed$ to the \emph{mixed exterior algebra}, where we identify the exterior algebras over $V$ and  $V^*$ under $\psi$. 

In Section \ref{chap: cluster structures on mixed grassmannians}, we construct the cluster algebra $\mcas$ associated with the given signature $\sigma$ and state that it is equal to the mixed Pl\"ucker ring (Theorem~\ref{thm: mixed grassmannian is a cluster algebra}). 
We start with a description of suitable ``coordinates" for the cluster variables in $\mcas$. In Section~\ref{sec: decorated flags}, we review the definition of a \emph{decorated flag}, viewed as a tuple of nested extensors. By viewing the vectors and covectors as cyclically ordered according to the chosen signature, we construct a tuple $\vec{\mcf}(\uu)$ of $n$ cyclically ordered decorated flags. Each of these decorated flags is constructed by taking consecutive mixed wedges of vectors and covectors appearing along the circle. Such a tuple satisfies very nice properties described in Proposition \ref{prop: recursive relation between flags associated with a signature}. 

In Section \ref{sec: decorated flag moduli space}, we introduce the notion of \emph{cyclic decorated flag moduli spaces}. We warn of a minor difference between our notion and the corresponding notion in \cite{CasalsLeSBWeng}. One of these spaces is of particular interest, namely $\mfm(\beta_\sigma)$, the space associated with the signature $\sigma$. It turns out that an element in $\mfm(\beta_\sigma)$ can always be represented by a tuple $\vec{\mcf}(\uu)$ of cyclically ordered decorated flags constructed in Section \ref{sec: decorated flags}. As a consequence, the space $\mfm(\beta_\sigma)$ is birationally equivalent to the configuration space of vectors and covectors parameterized by $\sigma$. Any function on $\mfm(\beta_\sigma)$ can be viewed as a function of the coordinates of these vectors and covectors.

Having made all these preparations, we are finally able to define in Section \ref{sec: cluster alg associated with weaves} a cluster algebra associated with a Demazure weave. To define the cluster variables, we start with a \emph{normalized decoration} (with coordinates described in the previous section) for the top word $\beta$ of a Demazure weave. We extend the decoration as we scan from top to bottom following Theorem \ref{thm: unique extension of decorations for a weave}. This process produces the cluster variables. There are two things we would like to point out. First, the word $\beta$ has to be a contiguous subword of $\beta_\sigma$ (the word associated with the signature $\sigma$), so that we can use the coordinates introduced in the previous section. Second, we do not prove at this stage that the cluster variables defined in this way form a seed (i.e., they are algebraically independent) for certain choices of subword $\beta$ of $\beta_\sigma$. This claim is established much later in Section~\ref{section: weyl generators}.

In Section \ref{sec: cluster structure on mixed Grassmannians}, we construct the cluster algebra $\mcas$ and formulate the main theorem (Theorem~\ref{thm: mixed grassmannian is a cluster algebra}). To construct $\mcas$, we start with the cyclic word $\beta_\sigma$, the word associated with the given signature $\sigma$. We  cut $\beta_\sigma$ at two locations $(p, q)$ into two contiguous (cyclic) subwords $\beta_1$ and $\beta_2$. For each of these two words, we construct a reduced Demazure weave $\ww_1$ (resp., $\ww_2$) with $\beta_1$ (resp., $\beta_2$) as the top word, so that the bottom words are reverse to each other. We then stitch these two Demazure weaves along their bottom word (rotating $\ww_2$ by $180^\circ$) into an ``amalgamated" weave~$\ww$, which is typically not a Demazure weave. The seed associated with $\ww$ is defined to be the amalgamation of the seeds associated with $\ww_1$ and $\ww_2$ along their common frozen variables. The cluster algebra $\mcas(p, q)$ associated with $\ww$ is defined to be the cluster algebra arising from the amalgamated seed. 

There are certain requirements to be satisfied for such an amalgamation to work properly. The decoration for the bottom of $\ww_1$ needs to be reverse to the decoration for the bottom of $\ww_2$; and the Lusztig cycles ending at the bottom of $\ww_1$ need to be able to concatenate properly with the Lusztig cycles ending at the bottom of $\ww_2$. This is where we use the condition that the dimension $d$ is odd (because of the sign issue). This condition is always assumed in what follows. Both requirements are satisfied if the cyclic distance between the two cutting points $p, q$ is at least $d+1$. Such a cut will be referred to as a \emph{valid cut}.

We claim that the cluster algebra $\mcas(p, q)$ does not depend on the choice of the cut $(p, q)$ as long as that cut is valid (Proposition \ref{prop: cluster algebra does not depend on the cut}). Hence we can refer to it as the cluster algebra $\mcas$ associated with $\sigma$, and state the main theorem (Theorem~\ref{thm: mixed grassmannian is a cluster algebra}): the cluster algebra $\mcas$ coincides with the mixed Pl\"ucker ring $R_\sigma$. 

In Section \ref{chap: proof of the main theorem}, we prove the main theorem. 
In Section \ref{sec: ca does not depend of the choice of the cut}, we establish the claim stated in the previous section that the cluster algebra $\mcas(p, q)$ does not depend on the cut $(p, q)$. The proof relies on two facts. The first fact is that the decoration and the cycles ending at the bottom of $\ww_1$ and $\ww_2$ are both very nice. The second fact is that we only need to prove the claim for two cuts of the form $(p, q)$ and $(p, q+1)$. With these two facts in mind, we study the portion of the weave $\ww$ near the two amalgamation bottom words, denoted by $\ww(p, q, q+1)$. The boundary word for this weave is $w_0 \rho w_0$ where $w_0$ is the longest word in $S_d$ and $\rho$ is the word $12\cdots (d-1)$ or $(d-1)(d-2)\cdots 1$. The key observation is that the weave $\ww(p, q, q+1)$ can be realized as a Demazure weave in two different ways, $w_0\rho \rarrow w_0$ or $\rho w_0 \rarrow w_0$, and the seeds for these two realizations are the same. As a consequence, the stitching process on the weaves yields the amalgamations of the corresponding seeds, cf.\ Figure~\ref{fig: Amalgamating three decorated Demazure weaves in two different ways, result in the same seeds}.

The fact that $\mcas$ does not depend on the cut together with the classification theorem (Theorem \ref{thm: demazure classification}) allow us to obtain a large family of seeds for $\mcas$: we can choose different valid cuts, and for each cut, we can choose different Demazure weaves for both $\ww_1$ and $\ww_2$. This is critical for the proof that all Weyl generators appear in~$\mcas$ as cluster or frozen variables. 

In Section \ref{sec: initial weave}, we study a special weave $\ww$ called the \emph{initial weave} that is particularly well-behaved. It is constructed as a concatenation of \emph{strips}; each strip is constructed as a concatenation of \emph{patches}. We then obtain a full description of the Lusztig cycles in $\ww$ (Proposition~\ref{prop: descriptions of cycles of the initial weave}). All of these cycles are unweighted subgraphs of $\ww$; moreover,  all of them are trees except for the cycles originating from the first patch of a certain type of strip; these exceptional cycles only have one merging point. By counting the number of cycles in the initial weave, we conclude that the dimension of $\mcas$ matches the dimension of the mixed Grassmannian.  

We then obtain a full list of explicit and relatively simple formulas for  the cluster and frozen variables in the \emph{initial seed}, the seed associated  to the initial weave (Proposition~\ref{prop: decoration flags of the initial weave}). These formulas imply that all these variables belong to the mixed Pl\"ucker ring $R_\sigma$. 

In Section \ref{sec: quivers and mutations}, we study the quiver associated with the initial weave. We compute the once mutated cluster variables, and obtain explicit formulas for all exchange relations for the initial seed. The ``generic" exchange relations are direct generalizations of the classical exchange relations involving Pl\"ucker coordinates in the Grassmannian for degree $4$ vertices (\emph{square move}) and degree $6$ vertices.

In Section \ref{section: weyl generators}, we show, under the assumption that $n>d^2$,  that all Weyl generators for $R_\sigma$ are cluster or frozen variables in $\mcas$. As mentioned earlier, this assumption can be relaxed when $\sigma$ is ``nice". For example, when $\sigma$ is separated with $a, b \ge d-1$, then the assumption can be relaxed to $n \ge 2d$. 

This result has two consequences. First, it implies that the seeds that we defined earlier are indeed valid seeds, i.e., the elements of an extended cluster are algebraically independent, as promised in Section \ref{sec: cluster structure on mixed Grassmannians}. Second, it implies that $R_\sigma$ is a subalgebra of $\mcas$, because Weyl generators generate $R_\sigma$. 

Finally, in Section \ref{sec: proof of the main theorem}, we prove the main theorem. We present two different proofs, both relying on certain results in~\cite{GLS}. The first proof utilizes Theorem~1.4 in~\cite{GLS}, which states that if one finds two disjoint seeds (no common cluster variables) in a cluster algebra such that all of their cluster and frozen variables lie in a factorial subalgebra of the cluster algebra, then the cluster algebra is equal to its factorial subalgebra. As we have proved that $R_\sigma$ (which is factorial) is a subalgebra of $\mcas$, we only need to find two disjoint seeds contained in $R_\sigma$. This is not particularly difficult as we have a wide variety of seeds to choose from. 

In the second proof of the main theorem, we first use Theorem 1.4 in \cite{GLS} to show that all cluster and frozen variables in $\mcas$ that lie in $R_\sigma$ are irreducible in $R_\sigma$. We then use the ``Starfish lemma" to conclude that $R_\sigma = \mcas$, by showing that the cluster and frozen variables for the initial seed and its neighbors all lie in $R_\sigma$. These once mutated cluster variables were calculated in Section \ref{sec: quivers and mutations}. 

In Section \ref{chap: properties and generalizations}, we explore the properties of $\mcas$ and discuss a generalization of our results. In Section \ref{sec: separated signatures}, we construct the cluster algebra $\mcas$ under the condition that $n \ge 2d$ and $\sigma$ is separated, and prove that $R_\sigma = \mcas$. Furthermore, we show that our cluster structures for separated signatures coincide with the one given by Carde \cite{Carde} using mixed plabic graphs. 
In Section \ref{sec: properties and further results}, we show that the cluster structure that we constructed on the mixed Grassmannian is natural, in the sense described earlier in the introduction. 
In Section \ref{sec: generalizations and conjectures}, by viewing vectors and covectors as $1$-extensors and $(d-1)$-extensors respectively, we consider the ring of $\slv$-invariant polynomial functions of several extensors of arbitrary (prescribed) levels.
It is well-known that such a ring is finitely generated and factorial (cf.\ \cite[Theorems 3.5 and 3.17]{VinbergPopov}). We conjecture that the invariant ring $R_\sigma$ carries a natural cluster algebra structure that can be described using a suitable adaptation of the construction of the cluster algebra $\mcas$ described in Section \ref{sec: cluster structure on mixed Grassmannians}.

Section \ref{appendix: signature} contains independent (with respect to the main theorem) results about the combinatorics of signatures.
In Section~\ref{sec: signatures and affine permutations}, we explore the relation between signatures and affine permutations. The set of signatures can be realized as a subset of biased affine permutations of bias $\le d$. Furthermore, the admissible signatures precisely correspond to the biased affine permutations with bias equal to $d$. 
In Section~\ref{subsec: adundant signatures}, we study the counterparts of bounded affine permutations under the correspondence described in Section~\ref{sec: signatures and affine permutations}.

\newpage

\section{Setting the Stage: Context and Preliminaries}\label{chap: preliminary}

\subsection{Admissible signatures}\label{sec: admissible signature}

In this section, we introduce the notion of a signature $\sigma$, which can be viewed as an ordering of vectors and covectors.
Let $a, b, n, d\in \Z$ be non-negative integers such that $n = a+b \ge 3$ and $d \ge 3$. We denote by $[i, j]$ the set of all integers $k\in \Z$ satisfying $i\le k \le j$. 

\begin{defn}\label{defn: signature}
	A  \emph{(size $n$) signature of type $(a, b)$} is a map:
	\begin{equation}
		\sigma: \Z \rightarrow \{-1, 1\}
	\end{equation}
	such that 
	\begin{enumerate}[label = {\textbf{(\alph*)}}, wide, labelwidth=!, labelindent=0pt]
		\item $\sigma(j + n ) = \sigma(j)$ for $j\in \Z$ (periodicity),
		\item $\#\{\sigma^{-1}(1)\cap [1, n]\} = a$ and $\#\{\sigma^{-1}(-1)\cap [1, n]\} = b$.
	\end{enumerate}
	For an index $j\in \Z$, we say $j$ (or $j \bmod d$) is a \emph{black vertex} (resp., \emph{white vertex}) if $\sigma(j) = 1$ (resp., $\sigma(j) = -1$).
\end{defn}

\begin{remk}\label{remk: interpretation of a signature}
	A signature $\sigma$ of type $(a, b)$ equivalently defines a cyclic word of length $n$ over $\{\bullet, \circ\}$ with $a$ black vertices $\bullet$ and $b$ white vertices $\circ$. We denote this as $\sigma = [\sigma_1\, \sigma_2\, \cdots\, \sigma_n]$, where $\sigma(j) = \sigma_j$ for $j \in [1, n]$.
\end{remk}

\begin{example}
	Figure~\ref{fig: example of a signature} shows the signature $\sigma = [\bullet\, \bullet\, \circ\, \bullet\, \circ\, \bullet]$ of type $(4, 2)$, where $\sigma(1) = \sigma(2) = \sigma(4) = \sigma(6) = 1$ and $\sigma(3) = \sigma(5) = -1$.
\end{example}

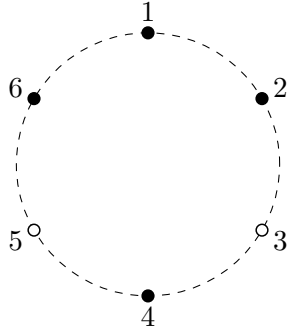
\begin{figure}[htp]\centering
	\begin{tikzpicture}[scale=.35]
		\draw [black, dashed] (0,0) circle (5cm);
		\foreach \angle/\label/\color in {90/1/black, 30/2/black, -30/3/white, -90/4/black, -150/5/white, 150/6/black} {
			\draw[line width=.2mm, fill=\color] (\angle:5) circle (6.3pt);
			\node at (\angle:5.8) {$\label$};
		}
	\end{tikzpicture}
	\caption{Example of a (cyclic) signature with six vertices.}
	\label{fig: example of a signature}
\end{figure}

\begin{defn}
	Let $\sigma$ be a signature. For $j\in \Z$ and $k\in [0, d-1]$, let $j'\ge j$ be the smallest integer (if exists) such that 
	\[
	\sum_{i = j}^{j'} \sigma(i) \equiv k \bmod d.
	\]
	We define $j' := \infty$ if no such $j'$ exists. Let $\ell(\sigma, j, k, d):= j' - j+1$. In other words, $\ell(\sigma, j, k, d)$ is the minimal length of a contiguous subsequence starting at $j$ whose partial sums of $\sigma$ modulo $d$ equal $k$ (and $\ell(\sigma, j, k, d) = \infty$ if no such contiguous subsequence exists).
\end{defn}
\begin{defn}\label{defn: admissible signature}
	A signature $\sigma$ is \emph{$d$-admissible} if $\ell(\sigma, j, k, d) < \infty$ for all $j\in \Z$ and all $k\in [0, d-1]$.
\end{defn}

\begin{lemma}\label{lemma: admissible equivalent def}
	A signature $\sigma$ is $d$-admissible if and only if $\ell(\sigma, j_0,k, d) < \infty$ for some $j_0\in \Z$ and all $k\in [0, d-1]$.
\end{lemma}
\begin{proof}
	Assume that $\ell(\sigma, j_0, k, d) < \infty$ for some $j_0$ and all $k$. For arbitrary $j \in \Z$ and $k' \in [0, d-1]$, choose $c \in \Z_{>0}$ such that $j_0 + cn \geq j$. Let 
	\[
	k'' = \sum_{i=j}^{j_0 + cn} \sigma(i).
	\]
	Since $\ell(\sigma, j_0 + cn, k' - k'', d) < \infty$, there exists $j' \geq j_0 + cn$ with 
	\[
	\sum_{i=j_0 + cn}^{j'} \sigma(i) \equiv k' - k'' \bmod d.
	\]
	Thus, 
	\[
	\sum_{i=j}^{j'} \sigma(i) \equiv k' \bmod d,
	\] 
	proving $\ell(\sigma, j, k', d) < \infty$.
\end{proof}

\begin{example}
	Let $n = 6, d = 3$, $\sigma_1 = [\bullet\, \bullet\, \circ\, \bullet\, \circ\ \bullet]$ and $\sigma_2 = [\bullet\, \circ\, \bullet\, \circ\, \bullet\ \circ]$. 
	For $\sigma_1$,  we have $\big(\ell(\sigma_1, 1, k, 3)\big)_{k = 0, 1, 2} = (7, 1, 2)$. Hence $\sigma_1$ is $3$-admissible by Lemma~\ref{lemma: admissible equivalent def}.
	For $\sigma_2$, we have $\big(\ell(\sigma_2, 1, k, 3)\big)_{k = 0, 1, 2} = (2, 1, \infty)$. Hence $\sigma_2$ is not $3$-admissible. 
\end{example}

\begin{defn}
	A signature is called \emph{alternating} if $\sigma(j) = -\sigma(j+1)$ for all $j\in \Z$. Note that an alternating signature is always of type $(a, a)$ for some $a\in \Z_{> 0}$.
\end{defn}

\begin{example}
	An alternating signature is not $d$-admissible for any $d \ge 3$. 
\end{example}

\begin{example}
	A signature $\sigma$ is $3$-admissible if and only if $\sigma$ is non-alternating. 
\end{example}

\begin{example}\label{example: d-1 consecutive implies admissible}
	Let $\sigma$ be a signature with at least $d-1$ consecutive vertices of the same color. Then $\sigma$ is $d$-admissible. 
\end{example}

\begin{lemma}\label{admissibility for signature of type a, b}
	Let $\sigma$ be a signature of type $(a, b)$ with $a \neq b$. Then $\sigma$ is $d$-admissible for any $d$.
\end{lemma}
\begin{proof}
	Without loss of generality, we assume that $a > b$. Let $j\in \Z$. Consider the sequence
	\[
	m_t = \sum_{i = j}^t \sigma(i), \text{ for }t \ge j.
	\]
	We have $m_j = \pm 1$, $|m_{t+1} - m_t| = 1$ and $m_{t+n} - m_t = a - b > 0$ for any $t\ge j$. This implies that $\{m_t\}_{t \ge j} \supseteq\Z_{>0}$. In particular, for any $k\in [0, d-1]$, there exists $j'\ge j$ such that $m_{j'} \equiv k \mod d$. Therefore $\sigma$ is $d$-admissible for any $d$.
\end{proof}

\begin{defn}\label{defn: separated signature}
	A signature of type $(a, b)$ is \emph{separated} if it contains $a$ consecutive black vertices followed by $b$ consecutive white vertices. 
\end{defn}

\begin{lemma}\label{lemma: separate signature is admissible}
	A separated signature with $n \ge 2d-2$ is $d$-admissible. 
\end{lemma}
\begin{proof}
	If $a \neq b$, then the result follows from Lemma \ref{admissibility for signature of type a, b}. If $a = b$, then we have $a \ge d-1$. The claim then follows from Example \ref{example: d-1 consecutive implies admissible}. 
\end{proof}

\begin{remk}\label{remk: ell(sigma, j, 0) < infty}
	Similar to the proof of Lemma~\ref{admissibility for signature of type a, b}, we can show that for any signature $\sigma$, we always have $\ell(\sigma, j, 0, d) < \infty$ for any $j\in \Z$, regardless of whether $\sigma$ is $d$-admissible or not. 
\end{remk}

\begin{defn}
	Let $\sigma$ be a signature. The \emph{$d$-length} of $\sigma$ is defined to be 
	\[
	\ell_d(\sigma):= \frac{1}{n} \sum_{j = 1}^n \ell(\sigma, j, 0, d).
	\]
	The number $\ell_d(\sigma)$ is always an integer.
\end{defn}

We have the following criterion for $d$-admissibility of a signature. 

\begin{theorem}\label{thm: admissible signature criterion}
	For any signature $\sigma$, we have $\ell_d(\sigma) \le d$. Moreover, $\sigma$ is $d$-admissible if and only if $\ell_d(\sigma) = d$. 
\end{theorem}
The proof of Theorem~\ref{thm: admissible signature criterion} can be found in Section~\ref{appendix: signature}, and the main theorem of the manuscript does not depend on this result. 

From now on, unless stated otherwise, we will only deal with $d$-admissible signatures. For properties involving general signatures, for example, the relationship between signatures and affine permutations, we refer the reader to Section~\ref{appendix: signature}.

\subsection{Mixed Pl\"ucker ring}\label{sec: slv invariants}

In this section, we introduce the main object of study, the \emph{mixed Pl\"ucker ring} $\rsigma$, equipped with a choice of signature $\sigma$, following \cite[Section 2]{FominPylyavskyy}. Let $V$ be a $d$-dimensional vector space. Let $a, b$ be non-negative integers with $n = a+b$. 

\begin{defn}\label{defn: slv action and mixed grassmannian}
The special linear group $\slv$ acts on $V$ naturally by left multiplication, and acts naturally on $V^*$ by $(gu^*)(v) = u^*(g^{-1}(v))$, for $u^*\in V^*, v\in V$ and $g\in \slv$. As a consequence, $\slv$ naturally acts on the vector space
\begin{equation*}
	V^a\times (V^*)^b = \underbrace{V\times \dots \times V}_\text{$a$ copies} \times 
	\underbrace{V^*\times \cdots \times V^*}_\text{$b$ copies}\,, 
\end{equation*}
and therefore on its coordinate ring $\C[V^a\times (V^*)^b]$.

Our goal is to define cluster structures on the ring of $\slv$-invariant polynomial functions on $V^a\times (V^*)^b$, namely on
\begin{equation*}
	\rab := \C[V^a\times (V^*)^b]^{\slv},
\end{equation*}
which we refer to as the \emph{mixed Pl\"ucker ring} (or the coordinate ring of the affine \emph{mixed Grassmannian}).
\end{defn}

The ring $\rab$ is a foundational object in classical invariant theory; see,
e.g., \cite[Chapter 2]{Dolgachev}, \cite{Li, Olver}, \cite[\S 9]{VinbergPopov}, \cite[\S 11]{Procesi}, \cite{SturmfelsBernd, Weyl}. It was explicitly
described by Hermann Weyl \cite{Weyl} in terms of generators and relations. For our purposes, only the generators are essential.

\begin{theorem}[The First Fundamental Theorem of Invariant Theory]\label{thm: the first fundamental theorem of invariant theory}
Let $u_1, \dots, u_a$\newline$\in V$ be $a$ vectors with indeterminate coordinates and $v_1^*, \dots, v_b^*\in V^*$ be $b$ covectors with indeterminate coordinates. Then the ring $\rab$ is generated by the following multilinear polynomials:
\begin{itemize}[wide, labelwidth=!, labelindent=0pt]
	\item the $\binom{b}{d}$ ``dual" Pl\"ucker coordinates $P_I^* = \det(v^*_{i_1}, \dots, v^*_{i_d})$,
	\item the $\binom{a}{d}$  Pl\"ucker coordinates $P_J = \det(u_{j_1}, \dots, u_{j_d})$, and 
	\item the $ab$ pairings $Q_{ij} = \langle u_i, v_j^* \rangle$.
\end{itemize}
\end{theorem}

\begin{defn}\label{defn: Weyl generators}
	The generators in Theorem~\ref{thm: the first fundamental theorem of invariant theory} are called the \emph{Weyl generators} for $\rab$. 
\end{defn}
Notice that when $b = 0$ (resp., $a = 0$), the ring $R_{a, 0}(V)$ (resp., $R_{0, b}(V)$) recovers the homogeneous coordinate ring of the Grassmannian $\text{Gr}_{d, n}$ under its Pl\"ucker embedding~\cite[Corollary 2.3]{Dolgachev}.

In our discussions of the rings $\rab$, it will be important to distinguish between
their incarnations that involve different orderings of the contravariant and covariant
arguments. To this end, we utilize signatures, cf.\ Definition~\ref{defn: signature} and Remark~\ref{remk: interpretation of a signature}.

\begin{defn}\label{defn: V sigma and rsigma}
Let $\sigma$ be a signature of type $(a, b)$. The \emph{configuration space of $a$ vectors and $b$ covectors with respect to $\sigma$}, denoted by $V^\sigma$, is defined to be the rearrangement of the direct product  $V^a\times(V^*)^b$ where the $j$-th factor is $V$ (resp., $V^*$) if $\sigma(j) = 1$ (resp., $\sigma(j) = -1$), for $1\le j \le n$. The associated \emph{mixed Pl\"ucker ring} is the ring of $\slv$-invariant polynomials on $V^\sigma$:
\[\rsigma = \C[V^\sigma]^{\slv}.\]
\end{defn}

\begin{example}
	For $\sigma = [\circ\, \bullet\, \bullet\, \circ\, \circ]$ of type $(2, 3)$, we have
	\[V^\sigma = V^*\times V\times V\times V^*\times V^*,\] 
	and $\rsigma = \C[V^*\times V\times V\times V^*\times V^*]^{\slv}$ is the ring of $\slv$-invariant polynomial functions $f$ of the form
	\[
	f: V^*\times V\times V\times V^*\times V^* \rightarrow \C.
	\]
\end{example}
Note that $\rsigma \cong \rab$ as rings. The reason that we care about signatures is that we will show that for any choice of a $d$-admissible signature $\sigma$, the ring $\rsigma$ carries a natural cluster structure that depends on the choice of $\sigma$.

\begin{defn}\label{defn: multidegree}
There is a natural action of the $n$-dimensional torus on $\rsigma$, which defines a multi-grading on the ring $\rsigma$. If an invariant $f\in \rsigma$ is multi-homogeneous of degrees $d_1, \cdots, d_{n}$ in its $n$ arguments, then the \emph{multidegree of $f$} is defined to be the tuple
\begin{equation*}
	\multideg(f):= (d_1, \dots, d_{n}).
\end{equation*}
\end{defn}

Being a subring of a polynomial ring, $\rsigma$ is a domain. By the First Fundamental Theorem of invariant theory, we know $\rsigma$ is finitely generated. Moreover, it is a UFD:
\begin{lemma}[{\cite[Theorem 3.17]{VinbergPopov}}] \label{lemma: R(V) is a UFD}
	The invariant ring $\rsigma$ is a finitely generated unique factorization domain. 
\end{lemma}

\subsection{Cluster algebras}\label{sec: cluster algebras}

In this section, we briefly review the basics of cluster algebras relevant to this manuscript. For further details, see \cite{FominZelevinsky1, FominZelevinsky2, FominZelevinsky4} and \cite{FominWilliams1-3, FominWilliams4-5, FominWilliams6}.  

The goal here is to review the definition of cluster algebras and the ``Starfish Lemma", cf.\ Proposition~\ref{prop:cluster-criterion}, which will be used to establish that a given ring endowed with certain algebraic and combinatorial data can be viewed as a cluster algebra. The notion of \emph{amalgamation} (cf.\ Definition~\ref{defn: amalgamating two seeds}) is introduced at the end of the section, following~\cite{ScgraderShapiro}. This procedure is subsequently applied in Section~\ref{sec: cluster structure on mixed Grassmannians} to seeds obtained from two Demazure weaves.

In this manuscript, we restrict ourselves to cluster algebras associated with quivers. They are of geometric type and are defined over $\C$.

\begin{defn}
A \emph{quiver}~$Q$ is a finite oriented graph with no loops and no oriented 2-cycles. Certain vertices of $Q$ are marked as \emph{mutable}; the remaining ones are called \emph{frozen}. Contrary to the usual convention, we will keep the arrows between the frozen vertices. We will use $\Circle$ to represent a mutable vertex and $\Box$ to represent a frozen vertex. An example of a quiver is given in Figure~\ref{fig: example of a quiver}. 

\begin{figure}[H]
	\centering
	\begin{tikzcd}
		\Box& \Circle \arrow[d] \arrow[l]  & \Box \arrow[l] \\
		\Circle \arrow[u] \arrow[ur]  & \Circle \arrow[l]
	\end{tikzcd}
	\caption{Example of a quiver.}
	\label{fig: example of a quiver}
\end{figure}
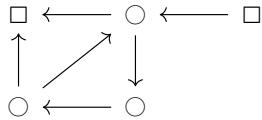
\end{defn}

\begin{defn}
	Let $z$ be a mutable vertex in a quiver $Q$. The \emph{quiver mutation} is a transformation that turns $Q$ into new quiver $Q' = \mu_z(Q)$ defined as follows: for each pair of directed edges $x \to z \to y$, add a new edge $x \to y$; then reverse all the edges incident to $z$; finally, remove all oriented 2-cycles until we are not able to do so. An example of quiver mutation is given in Figure~\ref{fig: quiver mutation example}. We remind the reader that we keep the arrows between frozen vertices. 
	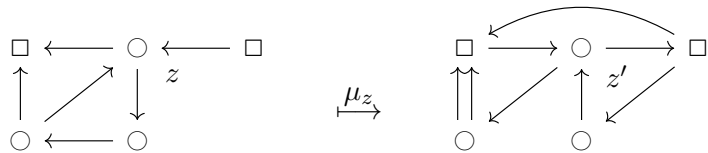
\begin{figure}[H]
		\centering
		\begin{tikzcd}
			\Box & \Circle \arrow[d] \arrow[l]  \arrow[dr,phantom, "z"' very near start] & \Box   \arrow[l]   &\phantom{}\arrow[d, phantom,  "\stackrel{\displaystyle\mu_z}{\longmapsto}"]&  \Box \arrow[r] & \Circle \arrow[dr, phantom, "{z'}" very near start] \arrow[dl] \arrow[r] & \Box \arrow[bend right=30]{ll} \arrow[dl] \\
			\Circle \arrow[u] \arrow[ur]  & \Circle \arrow[l] &  \phantom{} &\phantom{}&  \Circle \arrow[u, shift left] \arrow[u, shift right] & \Circle \arrow[u] & \phantom{}
		\end{tikzcd}
		\caption{Example of a quiver mutation at $z$.}
		\label{fig: quiver mutation example}
	\end{figure}
	
	Two quivers $Q_1$ and $Q_2$ are \emph{mutation equivalent}, denoted $Q_1 \sim Q_2$,  if there is a sequence of mutations that transform $Q_1$ into $Q_2$. 
	Notice that mutation is an \emph{involution}: mutate $Q$ twice at the same vertex to recover the quiver $Q$. 
\end{defn}

\begin{defn}
	Let $\msf$ be a field containing $\C$, called the \emph{ambient field}. (In typical examples, $\msf$ is a field of rational functions in several variables over $\C$.) A \emph{seed} $\seed$ in $\msf$ is a pair $(Q, \zz)$ where $Q$ is a quiver, and $\zz$ is a set of algebraically independent elements (over $\C$) in $\msf$, one for each vertex of $Q$. The elements in $\zz$ corresponding to mutable vertices are called \emph{cluster variables}. And they form a \emph{cluster}; the ones corresponding to frozen vertices are called \emph{frozen variables} or \emph{coefficient variables}; we call $\zz$ an \emph{extended cluster}. 
	
	A \emph{seed mutation} at a cluster variable $z$ transforms the seed $\seed = (Q, \zz)$ into $\seed' = (Q', \zz') = \mu_z(Q, \zz)$, where $Q' = \mu_z(Q)$, $\zz' = \zz \cup \{z'\} \setminus \{z\}$, and the new cluster variable $z'$ is defined by the \emph{exchange relation}:
	\begin{equation}\label{equ: exchange relation}
		zz' = \prod_{y \to z} y + \prod_{z \to y} y.
	\end{equation}
	A seed mutation is usually denoted by $\seed \stackrel{\mu_z}{\longmapsto} \seed'$. 
	An example of a seed mutation is shown in Figure~\ref{fig: seed mutation example}. 
	\begin{figure}[H]
		\centering
		\begin{tikzcd}
			\boxed{u} & z \arrow[d] \arrow[l] & \boxed{v}   \arrow[l]   &\phantom{}\arrow[d, phantom,  "\stackrel{\displaystyle\mu_z}{\longmapsto}"]&  \boxed{u} \arrow[r] & z'  \arrow[dl] \arrow[r] & \boxed{v} \arrow[bend right = 30]{ll} \arrow[dl] \\
			x \arrow[u] \arrow[ur]  & y \arrow[l] &  \phantom{} &\phantom{}&  x\arrow[u, shift left] \arrow[u, shift right] & y \arrow[u] & \phantom{}
		\end{tikzcd}
		\caption{A seed mutation at $z$, with the exchange relation $zz' = xv + uy$.}
		\label{fig: seed mutation example}
	\end{figure}
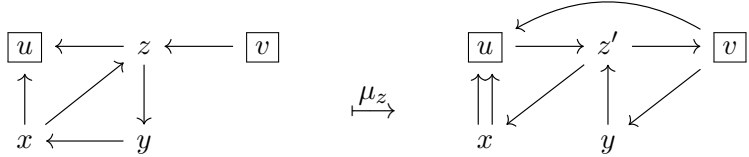
	We are not allowed to mutate at a frozen variable. Two seeds $\seed_1 = (Q_1, \zz_1)$ and $\seed_2 = (Q_2, \zz_2)$ are \emph{mutation equivalent}, denoted by $\seed_1 \sim \seed_2$,  if there is a sequence of mutations that transform $\seed_1$ into $\seed_2$. 
	Note that mutation is an \emph{involution}: mutating $\seed$ twice at the same vertex will take the seed back to itself (mutate at $z$ and then mutate at $z'$). Also notice that the frozen variables do not change under mutations. 
\end{defn}

\begin{defn}
	The \emph{cluster algebra} associated with a seed $\seed = (Q, \zz)$, denoted $\mca(\seed) = \mca(Q, \zz)$, is defined to be the (commutative) subring of $\msf$ generated over $\C$ by all cluster and frozen variables in seeds mutation equivalent to $\seed$. The \emph{(Krull) dimension} of a cluster algebra is the size of the extended cluster and the \emph{rank} is the size of the cluster. 
\end{defn}

\begin{remk}
	The frozen variables are not invertible in our definition of cluster algebras. 
\end{remk}

According to the definition above, one can construct a cluster algebra by starting with an arbitrary seed $\seed$, called the \emph{initial seed}, and repeatedly applying seed mutations in all possible directions. Then the cluster algebra is the subring generated by all the cluster variables and frozen variables in these seeds. 

The same commutative ring could carry different (non-isomorphic) cluster structures, see, e.g., \cite[Example 12.10]{FominZelevinsky2}. The same phenomenon will appear later in this manuscript. 

One important property of cluster algebras is the well-known \emph{Laurent Phenomenon}:

\begin{theorem}[\cite{FominZelevinsky1, FominZelevinsky2}]
	\label{thm: laurent phenomenon}
	Every element of $\mca(Q,\zz)$ is a Laurent polynomial in $\zz$. Furthermore, no frozen variable appears in any denominator of this Laurent expression (reduced to lowest terms). 
\end{theorem}

Let $R$ be an integral domain. We are interested in identifying $R$ as a cluster algebra. Let $K(R)$ be the fraction field of $R$; it will serve as the ambient field. The challenge is to find a seed $\seed = (Q, \zz)$ such that $R = \mca(Q, \zz)$. The next Proposition presents a set of sufficient conditions that ensures that a particular seed will make $R$ into a cluster algebra. 

Recall that an integral domain $R$ is \emph{normal} if it is integrally closed in~$K(R)$. Two elements of $R$ are \emph{coprime} if they are not contained in the same prime ideal of height~$1$.

\begin{prop}[``Starfish lemma", {\cite[Proposition 3.6]{FominPylyavskyy}}] 
	\label{prop:cluster-criterion}
	Let $R$ be a finitely generated $\C$-algebra and a normal domain.
	Let~$\seed=(Q,\zz)$ be a seed in $K(R)$ satisfying the following
	conditions:
	\begin{itemize}[wide, labelwidth=!, labelindent=0pt]
		\item
		all elements of $\zz$ belong to~$R$; 
		\item
		the cluster variables in $\zz$ are pairwise coprime (in~$R$); 
		\item
		for each cluster variable $z\in\zz$, the seed mutation~$\mu_z$ replaces $z$ with an element~$z'$ that lies in $R$ and is coprime to~$z$. 
	\end{itemize}
	Then $\mca(Q,\zz)\subseteq R$.
	If, in addition, $R$ has a set of generators each of which appears in the seeds mutation equivalent to $(Q,\zz)$, then $R=\mca(Q,\zz)$. 
\end{prop}

\begin{defn}[{Amalgamation of two seeds, cf.\ {\cite[\S2]{ScgraderShapiro}}}] \label{defn: amalgamating two seeds}
	Let $\seed_1 = (Q_1, \zz_1)$ and $\seed_2 = (Q_2, \zz_2)$ be seeds (potentially of different size) in $\msf$. We say that the seeds $\seed_1$ and $\seed_2$ \emph{can be amalgamated (along $\zz_1\cap \zz_2$)} if the following conditions are satisfied:
	\begin{itemize}[wide, labelwidth=!, labelindent=0pt]
		\item all elements in $\zz_0: = \zz_1\cap \zz_2$ are frozen in both $\zz_1$ and $\zz_2$;
		\item the elements of $\zz_1\cup \zz_2$ are algebraically independent.
	\end{itemize}
	In this case, \emph{the amalgamation of $\seed_1$ and $\seed_2$ along (the subset) $\zz_0$} is the seed 
	\[
	\seed= \glueseeds{\seed_1} {\seed_2} {\zz_0} = (Q, \zz) 
	\]
	defined as follows. 
	Let $I_1$ (resp. $I_2$) be the set of vertices in $Q_1$ (resp. $Q_2$) corresponding to $\zz_0$, and let $\phi: I_1\rightarrow I_2$ be the corresponding bijection. The new quiver $Q$ is constructed in two steps:
	\begin{enumerate}[wide, labelwidth=!, labelindent=0pt]
		\item[\textbf{(1)}] For $i\in I_1$, identify the vertices $i\in Q_1$ and $\phi(i)\in Q_2$ in the union $Q_1\sqcup Q_2$, and \emph{defrost} the resulting vertex in $Q$. 
		\item[\textbf{(2)}] For any pair $i, j\in I_1$ with $\epsilon_{ij}$ arrows $i\rightarrow j$ in $Q_1$ and $\epsilon_{\phi(i), \phi(j)}$ arrows $\phi(i)\rightarrow \phi(j)$ in $Q_2$, introduce $\epsilon_{ij} + \epsilon_{\phi(i), \phi(j)}$  arrows between the corresponding vertices in $Q$. (Recall that we do retain the arrows between frozen vertices.)
	\end{enumerate}
	Define the extended cluster $\zz = \zz_1 \cup \zz_2$, with $\zz_0 = \zz_1\cap \zz_2$ defrosted. 
\end{defn}

\begin{example}
	Let $\seed_1 = (Q_1, \zz_1)$ and $\seed_2 = (Q_2, \zz_2)$ be two seeds in $\msf$, where $Q_1$ (resp., $Q_2$) is the quiver  shown on the left (resp., right) in Figure~\ref{fig: two seeds}; $\zz_1 = \{x_1, x_2, x_3, x_7, x_8\}$ with $x_3, x_7, x_8$ frozen; $\zz_2 = \{x_4, x_5, x_6, x_7, x_8\}$ with $x_7, x_8$ frozen. Assume furthermore that $\{x_1, x_2, x_3, x_4, x_5, x_6, x_7, x_8\}$ are algebraically independent over $\C$ in $\msf$. 
	\begin{figure}[H]
		\centering
		\begin{tikzcd}
			1 & 2 \arrow[d] \arrow[l]  & \boxed{3}   \arrow[l]   &  4 \arrow[r] & 5  \arrow[r] & 6 \arrow[dl] \\
			\boxed{7} \arrow[u]  \arrow[ur]  & \boxed{8}\arrow[l]&  \phantom{} &\phantom{}&  \boxed{7}\arrow[u] \arrow[r]& \boxed{8}  \arrow[u] & \phantom{}
		\end{tikzcd}
		\caption{Quivers $Q_1$ (left) and $Q_2$ (right).}  
		\label{fig: two seeds}
	\end{figure}
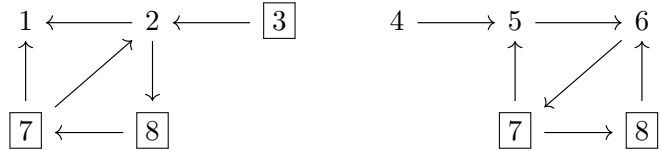
	Then $\seed_1$ and $\seed_2$ can be amalgamated along $\zz_0 = \zz_1\cap \zz_2 =  \{x_7, x_8\}$. The amalgamated seed $\seed =\glueseeds{\seed_1} {\seed_2} {\zz_0} = (Q, \zz)$, where $Q$ is shown in Figure~\ref{fig: example of gluing two seeds} and $\zz = \{x_1, x_2, x_4, x_5, x_6, x_7, x_8, x_3\}$ with $x_3$ frozen. 
	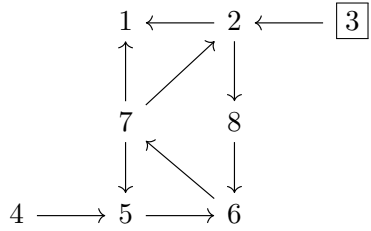
\begin{figure}[H]
		\centering
		\begin{tikzcd}
			&1 & 2 \arrow[d] \arrow[l]  & \boxed{3}   \arrow[l]   \\
			&	{7} \arrow[u] \arrow[d] \arrow[ur]  & {8} \arrow[d]&  \\
			4 \arrow[r] & 5 \arrow[r] & 6 \arrow[ul] & 	
		\end{tikzcd}
		\caption{The quiver of the amalgamated seed.}
		\label{fig: example of gluing two seeds}
	\end{figure}
	
\end{example}

\begin{remk}\label{remk: properties of amalgamation}
	Amalgamation has the following important properties.
	\begin{itemize}[wide, labelwidth=!, labelindent=0pt]
		\item \textbf{Commutation with mutations}. With the same notation as in Definition~\ref{defn: amalgamating two seeds}, let $z\in \zz_1$ be a cluster variable. Then 
		\[
		\glueseeds{\mu_z(\seed_1)}{\seed_2}{\zz_0} = \mu_z(\glueseeds{\seed_1}{\seed_2}{\zz_0}).
		\]
		\item \textbf{Associativity}. Let $\seed_i = (Q_i, \zz_i)$, $i = 1, 2, 3$, be three seeds in $\msf$. Suppose  that $\seed_1$ and $\seed_2$ can be amalgamated along $\zz_0:= \zz_1\cap \zz_2$; $\seed_2$ and $\seed_3$ can be amalgamated along  $\zz'_0 := \zz_2\cap \zz_3$. Furthermore, suppose that $\zz'_0\cap \zz_0 = \emptyset$. Then 
		\[
		\glueseeds{(\glueseeds{\seed_1}{\seed_2}{\zz_0})}{\seed_3}{\zz'_0} = \glueseeds{\seed_1}{(\glueseeds{\seed_2}{\seed_3}{\zz'_0})}{\zz_0}.
		\]
	\end{itemize}	
\end{remk}

\subsection{Combinatorics of (Demazure) weaves}\label{sec: combinatorics of weaves}

The diagrammatic planar calculus of weaves, introduced by Casals and Zaslow \cite{CasalsZaslow} and further developed in \cite{CasalsWeng,CasalsLeSBWeng,CGGS,CGGLSS}, provides tools to construct cluster structures in algebraic varieties. We begin with foundational definitions.

\begin{defn}[{\cite[Definition 2.25]{CasalsLeSBWeng}}] For $d\in\Z_{>0}$, a \emph{$d$-weave}  $\ww$ is a particular kind of a planar graph properly embedded in a topological 2-disk, with edges labeled by elements in $[1, d-1]$. Vertices on the disk boundary are \emph{boundary vertices}; others are \emph{internal}. Each internal vertex must conform to one of the types in Figure~\ref{fig:allowable vertices}. Boundary vertex have degree $1$. Edges of $\ww$ are also referred to as \emph{weave lines}. Certain boundary vertices are designated as \emph{marked boundary vertices}. Weave lines incident to boundary vertices are \emph{external weave lines}. See Figure~\ref{fig: example of a 4-weave} for an example of a $4$-weave.\\
	
	\noindent Elements of $[1,d-1]$ are referred to as the \emph{colors} of the edges of a $d$-weave. We often refer to $d$-weaves simply as \emph{weaves} if $d$ is clear from the context.
	\begin{figure}[H]
		\centering
		\includegraphics[trim = 18.5cm 22cm 20cm 20cm, clip = true, scale = 0.55]{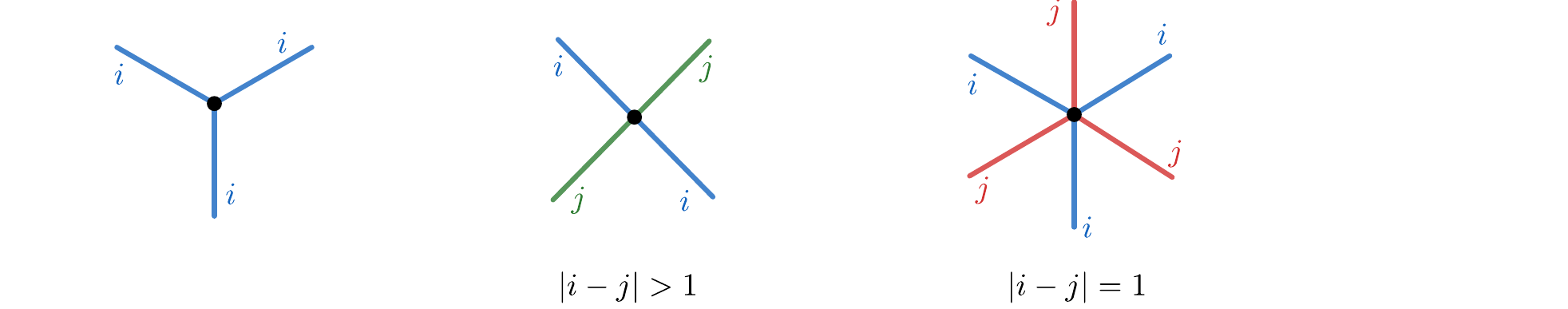}
		\caption{Allowable vertex types in a weave. For a 3-valent vertex, all labels must be of the same color. For a 4-valent vertex, the labels must alternate between two non-adjacent different colors. For a 6-valent vertex, the labels must alternate between two adjacent colors.}
		\label{fig:allowable vertices}
	\end{figure}
\end{defn}

\begin{example}
	A $2$-weave corresponds to a trivalent graph embedded in the $2$-disk (with marked boundary vertices). Figure~\ref{fig: example of a 4-weave} is an example of a $4$-weave.
	\begin{figure}
		\centering
		\includegraphics[trim = 13cm 11.55cm 18cm 12.05cm, clip = true, scale = 0.4]{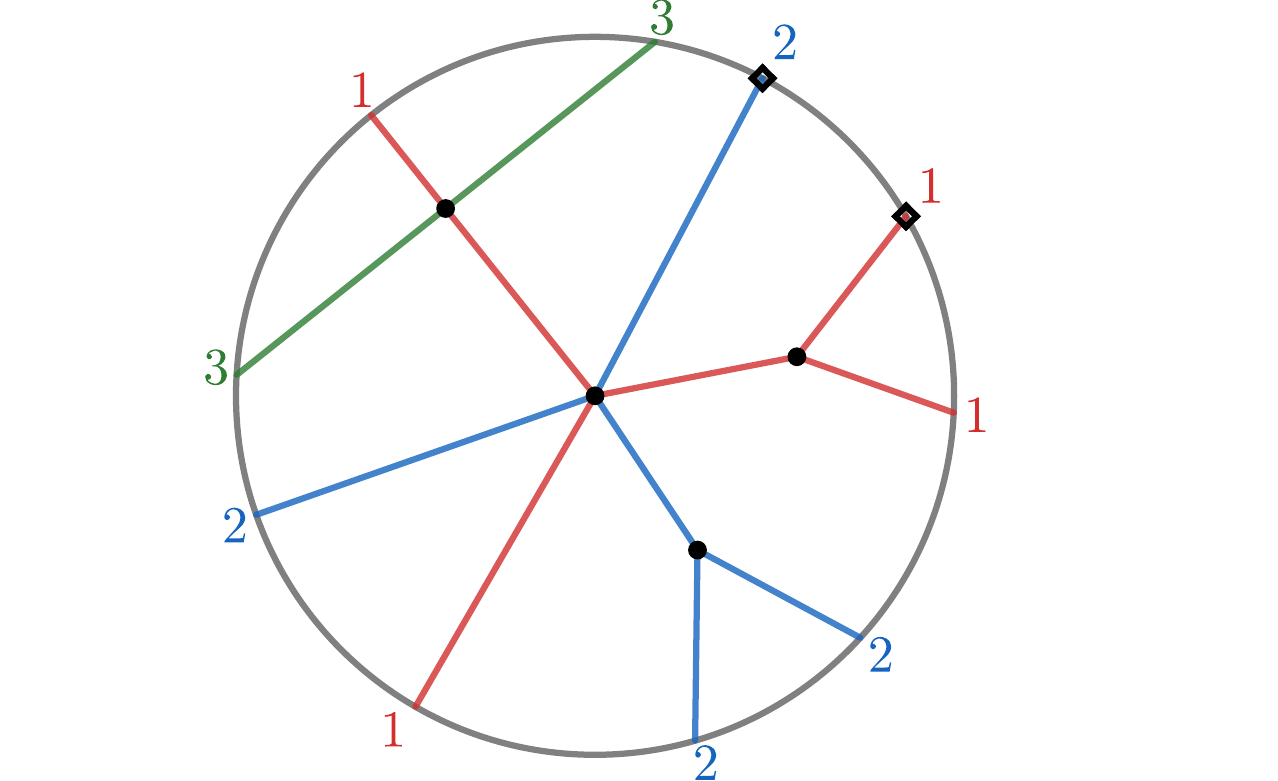}
		\caption{An example of a $4$-weave.}
		\label{fig: example of a 4-weave}
	\end{figure}
\end{example}

Weaves can be used to construct cluster structures in algebraic varieties. While for general weaves, there is no known procedure for doing this, there is a special family of weaves, called \emph{Demazure weaves}, that allow this kind of construction. 

\begin{defn}[{\cite[Definition 2.32]{CasalsLeSBWeng}}]\label{defn: demazure weave} Let $R$ be a rectangle on the Euclidean plane with horizontal and vertical sides. A \emph{Demazure weave} is a weave in $R$ such that:
	\begin{itemize}[wide, labelwidth=!, labelindent=0pt]
		\item All external weave lines are incident to the top/bottom boundary of $R$.
		\item Each weave line intersects every horizontal line in at most one point.
		\item Each trivalent vertex is incident to two weave lines above it and one below it.
		\item Each 4-valent vertex is incident to two weave lines above it and two below it.
		\item Each 6-valent vertex is incident to three weave lines above it and three below it.
		\item All marked boundary vertices must lie on the top boundary.
	\end{itemize}
	
	Unless noted otherwise, we will implicitly assume that the weave lines in a Demazure weave are oriented from top to bottom. 
\end{defn}
\begin{remk}
	The difference between our notion of Demazure weaves and the ones in the literature \cite{CasalsZaslow,CasalsWeng,CasalsLeSBWeng,CGGS,CGGLSS} is that we have designated a set of marked boundary vertices to a Demazure weave. These marked boundary vertices allow extra \emph{frozen cycles} (cf.\ Definition~\ref{defn:Lusztig cycles} and~\ref{defn: mutable and frozen cycles}) to originate from the top boundary and will become the frozen vertices for the amalgamated quiver. 
\end{remk}

\begin{example}
	The weave in Figure~\ref{fig: example of a 4-weave} is homeomorphic to a Demazure $4$-weave. One possible choice is shown in Figure~\ref{fig: example of a demazure 4-weave}.
	\begin{figure}
		\centering
		\includegraphics[trim = 7.5cm 12.4cm 10cm 9.6cm, clip = true, scale = 0.45]{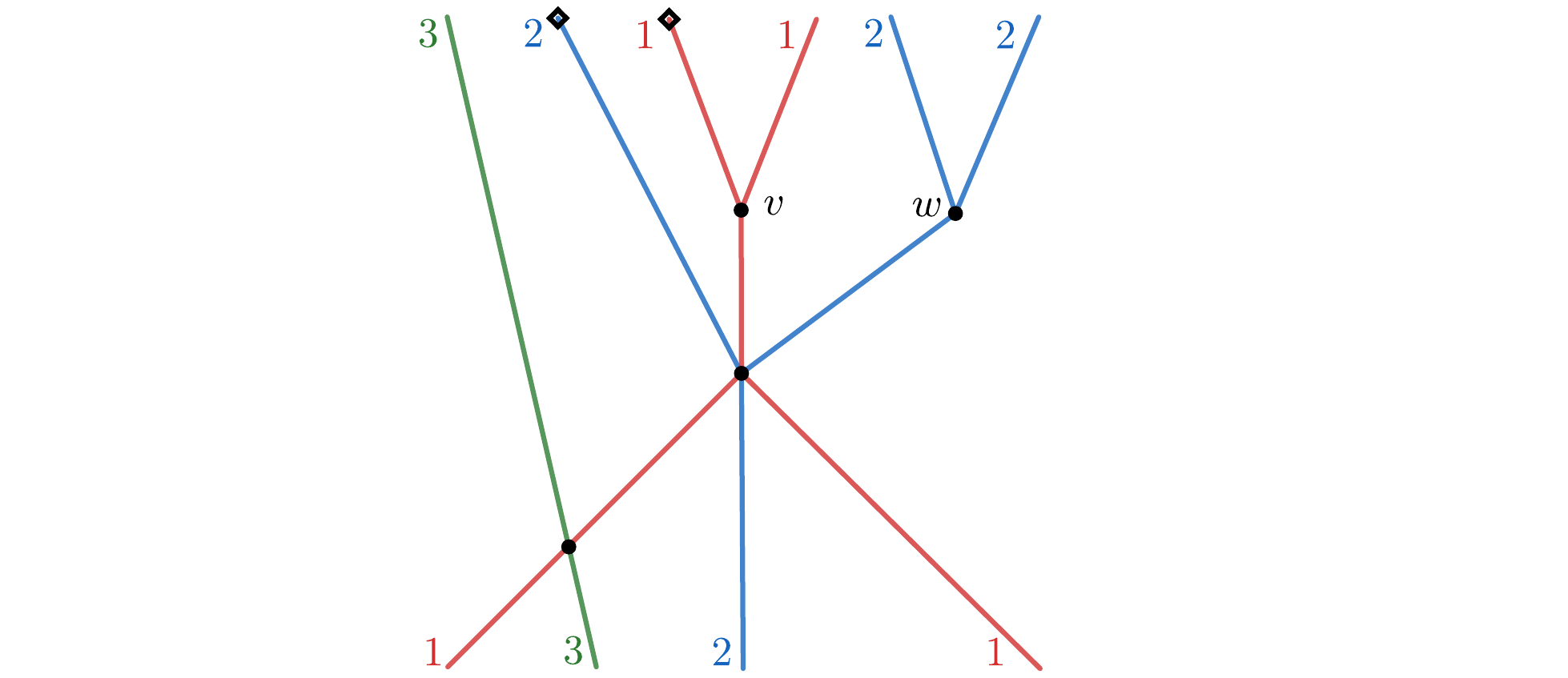}
		\caption{A Demazure $4$-weave homeomorphic to the weave in Figure~\ref{fig: example of a 4-weave}.}
		\label{fig: example of a demazure 4-weave}
	\end{figure}
	
\end{example}

To allow local analysis of Demazure weaves, we introduce the notion of \emph{row words} and \emph{partial weaves}. 

\begin{defn}\label{defn: scanning process and row words}
	Let $\ww$ be a Demazure weave. By toggling the vertices, we may assume that they are at different heights. Now consider a horizontal scanning line that connects the left and right boundaries of $\ww$. If a scanning line does not pass through any vertices, we can form a word by recording the color of weave lines that intersect with the scanning line, from left to right. Such a word is called a \emph{row word of $\ww$}. 
	
	Let $\ww_\text{top}$ (resp. $\ww_\text{bottom}$) be the row word when the scanning line is at the top (resp. bottom) boundary of $\ww$. We will usually write $\ww: \ww_\text{top} \rarrow \ww_\text{bottom}.$
	We will also use similar notation to describe a portion of a weave between two horizontal lines.
	
	As the scanning line moves from top to bottom, the row words are transformed as follows:
	\begin{itemize}[wide, labelwidth=!, labelindent=0pt]
	\item when the scanning line does not move through a vertex, the row word remains unchanged; 
	\item when the scanning line moves through a trivalent vertex, cf.\ Figure~\ref{fig:local transformations when passing through a vertex}, the row word is transformed as follows: $\cdots i  i \cdots \rarrow \cdots i \cdots$;
	\item when the scanning line moves through 4-valent vertex, cf.\ Figure~\ref{fig:local transformations when passing through a vertex}, the row word is transformed as follows: $\cdots i  j \cdots \rarrow \cdots j  i  \cdots$;
	\item when the scanning line moves through 6-valent vertex, cf.\ Figure~\ref{fig:local transformations when passing through a vertex}, the row word is transformed as follows: $\cdots i  j  i \cdots \rarrow \cdots j  i j \cdots$.
	\end{itemize}

	\begin{figure}[H]
		\centering
		\includegraphics[trim = 11.5cm 12cm 0cm 12cm, clip = true, scale = 0.47]{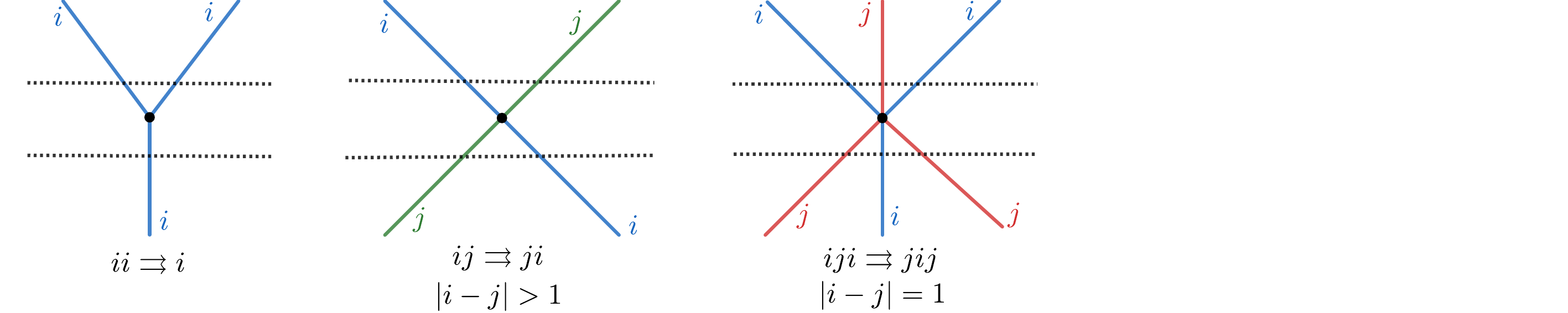}
		\caption{Row word transformations via vertex crossings. Dashed lines: scanning lines.}
		\label{fig:local transformations when passing through a vertex}
	\end{figure}
	The last two transformations are called \emph{braid moves}. Let $\beta, \beta'$ be two words. We say that they are \emph{braid equivalent}, denoted by $\beta\sim \beta'$, if they are related by a sequence of braid moves. Notice that for a Demazure weave $\ww: \beta\rarrow \beta'$, we have $\beta \sim \beta'$ if and only if $\ww$ contains only 4-valent and 6-valent vertices.
\end{defn}

\begin{defn}	
	Let $\ww: \beta\rarrow \beta'$ be a Demazure weave. Let $\beta_1, \beta_2$ be two row words for $\ww$, with $\beta_1$ above $\beta_2$. The portion of the weave that is between $\beta_1$ and $\beta_2$ is called \emph{a partial weave of $\ww$}, denoted by $\beta_1\stackrel{\ww}{\rarrow} \beta_2$ (or simply $\beta_1\rarrow \beta_2$ when $\ww$ is clear from the context). 
	
	A Demazure weave $\ww: \beta\rarrow \beta'$ can be interpreted as a concatenation of partial weaves of $\ww$ along the row words. We express this as
	\[
	\ww: \beta = \beta_0 \rarrow \beta_1 \rarrow \beta_2 \rarrow \cdots \rarrow \beta_{s-1} \rarrow \beta_s = \beta'
	\]
	where $\beta_i$ are row words for $\ww$, $0\le i \le s$.
\end{defn}

\begin{remk}
We will view a partial weave of $\ww$ with $\ww_{\text{top}}$ as the top word as a Demazure weave with the same set of marked boundary vertices (as $\ww$). However, a partial weave $\beta_1\rarrow \beta_2$ with $\beta_1 \neq \ww_{\text{top}}$ will not be viewed as a Demazure weave in general; as the information concerning the marking is lost in general. 
\end{remk}

\begin{example}
	let $\ww$ be the Demazure 4-weave shown in Figure~\ref{fig: example of a demazure 4-weave}. Toggle the trivalent vertex $v$ to be above the trivalent vertex $w$. Then $\ww$ can be decomposed as
	\[
	\ww: \ww_{\text{top}} = 321122 \rarrow  32122 \rarrow  3212 \rarrow  3121 \rarrow 1321 = \ww_{\text{bottom}}.
	\]
\end{example}

To construct a cluster algebra, we associate a quiver $Q(\ww)$ to a Demazure weave~$\ww$, cf.\ \cite[Section~4.4]{CGGLSS}. The vertices of $Q(\ww)$ correspond to \emph{Lusztig cycles} of $\ww$ (Definition~\ref{defn:Lusztig cycles}), while its arrows derive from the  intersection pairings between the corresponding Lusztig cycles (Definition~\ref{defn:local intersection pairing}).

\begin{defn}[{\cite[Definition 2.33]{CasalsLeSBWeng}}]\label{defn:Lusztig cycles} Let $\ww$ be a Demazure weave. Let $v$ be either a marked boundary vertex or a trivalent vertex in $\ww$. The \emph{(Lusztig) cycle} $\gamma = \gamma_v$ associated with $v$ is a function on the edges of $\ww$, constructed by scanning $\ww$ from top to bottom, as follows.
	\begin{itemize}[wide, labelwidth=!, labelindent=0pt]
		\item For any weave line $e$ that begins above $v$, we have $\gamma(e)=0$.
		\item For the unique weave line $e$ that begins at $v$, we have $\gamma(e)=1$.
		\item For any trivalent vertex $w$ below $v$, we have (cf.\ Figure~\ref{fig: trivalent 4-valent and 6-valent label}):
		\[\gamma(b)=\min\{\gamma(a), \gamma(c)\},
		\]
		where $b$ is the weave line that begins at $w$, and $a, c$ are weave lines end at $w$.
		\item For any 4-valent vertex $w$, we have (cf.\ Figure~\ref{fig: trivalent 4-valent and 6-valent label}):
		\[
		\gamma(c)=\gamma(a), \quad \gamma(b)=\gamma(d),
		\]where $a,b,c,d$ are weave lines incident to $w$ listed in a cyclic order.
		\item For any 6-valent vertex $w$, we have (cf.\ Figure~\ref{fig: trivalent 4-valent and 6-valent label}):
		\begin{align*}
		\gamma(b)&=\gamma(f)+\gamma(e)-\min\{\gamma(a),\gamma(e)\},\\
		 \gamma(c)&=\min\{\gamma(a),\gamma(e)\},\\ 
		 \gamma(d)&=\gamma(a)+\gamma(f)-\min\{\gamma(a),\gamma(e)\},
		\end{align*}
		where $a,b,c,d,e,f$ are weave lines incident to $w$ in a cyclic order, with $a,f,e$ above $w$ and $b,c,d$ below $w$.
	\end{itemize}
\end{defn}

	\begin{figure}[H]
		\centering
		\includegraphics[trim= 12.5cm 18cm 20cm 18.7cm, clip = true, scale = 0.58]{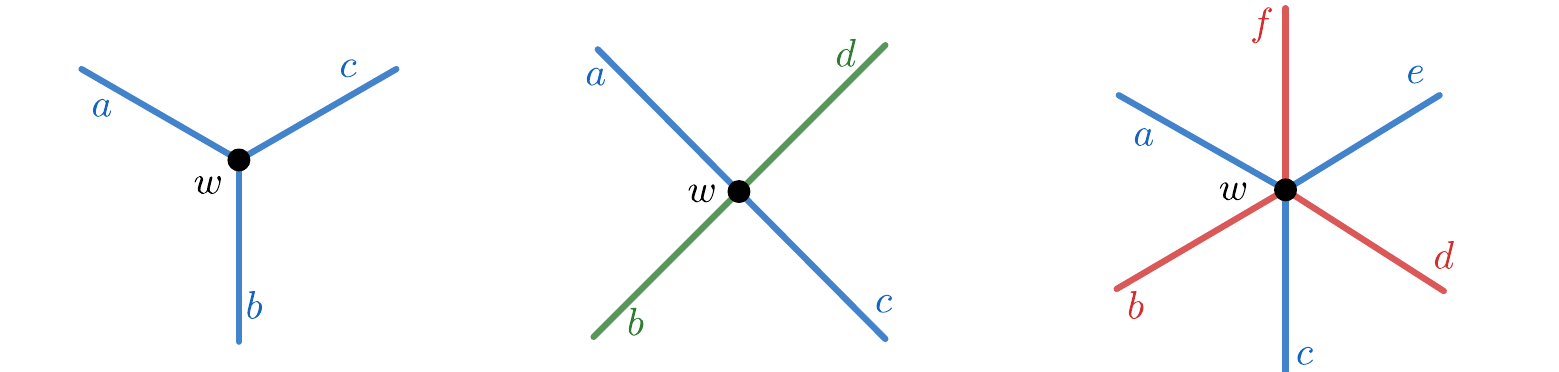}
		\caption{Edge labelings around trivalent, 4-valent and 6-valent vertices.}
		\label{fig: trivalent 4-valent and 6-valent label}
	\end{figure}

	Such a (Lusztig) cycle $\gamma$ can be interpreted as an oriented weighted subgraph ~$G_\gamma$ of ~$\ww$. An edge $e$ in $\ww$ belongs to $G_\gamma$ if $\gamma(e) \neq 0$; its weight is equal to $\gamma(e)$. Graphically cycles follow the rules shown in Figure~\ref{fig: local cycle types}.

\begin{figure}[H]
	\centering
	\includegraphics[trim= 0cm 9cm 4cm 4.5cm, clip = true, scale = 0.27]{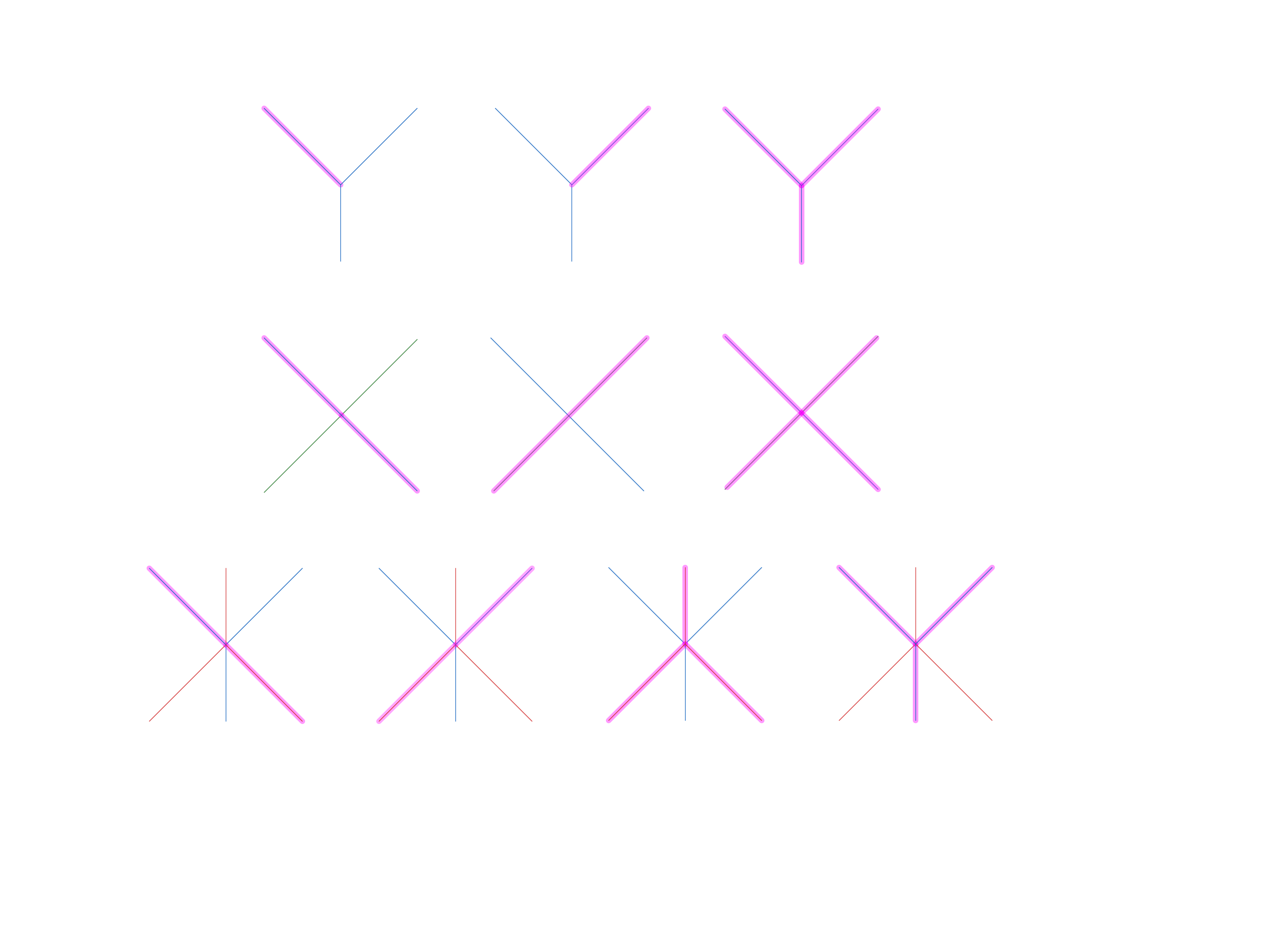}
	\caption{Local rules for a cycle. A thickened edge indicates that the value of the cycle on that edge is $1$; otherwise the value is $0$.}
	\label{fig: local cycle types}
\end{figure}

\begin{defn}\label{defn: mutable and frozen cycles}
	A cycle is \emph{frozen} if it meets with a boundary vertex. Otherwise it is \emph{mutable}. Notice that there are two types of frozen cycles: one that originating from the marked boundary vertices; the other that originating from a trivalent vertex, but ending at the bottom boundary of the weave. 
\end{defn}

\begin{example}
	The Demazure 4-weave in Figure~\ref{fig: example of a demazure 4-weave} has $4$ cycles $\gamma_1, \gamma_2, \gamma_3, \gamma_4$, as shown in Figure~\ref{fig: example of a demazure 3-weave with cycles}.
	The cycles $\gamma_1, \gamma_2$ (resp. $\gamma_3, \gamma_4$) originate from marked boundary vertices (resp. trivalent vertices). Notice that for every cycle, all weights are equal to $1$, so in this example they are (unweighted) subgraphs of the weave. Also notice that all $4$ cycles are frozen.
	\begin{figure}[H]
		\centering
		\subfloat{{\includegraphics[trim = 8cm 11.5cm 32cm 12cm, clip = true, scale = 0.45]{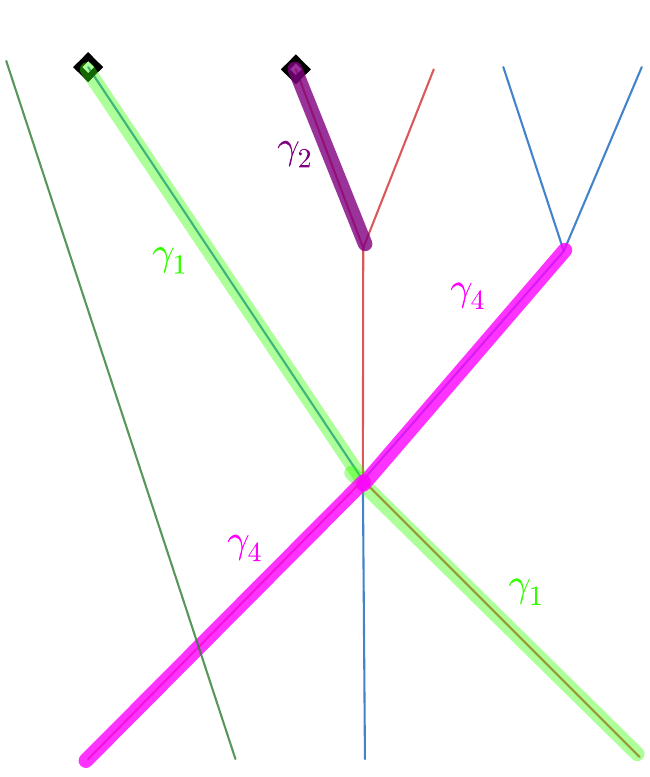}}}
		\quad
		\subfloat{{\includegraphics[trim = 8cm 11.5cm 32cm 12cm, clip = true, scale = 0.45]{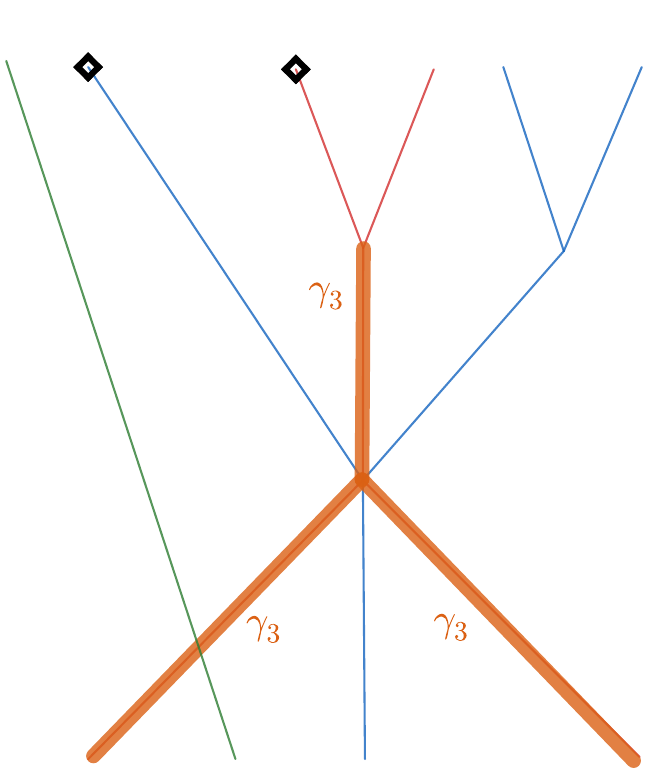}}}
		\caption{Cycles in the Demazure 4-weave from Figure~\ref{fig: example of a demazure 4-weave}.}
		\label{fig: example of a demazure 3-weave with cycles}
	\end{figure}
\end{example}

The number of arrows between two vertices of the quiver associated with a Demazure weave is the \emph{intersection pairing} between the corresponding cycles.

\begin{defn}[{\cite[Definition 2.28]{CasalsLeSBWeng}}]\label{defn:local intersection pairing} Let $V_\ww$ denote the set of internal vertices in a weave $\ww$. For any two cycles $\gamma, \gamma'$, their \emph{intersection pairing} is defined by
	\begin{equation}
		\inprod{\gamma}{\gamma'}:=\sum_{v\in V_\ww}\inprod{\gamma}{\gamma'}_v,
	\end{equation}
	where we use the notation
	\[
	\inprod{\gamma}{\gamma'}_v:=\left\{\begin{array}{ll}\vspace{1cm}  \begin{vmatrix}
			1 & 1 & 1 \\
			\gamma(a) & \gamma(b) & \gamma(c)\\
			\gamma'(a) & \gamma'(b) & \gamma'(c)
		\end{vmatrix} & \text{if } v \text{ is trivalent},\\
		\vspace{1cm}
		\frac{1}{2}\left(\begin{vmatrix} 1 & 1 & 1\\
			\gamma(a) & \gamma(c) & \gamma(e) \\
			\gamma'(a) & \gamma'(c) & \gamma'(e)
		\end{vmatrix} +\begin{vmatrix} 1 & 1 & 1\\
			\gamma(b) & \gamma(d) & \gamma(f) \\ 
			\gamma'(b) & \gamma'(d) & \gamma'(f)
		\end{vmatrix}\right) & \text{if } v \text{ is } 6\text{-valent}, \\
		0 & \text{otherwise,}
	\end{array}\right.
	\]
	and the edges labels are as shown in Figure~\ref{fig: trivalent and 6-valent labeled}.
	\end{defn}
	\begin{figure}
		\centering
		\includegraphics[trim= 10cm 14cm 18cm 15cm, clip = true, scale = 0.50]{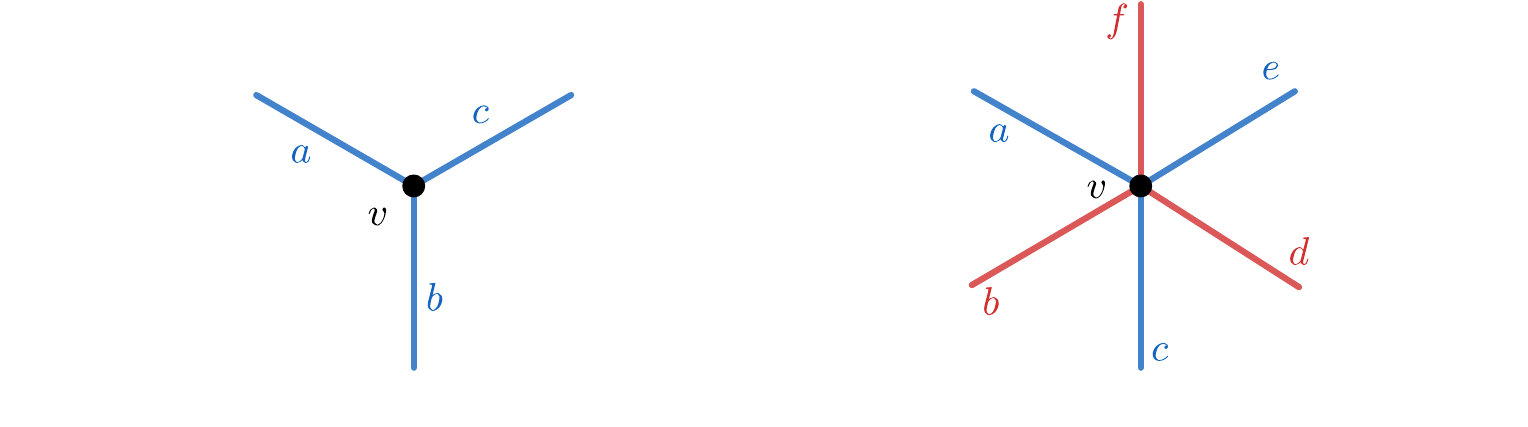}
		\caption{Trivalent and 6-valent vertices with edges labeled.}
		\label{fig: trivalent and 6-valent labeled}
	\end{figure}

\begin{figure}[H]
	\centering
	\includegraphics[trim= 0cm 6cm 0cm 12cm, clip = true, scale = 0.3]{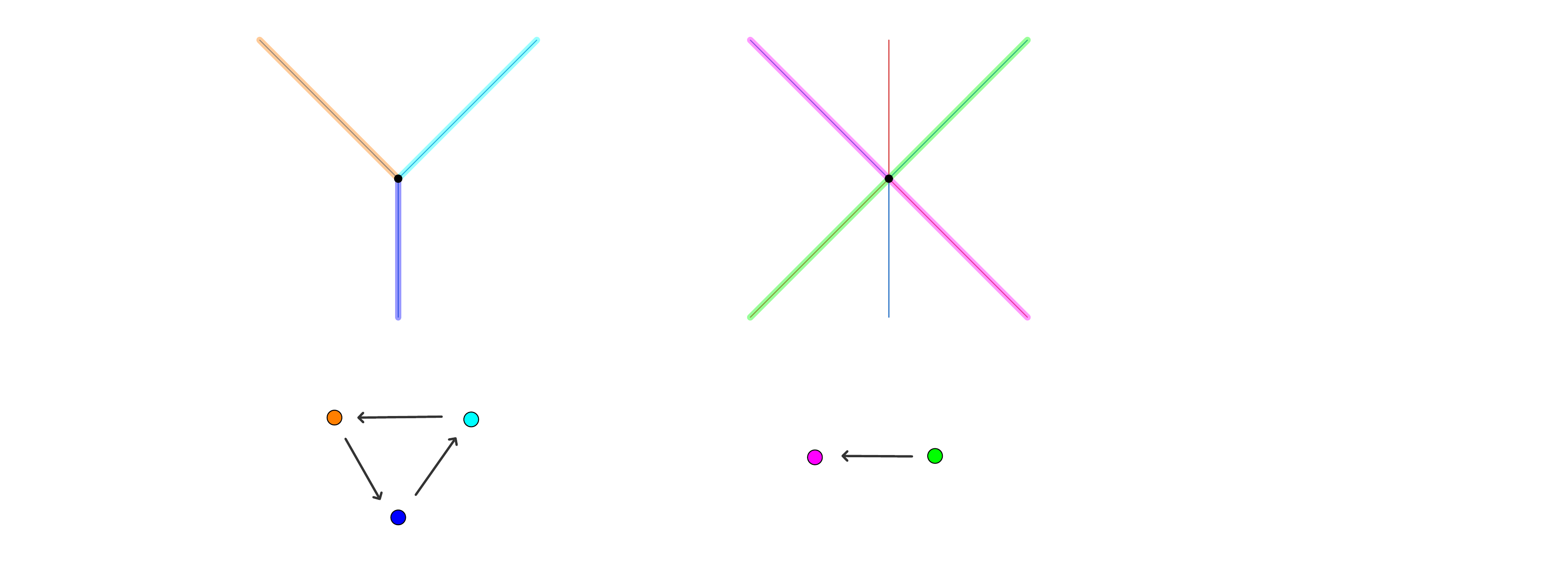}
	\caption{Quiver interpretation of intersection pairing of cycles at trivalent and 6-valent vertices.}
	\label{fig: local intersection pair graph}
\end{figure}

Notice that the intersection pairing is skew-symmetric by construction. Quiver interpretation of intersection pairing of cycles is shown in Figure~\ref{fig: local intersection pair graph}, where an arrow from $\gamma$ to $\gamma'$ contributes $+1$ to $\langle \gamma, \gamma'\rangle$ and $-1$ to $\langle \gamma', \gamma \rangle$. 

\begin{example} \label{example: demazure 3-weaves with cycles}
	The intersection pairings of cycles in Figure~\ref{fig: example of a demazure 3-weave with cycles} are given by:
	\[
	\begin{array}{c|cccc}
		\langle \gamma_i, \gamma_j \rangle & \gamma_1 &\gamma_2 & \gamma_3 &\gamma_4 \\
		\hline
		\gamma_1 &0 & 0 & 0 & -1\\
		\gamma_2 & 0& 0 & 1& 0\\
		\gamma_3 & 0& -1& 0 & 0\\
		\gamma_4 & 1& 0& 0 & 0
	\end{array}
	\]
\end{example}

 Now we can define the \emph{quiver associated with a Demazure weave}. 
\begin{defn}\label{defn: quiver associated with Demazure weaves}
	Let $\ww$ be a Demazure weave. \emph{The quiver associated with $\ww$}, denoted by $Q(\ww)$, is defined as follows. The set of vertices of $Q(\ww)$ is indexed by the set of cycles in $\ww$. A vertex of $Q(\ww)$ is designated frozen (resp., mutable) if the corresponding cycle is frozen (resp., mutable). The number of arrows between two vertices is the intersection pairing of the corresponding cycles, cf.\ Definition~\ref{defn:local intersection pairing}.
\end{defn}

\begin{remk}
	In \cite[Definition 4.24 and 4.26]{CGGLSS}, they have defined an extra intersection pairing between two frozen cycles ending at the bottom of the weave. These extra intersection pairings will be canceled out when we amalgamate two Demazure weaves with reverse bottom words, henceforth are omitted.
\end{remk}

\begin{example}
	The quiver $Q(\ww)$ associated with the Demazure weave $\ww$ shown in Figure~\ref{fig: example of a demazure 4-weave} is shown in Figure~\ref{fig: example of a quiver associated with a Demazure weave}.
	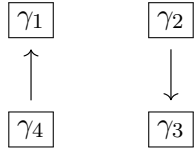
\begin{figure}[H]
		\centering
	\begin{tikzcd}
		\boxed{\gamma_1} & \boxed{\gamma_2} \arrow[d] \\
		\boxed{\gamma_4} \arrow[u] &\boxed{\gamma_3}
	\end{tikzcd}
\caption{The quiver $Q(\ww)$ associated with the Demazure weave $\ww$ shown in Figure~\ref{fig: example of a demazure 4-weave}.}
\label{fig: example of a quiver associated with a Demazure weave}
\end{figure}
\end{example}

We next describe equivalence and mutation equivalence between Demazure weaves and show that the quiver associated with Demazure weaves are well-behaved under equivalence and mutation equivalence.

\begin{defn}[{\cite[Section 4.2]{CGGLSS}}]\label{defn: weave equivalence moves}
	Two Demazure weaves with the same marked boundary vertices are \emph{equivalent} if they are related by a sequence of \emph{(local) equivalence moves}, defined as follows:
	\begin{enumerate}[wide, labelwidth=!, labelindent=0pt]
		\item[(i)] Let $\ww,\ww':\beta\rarrow\beta'$ contain only  4-valent and 6-valent vertices (equivalently $\beta\sim \beta'$). Then $\ww$ and $\ww'$ are equivalent. For example, the partial weaves $ij \rarrow ji \rarrow ij$ and $ij \rarrow ij$ (cf.\ left of Figure~\ref{fig: 4-valent and 6-valent only equivalent move}) are equivalent, where $|i-j|>1$; the partial weaves $iji\rarrow jij \rarrow iji$ and $iji \rarrow iji$ (cf.\ right of Figure~\ref{fig: 4-valent and 6-valent only equivalent move}) are equivalent, where $|i-j| = 1$.
		 \begin{figure}[H]
		 	\centering
		 	{\includegraphics[trim= 2cm 9cm 10cm 12cm, clip = true, scale = 0.3]{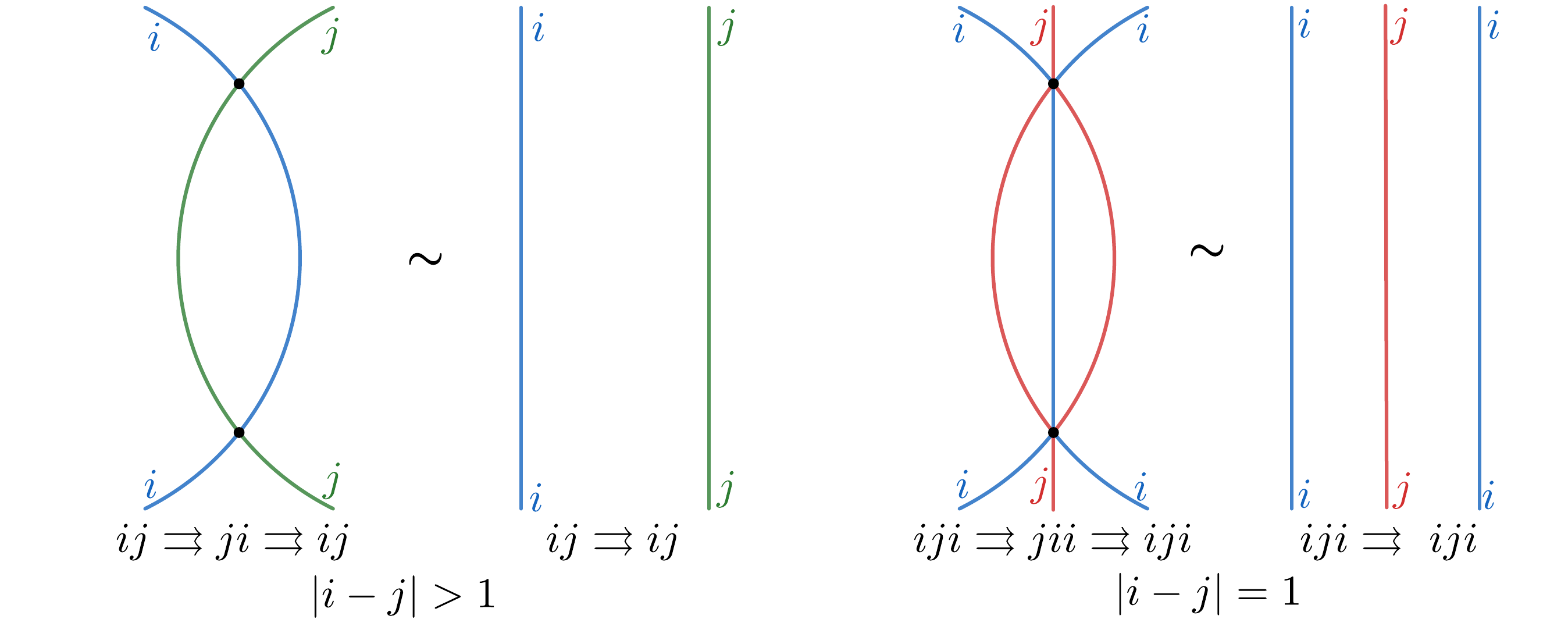}}
		 	\caption{Equivalence moves with only 4-valent and 6-valent vertices.}
		 	\label{fig: 4-valent and 6-valent only equivalent move}
		 \end{figure}
		\item[(ii)] [\textbf{Pushthrough from below.}] Suppose that $|i-j| = 1$. Then the partial weaves $ijij\rarrow jijj\rarrow jij$ and $ijij\rarrow iiji\rarrow iji\rarrow jij$ (cf.\ Figure~\ref{fig: equivalence pushthrough from below}) are equivalent. The mirror image version of this equivalence is defined in a similar way.
		\begin{figure}[H]
			\centering
			{\includegraphics[trim= 8cm 9.1cm 10cm 14.5cm, clip = true, scale = 0.45]{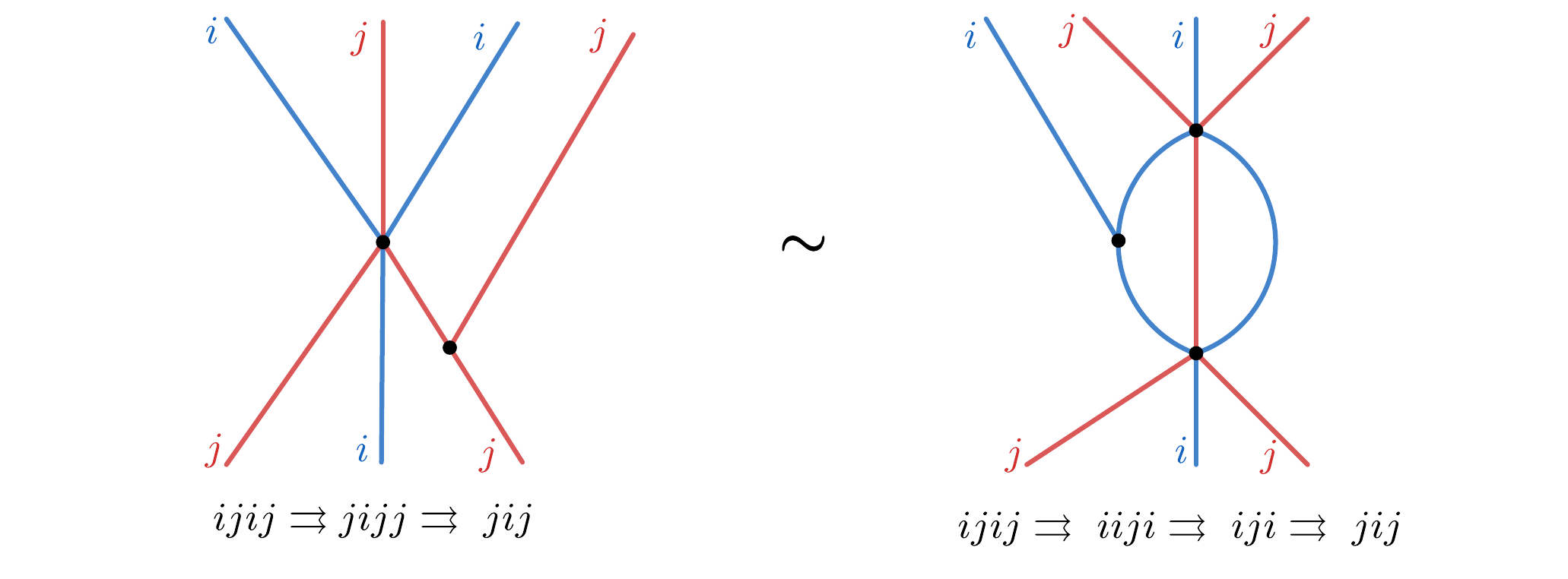}}
			\caption{Pushthrough from below, $|i-j| = 1$.}
			\label{fig: equivalence pushthrough from below}
		\end{figure}
		\item[(ii$'$)] [\textbf{Pushthrough from above.}] Suppose that $|i-j| = 1$. Then the partial weaves $ijii\rarrow iji \rarrow jij$ and $ijii\rarrow jiji\rarrow jjij\rarrow jij$ (cf.\ Figure~\ref{fig: equivalence pushthrough from above}) are equivalent. The mirror image version of this equivalence is defined in a similar way.
		\begin{figure}[H]
			\centering
			\includegraphics[trim= 7cm 6.1cm 5cm 15.1cm, clip = true, scale = 0.41]{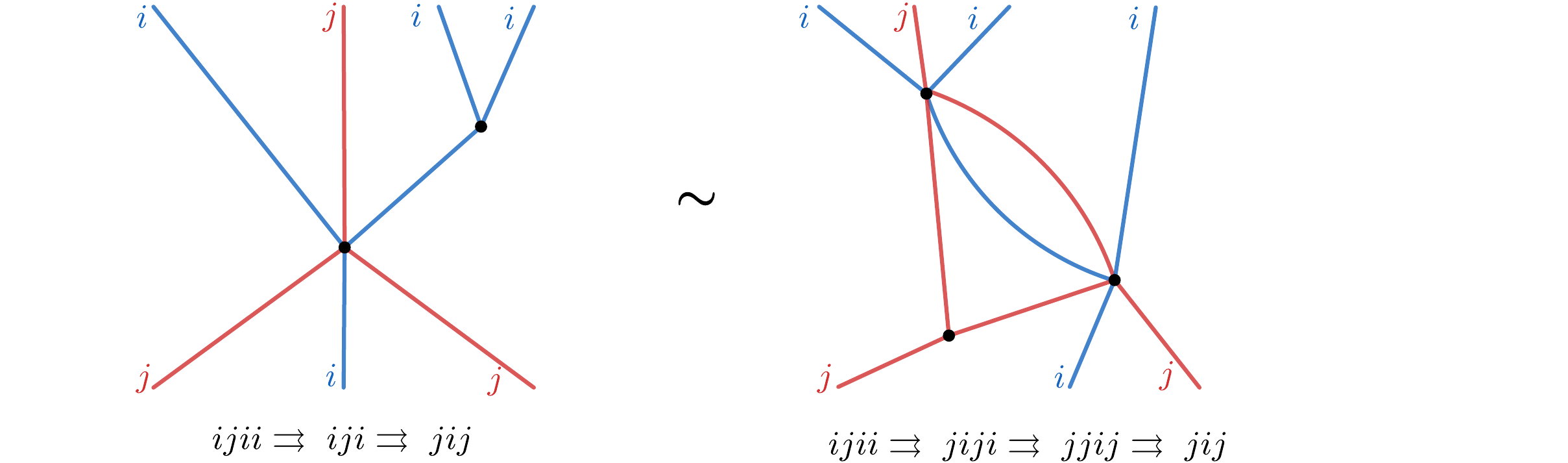}
			\caption{Pushthrough from above, $|i-j| = 1$.}
			\label{fig: equivalence pushthrough from above}
		\end{figure}
		\item[(iii)] Suppose that $|i-j| > 1$. Then the partial weaves $iji\rarrow iij\rarrow ij\rarrow ji$ and $iji\rarrow jii\rarrow ji$ (cf.\ Figure~\ref{fig: Equivalence Passthrough}) are equivalent. In other words, one can move a $j$-colored strand through an $i$-colored trivalent vertex.
	\end{enumerate}	
	\begin{figure}[H]
		\centering
		{\includegraphics[trim= 8cm 3cm 10cm 14.9cm, clip = true, scale = 0.41]{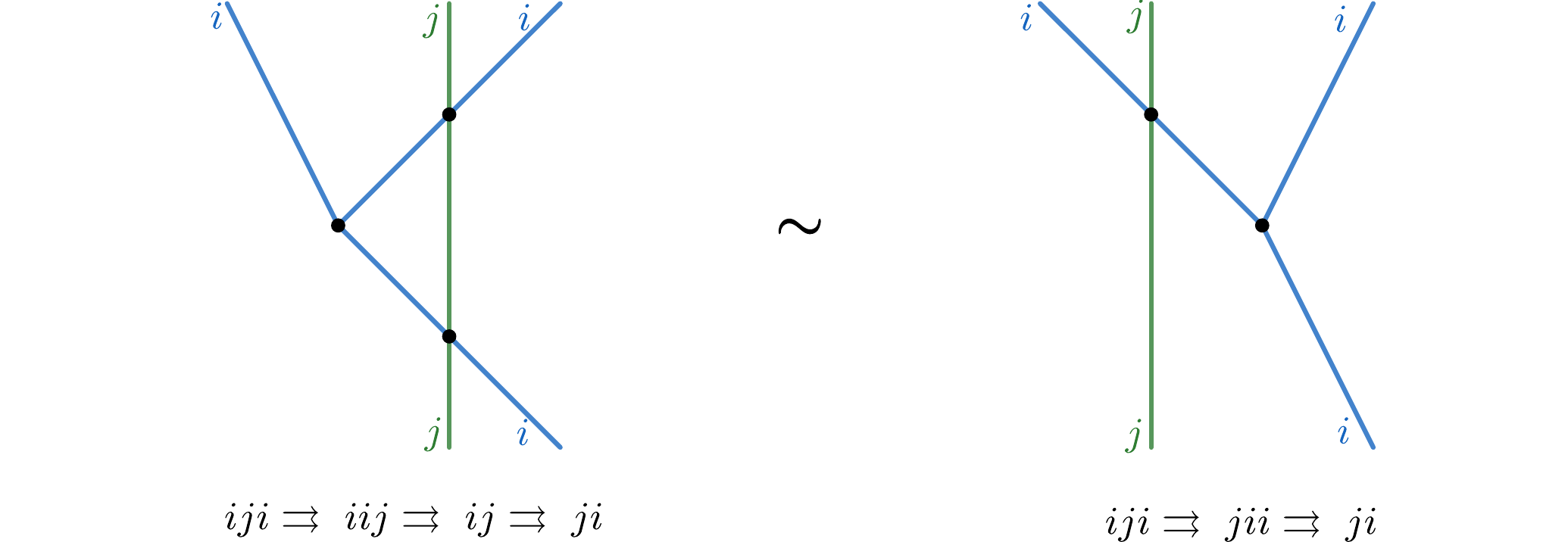}}
		\caption{Passthrough a trivalent vertex, $|i-j|>1$.}
		\label{fig: Equivalence Passthrough}
	\end{figure}
\end{defn}

\begin{defn}
	The move depicted in Figure~\ref{fig: local mutation move} is called a \emph{weave mutation move}. Two Demazure weaves are called \emph{mutation equivalent} if they are related by a sequence of mutation moves and equivalence moves. 
	\begin{figure}[H]
		\centering
		{\includegraphics[trim= 7cm 8cm 14.5cm 15.8cm, clip = true, scale = 0.41]{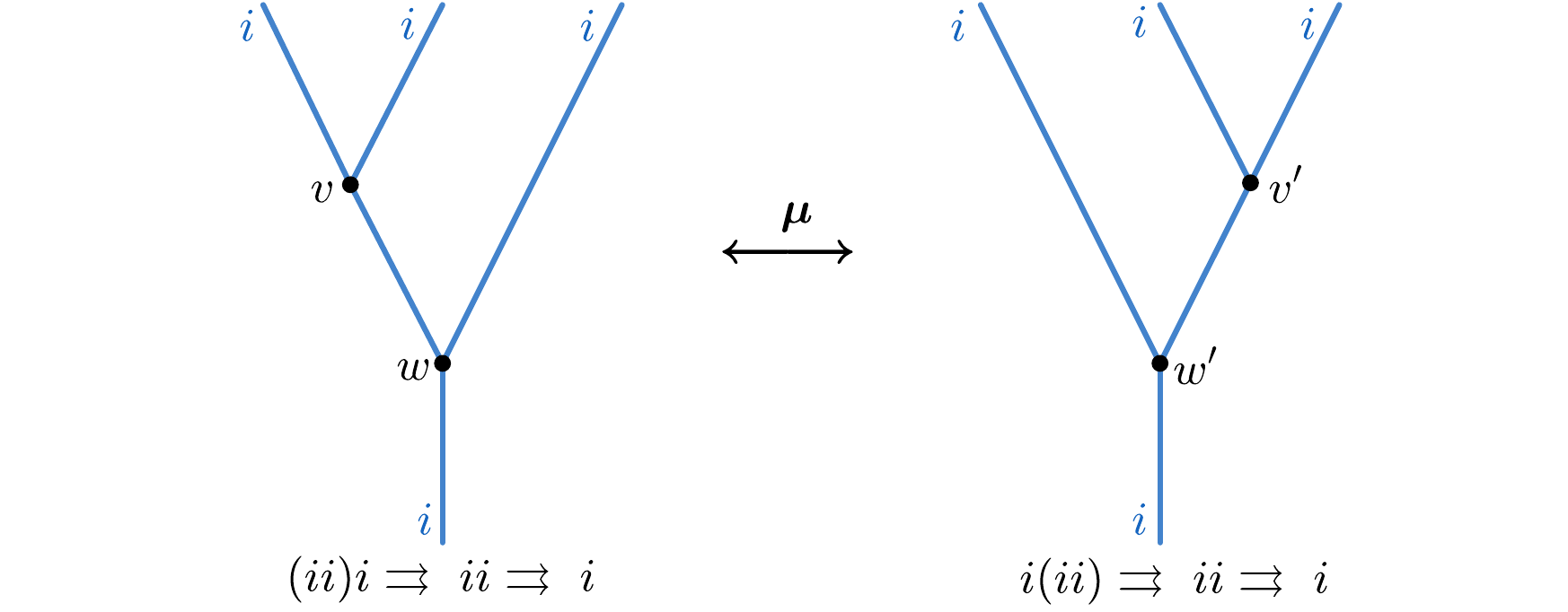}}
		\caption{Local mutation move.}
		\label{fig: local mutation move}
	\end{figure}
\end{defn}

\begin{remk}\label{remk: correspondence of cycles under mutation equivalence}
	For each local equivalence move and mutation move, there is a distinguished 1-1 correspondence between the trivalent vertices before and after the move.  For the local equivalence moves, there is at most one trivalent vertex, so the choice of correspondence is unique. For the local mutation move, $v$ and $w$ correspond to $v'$ and $w'$ respectively, cf.\ Figure~\ref{fig: local mutation move}. This 1-1 correspondence induces a 1-1 correspondence between the cycles of two mutation equivalent Demazure weaves. 
\end{remk}

We say an edge of a Demazure weave a \emph{bottom edge} if it is incident to the bottom boundary of the weave. 

\begin{lemma}\label{lemma: cycles end behavior remain unchanged under mutation}
	Mutation equivalent weaves preserve cycle weights on bottom edges. 
\end{lemma}
\begin{proof}
	Let $\ww, \ww': \beta\rarrow \beta'$ be two mutation equivalent Demazure weaves. Fix a sequence of local equivalence moves and local mutation moves that transform $\ww$ to $\ww'$. Let $\gamma$ be a cycle in $\ww$ and $\gamma'$ be the corresponding cycle in $\ww'$ (see Remark~ \ref{remk: correspondence of cycles under mutation equivalence}). We need to show that $\gamma(e) = \gamma'(e)$ for any bottom edge $e$. 
	
	As $\ww$ and $\ww'$ are related by a sequence of local equivalence moves and local mutation moves, we only need to prove the result if $\ww$ and $\ww'$ are related by a single equivalence move or a single mutation move. Here we demonstrate the proof for one move, the proof for other moves is similar. 
	
	Consider the local equivalence move (ii): Pushthrough from below, cf.\ Figure~\ref{fig: equivalence pushthrough from below same cycles}. Assume that $\gamma \neq \gamma_v$ (hence $\gamma' \neq \gamma_{v'}$). We have $\gamma = \gamma'$ on $a, b, c, d$, and we need to show that $\gamma = \gamma'$ on $l, h, n$. 
	\begin{figure}[H]
		\centering
		\includegraphics[trim = 3.5cm 12.5cm 0cm 13cm, clip = true, scale = 0.44]{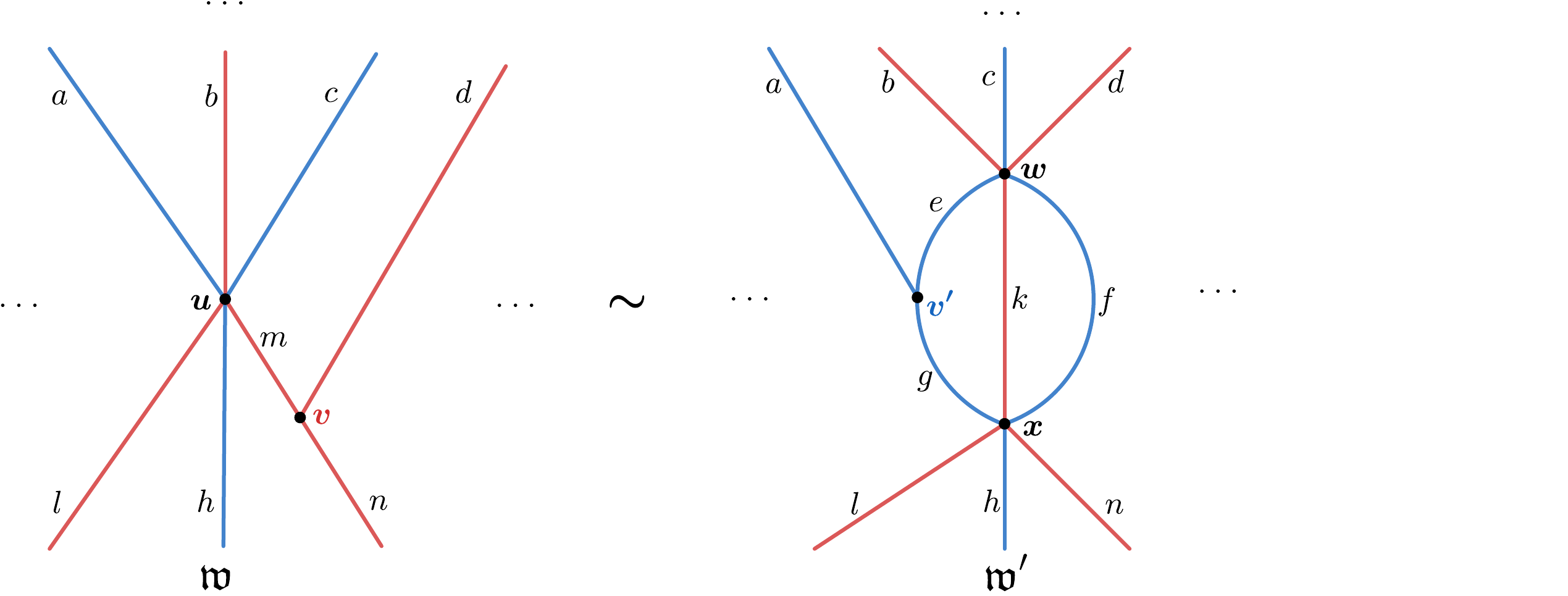}
		\caption{Two Demazure weaves related by a Pushthrough from below.}
		\label{fig: equivalence pushthrough from below same cycles}
	\end{figure}
	Using the definition of (Lusztig) cycles (cf.\ \ref{defn:Lusztig cycles}), we have
	\begin{equation*}
		\begin{split}
			\gamma'(h) &=\min\{\gamma'(g), \gamma'(f)\}\\
			&= \min\{\min\{\gamma'(a), \gamma'(e)\}, \gamma'(c)+\gamma'(b) - \min\{\gamma'(b), \gamma'(d)\}\}\\
			&= \min\{\gamma(a), \gamma'(e),\gamma(c)+\gamma(b) - \min\{\gamma(b), \gamma(d)\} \}\\
			&= \min\{\gamma(a), \gamma(c)+\gamma(d) - \min\{\gamma(b), \gamma(d)\},\gamma(c)+\gamma(b) - \min\{\gamma(b), \gamma(d)\} \}\\
			&= \min\{\gamma(a), \min\{\gamma(c)+\gamma(d) - \min\{\gamma(b), \gamma(d)\},\gamma(c)+\gamma(b) - \min\{\gamma(b), \gamma(d)\}\} \}\\
			& = \min\{\gamma(a), \gamma(c) - \min\{\gamma(b), \gamma(d)\} + \min\{\gamma(d)\},\gamma(b) \}\}\\
			& = \min\{\gamma(a), \gamma(c)\} = \gamma(h).
		\end{split}			
	\end{equation*}
	Now notice that $\gamma(h) + \gamma(l) = \gamma(b) + \gamma(c)$ and $\gamma'(h) + \gamma'(l) = \gamma'(k) + \gamma'(f) = \gamma'(b) + \gamma'(c)$, we get $\gamma(l) = \gamma'(l)$. Lastly, we have 
	\[
	\gamma(n) = \min\{\gamma(m), \gamma(d)\} = \min\{\gamma(a) + \gamma(b) - \min\{\gamma(a), \gamma(c)\}, \gamma(d)\}.
	\]
	and 
	\begin{equation*}
		\begin{split}
	\gamma'(n) &= \gamma'(g) + \gamma'(k) - \gamma(h)\\
	& = \min\{\gamma(a), \gamma'(e)\} + \min\{\gamma(b), \gamma(d)\} - \min\{\gamma(a), \gamma(c)\}\\
	& = \min\{\gamma(a), \gamma(c) + \gamma(d) - \min\{\gamma(b), \gamma(d)\}\} + \min\{\gamma(b), \gamma(d)\} - \min\{\gamma(a), \gamma(c)\}\\
	& = \min\{\gamma(a)+ \min\{\gamma(b), \gamma(d)\}, \gamma(c) + \gamma(d)\}  - \min\{\gamma(a), \gamma(c)\}\\
	& =  \min\{\gamma(a) + \gamma(b) - \min\{\gamma(a), \gamma(c)\}, \gamma(d)\}.
	\end{split}			
	\end{equation*}
The last identity is obtained by troplicalizing the identity
\[
\frac{t^{\gamma(a)}(t^{\gamma(b)}+t^{\gamma(d)}) + t^{\gamma(c)}t^{\gamma(d)}}{t^{\gamma(a)} + t^{\gamma(c)}} = \frac{t^{\gamma(a)}t^{\gamma(b)}}{t^{\gamma(a)} + t^{\gamma(c)}} + t^{\gamma(d)}.
\]
Hence $\gamma'(n) = \gamma(n)$. 

Finally, notice that $\gamma_v(l) = \gamma_v(h) = \gamma_{v'}(l) = \gamma_{v'}(h) = 0$ and $\gamma_v(n) = \gamma_{v'}(n) = 1$. This completes the proof.
\end{proof}

The next Proposition justifies the notion of mutation equivalence of Demazure weaves. 
\begin{prop}\label{prop: mutation equivalent weaves yield mutation equivalent quivers}
	Let $\ww, \ww': \beta\rarrow \beta'$ be two equivalent (resp. mutation equivalent) Demazure weaves. Then the quivers associated with $\ww, \ww'$ are the same (resp., mutation equivalent), i.e., $Q(\ww) = Q(\ww')$ (resp., $Q(\ww) \sim Q(\ww')$). 
\end{prop}

\begin{proof}
	Let us first assume that $\ww$ and $\ww'$ are equivalent Demazure weaves. By definition, they are related by a sequence of local equivalence moves. So we only need to show that $Q(\ww) = Q(\ww')$ if $\ww$ and $\ww'$ are related by a single local equivalence move. Here we will demonstrate the proof for one move, the proof for other moves is similar. 
	
	Consider the local equivalence move (ii): Pushthrough from below, cf.\ Figure~\ref{fig: equivalence pushthrough from below same cycles}. Let $\gamma_1, \gamma_2$ be two different cycles in $\ww$, the corresponding cycles in $\ww'$ are denoted by $\gamma'_1, \gamma'_2$ respectively. We need to show that $\langle \gamma_1, \gamma_2 \rangle_{\ww} = \langle \gamma'_1, \gamma'_2\rangle_{\ww'}$. Note that for any vertex $z$ that $\ww$ and $\ww'$ have in common, we have $\langle \gamma_1, \gamma_2 \rangle_{z} = \langle \gamma'_1, \gamma'_2\rangle_{z}$. So we only need to prove the following identity: 
	\[
	\langle \gamma_1, \gamma_{2} \rangle_{u} +\langle \gamma_1, \gamma_{2} \rangle_{v}= \langle \gamma_1', \gamma_{2}'\rangle_{w}+\langle \gamma'_1, \gamma'_{2}\rangle_{v'}+\langle \gamma'_1, \gamma'_{2}\rangle_{x}.
	\]
	We refer reader to \cite[Lemma 4.30]{CGGLSS} for a proof for the identity above.
	\begin{figure}
		\centering
		\includegraphics[trim =1cm 7.8cm 0cm 11.5cm, clip = true, scale = 0.4]{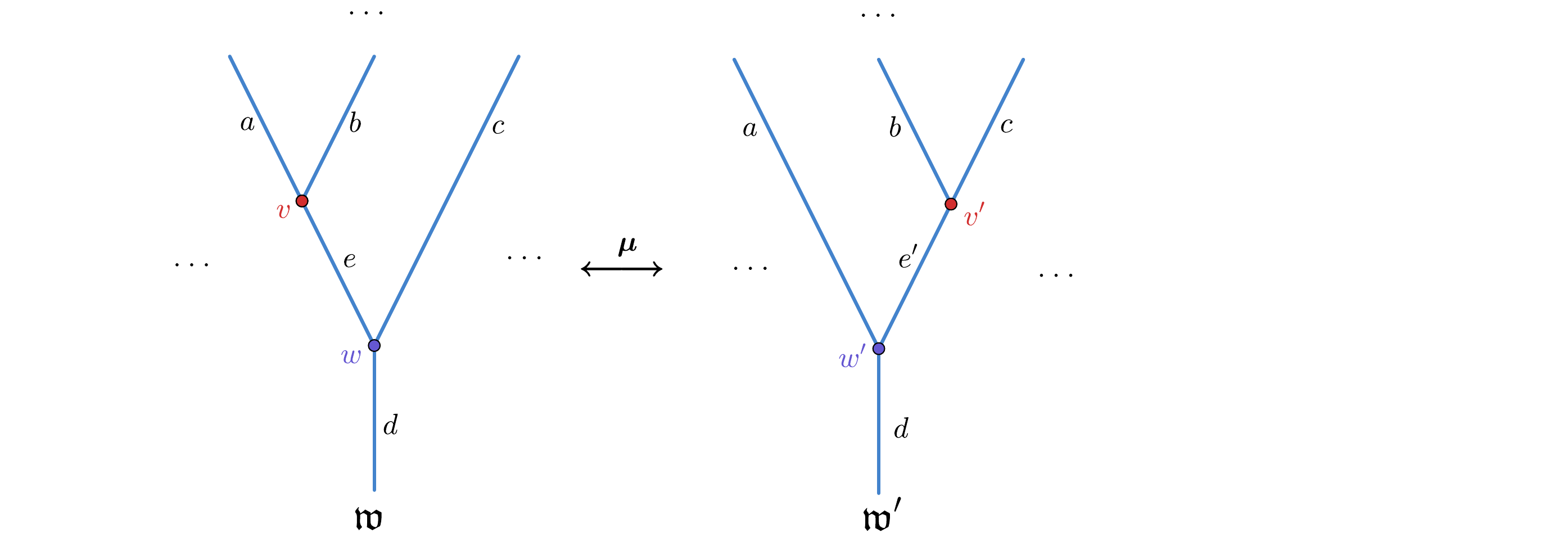}
		\caption{A local mutation move.}
		\label{fig: local mutation move seeds version 1}
	\end{figure}
	Now let us assume that $\ww$ and $\ww'$ are related by a local mutation move, cf.\ Figure~\ref{fig: local mutation move seeds version 1}. We will show that $Q(\ww)$ and $Q(\ww')$ are related by a single quiver mutation. More precisely, we will show that $Q(\ww') = \mu_{\gamma_v}(Q(\ww))$. 
	
	Let $\gamma_1, \gamma_2$ be two (different) cycles other than $\gamma_v$ in $\ww$, the corresponding cycles in $\ww'$ are denoted by $\gamma'_1, \gamma'_2$ respectively. We need to prove the following two identities:
	\[
	\langle \gamma'_1, \gamma_{v'} \rangle_{\ww'} = - \langle \gamma_1 , \gamma_{v} \rangle_{\ww};
	\]
	and
	\[
	\langle \gamma'_1, \gamma'_2\rangle_{\ww'} = \langle \gamma_1, \gamma_2 \rangle_{\ww} + [\langle \gamma_1, \gamma_v \rangle_{\ww}]_+[\langle \gamma_v, \gamma_2\rangle_\ww]_+ - [\langle \gamma_1, \gamma_v \rangle_\ww]_-[\langle \gamma_v, \gamma_2\rangle_\ww]_-.
	\]
	Here $[x]_+ := \max\{x, 0\}$ and $[x]_{-}:= \min\{x, 0\}$. Similarly we refer reader to \cite[Lemma 4.33]{CGGLSS} for a proof for these identities. \qedhere
\end{proof}

We close this section by discussing \emph{reduced Demazure weaves} and the \emph{classification theorem} for reduced Demazure weaves.

\begin{defn}
	Let $d \in \Z_{>0}$. The \emph{$0$-Hecke monoid} $H^d_0$ has generators $\tau_i$, $i\in [1, d-1]$ and relations
	\begin{equation}
		\begin{split}
			\tau_i^2 &= \tau;\\
			\tau_i \tau_j &= \tau_j\tau_i, \ |i-j|>1;\\
			\tau_i\tau_{i+1}\tau_i & = \tau_{i+1}\tau_i\tau_{i+1}.
		\end{split}
	\end{equation}
	A \emph{reduced word} for an element $\delta\in H^d_0$ is defined to be the shortest sequence of generators that multiply to $\delta$ in $H^d_0$. Notice that they are the same as reduced words in the symmetric group. 
	
	For a word $\beta$ in the alphabet $\{\tau_1, \cdots, \tau_{d-1}\}$, the \emph{Demazure product} $\delta(\beta)\in H_0^d$ of $\beta$ is defined to be the product in $H^d_0$. 
\end{defn}

We will use $\tau_i$ and $i$ interchangeably when representing a word. 
\begin{lemma}
	Let $\ww: \beta \rarrow \beta'$ be a Demazure weave with $\beta = \ww_{\text{top}}$ and $\beta' = \ww_{\text{bottom}}$. Then for any row word $\tilde{\beta}$, we have $\delta(\tilde{\beta}) = \delta(\beta)$. In particular, we have $\delta(\beta) = \delta(\beta')$. 
\end{lemma}
\begin{proof}
	By Definition~\ref{defn: scanning process and row words}, $\beta$ and $\tilde{\beta}$ are related by a sequence of transformations: $ii \rarrow i$,  $ij \rarrow ji$ (with $|i-j|>1$) and $iji \rarrow jij$ (with $|i-j| = 1$), and these transformations preserve products in $H_0^d$, hence resulting in the same Demazure product. 
\end{proof}
\begin{defn}
	A Demazure weave $\ww: \beta \rarrow \beta'$ is \emph{reduced} if $\beta'$ is a reduced word for $\delta(\beta)$. 
\end{defn}

\begin{remk}
	Different choices of a reduce word for $\delta(\beta)$ are braid equivalent. This follows from the well-known results for reduce words in the symmetric group $S_d$. 
\end{remk}
\begin{example}
	the Demazure weave  $\ww: 321122 \rarrow 1321$ in Figure~\ref{fig: example of a demazure 4-weave} is reduced, since $1321$ is a reduced word. 
\end{example}

One of the most important results about Demazure weaves is the following classification theorem. It ensures that cluster structures are independent of weave presentation.

\begin{theorem}[{\cite[Theorems 4.12]{CGGS}}]
	\label{thm: demazure classification}
	Let $\ww,\ww':\beta\rarrow\beta'$ be reduced Demazure weaves with the same marked boundary vertices. Then $\ww$ and $\ww'$ are mutation equivalent. Consequently, we have $Q(\ww) \sim Q(\ww')$ by Proposition~\ref{prop: mutation equivalent weaves yield mutation equivalent quivers}. 
\end{theorem}

\begin{remk}
	The assumption that $\ww, \ww'$ are reduced (i.e., $\beta'$ is reduced) is important. Indeed, the two Demazure weaves $(ii)i \rarrow ii$ and $i(ii) \rarrow ii$ are neither equivalent or mutation equivalent. 
\end{remk}

\subsection{Mixed exterior algebra}\label{sec: mixed exterior algebra}

This section reviews foundational concepts of \emph{exterior algebras} over $V$ and $V^*$
following~\cite[Chapters 5 and 6]{Greub}, emphasizing their duality. We then introduce the \emph{mixed wedge} operator $\mixwed$, a key tool for streamlining notation in cluster variable computations and exchange relations.

Let $V$ be a $d$-dimensional vector space over $\C$ and let \[
\bigwedge V = \bigoplus_{k = 0}^{d} {\bigwedge}^k V 
\]
denote the \emph{exterior algebra} over $V$, here $\bigwedge^0 V = \C$. Elements in $\bigwedge^k V$ are called \emph{$k$-extensors}. The exterior product is denoted by $\wedge$ as usual. 
For $u\in \bigwedge V$, we call $u$ \emph{decomposable in $\bigwedge V$} if either $u\in \C$ or there exist $v_1, \cdots, v_r \in V$ such that $u = v_1 \wedge \cdots \wedge v_r$.

The dual space $V^*$ similarly admits an exterior algebra  $\bigwedge V^*$ with (dual) exterior product $\wedge^*$ and {decomposable} elements. The following pairing bridges these structures:

\begin{defn}\label{defn: pairing of vectors and covectors, determinant}
The pairing
\begin{equation}\label{pairing}
	\langle \ , \  \rangle: \bigwedge V \times \bigwedge V^* \rightarrow \C
\end{equation}
is defined by 
\begin{equation}\label{formula: pairing formula}
	\begin{split}
		\langle \lambda, \mu\rangle &= \lambda\mu,  \text{ for } \lambda, \mu \in \C; \\
		\langle v_1\wedge \cdots \wedge v_k, w^*_1\wedge^* \cdots \wedge^* w^*_k \rangle &= \det(w^*_i(v_j)), \text{ for } v_1, \cdots, v_k \in v, w_1^*, \cdots, w_k^*\in V^*;\\
		\langle v_1\wedge \cdots \wedge v_p, w^*_1\wedge^* \cdots \wedge^* w^*_q \rangle &= 0, \text{ for } v_1, \cdots, v_p \in v, w_1^*, \cdots, w_q^*\in V^* \text{ and } p \neq q.
	\end{split}
\end{equation}
Fix a \emph{volume form} $e\in \bigwedge^d V$ and a \emph{dual volume form} $e^*\in \bigwedge^d V^*$ such that $\langle e, e^* \rangle  = 1$. For a top form $c\in \bigwedge^d V$, the \emph{determinant} of $c$ is defined by $\det(c) := \langle c, e^* \rangle \in \C$. For $c = v_1\wedge \cdots \wedge v_n$, we will write $\det(c) = \det(v_1, \cdots, v_n)$. Similarly, the determinant of $c^*\in \bigwedge^d V^*$ is defined by $\det(c^*): = \langle e, c^* \rangle$. 
\end{defn}
For simplicity, we identify $\bigwedge^d V$ with $\bigwedge^0 V = \C$ under the map $\langle \cdot, e^* \rangle $. In other words, we identify a top form $c$ with its determinant $\det(c)$. We identify $\bigwedge^d V^*$ with $\bigwedge^0 V^* = \C$ in a similar way.

The pairing in Definition~\ref{defn: pairing of vectors and covectors, determinant} naturally induces identifications of  $(\bigwedge^{d-k} V^*)^*$ with $\bigwedge^{d-k} V$, for $0\le k \le d$, leading to:

\begin{defn}\label{defn: identification map from V to V^*}
 The exterior product map
\begin{equation}
	{\bigwedge}^k V^* \times {\bigwedge}^{d-k} V^* \stackrel{\wedge^*}{\longrightarrow} {\bigwedge}^d V^* = \C
\end{equation}
	induces isomorphisms
\begin{equation}
	\psi_k: {\bigwedge}^k V^* \longrightarrow ({\bigwedge}^{d-k} V^*)^* = {\bigwedge}^{d-k} V
\end{equation}
via the pairing, yielding a vector space isomorphism:
\[
\psi: = \oplus \psi_k : \bigwedge V^* \rightarrow \bigwedge V.
\]
The \emph{identification map} $\psi$ intertwines two exterior algebra $\bigwedge V$ and $\bigwedge V^*$.
\end{defn}

With this identification, we define a product structure on the exterior algebra $\bigwedge V$ mirroring the (dual) wedge product on the dual exterior algebra $\bigwedge V^*$:

\begin{defn}\label{defn: intersection product}
The \emph{intersection product} $\cap$ on $\bigwedge V$ is the pullback of $\wedge^*$ on $\bigwedge V^*$ via the identification map $\psi$:
\begin{equation}
	u \cap v := \psi(\psi^{-1}(u)\wedge^* \psi^{-1}(v)), \text{ for } u, v \in \bigwedge V.
\end{equation}
In other words, the diagram
\[
\begin{tikzcd}[column sep=large, row sep=large]
	\bigwedge V^* \times \bigwedge V^* \arrow[r, "\psi \times \psi"] \arrow[d, "\wedge^*"] & \bigwedge V\times \bigwedge V \arrow[d, "\cap"]\\
	\bigwedge V^* \arrow[r, "\psi"] & \bigwedge V
\end{tikzcd}
\]
commutes.
The exterior algebra $\bigwedge V$ endowed with the two operators $\wedge, \cap$ is usually called \emph{the Grassmann-Cayley algebra}, cf.\ \cite[Section 3.3]{SturmfelsBernd}. The operators $\wedge, \cap$ are then called the meet and join, respectively; we will not use this terminology. 
\end{defn}

\begin{remk}
	We can similarly define $\cap^*$ in $\bigwedge V^*$. Under the identification map $\psi$, applying $\wedge^*$ in $\bigwedge V^*$ is the same as applying $\cap$ in $\bigwedge V$; and applying $\cap^*$ in $\bigwedge V^*$ is the same as applying $\wedge$ in $\bigwedge V$. We will interchangeably use $\wedge$ with $\cap^*$ (and $\wedge^*$ with  $\cap$) depending on the context. 
\end{remk}

The intersection product’s explicit form reveals its combinatorial and geometric nature as an intersection:

\begin{prop}\label{prop: intersection production expansion formula}
If $u = u_1 \wedge \cdots \wedge u_p\in \bigwedge^p V$ and $v = v_1 \wedge \cdots \wedge v_q\in \bigwedge^q V$ with $p+q\ge d$. Then $u\cap v\in \bigwedge^{p+q-d} V$. Specifically:
\begin{equation}
	u\cap v = \sum_{\sigma} \sgn(\sigma)\det(u_{\sigma(1)}, \cdots, u_{\sigma(d-q)}, v_1, \cdots, v_q) u_{\sigma(d-q+1)}\wedge \cdots \wedge u_{\sigma(p)}
\end{equation}
where the sum is taken over all permutations $\sigma$ of $\{1, 2, \cdots, p\}$ such that $\sigma(1)<\cdots<\sigma(d-q)$ and $\sigma(d-q+1)<\cdots<\sigma(p)$. 

In particular, if $p+q = d$,  then 
\begin{equation}
	u\cap v = \det(u\wedge v).
\end{equation}
Since we identify $u\wedge v$ with $\det(u\wedge v)$, we get, in this case, the equality $u\wedge v = u\cap v$.
Also we have
\begin{equation}\label{equation: intersection is anti-commutative}
	u\cap v = (-1)^{(d-p)(d-q)} v\cap u.
\end{equation}

\end{prop}
\begin{proof}
	See \cite[Section 6.12]{Greub} and \cite[Section 3.3]{SturmfelsBernd}.
\end{proof}

The identification $\psi$ preserves decomposability:
\begin{lemma}\label{lemma: decomposable well-defined under psi}
	An extensor $v\in \bigwedge V$ is decomposable in $\bigwedge V$ if and only if $w^* = \psi^{-1}(v)$ is decomposable in $\bigwedge V^*$. 
\end{lemma}
\begin{proof}
	Let $v\in \bigwedge^{d-k} V$ be a $(d-k)$-extensor, $k\in [0, d]$. For $k = 0, 1, d-1, d$, we know $v$ is always decomposable in $\bigwedge V$ and $w^* = \psi^{-1}(v) \in \bigwedge^{k} V^*$ is always decomposable in $\bigwedge V^*$. Now let us assume that $k\in [2, d-2]$. Assume that $w^*$ is decomposable in $\bigwedge V^*$, i.e., we have $w^* = w^*_1\wedge^* w^*_2 \wedge^* \dots \wedge^* w^*_k$ for some $w^*_i \in V^*$, $1\le i \le k$. Let $v_i:= \psi(w^*_i) \in \bigwedge^{d-1} V$, $1\le i \le k$, then by Definition~\ref{defn: intersection product}, we have
	\begin{align*}
	v = \psi(w^*) =& \psi(w^*_1\wedge^* w^*_2\wedge^* \dots \wedge^* w^*_k) \\
	=& \psi(\psi^{-1}(v_1)\wedge^* \psi^{-1}(v_2)\wedge^* \dots \wedge^* \psi^{-1}(v_k))
	= v_1\cap v_2 \cap \dots \cap v_k.
	\end{align*}
	We claim that the  intersecting a $(d-1)$-extensor $v_1$ with a decomposable (in $\bigwedge V$)~$j$-extensor $u$ is again decomposable (in $\bigwedge V$), $1\le j \le d-1$. 
	
	We prove the claim by an induction on $j$. If $j = 1$, then $v_1\cap u \in \bigwedge^d V = \C$ is decomposable. Assume that $v_1 \cap u'$ is decomposable for any decomposable $(j-1)$-extensor $u'$. Let $u$ be a decomposable $j$-extensor. Then we can write $u = u_1\wedge u_2\wedge \dots \wedge u_j$ with $u_i\in V$, $1\le i \le j$. Now by Proposition~\ref{prop: intersection production expansion formula}, if $u_i\wedge v_1 = 0$ for all $i\in [1, j]$, then $u\cap v_1 = 0$ and we are done. Otherwise, without loss of generality, we assume that $u_{j-1}\wedge v_1 \neq 0$. Then again by Proposition~\ref{prop: intersection production expansion formula}, we have
	\[
	u \cap v_1 = (u_1\wedge u_2\wedge \dots \wedge u_j)\cap v_1 = (u_j\wedge v_1)\cdot(u_1\wedge u_2\wedge \dots \wedge u_{j-1}) - \left((u_1\wedge \dots \wedge u_{j-1})\cap v_1 \right) \wedge u_j.
	\]
	Notice that
	\[
	(u_j\wedge v_1)\cdot(u_1\wedge u_2\wedge \dots \wedge u_{j-1}) = \frac{u_j\wedge v_1}{u_{j-1}\wedge v_1}\left( \left((u_1\wedge \dots \wedge u_{j-1})\cap v_1\right)\wedge u_{j-1}\right).
	\]
	Therefore
	\[
	u\cap v_1 = \left((u_1\wedge \dots \wedge u_{j-1})\cap v_1 \right) \left(\frac{u_j\wedge v_1}{u_{j-1}\wedge v_1}u_{j-1} - u_j\right). 
	\]
	is a wedge of two decomposable extensors, hence is decomposable. 
	
	Now by the claim, it is clear that $v$ is decomposable (by an induction on $k$). This completes the proof that $w^*$ is decomposable in $\bigwedge V^*$ implies $v$ is decomposable in $\bigwedge V$. The reverse implication follows symmetrically.
\end{proof}


Combining $\wedge$ and $\cap$ into a unified operator:
\begin{defn}
	The \emph{mixed wedge operator} $\mixwed$ is the bilinear map
	\begin{equation}
		\mixwed: \bigwedge V \times \bigwedge V \rightarrow \bigwedge V
	\end{equation}
	such that for $u\in \bigwedge^p V$ and $v\in \bigwedge^q V$, $p, q \in [0, d]$, we have 
	\begin{equation}
		u\mixwed v := \begin{cases}
			u\wedge v, & \text{if }\  p+q \le d,\\
			u\cap v, & \text{if }\  p+q\ge d.
		\end{cases}
	\end{equation}
The mixed wedge operator $\mixwed$ extends naturally to the \emph{mixed exterior algebra},  where we identify $\bigwedge V$ and $\bigwedge V^*$ under $\psi$. 
\end{defn}

Note that $\mixwed$ is commutative up to sign, i.e., we have 
\begin{equation}
	u\mixwed v := \begin{cases}
		(-1)^{pq}v\mixwed u, & \text{if }\  p+q \le d,\\
		(-1)^{(d-p)(d-q)}v\mixwed u, & \text{if }\ d\le p+q\le 2d.
	\end{cases}
\end{equation}
for $u\in \bigwedge^p V$ and $v\in \bigwedge^q V$, $0\le p, q\le d$; cf.\ Equation (\ref{equation: intersection is anti-commutative}).

\begin{remk}
The product $\mixwed$ is not always associative. Let $u\in \bigwedge^p V, v\in \bigwedge^q V, w\in \bigwedge^r V$.  If $p+q+r\le d$ or $p+q+r \ge 2d $, then the mixed wedge is associative:
\[
(u \mixwed v)\mixwed w = u \mixwed (v \mixwed w),
\]
because we are either applying $\wedge$ twice, or applying $\cap$ twice. On the other hand, if $d< p+q+r< 2d$, then in general 
\[
(u \mixwed v)\mixwed w \neq u \mixwed (v \mixwed w),
\]
because we are mixing $\wedge$ and $\cap$, making the product non-associative. 
\end{remk}

Due to the non-associativity of $\mixwed$, we resolve ambiguity by adopting the \textbf{right-to-left} evaluation convention when parentheses are omitted. For example, $u\mixwed v\mixwed w := u \mixwed (v \mixwed w)$.

\begin{example}
	Let $u\in \bigwedge^k V$ and $u^*\in \bigwedge^k V^*$. Then 
	\[
	u \mixwed u^* = u\wedge u^* = \langle u, u^* \rangle
	\]
	and 
	\[
	u^* \mixwed u = u^* \wedge u =   (-1)^{k(d-k)} u\wedge u^* = (-1)^{k(d-k)}\langle u, u^* \rangle.
	\]
\end{example}

\begin{example}
	Let $v_1, \cdots, v_p \in V$ and $v \in \bigwedge^q V$ with $p+q\le d$. Then
	\[
	v_1\mixwed \cdots \mixwed v_p \mixwed v = (v_1\wedge \cdots \wedge v_p) \wedge v.
	\]
\end{example}
\begin{lemma}\label{lemma: mixwed wedge of decomposable is decomposable}
	The mixed wedge of two decomposable extensors is decomposable. 
\end{lemma}
\begin{proof}
	This follows from Lemma~\ref{lemma: decomposable well-defined under psi}. Notice that the mixed wedge operator is either $\wedge$ in $\bigwedge V$ or $\cap = \wedge^*$ in $\bigwedge V^*$; in either case, the wedge (resp., dual wedge) of two decomposable extensors (in their corresponding space) is again decomposable. 
\end{proof}

The lemma below plays a pivotal role in simplifying mixed wedge products, particularly in the analysis of cluster variables and their exchange relations.

\begin{lemma}\label{lemma: derivative formula for mixwed wedge}
	Let $v\in V$,  $u^*_1 \in \bigwedge^p V^*$ and $u^*_2\in \bigwedge^q V^*$ with $p, q \ge 1, p+q \le d$.  Then 
	\begin{equation}\label{formula: derivative formula for mixed wedge}
		v\mixwed (u^*_1\mixwed u^*_2) = (-1)^{q}(v\mixwed u^*_1)\mixwed u^*_2 +  u^*_1 \mixwed (v\mixwed u^*_2).
	\end{equation}
	Similarly, if $v_1 \in \bigwedge^p V$, $v_2\in \bigwedge^q V$ and $u^*\in V^*$ with $p, q\ge 1, p+q \le d$, then
	\begin{equation}\label{formulat: derivative formula for mixed wedge, dual}
		u^*\mixwed(v_1\mixwed v_2) = (-1)^q(u^*\mixwed v_1)\mixwed v_2 + v_1 \mixwed (u^*\mixwed v_2).
	\end{equation}
\end{lemma}
\begin{proof}
	See \cite[Proposition 5.14.1]{Greub}.
\end{proof}

The following computation illustrates mechanics:

\begin{example}
	Let $d = 3$, $v_1, v_2, v_3 \in V$ and $u^*_1, u^*_2, u^*_3\in V^*$, then we have 
	\begin{equation}
		v_1\mixwed v_2\mixwed v_3 = \det(v_1, v_2, v_3), \quad u^*_1\mixwed u^*_2\mixwed u^*_3 = \det(u^*_1, u^*_2, u^*_3).
	\end{equation}
	Let us compute
	\begin{equation}
		I = v_1\mixwed v_2 \mixwed u^*_1 \mixwed v_3 \mixwed u^*_2 \mixwed u^*_3
	\end{equation}
	as a polynomial in the Weyl generators (cf.\ Definition \ref{defn: Weyl generators}). 
	Recall that we apply $\mixwed$ from right to left. We first calculate 
	\begin{align*}
	v_3 \mixwed (u^*_2 \mixwed  u^*_3) \stackrel{(\ref{formula: derivative formula for mixed wedge})}{=\joinrel=} -u^*_2(v_3)u^*_3 + u^*_3(v_3)u^*_2,
	\end{align*}
	then we compute
	\begin{align*}
	u^*_1 \mixwed (-u^*_2(v_3)u^*_3 + u^*_3(v_3)u^*_2) = - u^*_2(v_3)u^*_1 \wed u^*_3 + u^*_3(v_3)u^*_1\wed u^*_2,
	\end{align*}
	and
	\begin{align*}
	& v_2 \mixwed (- u^*_2(v_3)u^*_1 \wed u^*_3 + u^*_3(v_3)u^*_1\wed u^*_2) \\
	=&- u^*_2(v_3) (v_2 \mixwed (u^*_1 \mixwed  u^*_3)) + u^*_3(v_3) (v_2 \mixwed (u^*_1\wedge^*u^*_2))\\
	 =&  u^*_1(v_2)u^*_2(v_3)u^*_3 - u^*_2(v_3)u^*_3(v_2)u^*_1 - u^*_1(v_2)u^*_3(v_3)u^*_2 + u^*_2(v_2)u^*_3(v_3)u^*_1\quad (\text{by (\ref{formula: derivative formula for mixed wedge})}).
	\end{align*}
	Finally, we compute
	\begin{align*}
	 I = &v_1 \mixwed (u^*_1(v_2)u^*_2(v_3)u^*_3 - u^*_2(v_3)u^*_3(v_2)u^*_1 - u^*_1(v_2)u^*_3(v_3)u^*_2 + u^*_2(v_2)u^*_3(v_3)u^*_1)\\
	  =& u^*_1(v_2)u^*_2(v_3)u^*_3(v_1) - u^*_1(v_1)u^*_2(v_3)u^*_3(v_2) - u^*_1(v_2)u^*_2(v_1) u^*_3(v_3)+ u^*_1(v_1)u^*_2(v_2)u^*_3(v_3),
	\end{align*}
	and we are done. Note that the last two steps can be combined into a single step if we group $v_1\wedge v_2$ first:
	\begin{align*}
		I &= v_1\mixwed v_2 \mixwed u^*_1 \mixwed v_3 \mixwed u^*_2 \mixwed u^*_3 \\
		& = (v_1\mixwed v_2) \mixwed u^*_1 \mixwed v_3 \mixwed (u^*_2 \mixwed u^*_3)\\
		& = (v_1\mixwed v_2) \mixwed (-u^*_2(v_3)u^*_1\mixwed u^*_3 + u^*_3(v_3)u^*_1\mixwed  u^*_2)\\
		& = -u^*_2(v_3)\langle v_1\wedge v_2, u^*_1\wedge^*  u^*_3\rangle + u^*_3(v_3)\langle v_1\wed v_2, u^*_1\wedge^* u^*_2 \rangle \\
		& = u^*_1(v_2)u^*_2(v_3)u^*_3(v_1) - u^*_1(v_1)u^*_2(v_3)u^*_3(v_2) - u^*_1(v_2)u^*_2(v_1)u^*_3(v_3) + u^*_1(v_1)u^*_2(v_2)u^*_3(v_3).
	\end{align*}	
	Alternatively, we can use formula (\ref{formula: derivative formula for mixed wedge}) in a slightly different way:
	\begin{align*}
		I = & v_1\mixwed v_2 \mixwed u^*_1 \mixwed v_3 \mixwed u^*_2 \mixwed u^*_3\\
		 =& v_1\mixwed v_2 \mixwed (u^*_1 \mixwed (v_3 \mixwed (u^*_2 \mixwed u^*_3)))\\
		=& v_1\mixwed v_2 \mixwed (v_3 \mixwed (u^*_1 \mixwed (u^*_2 \mixwed u^*_3))) -v_1\mixwed v_2 \mixwed (u^*_1(v_3)( u^*_2 \mixwed u^*_3) )\\
		=& \det(v_1, v_2 , v_3)\det(u^*_1, u^*_2,  u^*_3) - u^*_1(v_3) \langle v_1 \wedge v_2, u^*_2 \wedge^* u^*_3\rangle \\
		=&\det(v_1, v_2 , v_3)\det(u^*_1, u^*_2,  u^*_3) - u^*_1(v_3)u_2^*(v_1)u_3^*(v_2) + u^*_1(v_3)u_2^*(v_2)u_3^*(v_1).
	\end{align*}	
	Notice that we get a different representation of $I$ as a polynomial in the Weyl generators. 
\end{example}
The following example shows that a mixed wedge product can factor or vanish due to 
repeated terms.
\begin{example}
	Let $d = 3$, $v_1, v_2, v_3 \in V$ and $u^* \in V^*$. Then we have 
	\[v_3 \mixwed v_2\mixwed u^* \mixwed v_1 \mixwed v_2 =  u^*(v_2)\cdot\det(v_3, v_2,  v_1)
	\]
	and 
	\[
	v_1 \mixwed v_2 \mixwed u^* \mixwed v_1 \mixwed v_2 = 0.
	\]
\end{example}

\newpage

\section{Cluster Structures in Mixed Grassmannians}\label{chap: cluster structures on mixed grassmannians}

\subsection{Decorated flags}\label{sec: decorated flags}

In this section, we introduce the notion of \emph{decorated flags} (see, e.g., \cite{CasalsLeSBWeng, CasalsWeng, Weng}),
which encode nested subspaces of $V$ with additional data. These structures arise naturally in the study of cluster algebras via the quotient $\text{SL}(V)/N$, the \emph{decorated flag variety}, where $N$ is the upper triangular unipotent subgroup of $\slv = \text{SL}(n, \C)$. We interpret decorated flags combinatorially as tuples of nested (decomposable) extensors, laying groundwork for their role as ``coordinates" in subsequent cluster variable constructions.

\begin{defn}{\label{flag}}
	A (complete) \emph{decorated flag} is an element of the \emph{decorated flag variety} $\slv/N$. Equivalently, a \emph{decorated flag} over $V$ is a tuple
	\begin{equation}
		\mcf= (\mcf^1, \mcf^2, \dots, \mcf^{d-1})
	\end{equation}
	such that
	\begin{enumerate}[wide, labelwidth=!, labelindent=0pt]
		\item[\textbf{(a)}] for each $1\le k \le d-1$, we have $\mcf^k \in \bigwedge^k V - \{0\}$;
		\item[\textbf{(b)}] for each $1\le k <d-1$, there exists $v\in V$ such that $\mcf^{k+1} = \mcf^k \wedge v$.		
	\end{enumerate}
\end{defn}

\begin{remk}
	The group $\slv$ acts on decorated flags via its natural action on $\bigwedge V$, preserving the nested condition in Definition \ref{flag}\textbf{(b)}. 
\end{remk}

The duality between decorated flags in $V$ and $V^*$ emerges via the identification map $\psi$.  Applying $\psi^{-1}$ reverses flag gradings while preserving nesting. 

\begin{lemma}\label{lemma: k+1 = k wedge v equivalent k = k+1 wedge v star}
	Let $1\le k \le d-1$, and let $\mcf^k\in \bigwedge^k V$ and $\mcf^{k+1} \in \bigwedge^{k+1}$ be nonzero decomposable extensors. Then $\mcf^{k+1} = \mcf^k \wedge v$ for some $v\in V$ if and only if $\mcf^k = \mcf^{k+1}\mixwed v^*$ for some $v^*\in V^*$. 
\end{lemma}
\begin{proof}
	Assume that $\mcf^{k+1} = \mcf^k \wedge v$ and write $\mcf^k = v_1\wedge \cdots \wedge v_k$ with $v_i \in V$. Now take $v^*\in V^*$ such that $v^*(v_i) = 0$ for $1\le i \le k$ and $v^*(v) = (-1)^{k}$, then we have
	\begin{equation}
		\mcf^{k+1} \mixwed v^* =(-1)^{d-k-1}v^* \mixwed \mcf^{k+1} =(-1)^{d-k-1}(\mcf^k \mixwed (v^*\mixwed v)) = \mcf^k
	\end{equation}
	by lemma~\ref{lemma: derivative formula for mixwed wedge}. The converse is similar. 
\end{proof}

\begin{remk}\label{remk: pull back of flags over V are flags over V dual}
	Let $\mcf$ be a decorated flag over $V$. It's dual in $V^*$
	\[
	\mcf^*: = \psi^{-1}(\mcf) := (\psi^{-1}(\mcf^{d-1}), \dots, \psi^{-1}(\mcf^2), \psi^{-1}(\mcf^1))
	\]
	is a decorated flag in $V^*$, where $\psi$ is the identification map, cf.\ Definition~\ref{defn: identification map from V to V^*}. Indeed, by Lemma~\ref{lemma: k+1 = k wedge v equivalent k = k+1 wedge v star}, we have
	\begin{equation*}
		({\mcf^*})^{k+1} = \psi^{-1}(\mcf^{d-k-1}) = \psi^{-1}(\mcf^{d-k} \mixwed v^*) = \psi^{-1}(\mcf^{d-k}) \wedge^* v^*= (\mcf^*)^{k}\wedge^* v^*
	\end{equation*}
	for some $v^*\in V^*$.
\end{remk}

The following result is foundational to our framework: the vanishing and commutativity properties of mixed wedges in decorated flags ensure algebraic consistency, which is critical for explicit computations of cluster variables and their exchange relations.

\begin{prop}\label{prop: wedge of any two within a flag is zero}
	Let $\mcf$ be a decorated flag. Let $1\le k, k'\le d-1$. Then
	\begin{equation}\label{equation: wedges in the same flag is zero}
		\mcf^k\mixwed \mcf^{k'} = 0. 
	\end{equation}
Moreover, for any $v \in \bigwedge V$ we have 
\begin{equation}\label{equation: wedges in the same flag commutes}
	\mcf^k\mixwed (\mcf^{k'} \mixwed v) = \mcf^{k'}\mixwed (\mcf^k  \mixwed v).
\end{equation}
\end{prop}
\begin{proof}
	As $\mixwed$ is anti-commutative, we may assume that $k\le k'$. Suppose that $k+k'\le d$, we have $\mcf^{k'} = \mcf^k \wedge u$ with $u\in \bigwedge^{k'-k}V$, hence 
	\[
	\mcf^k\mixwed \mcf^{k'} = \mcf^k\wedge  \mcf^{k'} = \mcf^k\wedge \mcf^{k} \wedge u= 0.
	\]
	Suppose that $k+k'> d$, then $(d-k) + (d-k') < d$, and by duality (cf.\ Remark~\ref{remk: pull back of flags over V are flags over V dual}), it reduces to the previous case. This completes the proof of (\ref{equation: wedges in the same flag is zero}).
	
	Now let us prove (\ref{equation: wedges in the same flag commutes}). By linearity we may assume that $v\in \bigwedge^i V$. By symmetry we may assume that $k\le k'$. Write $\mcf^{k'} = \mcf^k \wedge u$ with $u\in  \bigwedge^{k'-k} V$. If $i+k+k'\le d$ or $i+k+k' \ge 2d$, then the mixed wedge operator is associative, and both sides vanish by (\ref{equation: wedges in the same flag is zero}). Hence we may assume that $d< i + k + k' < 2d$. 
	
	Suppose that $i < d-k'$. Then $i+k \le i+k' < d$. By the shuffling formula for the intersection product, cf.\ Proposition~\ref{prop: intersection production expansion formula}, we can conclude that
	\[
	\mcf^k\mixwed (\mcf^{k'} \mixwed v)  = \mcf^k \cap (\mcf^k\wedge u \wedge v) = 0
	\]
	and 
	\[
	 \mcf^{k'}\mixwed (\mcf^k  \mixwed v) = (\mcf^k \wedge u)\cap (\mcf^k \wedge v) = 0.
	\]
	By duality, the result follows in the case when $i > d-k$. 
	
	We next assume that $d-k' \le i \le d-k$. This is the case when both sides do not vanish. We have
	\[
	\mcf^k\mixwed (\mcf^{k'} \mixwed v) = \mcf^k \wedge \big((\mcf^k\wedge u)\cap v\big) = \mcf^k \wedge \big( u \cap (\mcf^k\wedge v)\big) 
	\]
	where the last equality follows from the shuffling formula for the intersection product. Similarly we have
	\[
	\mcf^{k'}\mixwed (\mcf^k  \mixwed v) = (\mcf^k\wedge u) \cap (\mcf^k \wedge v) = \mcf^k \wedge \big( u \cap (\mcf^k\wedge v)\big).
	\]
	This completes the proof of (\ref{equation: wedges in the same flag commutes}).
\end{proof}

\begin{remk}
	When $d-k' \le i \le d-k$, identity (\ref{equation: wedges in the same flag commutes}) is a generalized version of the following identity for vector spaces. Let $U, V, W$ be vector subspaces of $\C^n$ with $U\subseteq V$, then 
	\[
		U + (V\cap W) = V\cap(U+W).
		\]
\end{remk}

We now associate a tuple of $n$ cyclically ordered decorated flags to a chosen signature, crucial for cluster coordinates.

\begin{defn}\label{defn: element of Vsigma}
Let $\sigma$ be a size $n$ signature of type $(a, b)$, cf.\ Definition~\ref{defn: signature}. Recall from Definition~\ref{defn: V sigma and rsigma} that the configuration space $V^\sigma$ is the rearrangement of the direct product  $V^a\times(V^*)^b$ where the $j$-th factor is $V$ (resp., $V^*$) if $\sigma(j) = 1$ (resp., $\sigma(j) = -1$), for $1\le j \le n$. Accordingly we will usually write an element $\uu\in V^\sigma$ as a tuple
\[
\uu = (u_1, u_2, \dots, u_n),
\]
where $u_j \in V$ if $\sigma(j) = 1$ and $u_j \in V^*$ if $\sigma(j) = -1$, $1\le j \le n$. We will extend the notation $u_j$ to all $j\in \Z$ by periodicity, i.e., $u_{j+n} = u_{j}$ for $j\in \Z$. Sometimes we will write $u^*_j$ instead of $u_j$ when $\sigma(j) = -1$ in order to indicate that $u_j$ is a covector.
\end{defn}

\begin{example}
	Let $\sigma = [\bullet\, \bullet\, \circ\, \circ\, \bullet]$. Recall that $\bullet$ stands for $1$ and $\circ$ stands for $-1$. Then $V^\sigma = V\times V\times V^*\times V^*\times V$, and $\uu = (u_1, u_2, u_3^*, u_4^*, u_5)\in V^\sigma$.
\end{example}

\begin{defn}\label{defn: generic point}
	Let $\uu\in V^\sigma$ and $j\in \Z$. For $j\in [1, n]$, define 
	\[
	f_j(\uu) := u_j \mixwed u_{j+1} \mixwed \cdots \mixwed u_{j'-1}\mixwed u_{j'} \in \C
	\] 
	where $j' \ge j$ is the smallest integer such that
	\begin{equation*}
		\sum_{i = j}^{j'} \sigma(i) \equiv 0 \bmod d.
	\end{equation*}
	Recall that by convention we take $\mixwed$ from right to left, i.e., the mixed wedge above should be viewed as $u_j \wed (u_{j+1}\wed \cdots \wed (u_{j'-1}\wed u_{j'}))$. 
	
	These $f_j$'s will be the frozen variables for the cluster algebra. 
	Let 
	\[
	f(\uu) = \prod_{j = 1}^n f_j(\uu).
	\]
	We say $\uu$ is \emph{generic} if $f(\uu) \neq 0$. Genericity ensures non-degenerate flag constructions.
\end{defn}

\begin{defn}\label{defn: tuple of flags mcf associated with sigma and u}
	Let $\sigma$ be a $d$-admissible signature (cf.\ Definition~\ref{defn: admissible signature}) and let $\uu\in V^\sigma$ such that $f(\uu) \neq 0$. The tuple of decorated flags
	\begin{equation}
		\vec{\mcf} = \vec{\mcf}(\uu) = (\mcf_1, \mcf_2, \dots, \mcf_{n}, \mcf_{n+1}= \mcf_1)
	\end{equation}
	is defined as follows.
	
	For every $j\in [1, n]$ and $k\in [1, d-1]$, take the smallest integer $j'\ge j$ such that 
	\begin{equation}
		\sum_{i = j}^{j'} \sigma(i) \equiv k \mod d.
	\end{equation}
	Then set
	\begin{equation}
		\mcf_j^k := u_j \mixwed u_{j+1} \mixwed \cdots \mixwed u_{j'-1}\mixwed u_{j'} \in {\bigwedge}^k V,
	\end{equation}
	Let
	\begin{equation}
		\mcf_j := (\mcf_j^1,\mcf_j^2, \dots, \mcf_j^{d-1}).
	\end{equation}
We will extend $\mcf_j$ to all $j\in \Z$ by periodicity, i.e., $\mcf_{j+n}= \mcf_{j}$ for $j\in \Z$. 

To be precise, we cannot call $\mcf_j$ a decorated flag yet as we haven't check that the conditions for a decorated flag are satisfied for $\mcf_j$. This is done in the next Proposition. 
\end{defn}

\begin{prop}\label{prop: mcf is a decorated flag}
	Let $j\in \Z$. Then $\mcf_j$ is a decorated flag over $V$ for generic $\uu\in V^\sigma$. 
\end{prop}

The proof of Proposition \ref{prop: mcf is a decorated flag} will rely on two lemmas, but first let us work an example.

\begin{example}
	Let $\sigma = [\bullet\, \bullet\, \bullet\, \bullet\, \circ]$ with $n = 5, d = 4$. Then 
	\begin{align*}
		\mcf_1 &= (u_1, u_1\mixwed u_2, u_1 \mixwed u_2\mixwed u_3),  &f_1 &= \det(u_1, u_2, u_3, u_4);\\
		\mcf_2 & = (u_2, u_2\mixwed u_3, u_2\mixwed u_3\mixwed u_4),  &f_2 &= u_2\mixwed u_3 \mixwed u_4 \mixwed u_5^*\mixwed u_1 \mixwed u_2;\\
		\mcf_3 & = (u_3, u_3 \mixwed u_4, u_3\mixwed u_4\mixwed u^*_5\mixwed u_1\mixwed u_2), &f_3 &= u_3 \mixwed u_4 \mixwed u_5^* \mixwed u_1 \mixwed u_2 \mixwed u_3;\\
		\mcf_4 & = (u_4, u_4\mixwed u^*_5\mixwed u_1\mixwed u_2, u_4\mixwed u^*_5\mixwed u_1\mixwed u_2\mixwed u_3), &f_4 &= u_4 \mixwed u_5^* = u^*_5(u_4);\\
		\mcf_5 & = (u^*_5\mixwed u_1\mixwed u_2, u^*_5\mixwed u_1\mixwed u_2\mixwed u_3, u^*_5), &f_5 &= u_5^* \mixwed u_1 = -u_5^*(u_1).
	\end{align*}
	The condition that $f_j \neq 0$ for all $j\in [1, 5]$ ensures that $\mcf_j$ are all non-degenerate. The nested condition is quite subtle. Take $\mcf_5$ as an example, notice that 
	\[
	u_5^* \mixwed u_1 \mixwed u_2 \mixwed u_3 = u_5^* \mixwed (u_1 \mixwed u_2 \mixwed u_3)  \neq (u_5^*\mixwed u_1 \mixwed u_2) \mixwed u_3.
	\]
	Instead, we have
	\[
	u_5^* \mixwed u_1 \mixwed u_2 \mixwed u_3 = (u_5^*\mixwed u_1 \mixwed u_2) \mixwed (-u_3 + \frac{u_5^*(u_3)}{u_5^*(u_1)}u_1).
	\]
\end{example}

\begin{lemma}\label{lemma: iterated mixed wedges are decomposable}
	For $j\in \Z$ and $k\in [1, d-1]$, the $k$-extensor $\mcf_j^k$ is decomposable. 
\end{lemma}
\begin{proof}
	This follows from Lemma~\ref{lemma: mixwed wedge of decomposable is decomposable}, as $\mcf_j^k$ is defined to the mixed wedges of vectors and covectors. 
\end{proof}

Let $u$ (resp., $v$) be either a vector or a covector. We say $u$ and $v$ are \emph{of the same type} if they are both vectors or both covectors.

\begin{lemma}\label{lemma: we can change the order of wedge by adjusting the last vectors}
	Let $r\ge 2$ and $w = w_1 \mixwed w_2 \mixwed \cdots \mixwed w_{r-1} \mixwed w_{r} \neq 0$ where $w_j \in V$ or $w_j\in V^*$, $1\le j\le r$. Suppose that $w_j \mixwed w_{j+1} \mixwed \cdots \mixwed w_r \notin \bigwedge^0 V$ (and $\bigwedge^d V$), for all $2\le j\le r$.  Then there exists $w'_r$ of the same type as $w_r$ such that $w = (w_1 \mixwed w_2 \mixwed \cdots \mixwed w_{r-1}) \mixwed w'_{r}$.
\end{lemma}
\begin{proof}
	We prove the statement by induction on $r$. When $r = 2$, we take $w'_2 = w_2$. 
	Now assume that the statement  is true for $r$, consider $w = w_1 \mixwed w_2 \mixwed \cdots \mixwed w_{r} \mixwed w_{r+1}$.  Without loss of generality, we may assume that $w_{r+1}\in V$. We have 
	\[
	w = w_1 \mixwed w_2 \mixwed \cdots \mixwed w_{r} \mixwed w_{r+1} = w_1 \mixwed (w_2 \mixwed \cdots \mixwed w_{r} \mixwed w_{r+1}) = w_1 \mixwed w'
	\]
	with $w' = w_2 \mixwed \cdots \mixwed w_{r} \mixwed w_{r+1}$. Apply induction hypothesis to $w'$, there exists $w_{r+1}'\in V$ such that $w' = (w_2 \mixwed \cdots \mixwed w_{r}) \mixwed w'_{r+1}$. Therefore
	\[
	w = w_1 \mixwed w_2 \mixwed \cdots \mixwed w_{r} \mixwed w_{r+1} =  w_1 \mixwed w' = w_1 \mixwed ((w_2 \mixwed \cdots \mixwed w_{r}) \mixwed w'_{r+1}) = w_1 \mixwed w'' \mixwed w'_{r+1}
	\]
	where $w'' = w_2 \mixwed \cdots \mixwed w_{r}$.  
	
	Now by Lemma~\ref{lemma: iterated mixed wedges are decomposable}, $w''$ is decomposable. Suppose that $w_1\in V$. Then we have
	\[
	w = w_1 \wedge w'' \wedge w'_{r+1} = (w_1 \wedge w'') \wedge w'_{r+1} = (w_1 \mixwed w_2 \mixwed \cdots \mixwed w_{r}) \mixwed w'_{r+1}.
	\]
	and we are done. Suppose that $w_1\in V^*$. Similar to the proof of Lemma~\ref{lemma: decomposable well-defined under psi}, there exists $w''_{r+1} \in V$ such that 
	\begin{align*}
	w &= w_1 \mixwed w'' \mixwed w'_{r+1} = w_1 \cap (w'' \wedge w'_{r+1}) =(w_1 \cap w'') \wedge w''_{r+1} \\
	&= (w_1 \mixwed w'') \mixwed w''_{r+1} =(w_1 \mixwed w_2 \mixwed \cdots \mixwed w_{r}) \mixwed w''_{r+1}. \qedhere
	\end{align*}
\end{proof}

\begin{proof}[Proof of Proposition~\ref{prop: mcf is a decorated flag}]
	Let $j\in [1, n]$ and $k\in [1, d-1]$. We first show that $\mcf_j^k \neq 0$. 
	As in Definition~\ref{defn: tuple of flags mcf associated with sigma and u}, let $j'(k) \ge j$ be the smallest integer such that 
	\[
	\sum_{i = j}^{j'(k)} \sigma(i) \equiv k \bmod d.
	\]
	Then $\mcf_j^k = u_j \mixwed u_{j+1} \mixwed \cdots \mixwed u_{j'(k)-1}\mixwed u_{j'(k)}$.
	Now let $j_0 < j'(k)$ be the largest integer such that 
	\[
	\sum_{i = j_0}^{j'(k)} \sigma(i) \equiv 0 \bmod d.
	\]
	Notice that $j_0 < j$ and 
	\[
	f_{j_0} = u_{j_0}\mixwed u_{j_0+1} \mixwed \cdots \mixwed u_{j'(k)-1}\mixwed u_{j'(k)} = u_{j_0}\mixwed u_{j_0+1} \mixwed \cdots \mixwed u_{j-1} \mixwed \mcf_j^k.
	\]
	Since $f_{j_0}\neq 0$, we conclude that $\mcf_j^k\neq 0$.

	We next show that the nested condition for decorated flags is satisfied for $\mcf_j$. Similarly we denote by $j'(k+1) \ge j$ the smallest integer such that 
	\[
	\sum_{i = j}^{j'(k)} \sigma(i) \equiv k+1 \bmod d.
	\]
Then $\mcf_j^{k+1} = u_j \mixwed u_{j+1} \mixwed \cdots \mixwed u_{j'(k+1)-1}\mixwed u_{j'(k+1)}$. 
	By duality, we may assume that $j'(k) < j'(k+1)$. Then we have 
	\begin{align*}
		\mcf_j^{k+1} &= u_j \mixwed u_{j+1} \mixwed \cdots \mixwed u_{j'(k+1)-1}\mixwed u_{j'(k+1)}\\
		&= u_j \mixwed u_{j+1} \mixwed \cdots \mixwed u_{j'(k)} \mixwed (u_{j'(k)+1} \mixwed \cdots \mixwed u_{j'(k+1)-1}\mixwed u_{j'(k+1)})\\
		& = u_j \mixwed u_{j+1} \mixwed \cdots \mixwed u_{j'(k)} \mixwed w
	\end{align*}
	where $w = u_{j'(k)+1} \mixwed \cdots \mixwed u_{j'(k+1)-1}\mixwed u_{j'(k+1)}\in \bigwedge^1 V = V$. Notice that the condition of Lemma~\ref{lemma: we can change the order of wedge by adjusting the last vectors} is satisfied for $\mcf_j^{k+1}$ by the definition of $j'(k+1)$. Hence by Lemma~\ref{lemma: we can change the order of wedge by adjusting the last vectors}, there exists $v\in V$ such that $\mcf_j^{k+1} =(u_j \mixwed u_{j+1} \mixwed \cdots \mixwed u_{j'(k)}) \mixwed v = \mcf_j^k \wedge v$. 
\end{proof}

The tuple $\vec{\mcf}(\uu)$ exhibits a recursive structure where each decorated flag $\mcf_j$ can be uniquely reconstructed from its successor and either its first component (when $\sigma(j) = 1$) or its last component (when $\sigma(j) = -1$). This recursive dependency is universal: any cyclic tuple of decorated flags satisfying these conditions must be of the form $\vec{\mcf}(\uu)$ for some generic $\uu \in V^\sigma$.

\begin{prop}\label{prop: recursive relation between flags associated with a signature}
		Let $\sigma$ be a $d$-admissible signature. Let $\uu\in V^\sigma$ with $f(\uu) \neq 0$. Then the tuple of decorated flags
	\begin{equation}
		\vec{\mcf}(\uu) = (\mcf_1, \mcf_2, \dots, \mcf_{n}, \mcf_{n+1}= \mcf_1)
	\end{equation}
	satisfies the following recursive relations. Let $j\in [1, n]$.
	\begin{itemize}[wide, labelwidth=!, labelindent=0pt]
		\item Suppose that $\sigma(j) = 1$. Then $\mcf_j$ is uniquely determined by $\mcf_{j+1}$ and $\mcf_j^1$ via
		\begin{equation}\label{equation: wedge iteration condition black}
			\mcf_j^k = \begin{cases}
				\mcf_j^1, &\quad \text{if} \quad  k=1;\\
				\mcf_j^1 \mixwed \mcf_{j+1}^{k-1}, &\quad \text{if} \quad 1< k \le d-1.	
			\end{cases}
		\end{equation}
	\item Suppose that $\sigma(j) = -1$. Then $\mcf_j$ is uniquely determined by $\mcf_{j+1}$ and $\mcf_j^{d-1}$ via
	\begin{equation}\label{equation: wedge iteration condition white}
		\mcf_j^k = \begin{cases}
			\mcf_j^{d-1}, &\quad \text{if} \quad k = d-1; \\
			\mcf_j^{d-1} \mixwed \mcf_{j+1}^{k+1}, &\quad \text{if} \quad 1\le k < d-1.
		\end{cases}
	\end{equation}
	\end{itemize}
	Conversely, let $\vec{\mcg} = (\mcg_1, \mcg_2, \cdots, \mcg_n, \mcg_{n+1} = \mcg_1)$ be a tuple of decorated flags satisfying the recursive relations (\ref{equation: wedge iteration condition black}) and (\ref{equation: wedge iteration condition white}). Then $\vec{\mcg} = \vec{\mcf}(\uu)$ for some $\uu\in V^\sigma$. 
\end{prop}
\begin{proof}
	Let $j\in [1,n]$. If $\sigma(j) = 1$, then $\mcf_j^1 = u_j$. For $1< k \le d-1$, let $j'\ge j$ be the smallest integer such that 
	\begin{equation}
		\sum_{\ell = j}^{j'} \sigma(\ell) \equiv k \mod d.
	\end{equation}
	Since $\sigma(j) = 1$, we notice that $j' \ge j+1$ is also the smallest integer such that 
		\begin{equation}
		\sum_{\ell = j+1}^{j'} \sigma(\ell) \equiv k-1 \mod d.
	\end{equation}
	Hence $\mcf_{j+1}^{k-1} = u_{j+1} \mixwed \cdots \mixwed u_{j'-1}\mixwed u_{j'}$,
	and therefore 
	\begin{equation}
		\mcf_{j}^k = u_j\mixwed u_{j+1} \mixwed \cdots \mixwed u_{j'-1}\mixwed u_{j'} =  u_j \mixwed \mcf_{j+1}^{k-1}=\mcf_j^1 \mixwed \mcf_{j+1}^{k-1}.
	\end{equation}
	Similarly if $\sigma(j) = -1$, we have $\mcf_j^{d-1} = u^*_j$ and 
	\begin{equation}
		\mcf_{j}^k = u^*_j \mixwed \mcf_{j+1}^{k+1} =\mcf_j^{d-1} \mixwed \mcf_{j+1}^{k+1}
	\end{equation}
	for $1\le k < d-1$.
	
	Conversely, let $j\in [1, n]$, define $u_j := \mcg_j^1 \in V$ (resp. $u^*_j := \mcg_j^{d-1}\in V^*$) if $\sigma(j) = 1$ (resp. $\sigma(j) = -1$). Let $\uu = (u_1, u_2, \cdots, u_n) \in V^\sigma$. Then it is easy to check that $\vec{\mcg} = \vec{\mcf}(\uu)$. 
\end{proof}

In the next section, we will show that the universality in Proposition \ref{prop: recursive relation between flags associated with a signature} induces an isomorphism between:
\begin{enumerate}[wide, labelwidth=!, labelindent=0pt]
	\item the \emph{generic configuration space} $V^\sigma \setminus \{f = 0\}$, and 
	\item the moduli space of cyclic decorated flag tuples satisfying the recursive relations.
\end{enumerate}

This isomorphism, formalized in Theorem~\ref{thm: bijection between configuration spaces of signature and decorated flag moduli spaces}, achieves our goal of providing explicit coordinates for cluster variables via decorated flags.

\subsection{Decorated flag moduli space}\label{sec: decorated flag moduli space}

In this section, we introduce the notion of decorated flag moduli space. These flag moduli spaces were originally considered in the context of the microlocal theory of sheaves, cf.\ \cite[Section 2.8]{CasalsWeng} and references therein. 
While our construction of decorated flag moduli spaces shares similarities with \cite[Section 6.2]{CasalsLeSBWeng}, we employ distinct conventions that better suit our framework.

\begin{defn}[{\cite[Section 4.6.2]{CasalsWeng}}]\label{defn: wedge with common subspace quotient out}
	Let $u \in \bigwedge^p V, v \in \bigwedge^q V$. Let  $w\in \bigwedge^r V$ with $p+q - d \le r\le \min\{p, q\}$, such that $u = u_1\wedge w$ and $v = w\wedge v_1$ for some $u_1\in \bigwedge^{p-r} V, v_1\in \bigwedge^{q-r} V$. Then \emph{the wedge of $u$ and $v$ over $w$} is defined to be
	\begin{equation}
		\qwed u v w := u_1\wedge w\wedge v_1 \in {\bigwedge}^{p+q-r} V.
	\end{equation}
	Notice that $\qwed u v w = u_1\wedge w \wedge v_1 = u\wedge v_1 = u_1\wedge v$, it follows that the definition of $\qwed uvw$ does not depend on the choice of $u_1$ and $v_1$. Also note that the operation is anti-commutative in $u$ and $v$, i.e., $\qwed u v w = (-1)^{(p-r)(q-r)} \qwed v u w$.
\end{defn}

To understand the algebraic behavior of this wedge product, we establish fundamental cancellation laws:

\begin{lemma}\label{lemma: cancellation law for quotient wedge}
	Under the same assumption as in the Definition~\ref{defn: wedge with common subspace quotient out}, and assume that $v = \qwed w x y $ for some $x\in \bigwedge^s V, y\in \bigwedge^t V$. Then we have
	\begin{equation}\label{equation: cancellation law for quotient wedge}
		\qwed{u}{(\qwed w x y)}{w} = \qwed u x y.
	\end{equation}
	Symmetrically, if $u = \qwed{x'}{w}{y'}$ for some $x'\in \bigwedge^s V, y\in \bigwedge^t V$, then 
	\begin{equation}
		\qwed{(\qwed{x'}{w}{y'})}{v}{w} = \qwed{x'}{v}{y'}.
	\end{equation}
\end{lemma}
\begin{proof}
	If we write $x =  y \wedge x_1$ for some $x_1\in \bigwedge^{s-t}$, then we have
	\[
	\qwed{u}{(\qwed w x y)}{w} = \qwed u {(w\wedge x_1)} w= u \wedge x_1 =  \qwed u x y. 
	\]
	The symmetric case follows analogously.
\end{proof}

\begin{lemma}\label{lemma: cancellation law for quotient wedge, version2}
	Let $u, w, x, y$ be four extensors such that $u = w'\wedge w$ and $y = x \wedge x'$ for some extensors $w', x'$. Then we have
	\[
	\qwed{(u\wedge x)}{(w\wedge y)}{w\wedge x} = u\wedge y.
	\]
\end{lemma}
\begin{proof}
	We have 
	\[
	\qwed{(u\wedge x)}{(w\wedge y)}{w\wedge x} = \qwed{(w'\wedge w\wedge x)}{(w\wedge y)}{w\wedge x} = w'\wedge w\wedge y = u \wedge y. \qedhere
	\]
\end{proof}

Having established these algebraic tools, we now analyze geometric relationships between decorated flags:

\begin{defn}\label{defn: relative position}
	Let $\mcg, \mch$ be two decorated flags and $k\in [1, d-1]$. We say $\mcg$ and $\mch$ are \emph{in relative position} $s_k$, denoted by $\mcg \rel{k} \mch$, if $\mcg_i = \mch_i$ for $i\neq k$ and $\mcg_k \neq \lambda \mch_k$ for any $\lambda \in \C$.
\end{defn}

\begin{remk}
	Our relative position definition differs from \cite[Definition 6.8]{CasalsLeSBWeng} by omitting their uniform \emph{crossing value} (cf.\ Definition~\ref{defn: crossing values}) normalization. This modification preserves greater geometric flexibility.
\end{remk}

\begin{defn}
	Let $u, v \in \bigwedge^k V$ and $\mu \in \C$ such that $u = \mu v$. Then we define
	\[
	\frac u v := \mu\in \C.
	\]
\end{defn}

\begin{lemma}\label{lemma: crossing values of flags well-defined}
	Let $\mcg, \mch$ be two decorated flags that are in relative position $s_k$ for some $k\in [1, d-1]$. Then 
	\[
	\frac{\qwed{\mcg^k}{\mch^k}{\mcg^{k-1}}}{\mcg^{k+1}} = \mu \in \C - \{0\}.
	\]
	Here we adopt the convention that $\mcg^0 = \mcg^{d} = 1$.
\end{lemma}
\begin{proof}
	Let $u, v, w\in V$ such that $\mcg^{k} = \mcg^{k-1}\wedge u$, $\mcg^{k+1} = \mcg^{k-1}\wedge u \wedge v$ and $\mch^{k} = \mcg^{k-1}\wedge w$.  Since $\mch^{k+1} = \mcg^{k+1} = \mcg^{k-1}\wedge u \wedge v$, there exists $\lambda, \mu \in \C$ such that $\mch^{k} = \mcg^{k-1}\wedge (\lambda u + \mu v)$. Hence
	\[
	\qwed{\mcg^k}{\mch^k}{\mcg^{k-1}} = (\mcg^{k-1}\wedge u)\wedge (\lambda u + \mu v) = \mu\cdot (\mcg^{k-1}\wedge u \wedge v) = \mu \mcg^{k+1}.
	\]
	Notice that $\mch^k \neq \lambda \mcg^k$, hence $\mu \neq 0$.
\end{proof}

\begin{defn}[{\cite[Section 4.6.2]{CasalsWeng}}] \label{defn: crossing values}
	Following Lemma~\ref{lemma: crossing values of flags well-defined}, the \emph{crossing value} of $\mcg$ and $\mch$, denoted by $\cv \mcg \mch k$, is defined to be the constant $\mu$:
	\begin{equation}
		\cv \mcg \mch k := \frac{\qwed{\mcg^k}{\mch^k}{\mcg^{k-1}}}{\mcg^{k+1}} \in \C.
	\end{equation}
	Notice that $\cv \mcg \mch k = - \cv \mch \mcg k$. 
\end{defn}

\begin{example}
	Let $d = 3$, $\mcg = (u_2, u_1\wedge u_2)$ and $\mch = (u_2, u_2\wedge u_3)$. Then 
	\[
	\cv \mcg \mch 2= \frac{\qwed{(u_1\wedge u_2)}{(u_2\wedge u_3)}{u_2}}{1} = u_1\wedge u_2 \wedge u_3 = \det(u_1, u_2, u_3).
	\]
\end{example}

\begin{remk}\label{remk: duality of crossing values}
	Similarly, we can define the \emph{dual crossing values} $\dcv \mcg \mch k$ by viewing $\mcg, \mch$ as flags in $V^*$. To be precise,  
	let $u^* \in \bigwedge^p V^*, v^* \in \bigwedge^q V^*$. Let  $w^*\in \bigwedge^r V^*$ with $p+q - d \le r\le \min\{p, q\}$ such that $u^* = u^*_1\wedge^* w^*$ and $v^* = w^*\wedge v^*_1$ for some $u^*_1\in \bigwedge^{p-r} V^*, v^*_1\in \bigwedge^{q-r} V^*$, then the \emph{dual wedge of $u^*$ and $v^*$ over $w^*$} is defined as
	\begin{equation}
		\qwedd {u^*} {v^*} {w^*} := u^*_1\wedge^* w^*\wedge^* v^*_1 \in {\bigwedge}^{p+q-r} V^*.
	\end{equation}
	The \emph{dual crossing value of $\mcg$ and $\mch$}, denoted by $\dcv \mcg \mch k$, is defined as
	\begin{equation}
		\dcv \mcg \mch k:= \frac{\qwedd{\mcg^k}{\mch^k}{\mcg^{k+1}}}{\mcg^{k-1}} \in \C.
	\end{equation}
\end{remk}

Without any surprise, the duality implies that $\cv \mcg \mch k = \dcv \mcg \mch k$.
\begin{lemma}
	Let $\mcg, \mch$ be two decorated flags that are in relative position $s_k$ for some $k\in [1, d-1]$. Then $\cv \mcg \mch k = \dcv \mcg \mch k$.
\end{lemma}
\begin{proof}
	As in the proof of Lemma~\ref{lemma: crossing values of flags well-defined}, we can write $\mcg^{k} = \mcg^{k-1}\wedge u$, $\mcg^{k+1} = \mcg^{k-1}\wedge u \wedge v$ and $\mch^{k} = \mcg^{k-1}\wedge (\lambda u + \mu v)$ with $u, v\in V$, $\lambda, \mu\in \C$ and 
	$\cv \mcg \mch k = \mu$. 
	
	Now let $w^* \in V^*$ such that $w^*\mixwed \mcg^{k-1} = 0, w^*(u) = (-1)^{k-1}\mu$ and $w^*(v) = (-1)^k\lambda$. Then $\mch^k = \mcg^{k+1} \mixwed w^*$. Indeed, by Lemma~\ref{lemma: derivative formula for mixwed wedge} we have
	\begin{equation*}
		\begin{split}
			\mcg^{k+1} \mixwed w^* &= (\mcg^{k-1}\wedge u\wedge v) \mixwed w^*\\
			&=(\mcg^{k-1}\mixwed w^*)\mixwed (u\wedge v) + (-1)^{k-1} \mcg^{k-1}\mixwed (u\wedge v)\mixwed w^*\\
			&= (-1)^{k-1}\mcg^{k-1}\mixwed ((u\mixwed w^*)\mixwed v - (u\mixwed (v\mixwed w^*)))\\
			&= \mcg^{k-1}\wedge (\lambda u + \mu v) = \mch^k.
		\end{split}
	\end{equation*}
	Hence by Lemma~\ref{lemma: derivative formula for mixwed wedge} again, we have
	\[
	\qwedd{\mcg^k}{\mch^k}{\mcg^{k+1}} = \mcg^k \mixwed  w^* = (\mcg^{k-1}\wedge u) \mixwed w^* =( \mcg^{k-1}\mixwed w^*)\mixwed u + (-1)^{k-1} \mcg^{k-1}\mixwed (u\mixwed w^*) = \mu \mcg^{k-1}.
	\]
	Therefore $\dcv \mcg \mch k = \mu = \cv \mcg \mch k$.
\end{proof}

We now synthesize these concepts into our main geometric objects: \emph{the decorated flag moduli spaces}. 

\begin{defn}\label{defn: decoration of word}
	Let $\beta = \beta_1\beta_2\cdots\beta_s$ be a word with $\beta_k \in [1, d-1]$, $1\le k \le s$. A \emph{decoration for $\beta$} is a tuple of decorated flags $\dec_{\beta} = (\mcg_1, \mcg_2, \cdots, \mcg_{s+1})$ such that $\mcg_{k} \rel{{\beta_{k}}}\mcg_{k+1}$ (cf.\ Definition~\ref{defn: relative position}), $1\le k \le s$. We accordingly write
	\[
	\mcg_1 \rel{\beta_1} \mcg_2 \rel{\beta_2} \cdots \rel{\beta_{s-1}} \mcg_s \rel{\beta_s} \mcg_{s+1}.
	\]
	We also say that $\beta$ is decorated ``from $\mcg_1$ to $\mcg_{s+1}$". A decoration is called \emph{cyclic} if $\mcg_1 = \mcg_{s+1}$. A word with a decoration is called a \emph{decorated word}. 
\end{defn}

\begin{defn}\label{defn: marked word}
	A \emph{marked word} is a word with certain characters marked. We usually use overline above a character to indicate that it is marked. For example, in the marked word $13\overline{2}4\overline{3}1$, the third and fifth characters are marked. 
	
	Let $\ww$ be a Demazure weave, then the top word $\beta = \ww_{\text{top}}$ of $\ww$ will be treated as a marked word; a character in $\beta$ is marked if the corresponding weave line is connected to a marked boundary vertex. 
\end{defn}

\begin{defn}\label{defn: decoration of marked word}
	Let $\beta = \beta_1\beta_2\cdots\beta_s$ be a marked word. A \emph{normalized (resp., normalized cyclic) decoration for a marked word $\beta$} is a decoration (resp., cyclic decoration) $\dec_\beta = (\mcg_1, \mcg_2, \cdots, \mcg_{s+1})$ for $\beta$ such that all unmarked crossing values are equal to ~$1$, i.e., if $\beta_k$ is unmarked, then $\cv{\mcg_{k}}{\mcg_{k+1}}{\beta_k} = 1$. 
\end{defn}

\begin{defn}\label{defn: decorated flag moduli space}
	Let $\beta$ be a marked word. \emph{The decorated flag moduli space} $\mfmo(\beta)$ is the configuration space
	\begin{equation}
		\mfmo(\beta):= \{\text{Normalized decorations $\dec_\beta$ for }\beta \}.
	\end{equation}
	\emph{The cyclic decorated flag moduli space} $\mfm(\beta)$ is the configuration space
	\begin{equation}
		\mfm(\beta):= \{\text{Normalized cyclic decorations $\dec_\beta$ for }\beta \}.
	\end{equation}
\end{defn}
\begin{remk}
	The crossing value is $\slv$-invariant, i.e., for any $g\in \slv$, we have $\cv{g\mcg}{g\mch}{k} = \cv \mcg \mch k$. As a consequence, $\slv$ acts naturally on (cyclic) decorated flag moduli spaces. 
\end{remk}

\begin{remk}\label{remk: restriction of decorations to a subword}
	Let $\beta'$ be a contiguous marked subword of a marked word $\beta$. Then a normalized cyclic decoration for $\beta$ restricts to a normalized decoration for $\beta'$, i.e.,  we have a \emph{restriction map}
	\[
	\mfm(\beta) \rightarrow \mfmo(\beta').
	\]
\end{remk}

We next define the (cyclic) decorated flag moduli space associated with a signature. 
Let $\sigma$ be a size $n$ $d$-admissible signature of type $(a, b)$, cf.\ Definition~\ref{defn: admissible signature}. 
\begin{defn}\label{defn: configuration space of a signature, and defn of beta sigma}
	Let $\rho = 12\cdots \overline{(d-1)}$ and $\rho^* = {(d-1)}\cdots 2\overline{1}$ be marked words (with the last character marked).  The \emph{marked word associated with a signature $\sigma$} is the marked word 
	\[
	\beta_\sigma = \prod_{j = 1}^n \rho_{j}, \quad \rho_j \in \{\rho, \rho^*\}
	\]
	where $\rho_j = \rho$ (resp. $\rho_j =  \rho^*$) if $\sigma(j) = 1$ (resp. $\sigma(j) = -1$). In other words, $\beta_\sigma$ is a concatenation of $n$ copys of $\rho$ or $\rho^*$ such that the $j$-th subword is $\rho$ (resp. $\rho^*$) if $\sigma(j) = 1$ (resp. $\sigma(j) = -1$), $1\le j \le n$.
\end{defn}

\begin{defn}
	The \emph{(cyclic) decorated flag moduli space $\mfm(\beta_\sigma)$ associated with a $d$-admissible signature $\sigma$} is the cyclic decorated flag moduli space for the word $\beta_\sigma$. 
\end{defn}

Recall from Definition~\ref{defn: element of Vsigma} that the configuration space
$V^\sigma$ is the rearrangement of the direct product  $V^a\times(V^*)^b$ with respect to $\sigma$. A point $\uu\in V^\sigma$ is generic if $f(\uu) \neq 0$ (cf.\ Definition~\ref{defn: generic point}). The \emph{generic configuration space} $V^\sigma \setminus \{f = 0\}$ associated with $\sigma$ is the space of generic points in $V^\sigma$. 

The cyclic decorated flag moduli space $\mfm(\beta_\sigma)$ admits a geometric invariant theory interpretation as the generic configuration space associated with $\sigma$ (cf.\ Theorem \ref{thm: bijection between configuration spaces of signature and decorated flag moduli spaces}).

\begin{defn}\label{defn: interval words}
	An \emph{interval word} for an interval $[i, j]$ is either 
	\[
	I_i^j : = i (i+1) \cdots (j-1){j}\quad  \text{ or } \quad I_j^i:= {j}{(j-1)}\cdots {(i+1)}i,
	\] 
	where $1\le i \le j \le d-1$. We call $I_i^j$ (resp. $I_j^i$) \emph{increasing} (resp. \emph{decreasing}). The notion of interval words will be used again in later sections. 
\end{defn}

\begin{lemma}\label{lemma: decoration for interval word}
	Let $\beta$ be a marked interval word. Let $(\mcg_1, \mcg_2, \cdots, \mcg_{s+1})$ be a normalized decoration of $\beta$. Then the decorated flags $\mcg_{k}$ for $k\in[1, s+1]$ are uniquely determined by $\mcg_1$ and $\mcg_{s+1}$. In other words, a decoration for an interval word can be represented by a tuple $(\mcg_1, \mcg_{s+1})$, and we usually write $\mcg_1 \rel{\beta} \mcg_{s+1}$.
\end{lemma}
\begin{proof}
	Without loss of generality, we assume that $\beta = I_i^j$ with $1\le i\le j \le d-1$. Then 
	\[
		\mcg_k = (\mcg_{1}^1, \dots, \mcg_1^{i-1}, \mcg_{s+1}^i, \dots, \mcg_{s+1}^{k+i-2}, \mcg_1^{k+i-1}, \cdots, \mcg_1^{d-1}), \quad k \in [1, s+1]. 
		\]	
		is uniquely determined by $\mcg_1$ and $\mcg_{s+1}$. 
\end{proof}

\begin{lemma}\label{lem: pre-assignment of top boundary flags}
	For generic $\uu\in V^\sigma$ (cf.\ Definition \ref{defn: generic point}), let 
	\[
	\vec{\mcf}(\uu) =  (\mcf_1, \mcf_2, \dots, \mcf_n, \mcf_{n+1} = \mcf_1)
	\]
	be the tuple defined in Definition~\ref{defn: tuple of flags mcf associated with sigma and u}. Let $\{\mcf_{j,k}\}_{j\in \Z, 1\le k \le d-1}$ be a set of decorated flags defined as follows:
	\begin{itemize}[wide, labelwidth=!, labelindent=0pt]
		\item $\mcf_{j+n, k} = \mcf_{j, k}$;
		\item $\mcf_{j,1} = \mcf_j$, $1\le j \le n+1$;
		\item If $\sigma(j) = 1$, define $\mcf_{j, k}^i = \mcf_{j}^i$ for $i \ge  k$ and $\mcf_{j, k}^i = \mcf_{j+1}^i$ for $i< k$;
		\item If $\sigma(j) = -1$, define $\mcf_{j, k}^i = \mcf_{j}^i$ for $i \le d-k$ and $\mcf_{j, k}^i = \mcf_{j+1}^i$ for $i > d-k$.
	\end{itemize}
	Then 
	\[
	\dec = (\mcf_{1, 1}, \mcf_{1, 2}, \dots, \mcf_{1, d-1}, \mcf_{2, 1}, \dots, \mcf_{2, d-1}, \dots, \mcf_{n, 1}, \dots, \mcf_{n, d-1}, \mcf_{n+1,1} = \mcf_{1, 1})
	\] 
	is a normalized cyclic decoration for $\beta_\sigma$.  
\end{lemma}

\begin{proof}
	First we check that $\mcf_{j, k}$ is decorated flag, for $j\in \Z$ and $k\in [1, d-1]$. All these extensors are nonzero as $f(\uu) \neq 0$. Without loss of generality, we assume that $\sigma(j) = 1$. By Proposition~ \ref{prop: recursive relation between flags associated with a signature}, we have
	\[
	\mcf_{j, k}^{k} = \mcf_{j}^{k} = \mcf_j^1 \wedge \mcf_{j+1}^{k-1} = \mcf_j^1 \wedge \mcf_{j, k}^{k-1}.
	\]
	Hence the nested condition is satisfied and $\mcf_{j, k}$ is a decorated flag. 
	
	For $j\in \Z$ and $k \in [1, d-2]$. Notice that $\mcf_{j,k}$ and $\mcf_{j,k+1}$ are in relative position $s_k$ (resp. $s_{d-k}$) if $\sigma(j) = 1$ (resp. $\sigma(j) = -1$); and the crossing value (cf.\ Definition~\ref{defn: crossing values}) is equal to $1$. Indeed, without loss of generality, we assume that $\sigma(j) = 1$, then
	\[
	\cv{\mcf_{j, k}}  {\mcf_{j, k+1}}{k} = \frac{\qwed{\mcf_j^k}{\mcf_{j+1}^k}{\mcf_{j+1}^{k-1}}}{\mcf_{j}^{k+1}} = \frac{\qwed{({\mcf_j^1 \wedge \mcf_{j+1}^{k-1}})}{\mcf_{j+1}^k}{\mcf_{j+1}^{k-1}}}{\mcf_j^1\wedge \mcf_{j+1}^{k}}=1. 
	\]
	Here the last equality follows from Lemma~\ref{lemma: cancellation law for quotient wedge}. Notice that $\mcf_{j, d-1}$ and $\mcf_{j+1, 1}$ are in relative position $s_{d-1}$ (resp. $s_1$) if $\sigma(j) = 1$ (resp. $\sigma(j) = -1$); and the crossing value 
	is equal to $f_j \neq 0$ (cf.\ Definition \ref{defn: generic point}). Indeed, without loss of generality, we assume that $\sigma(j) = 1$, then
	\[
	\cv{\mcf_{j, d-1}}  {\mcf_{j+1,1}}{d-1} = \frac{\qwed{\mcf_j^{d-1}}{\mcf_{j+1}^{d-1}}{\mcf_{j+1}^{d-2}}}{1} = \qwed{({\mcf_j^1 \wedge \mcf_{j+1}^{d-2}})}{\mcf_{j+1}^{d-1}}{\mcf_{j+1}^{d-2}}=\mcf_j^1\wedge \mcf_{j+1}^{d-1} = f_j. 
	\]
	This completes that proof that  $\dec$ is a normalized cyclic decoration for $\beta_\sigma$. 
\end{proof}

\begin{theorem}\label{thm: bijection between configuration spaces of signature and decorated flag moduli spaces}
	There is a $\emph{\text{SL}}(V)$-equivariant isomorphism
	\begin{equation}
	\Phi: \mfm(\beta_\sigma) \stackrel{\sim}{\longrightarrow} V^\sigma\setminus \{f = 0\}.
	\end{equation}
\end{theorem}
\begin{proof}
	Let 
	\[
	 \dec_{\beta_\sigma} = (\mcg_{1, 1}, \mcg_{1, 2}, \cdots, \mcg_{1, d-1}, \mcg_{2, 1}, \cdots, \mcg_{2, d-1}, \cdots, \mcg_{n, 1}, \cdots, \mcg_{n, d-1}, \mcg_{n+1, 1} = \mcg_{1, 1})
	\]
	be a normalized cyclic decoration for $\beta_\sigma$. Let $\mcg_j:= \mcg_{j, 1}$, $j\in [1, n+1]$. We define 
	\[
	\Phi(\dec_{\beta_\sigma}) = (\mcg_1^{\sigma(1)}, \mcg_2^{\sigma(2)}, \dots, \mcg_n^{\sigma(n)}) \in V^\sigma.
	\]
	Here we adopt the convention that $\mcg^{-1} := \mcg^{d-1}$.  Notice that the crossing value between $\mcg_{j, d-1}$ and $\mcg_{j+1, 1}$ is equal to $f_j(\Phi(\dec_{\beta_\sigma}))\neq 0$, for $j\in \Z$. Hence $f(\Phi(\dec_{\beta_\sigma})) \neq 0$ and therefore $\Phi(\dec_{\beta_\sigma}) \in V^\sigma \setminus \{f = 0\}$.

	The inverse of $\Phi$ is given in Lemma~\ref{lem: pre-assignment of top boundary flags}. Let us verify that they are indeed inverse to each other. 
	Consider the tuple
	\[
	\vec{\mcg} = (\mcg_1, \mcg_2, \dots, \mcg_n, \mcg_{n+1} = \mcg_1).
	\]
	Notice that by Lemma~\ref{lemma: decoration for interval word}, the flags $\mcg_{j, k}$ can be uniquely recovered from $\mcg_{j, 1} =\mcg_j$. So we only need to show that $\vec{\mcg} = \vec{\mcf}(\Phi(\dec_{\beta_\sigma}))$ (cf.\ Definition~\ref{defn: tuple of flags mcf associated with sigma and u}). 
	
	Let  $j\in [1, n]$. Suppose that $\sigma(j) = 1$. The normalized condition on the crossing values translate into the following identities:  for $k\in [1, d-2]$, we have $\cv{\mcg_{j, k}}{\mcg_{j, k+1}}{k} = 1$. Equivalently, 
	\begin{equation}\label{equ: crossing = 1}
	\mcg_{j, k}^{k+1} = \qwed{\mcg_{j, k}^k}{\mcg_{j, k+1}^k}{\mcg_{j, k}^{k-1}}, \quad k\in[1, d-2].
	\end{equation}
	Notice that $\mcg_{j, k}^{k} = \mcg_{j}^k, \mcg_{j, k+1}^k = \mcg_{j+1}^k, \mcg_{j,k}^{k-1} = \mcg_{j+1}^{k-1}$ and $\mcg_{j, k}^{k+1} = \mcg_{j}^{k+1}$. Then Equations ~(\ref{equ: crossing = 1}) are equivalent to
	\begin{equation}\label{equ: inductive formula}
	\mcg_j^{k+1} = \qwed{\mcg_j^k}{\mcg_{j+1}^k}{\mcg_{j+1}^{k-1}}, \quad k\in [1, d-2].
	\end{equation}
	Then by an induction on $k$, we can show that Equations (\ref{equ: inductive formula}) are equivalent to 
	\begin{equation}\label{equ: recursive formula for flags, black}
	\mcg_j^{k+1} = \mcg_j^1 \wedge \mcg_{j+1}^k, \quad k\in [1, d-2].
	\end{equation}
	Suppose that $\sigma(j) = -1$.  Similarly we have
	\begin{equation}\label{equ: recursive formula for flags, white}
	\mcg_{j}^{k} = \mcg_j^{d-1} \mixwed \mcg_{j+1}^{k+1}, \quad k \in [1, d-2].
	\end{equation}

	To sum up, the tuple $\vec{\mcg}$
	satisfies the recursive relations (\ref{equ: recursive formula for flags, black}) and (\ref{equ: recursive formula for flags, white}). Hence by Proposition~\ref{prop: recursive relation between flags associated with a signature}, $\vec{\mcg} = \vec{\mcf}(\Phi(\dec_{\beta_\sigma}))$.  
\end{proof}

Recall from Definition~\ref{defn: V sigma and rsigma} that the mixed Pl\"ucker ring
\[R_\sigma(V) = \C[V^\sigma]^{\slv},
\]
is the ring of $\slv$-invariant polynomial functions on $V^\sigma$. The fraction field of $R_\sigma(V)$ is denoted by $\K$. It is well-known that $\K$ is equal to the field of $\slv$-invariant rational functions on $V^\sigma$ (cf.\ \cite[Theorem 3.3]{VinbergPopov}):
\[
K_\sigma(V) = \C(V^\sigma)^{\slv}.
\]

\begin{cor}\label{cor: decorated flag moduli space and the mixed grassmannian}
	The ring of $\emph{\text{SL}}(V)$-invariant regular functions on $\mfm(\beta_\sigma)$ is isomorphic to the mixed Pl\"ucker ring with $f$ inverted; and the field of $\emph{\text{SL}}(V)$-invariant rational functions on $\mfm(\beta_\sigma)$ is isomorphic to the fraction field of the mixed Pl\"ucker ring: 
	\[
	\C[\mfm(\beta_\sigma)]^{\emph{\text{SL}}(V)} \cong \rsigma[\frac 1 f], \quad \C(\mfm(\beta_\sigma))^{\emph{\text{SL}}(V)} \cong K_\sigma(V).
	\]
\end{cor}
\begin{proof}
	This follows from Theorem~\ref{thm: bijection between configuration spaces of signature and decorated flag moduli spaces}. 
\end{proof}

\begin{remk}
	The product $f = \prod_{j = 1}^n f_j$ corresponds to frozen variables in the cluster algebra associated to $\sigma$. While inverting $f$ would align with geometric perspective of cluster algebras, we deliberately maintain $f$ as a geometric constraint rather than a formal inversion. This preserves the algebro-geometric interpretation of $\mfm(\beta_\sigma)$ as an open subvariety of $V^\sigma$, where non-vanishing of $f$ encodes the stability condition for flag configurations. The cluster structure nevertheless persists through the identification of $f_j$'s as generalized Plücker coordinates.
\end{remk}

\subsection{Cluster algebras from Demazure weaves}\label{sec: cluster alg associated with weaves}

Let $\ww$ be a Demazure weave with marked top word $\beta$.
We now construct its associated cluster algebra through the following geometric framework:
\begin{enumerate}[wide, labelwidth=!, labelindent=0pt]
	\item \textbf{Quiver Foundations}: Recall from Definition~\ref{defn: quiver associated with Demazure weaves} that $\ww$ determines a quiver $Q(\ww)$ encoding its combinatorial structure.
	\item \textbf{Geometric realization}: Cluster variables will be realized as rational functions on the decorated flag moduli space $\mfm(\beta)$ (Definition~\ref{defn: seed from weave}).
\end{enumerate}


The results in this section are parallel to the results in \cite[Section 5]{CGGLSS}, but with different notations and conventions. 

\begin{defn}\label{defn: decorated demazure weave}
	A Demazure weave $\ww$ divides the ambient rectangle $R$ into faces. A \emph{decoration $\dec$ for $\ww$} is a set of decorated flags indexed by faces of $\ww$, satisfying the following condition:	
		 there exists $\Delta_\gamma\in \C$, one for each (Lusztig) cycle $\gamma$, such that the \emph{edge-crossing relation} is satisfied for every
		  weave line $e$:  suppose that $e$ is of color~$k$ and separates faces $U$ (left) and $V$ (right), then $\mcf_U\rel{k} \mcf_V$ and furthermore, the crossing value (cf.\ Definition~\ref{defn: crossing values}) is given by
		\begin{equation}\label{equation: edge value is equal to the product of cycles passing through}
			\Delta_e:= 	\cv {\mcf_U}{\mcf_V} k= \prod_{\gamma} \Delta_\gamma^{\gamma(e)}.
		\end{equation}

	A \emph{decorated Demazure weave} is a pair $(\ww, \dec)$ consisting of a Demazure weave $\ww$ and a decoration $\dec$ for $\ww$. We simply write $\ww$ if $\dec$ is clear from the context. 
\end{defn}

\begin{example}
	See Example~\ref{example: describing a seed from a decorated weave} for an explicit decorated Demazure weave with computed crossing values. 
\end{example}

\begin{remk}\label{remk: decoration of weave is the same as deco of the top word}
	Let $\beta$ be a marked word. We can view $\beta$ as a trivial Demazure weave $\ww_{\beta}: \beta \rarrow \beta$. 
	A decoration for the weave $\ww_{\beta}$ corresponds precisely to a {normalized decoration for $\beta$} (cf.\ Definition~\ref{defn: decoration of marked word}). Indeed, if a weave line $e$ (say of color $k$) in $\ww_\beta$ is incident to a unmarked boundary vertex, then no cycles passes through edge $e$, i.e., $\gamma(e) = 0$ for all $\gamma$; hence the crossing value $\Delta_e = \cv {\mcf_U}{\mcf_V} k = 1$. 
	
	Let $\ww$ be a Demazure weave with marked top word $\beta$. Then a decoration $\dec$ for $\ww$ restricts to a normalized decoration for $\beta$. Theorem~\ref{thm: unique extension of decorations for a weave} implies that the converse is also true: a normalized decoration for $\beta$ extends (uniquely) to a decoration for $\ww$. This bridges local flag configurations with global cycle invariants.
\end{remk}

\begin{theorem}\label{thm: unique extension of decorations for a weave}
	Let $\ww$ be a Demazure weave with marked top word $\beta = \ww_{\text{top}}$. Then a generic normalized decoration $\dec_{\beta}$ for $\beta$ extends uniquely to a decoration $\dec$ for $\ww$. Moreover, the numbers $\Delta_\gamma$ (cf.\ Definition~\ref{defn: decorated demazure weave}) are uniquely determined by $\dec_\beta$ (equivalently, by $\dec$). 
\end{theorem}

\begin{proof}
	We scan the Demazure weave $\ww$ from top to bottom, and extend the decoration as we meet new faces. In the meantime, we will define $\Delta_\gamma$ for new cycles $\gamma$, such that (\ref{equation: edge value is equal to the product of cycles passing through}) hold for new edges. To start with, since $\dec_\beta$ is normalized, we do have a valid decoration for the top of $\ww$, cf.\ Remark~\ref{remk: decoration of weave is the same as deco of the top word}. 
	
	Now as we scan downwards, there is nothing to be done until we will meet a new vertex $v$. 
	\begin{itemize}[wide, labelwidth=!, labelindent=0pt]
	\item If $v$ is trivalent, let $e$ be the edge below $v$. There is no new face, so we don't need to extend our decoration. However, we do have a new cycle $\gamma_v$, originated at $v$ with $\gamma_v(e) = 1$. We define 
	\[
	\Delta_{\gamma_v}: = {\Delta_e}\big/{\prod_{\gamma \in Y_v} \Delta_{\gamma}^{\gamma(e)}},
	\]
	where $Y_v$ is the set of all cycles originating above $v$. Notice that if a cycle $\gamma$ originates below $v$, then $\gamma(e) = 0$. Hence
	\[
		\Delta_e = \Delta_{\gamma_v} \prod_{\gamma \in  Y_v} \Delta_{\gamma}^{\gamma(e)} = \prod_{\gamma} \Delta_\gamma^{\gamma(e)}.
	\]
	
	\item If $v$ is $4$-valent, as in Figure~\ref{fig: 4-valent and 6-valent configuration}, let $U, V, W, X$ be the faces to the left, right, top, bottom of $v$; let $e$ of color $i$ (resp. $f$ of color $j$) be the edge separating $V$ (resp. $U$) and $W$, $|i-j|>1$; and $e'$ (resp. $f'$) be the edge separating $U$ (resp. $V$) and $X$. We have a new face $X$, so we extend our decoration by defining $\mcf_X$ as follows. Let $(\mcf_X)^k = (\mcf_{U})^k$ for $k\neq i$ and $(\mcf_X)^k = (\mcf_{V})^k$ for $k\neq j$. This is well-defined as $(\mcf_U)^k = (\mcf_{V})^k$ for $k\neq i, j$, and is clearly a decorated flag. There is no new cycle, and we need to verify (\ref{equation: edge value is equal to the product of cycles passing through}) for new edges $e'$ and $f'$. 

	Notice that $\Delta_{e'} = \cv{\mcf_U}{\mcf_X}{i} = \cv{\mcf_W} {\mcf_V}{i} = \Delta_{e}$; and $\gamma(e) = \gamma(e')$ for any cycle $\gamma$. Therefore $\Delta_{e'} = \Delta_e= \prod_{\gamma} \Delta_\gamma^{\gamma(e)} = \prod_{\gamma} \Delta_\gamma^{\gamma(e')}$. Similarly, we have $\Delta_{f'} = \prod_{\gamma} \Delta_\gamma^{\gamma(f')}$.
	\begin{figure}[H]
		\centering
		\includegraphics[trim =8cm 13cm 8cm 11cm, clip = true, scale = 0.45]{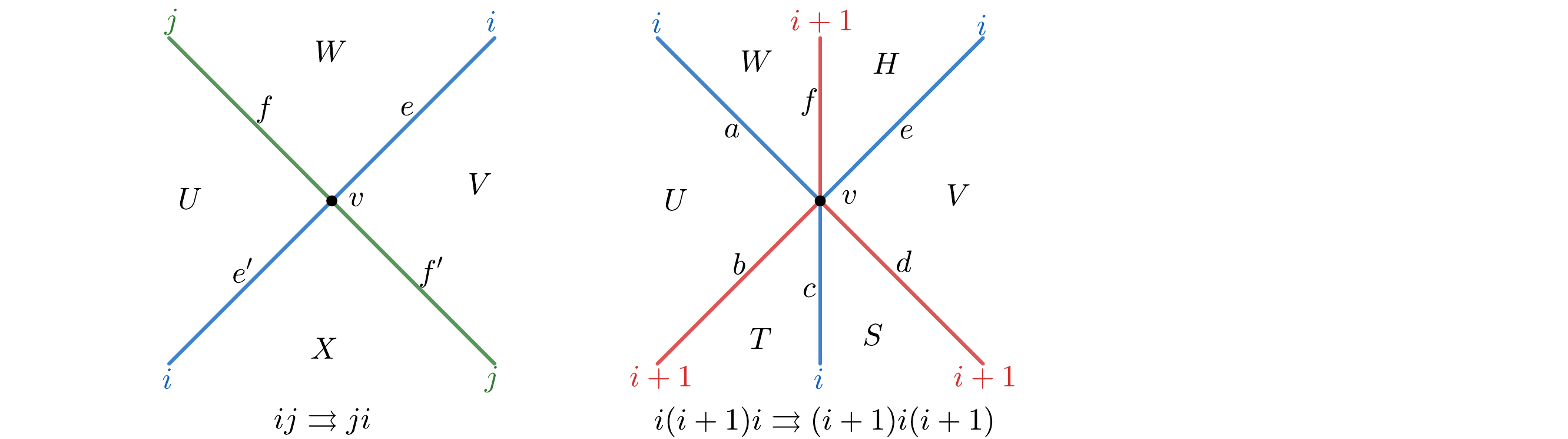}
		\caption{Decoration around a $4$-valent vertex (left) and at a $6$-valent vertex (right).}
		\label{fig: 4-valent and 6-valent configuration}
	\end{figure}
	\item If $v$ is $6$-valent. By duality (cf.\ Remark~ \ref{remk: duality of crossing values}), we may assume that it is the case of $i (i+1) i \rarrow (i+1) i (i+1)$, cf.\ Figure~\ref{fig: 4-valent and 6-valent configuration}. We have two new faces $T, S$, we extend the decoration by defining $\mcf_T$ and $\mcf_S$ as follows: $(\mcf_T)^k  := (\mcf_S)^k := (\mcf_U)^k$ for $k \neq i, i+1$, $(\mcf_T)^i := (\mcf_U)^i$, $(\mcf_S)^i := (\mcf_V)^i$, and 
	\[
	(\mcf_T)^{i+1} := (\mcf_S)^{i+1} := (\prod_{\gamma \in Y_v} \Delta_{\gamma}^{\gamma(c)})^{-1}\cdot(\qwed{\mcf_U^i}{\mcf_V^i}{\mcf_U^{i-1}}).
	\]
	There is no new cycle, and we need to verify (\ref{equation: edge value is equal to the product of cycles passing through}) for new edges $b, c$ and $d$. By definition, we have
	\[
	\Delta_c =\cv  {\mcf_T}  {\mcf_S}{i} = \frac{\qwed{\mcf_T^i}{\mcf_S^i}{\mcf_{T}^{i-1}}}{(\mcf_T)_{i+1}}= \frac{\qwed{\mcf_U^i}{\mcf_V^i}{\mcf_U^{i-1}}}{(\mcf_T)_{i+1}} = \prod_{\gamma\in Y_v} \Delta_{\gamma}^{\gamma(c)}=\prod_{\gamma} \Delta_{\gamma}^{\gamma(c)},
	\]
	where the last equality follows from the fact that if $\gamma$ is originated below $v$, then $\gamma(c) = 0$. Similarly, we have
	\[
	\Delta_b = \cv {\mcf_U}  {\mcf_T}{i+1} = \frac{\qwed{\mcf_U^{i+1}}{\mcf_T^{i+1}}{\mcf_U^i}}{(\mcf_U)_{i+2}} = \frac{\qwed{\mcf_U^{i+1}}{\big(\Delta_c^{-1}\cdot(\qwed{\mcf_U^i}{\mcf_V^i}{\mcf_U^{i-1}}) \big)}{\mcf_U^i}}{(\mcf_U)_{i+2}}\\
	= \frac{\qwed{\mcf_U^{i+1}}{\mcf_V^{i}}{\mcf_U^{i-1}}}{\Delta_c\cdot (\mcf_U)_{i+2}};
	\]
	where the last equality follows from Lemma~\ref{lemma: cancellation law for quotient wedge}. Similarly we have
	\[
	\Delta_e\Delta_f= \frac{\qwed{\mcf_H^i}{\mcf_V^i}{\mcf_U^{i-1}}}{\mcf_V^{i+1}}\cdot \frac{\qwed{\mcf_U^{i+1}}{\mcf_V^{i+1}}{\mcf_H^i}}{\mcf_U^{i+2}} = \frac{\qwed{\mcf_U^{i+1}}{(\qwed{\mcf_H^i}{\mcf_V^i}{\mcf_U^{i-1}})}{\mcf_H^i}}{\mcf_U^{i+2}}=\frac{\qwed{\mcf_U^{i+1}}{\mcf_V^{i}}{\mcf_U^{i-1}}}{(\mcf_U)_{i+2}}.
	\]
	Hence $\Delta_e\Delta_f = \Delta_b\Delta_c$. Now notice that $\gamma(b) + \gamma(c) = \gamma(e) + \gamma(f)$ for any cycle $\gamma$, so
	\[
	\Delta_b = \frac{\Delta_e\Delta_f}{\Delta_c} = \prod_{\gamma} \Delta_\gamma^{\gamma(e) + \gamma(f) - \gamma(c)} = \prod_{\gamma} \Delta_\gamma^{\gamma(b)}.
	\]
	Similarly 
	\[
	\Delta_d = \frac{\Delta_a\Delta_f}{\Delta_c} =  \prod_{\gamma} \Delta_\gamma^{\gamma(d)}.
	\]
	\end{itemize}	

	By continuing the scanning, we can extend the decoration to $\ww$; and for each cycle $\gamma$, we have defined $\Delta_\gamma$ such that $\Delta_e = \prod_{\gamma} \Delta_\gamma^{\gamma(e)}$ hold for all weave lines $e$. Therefore, we can extend a normalized decoration of $\beta$ to a decoration for $\ww$. The uniqueness of the extension is clear by the construction.
	
	Finally, notice that $\Delta_{\gamma_v}$ is uniquely determined by $\Delta_e$ and $\{\Delta_{\gamma}\}_{\gamma\in Y_v}$, where $e$ is the edge below $v$. So recursively, $\Delta_{\gamma_v}$ is uniquely determined by $\{\Delta_e\}_{e\in \ww}$, 
	which are uniquely determined by $\dec_{\beta}$.
\end{proof}

\begin{remk}
	Let $\ww$ be a Demazure weave with marked top word $\beta$. Let $\gamma$ be a cycle in $\ww$. By Theorem~\ref{thm: unique extension of decorations for a weave}, $\Delta_\gamma$ can be viewed as a rational function on the space of normalized decorations for $\beta$. Notice that $\Delta_e$ is invariant under $\slv$ actions, so $\Delta_\gamma$ is a $\slv$-invariant rational function on $\mfmo(\beta)$, i.e., $\Delta_\gamma \in \C(\mfmo(\beta))^{\slv}$.
\end{remk}

\begin{remk}\label{remk: cycles are in Ksigma}
	Throughout this manuscript, every marked top word $\beta$ for a Demazure weave $\ww$ that we consider will be a contiguous subword of $\beta_\sigma$ for some signature $\sigma$, cf.\ Definition~\ref{defn: configuration space of a signature, and defn of beta sigma}. Therefore a normalized cyclic decoration for $\beta_\sigma$ induces a normalized decoration for $\beta$, i.e., we have a natural map 
	\[
	 \mfm(\beta_\sigma) \rightarrow \mfmo(\beta).
	\]
	This induces an inclusion of the corresponding fields of $\text{SL}(V)$-invariant rational functions:
	\[
	\C(\mfmo(\beta))^{{\text{SL}}(V)} \hookrightarrow \C(\mfm(\beta_\sigma))^{{\text{SL}}(V)} \cong  \K.
	\]
	The isomorphism $\C(\mfm(\beta_\sigma))^{{\text{SL}}(V)} \cong \K$ follows from Corollary~\ref{cor: decorated flag moduli space and the mixed grassmannian} (recall that $\K$ is the fraction field of $\rsigma$).  Hence for any cycle $\gamma$ in $\ww$, we can view $\Delta_\gamma$ as an element of $\K$. 
\end{remk}

\begin{defn} \label{defn: seed from weave}
	Let $\beta$ be a marked contiguous subword of $\beta_\sigma$ for a signature $\sigma$. Let $\ww$ be a Demazure weave with marked top word $\beta$. 
	We define \emph{the seed} 
	\[
	\seed(\ww) = (Q(\ww), \zz(\ww))
	\]
	\emph{associated with a Demazure weave $\ww$}, as a seed in $\K$ given as follows. 
	
	The quiver $Q(\ww)$ is defined as in Definition~\ref{defn: quiver associated with Demazure weaves}. 
	Recall that the set of vertices of $Q(\ww)$ is indexed by the set of cycles in $\ww$. A vertex of $Q(\ww)$ is designated frozen (resp., mutable) if the corresponding cycle is frozen (resp., mutable). The number of arrows between two vertices is the intersection pairing of the corresponding cycles, cf.\ Definition~\ref{defn:local intersection pairing}.
	
	We define the cluster $\zz(\ww)$ by setting
	the cluster (or frozen) variable associated with a vertex in $Q(\ww)$ to be $\Delta_\gamma \in \K$, where $\gamma$ is the cycle corresponding to that vertex. 
	
	\emph{The cluster algebra $\mca(\ww)$ associated with a  Demazure weave $\ww$} is the cluster algebra associated with the seed $\seed(\ww)$. 
\end{defn}

\begin{remk}
	Strictly speaking we cannot  call $\seed(\ww)$ a seed since we don't yet know if the  cluster/frozen variables $\Delta_\gamma$ are algebraically independent. We will prove that the choices described in later sections guarantee that $\seed(\ww)$ is a seed, cf.\ Lemma~\ref{lemma: alg indep of the cluster}. Nevertheless, we would like to point out that the results for the seed mutations associated with Demazure weaves in the next few sections do not rely on Lemma~\ref{lemma: alg indep of the cluster}. To be precise, when we say two seeds $\seed_1 = (Q_1, \zz_1)$ and $\seed_2 = (Q_2, \zz_2)$ are mutation equivalent, what we mean is that $Q_1$ and $Q_2$ are related by a sequence of mutations, and applying the corresponding exchange transformations to the collection $\zz_1$ produce the collection $\zz_2$. 
\end{remk}

\begin{example}
For an explicit illustration of the seed construction in Definition~\ref{defn: seed from weave}, including the assignment of cluster variables to cycles, computation of crossing values, and quiver mutation dynamics, see Example~\ref{example: describing a seed from a decorated weave}. This step-by-step walkthrough demonstrates how a decorated Demazure weave explicitly determines a cluster seed, bridging the combinatorial topology of weaves with the algebraic structure of cluster algebras.
\end{example}

In the remainder of this section, we establish that mutation-equivalent Demazure weaves produce mutation-equivalent seeds, thereby demonstrating the invariance of their associated cluster algebras under weave mutations.

\begin{defn}
	Two decorated Demazure weaves are \emph{equivalent} (resp., \emph{mutation equivalent}) if the underlying weaves are equivalent (resp., mutation equivalent) and their decorations for the marked top word are the same. 
\end{defn}
\begin{lemma} \label{lemma: mutation equivalent weave same bottom flags}
	Let $\ww, \ww': \beta \rarrow \beta'$ be two mutation equivalent decorated Demazure weaves. Then they have the same decoration for the bottom word. 
\end{lemma}
\begin{proof}
	We only need to show that under a single local equivalence/mutation move, the decorated bottom words remain the same. Here we demonstrate the proof for one move, the proof for other moves is similar. 
	
	Consider the local equivalence move Pushthrough from below, cf.\ Figure~\ref{fig: equivalent pushthrough from below bottom flag unchanged}. Without loss of generality, we assume that $j = i+1$. We need to show that $\mch = \mch'$ and $\mck = \mck'$. 
	
	Notice that $\mch^k = \mcg_k = (\mch')^k$ and $\mck^k = \mathcal{I}^k = (\mck')^k$ for $k \neq i+1$, so we only need to show that $\mch^{i+1} = (\mch')^{i+1}$ and $\mck^{i+1} = (\mck')^{i+1}$. 
	\begin{figure}
		\centering
		\includegraphics[trim = 3cm 10cm 0cm 11cm, clip = true, scale = 0.4]{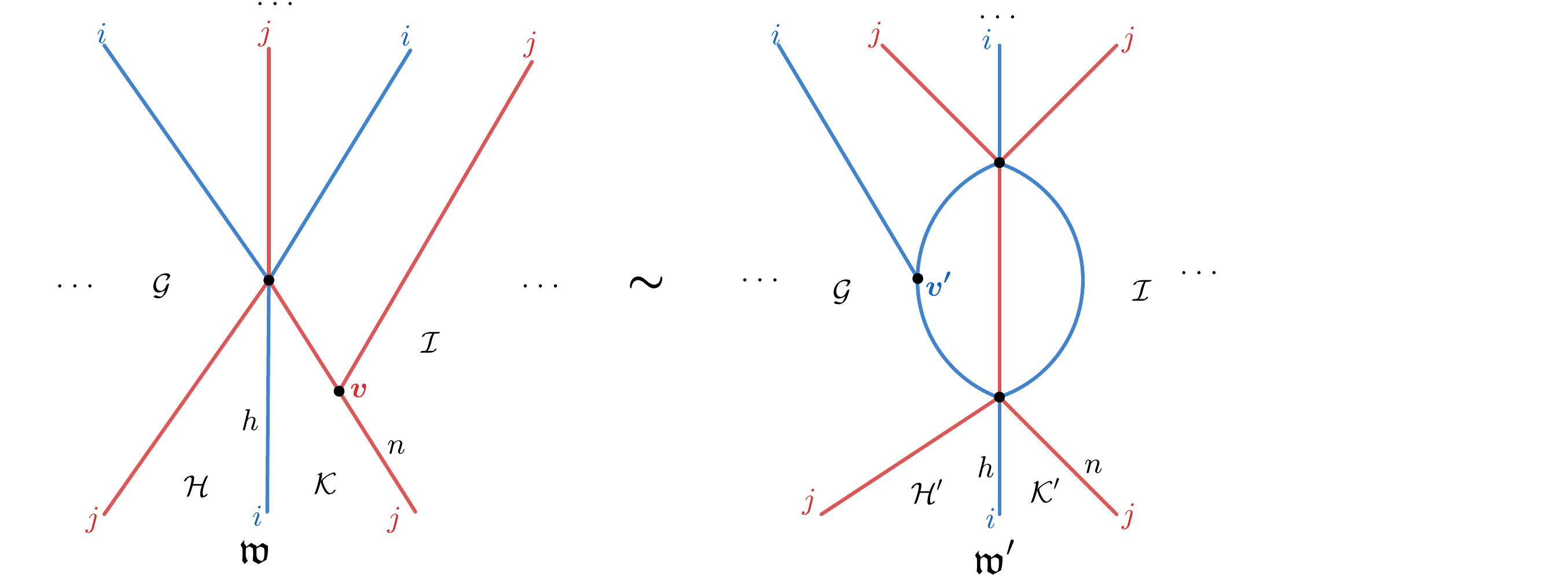}
		\caption{Two decorated Demazure weaves related by a single local move Pushthrough from below.}
		\label{fig: equivalent pushthrough from below bottom flag unchanged}
	\end{figure}
	By Lemma~\ref{lemma: cycles end behavior remain unchanged under mutation}, we have $\gamma(h) = \gamma'(h)$ for any cycle $\gamma$ in $\ww$ and the corresponding cycle $\gamma'\in \ww'$. Let $Y_v$ (resp. $Y'_{v'}$) be the cycles in $\ww$ (resp. $\ww'$) that originates above $v$ (resp. $v'$); there is a 1-1 correspondence between $Y_v$ and $Y'_{v'}$. Notice that if $\gamma \notin Y_v$ (and $\gamma'\notin Y'_{v'}$), then $\gamma(h) = 0$ and $\gamma'(h) =  0$.
	Hence
	\[
		\Delta_h(\ww) =\prod_{\gamma\in Y_v}\Delta_{\gamma}^{\gamma(h)} =  \prod_{\gamma'\in Y'_{v'}}\Delta_{\gamma'}^{\gamma'(h)} = \Delta_{h}(\ww').
	\]
	Then by Definition~\ref{defn: crossing values}, we have
	\[
	\Delta_h(\ww) = \cv{\mch}{\mathcal{K}}{i} = \frac{\qwed{\mch^i}{\mck^i}{\mch^{i-1}}}{\mch^{i+1}} = \frac{\qwed{\mcg^i}{\mathcal{I}^i}{\mcg^{i-1}}}{\mch^{i+1}}; \]
	and
	\[ 
	\Delta_{h}(\ww') =\cv {\mch'}  {\mathcal{K'}}{i} = \frac{\qwed{{\mch'}^i}{{\mck'}^i}{{\mch'}^{i-1}}}{{\mch'}^{i+1}} = \frac{\qwed{\mcg^i}{\mathcal{I}^i}{\mcg^{i-1}}}{{\mch'}^{i+1}}.
	\]
	Hence $\Delta_h(\ww) = \Delta_{h}(\ww')$ implies $\mch^{i+1} = {\mch'}^{i+1}$. Similarly we have $\mck^{i+1} = {\mck'}^{i+1}$.\qedhere
\end{proof}

\begin{theorem}\label{thm: mutation equivalence demazure weaves yield mutation equivalent seeds}
	The seeds associated with equivalent (resp., mutation equivalent) Demazure weaves are the same (resp., mutation equivalent). Consequently, the cluster algebra associated with mutation equivalent Demazure weaves are the same. 
\end{theorem}

\begin{proof}
	Let us first show that the seeds associated with equivalent Demazure weaves are the same. Similar to the proof of Lemma~\ref{lemma: mutation equivalent weave same bottom flags}, we only demonstrate the proof for one move, cf.\ Figure~\ref{fig: equivalent pushthrough from below bottom flag unchanged}.  By Proposition~\ref{prop: mutation equivalent weaves yield mutation equivalent quivers}, $Q(\ww) = Q(\ww')$. 
	
	Let $Y_v$ (resp. $Y'_{v'}$) be the cycles in $\ww$ (resp. $\ww'$) that originates above $v$ (resp. $v'$). If $\gamma \in Y_v$, and let $\gamma' \in Y'_{v'}$ be the corresponding cycle in $\ww'$, then $\Delta_\gamma = \Delta_{\gamma'}$ as $\ww$ and $\ww'$ are the same above the local picture. Now let us show that $\Delta_{\gamma_v} = \Delta_{\gamma_{v'}}$.
	
	By Lemma~\ref{lemma: mutation equivalent weave same bottom flags}, $\mathcal{K} = \mathcal{K'}$, hence 
	\[\Delta_n(\ww)= \cv{\mathcal{K}}{\mathcal{I}}{j} = \cv{\mck}{\mathcal{I}}{j} = \Delta_{n}(\ww').\]
	Notice that
	\[
	\Delta_n(\ww) = \Delta_{\gamma_v}^{\gamma_v(n)}\prod_{\gamma \neq \gamma_v}\Delta_{\gamma}^{\gamma(n)} \quad \text{and} \quad   \Delta_{n}(\ww')= \Delta_{\gamma_{v'}}^{\gamma_{v'}(n)}\prod_{\gamma' \neq \gamma_{v'}} \Delta_{\gamma'}^{\gamma'(n)}.
	\]
	By Lemma~\ref{lemma: cycles end behavior remain unchanged under mutation}, we have $\gamma(n) = \gamma'(n)$ for any cycle $\gamma$ in $\ww$ and it is corresponding cycle $\gamma'$ in $\ww'$. Notice that if $\gamma \notin Y_v$ and $\gamma \neq \gamma_v$, then $\gamma(n) = 0$ and $\gamma'(n) = 0$. 
	Also notice that $\gamma_v(n) = \gamma_{v'}(n) = 1$. So we can conclude that $\Delta_{\gamma_v} = \Delta_{\gamma_{v'}}$.
	
	Finally if $\gamma \notin Y_v$ and $\gamma \neq \gamma_v$, then $\gamma$ and $\gamma'$ originate below $v$ and $v'$, respectively. Now as we have shown that $\Delta_{\gamma_v} = \Delta_{\gamma_{v'}}$, we conclude that $\Delta_\gamma = \Delta_{\gamma'}$ as they are both defined recursively. 
	This completes the proof that the seeds associated with equivalent Demazure weaves are the same.
	
	\begin{figure}[H]
		\centering
		\includegraphics[trim = 0.5cm 7.5cm 0cm 13cm, clip = true, scale = 0.38]{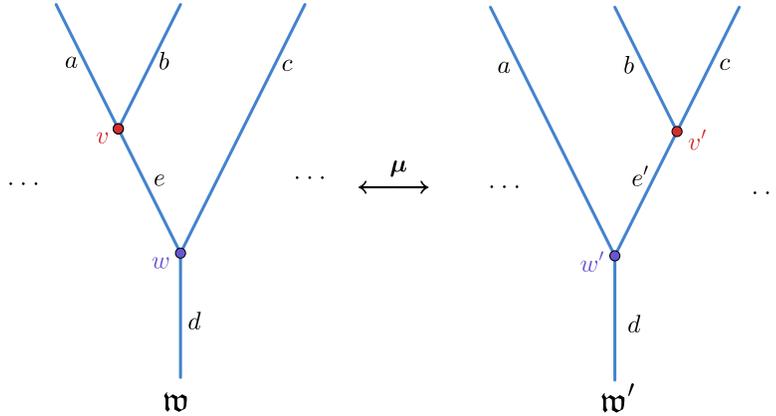}
		\caption{A local mutation move.}
		\label{fig: local mutation move seeds version}
	\end{figure}
	Now we need to show that the seeds associated with mutation equivalent Demazure weaves are mutation equivalent. Since equivalence move does not change seeds, we only need to show that under a single mutation move (cf.\ Figure~\ref{fig: local mutation move seeds version}), the seeds associated with the corresponding Demazure weaves are mutation equivalent. To be precise, we are going to show that $\seed(\ww') = \mu_{\gamma_{v}}(\seed(\ww))$. 
	
	By Proposition~\ref{prop: mutation equivalent weaves yield mutation equivalent quivers}, we know that $Q(\ww') = \mu_{\gamma_v}(Q(\ww))$. Now we show that the cluster variables other than $\Delta_{\gamma_v}$ (resp., $\Delta_{\gamma_{v'}}$) coincide, and $\Delta_{\gamma_v}$ and $\Delta_{\gamma_{v'}}$ satisfy the exchange relation. 
	
	Let $\gamma$ be a cycle in $\ww$ that originates above $v$ and $\gamma'$ be the corresponding cycle in $\ww'$. Then $\Delta_\gamma = \Delta_{\gamma'}$ as they are recursively defined by the cycles above, and $\ww$ and $\ww'$ are the same above the local picture. 
	
	Let $\gamma$ be a cycle in $\ww$ that originates below $v$. Then $\Delta_\gamma$ is recursively defined by all the cycles above it other than $\gamma_v$. This is because $\gamma_v$ ends at $w$, it will not enter the edge where $\gamma$ originates. Similarly, the corresponding cycle $\Delta_{\gamma'}$ in $\ww'$ is recursively (under the same formula) defined by all cycles above it other than $\gamma_{v'}$. Therefore, $\Delta_\gamma = \Delta_{\gamma'}$. 
	
	Let $Y$ (resp. $Y'$) be the set of cycles $\gamma \neq \gamma_v$ (resp. $\gamma'\neq \gamma_{v'}$) in $\ww$ (resp. $\ww'$). Then the discussion above shows that $\Delta_{\gamma} = \Delta_{\gamma'}$ for $\gamma\in Y$ and $\gamma'\in Y'$.
	
	Finally let us show that $\Delta_{\gamma_v}$ and $\Delta_{\gamma_{v'}}$ satisfy the exchange relation:
	\[
	\Delta_{\gamma_v}\Delta_{\gamma_{v'}}= \prod_{\gamma\in Y}\Delta_{\gamma}^{[\langle \gamma, \gamma_v \rangle]_+} + \prod_{\gamma\in Y}\Delta_{\gamma}^{-[\langle \gamma, \gamma_v \rangle]_-}.
	\]
	By $3$-term Pl\"ucker relations, we have
	\begin{equation}
		\Delta_{e}\Delta_{e'} = \Delta_{a}\Delta_{c} + \Delta_b\Delta_d.
	\end{equation}
	Furthermore, we have 
	\[
	\Delta_e = \Delta_{\gamma_v}\prod_{\gamma\in Y}\Delta_{\gamma}^{\gamma(e)} \quad \text{and}\quad \Delta_{e'} = \Delta_{\gamma_{v'}}\prod_{\gamma'\in Y'} \Delta_{\gamma'}^{\gamma'(e')};
	\]
	and let $\ell \in \{a, b, c, d\}$, notice that $\gamma_v(\ell) = 0$. Hence
	\[
	\Delta_\ell = \Delta_{\gamma_v}^{\gamma_v(\ell)}\prod_{\gamma\in Y}\Delta_{\gamma}^{\gamma(\ell)} =\prod_{\gamma\in Y}\Delta_{\gamma}^{\gamma(\ell)}.
	\]
	Therefore we have
	\[
	\Delta_{\gamma_v}\Delta_{\gamma_{v'}}\prod_{\gamma\in Y} \Delta_{\gamma}^{\gamma(e)+\gamma'(e')} = \prod_{\gamma\in Y}\Delta_{\gamma}^{\gamma(a) + \gamma(c)} + \prod_{\gamma\in Y}\Delta_{\gamma}^{\gamma(b) + \gamma(d)},
	\]
	where on the left the identity, we are using the fact that $\Delta_\gamma = \Delta_{\gamma'}$. 
	Hence
	\[
	\Delta_{\gamma_v}\Delta_{\gamma_{v'}}= \prod_{\gamma\in Y}\Delta_{\gamma}^{\gamma(a) + \gamma(c) - (\gamma(e)+\gamma'(e'))} + \prod_{\gamma\in Y}\Delta_{\gamma}^{\gamma(b) + \gamma(d)-(\gamma(e)+\gamma'(e'))}.
	\]
	Notice that
	\[
	\gamma(a) + \gamma(c) - (\gamma(e)+\gamma'(e')) = \gamma(a) + \gamma(c)  - \min\{\gamma(b), \gamma(c)\} - \min\{\gamma(a), \gamma(b)\};
	\]
	\begin{multline*}
	\gamma(b) + \gamma(d)-(\gamma(e)+\gamma'(e')) \\= \gamma(b) + \min\{\gamma(a), \gamma(b),\gamma(c)\}- \min\{\gamma(b), \gamma(c)\} - \min\{\gamma(a), \gamma(b)\};
	\end{multline*}
	and
	\[
	\langle \gamma, \gamma_v \rangle = \gamma(a) - \gamma(b) + \gamma(c) - \min\{\gamma(a), \gamma(b), \gamma(c)\}.
	\]
	Hence by Lemma~\ref{lemma: min equality}, we have
	\[
	\gamma(a) + \gamma(c) - (\gamma(e)+\gamma'(e')) = [\langle \gamma, \gamma_v \rangle]_+;
	\]
	\[
	\gamma(b) + \gamma(d)-(\gamma(e)+\gamma'(e')) = -[\langle \gamma, \gamma_v \rangle]_-;
	\]
	and therefore we have
	\[
	\Delta_{\gamma_v}\Delta_{\gamma_{v'}}= \prod_{\gamma\in Y}\Delta_{\gamma}^{[\langle \gamma, \gamma_v \rangle]_+} + \prod_{\gamma\in Y}\Delta_{\gamma}^{-[\langle \gamma, \gamma_v \rangle]_-}.
	\]
	This completes the proof that $\seed(\ww') = \mu_{\gamma_{v}}(\seed(\ww))$.
\end{proof}

\begin{lemma}\label{lemma: min equality}
	Let $a, b, c \in \Z$, then we have the following identities.
	\begin{enumerate}
		\item $[a-b+c-\min\{a, b, c\}]_+ =  a + c - \min\{b, c\} - \min\{a, b\}$.
		\item $-[a-b+c-\min\{a, b, c\}]_- = b - \min\{b, c\} - \min\{a, b\} + \min\{a, b, c\}$.
	\end{enumerate}
	Here $[x]_+ := \max\{x, 0\}$ and $[x]_- := \min\{x, 0\}$.
\end{lemma}
\begin{proof}
	By tropicalizing the identity
	\[
	\frac{(t^a+t^b+t^c)t^b}{t^{a+c}} + t^0  = \frac{(t^b+t^c)(t^a+t^b)}{t^{a+c}},
	\]
	we get 
	\[\min\{\min\{a, b, c\} + b - a - c, 0\} = \min\{b, c\} + \min\{a, b\} - a - c,\]
	hence the first identity holds since $[x]_+ = -[-x]_-$.
	The second identity holds since $[x]_+ + [x]_- = x$.
\end{proof}

\begin{theorem}\label{thm: demazure weaves same top/bottom yield mutation equivalent seeds}
	Let $\ww, \ww': \beta \rarrow \beta'$ be two reduced Demazure weaves. Then the seeds $\seed(\ww)$ and $\seed(\ww')$ are mutation equivalent. Consequently, $\mca(\ww) = \mca(\ww')$.
\end{theorem}
\begin{proof}
	This follows from Theorem~\ref{thm: demazure classification} and Theorem~\ref{thm: mutation equivalence demazure weaves yield mutation equivalent seeds}.
\end{proof}

\subsection{Cluster structures in mixed Grassmannians}\label{sec: cluster structure on mixed Grassmannians}

In this section, we formulate (and illustrate) our main results concerning the existence of cluster structures on mixed Plücker rings.

Let $\sigma$ be a size $n$ $d$-admissible signature of type $(a, b)$, cf.\ Definition~\ref{defn: admissible signature}. 
\textbf{Critical constraints are}:
\begin{itemize}
	\item odd dimension: $d$ is odd;
	\item size condition: $n > d^2$.
\end{itemize}
These constraints will be assumed as always unless stated otherwise. The necessity of these constraints is explained in Remark~\ref{remk: d odd and n>d squared explained}.

We first define a cluster algebra $\mcas$ (Definition~\ref{defn: cluster algebra A sigma}) for the chosen signature and then state that the mixed Pl\"ucker ring $R_\sigma$ is equal to the cluster algebra $\mcas$ (Theorem~\ref{thm: mixed grassmannian is a cluster algebra}).

\begin{defn}
	Let $p, q\in [1, n]$. \emph{The $n$-cyclic distance between $p$ and $q$}, denoted by $\|p-q\|_n$, is defined as
	\[
	\|p-q\|_n := \min\{|p-q|, n - |p-q|\}.
	\]
	We usually write $\|p-q\|$ if $n$ is clear from the context. 
\end{defn}

\begin{defn}\label{defn: valid cut in general}
	Let $1\le p\le q \le n$. We say that a pair $(p, q)$ is a \emph{valid cut of $\sigma$} if $\|p - q\| \ge d+1$. 
\end{defn}

\begin{defn}\label{defn: reversed word}
	Let $\beta$ be a word. We use $\rev \beta$ to denote the reversed word of $\beta$. For example, if $\beta = 123121$, then $\rev \beta = 121321$. If $\beta$ is marked, then we require the same marking on $\rev \beta$ as on $\beta$. As an example, if $\beta = \overline{1}23\overline{1}2\overline{1}$, then $\rev \beta = \overline{1} 2 \overline{1} 32 \overline{1}$. 
\end{defn}

Recall from Definition~\ref{defn: configuration space of a signature, and defn of beta sigma} that the marked words $\rho = 12\cdots \overline{(d-1)}$ and $\rho^* = {(d-1)}\cdots 2\overline{1}$ (with the last character marked).  The word associated with $\sigma$ is
\[
\beta_\sigma = \prod_{j = 1}^n \rho_{j}, \quad \rho_j \in \{\rho, \rho^*\}
\]
where $\rho_j = \rho$ (resp. $\rho_j = \rho^*$) if $\sigma(j) = 1$ (resp. $\sigma(j) = -1$). 

We first defined two reduced Demazure weaves associated with $\sigma$ with respect to a valid cut $(p, q)$. 

\begin{defn}\label{defn: cutting along the disk to get two weaves}
Consider a valid cut $(p, q)$ of $\sigma$. It will break $\beta_\sigma$ into two marked subwords
\[
\beta_1 = \beta_\sigma(p, q) =  \prod_{j = p}^{q-1} \rho_j\quad \text{and} \quad \beta_2 = \beta_\sigma(q, p+n) =  \prod_{j = q}^{p+n-1} \rho_j,
\]
where $\rho_j = \rho$ (resp., $ \rho_j = {\rho^*}$) if $\sigma(j) = 1$ (resp., $\sigma(j) = -1$). 
Let $\ww_1, \ww_2$ be  two reduced Demazure weaves 
\[
\ww_1 = \ww_\sigma(p, q): \beta_1 \rarrow \beta_1', \quad \ww_2 = \ww_\sigma(q, p+n) : \beta_2 \rarrow \beta_2',
\] 
where $\beta_1'$ and $\beta_2'$ are reduced words such that $\beta_1' = \rev{\beta_2'}$ (this is possible as they are both braid equivalent to $w_0$, by Lemma~\ref{lemma: bottom word is w0}).
\end{defn}

Now we define the cluster algebra associated with $\sigma$ with respect to a valid cut $(p, q)$.
\begin{defn}\label{defn: cluster algebra cutting at p, q}
Let $\seed_1 = \seed(\ww_1)$ be the seed associated with $\ww_1 = \ww_\sigma(p, q)$ and $\seed_2 = \seed(\ww_2)$ be the seed associated with $\ww_2 = \ww_\sigma(q, p+n)$, cf.\ Definition~\ref{defn: seed from weave}. Then by Lemma~\ref{lemma: amalgamation condition is satisfied}, $\seed_1$ and $\seed_2$ can be amalgamated along
\[
\zz_0 = \{\mcf_p^1\wedge \mcf_q^{d-1}, \mcf_p^2\wedge \mcf_q^{d-2}, \dots, \mcf_p^{d-1}\wedge \mcf_q^1\}.
\]
Let 
$\seed_\sigma(p, q):= \glueseeds{\seed_1}{\seed_2}{\zz_0}
$
be the amalgamated seed, cf.\ Definition~\ref{defn: amalgamating two seeds}; and let $\mca_\sigma(p, q)$ be the cluster algebra associated with $\seed_\sigma(p, q)$. 
\end{defn}

\begin{remk}
	Notice that $\seed_\sigma(p, q)$ does not depend on the choice of the (reduced) bottom words for $\ww_1$ and $\ww_2$ (we do require them to be reverse to each other), as changing the bottom word will only change the intersection pairings between the frozen variables in $\zz_0$, and these intersection pairings in $\ww_1$ will cancel out with those in $\ww_2$ when we amalgamate them. Also notice that $\mca_\sigma(p, q)$ does not depend on the choice of $\ww_1$ and $\ww_2$, since different choices of reduced Demazure weaves will result in mutation equivalent seeds by Theorem~\ref{thm: mutation equivalence demazure weaves yield mutation equivalent seeds}. 
\end{remk}

\begin{remk}\label{remk: d odd and n>d squared explained}
	Let us first explain why we require $d$ to be odd. By Proposition~\ref{prop: decoration flags of the initial weave}, the frozen variables in $\ww_1$ associated with the cycles ending at the bottom boundary of $\ww_1$ are 
	\[
	\{\mcf_p^1\wedge \mcf_q^{d-1}, \mcf_p^2\wedge \mcf_q^{d-2}, \dots, \mcf_p^{d-1}\wedge \mcf_q^1\};
	\]
	while the corresponding frozen variables in $\ww_2$ are
	\[
	\{\mcf_q^{d-1}\wedge \mcf_p^{1}, \mcf_q^{d-2}\wedge \mcf_p^{2}, \dots, \mcf_q^{1}\wedge \mcf_p^{d-1}\}.
	\]
	Notice a sign difference between these two sets:
	\[
	\mcf_p^k\wedge \mcf_q^{d-k} = (-1)^{k(d-k)} \mcf_q^{d-k}\wedge \mcf_p^k\  \text{ for }  k\in [1, d-1].
	\]
	Requiring $d$ to be odd resolves the sign difference.

	Why would such a sign issue arise? The key observation is that the mixed Grassmannian only have cyclic symmetry up to a sign. Detailed discussion on the sign issue and one potential solution to the even dimension case can be found in Section \ref{sec: intro main results}.  
	

	As explained in Section \ref{sec: intro main results}, the condition $n>d^2$ is not essential for the construction of cluster structures in mixed Grassmannians, but is needed  for the proof of the main theorem. When the signature is ``nice", the condition $n>d^2$ can be relaxed significantly (e.g., cf.\ Theorem~\ref{thm: mixed grassmannian is a cluster algebra in the case of separated signatures}).  
\end{remk}

The subscript $\sigma$ will be dropped in $\beta_\sigma(p,q), \ww_\sigma(p,q), \seed_\sigma(p, q), \mca_\sigma(p, q)$, etc., when $\sigma$ is clear from the context.

The following two lemmas are needed for Definition~\ref{defn: cutting along the disk to get two weaves} and Definition~\ref{defn: cluster algebra cutting at p, q}, and will be proved later in Section~\ref{sec: initial weave} and~\ref{section: weyl generators} respectively. 
\begin{lemma}\label{lemma: bottom word is w0}
	In Definition~\ref{defn: cutting along the disk to get two weaves}, $\beta'_1$ and $\beta'_2$ are reduced words for $w_0\in S_d$.
\end{lemma}

\begin{lemma}\label{lemma: amalgamation condition is satisfied}
	In Definition~\ref{defn: cluster algebra cutting at p, q}, the two seeds $\seed_1$ and $\seed_2$ can be amalgamated along the subset $\zz_0 = \{\mcf_p^1\wedge \mcf_q^{d-1}, \mcf_p^2\wedge \mcf_q^{d-2}, \dots, \mcf_p^{d-1}\wedge \mcf_q^1\}$.
\end{lemma}

Let us work out an example in detail, describing the cluster algebra $\mcas(1, 5)$. 

\begin{example}\label{example: describing a seed from a decorated weave}
	Let $\dim V = d = 3, n = 8$, and $\sigma = [\bullet\, \circ\, \bullet\, \bullet\, \bullet\, \bullet\, \circ\ \bullet]$. Then
	\[
	V^\sigma = V \times V^*\times V\times V\times V\times V\times V^*\times V,
	\]
	and an element $\uu\in V^\sigma$ can be written as 
	\[
	\uu = (u_1, u^*_2, u_3, u_4, u_5, u_6, u^*_7, u_8).
	\]
	We have $\beta_\sigma = \rho\rho^*\rho\rho\rho\rho\rho^*\rho$ where $\rho = 1\overline{2}$ and $\rho^* = 2\overline{1}.$
	Take a valid cut $(p, q) = (1, 5)$. Then $\beta_1 = \rho\rho^*\rho\rho = 1\overline{2}2\overline{1}1\overline{2}1\overline{2}$. 
	
	Recall from Remark \ref{remk: cycles are in Ksigma}, Theorem~\ref{thm: bijection between configuration spaces of signature and decorated flag moduli spaces} and  Lemma~\ref{lem: pre-assignment of top boundary flags}, the decoration $\dec_{\beta_\sigma}$ for the word $\beta_\sigma$ is of the form as described in Lemma~\ref{lem: pre-assignment of top boundary flags}; and the decoration for $\beta_1$ is the restriction of $\dec_{\beta_\sigma}$ to $\beta_1$ (cf.\ Remark~\ref{remk: restriction of decorations to a subword}). 
	
	An example of a decorated Demazure weave $\ww_1: \beta_1\rarrow \beta'_1 = 212$ is shown in Figure~\ref{fig: example of beta1}.  The (Lusztig) cycles are shown in Figure~\ref{fig: Example of beta1 with cycles part1} and Figure~\ref{fig: Example of beta1 with cycles part2}.

	\begin{figure}[H]
		\centering
		\includegraphics[trim= 0cm 4cm 5cm 0cm, clip = true, scale = 0.48]{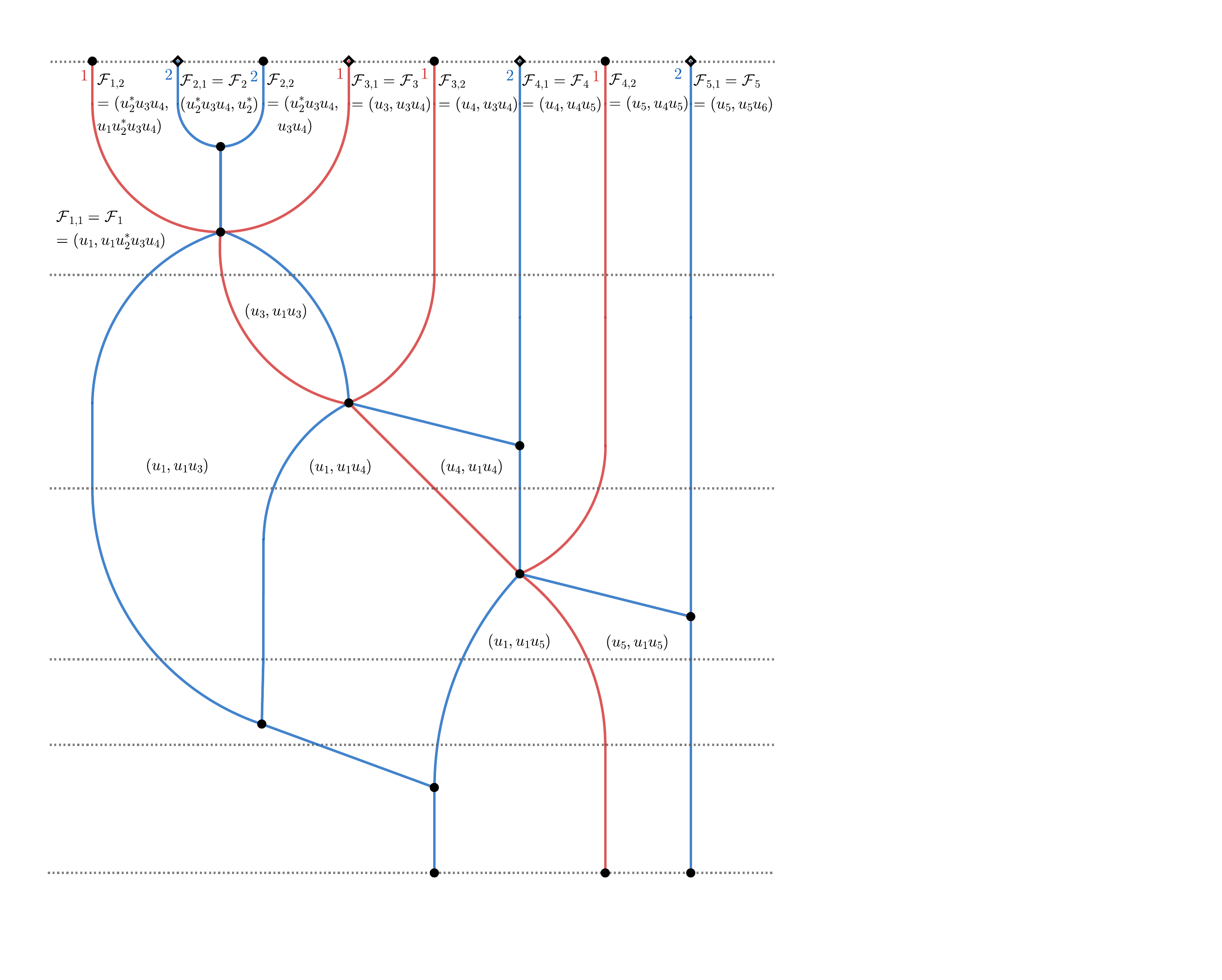}
		\caption{A decorated Demazure weave $\ww_1: \beta_1\rarrow \beta'_1 = 212$. The mixed wedge operators are omitted. For example, $u_1u^*_2u_3u_4$ stands for $u_1\mixwed u_2^*\mixwed u_3 \mixwed u_4$.}
		\label{fig: example of beta1}
	\end{figure}

\begin{figure}[H]
	\centering
	\includegraphics[trim= 0cm 4cm 5cm 0cm, clip = true, scale = 0.5]{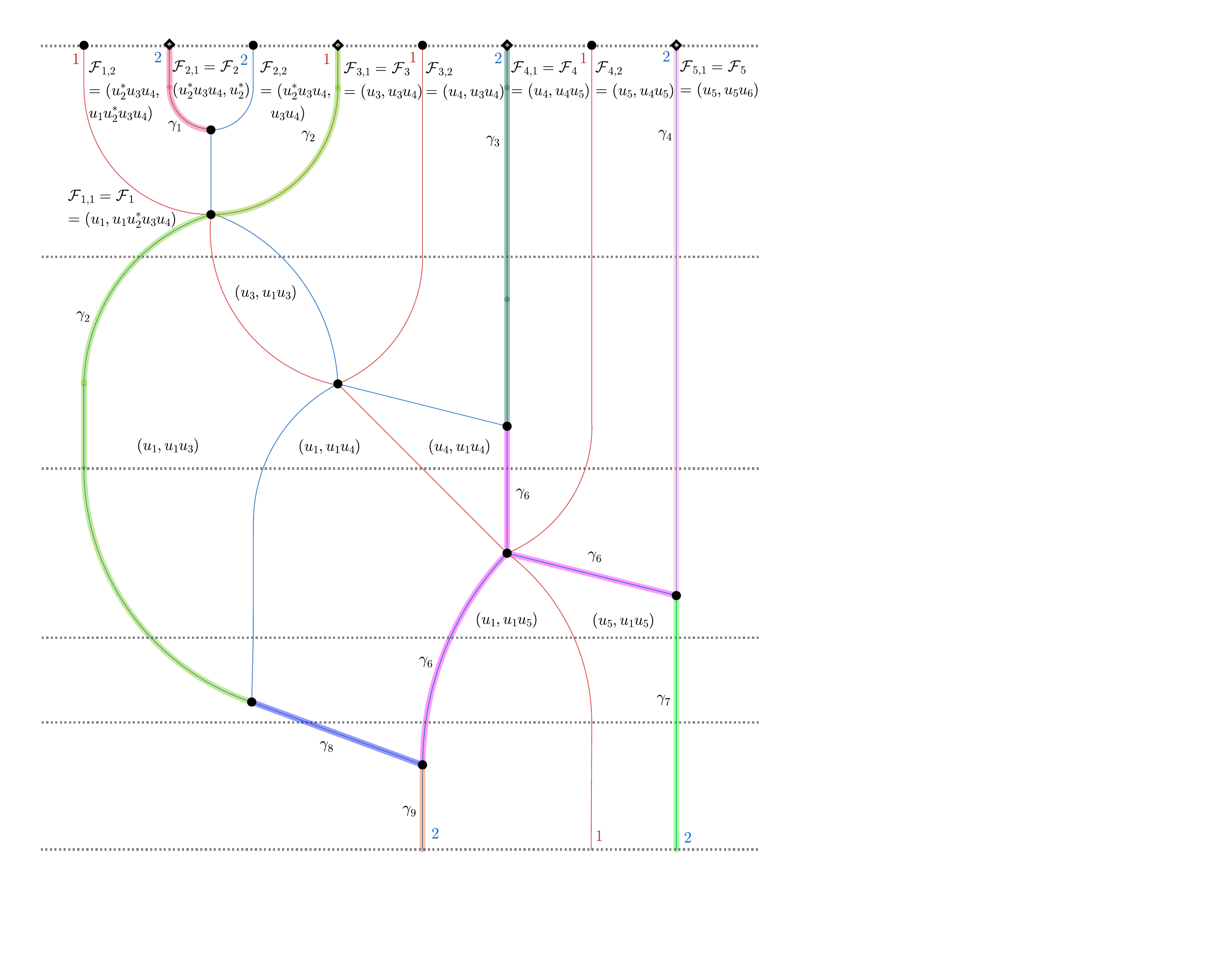}
	\caption{Cycles $\gamma_1, \gamma_2, \gamma_3, \gamma_4, \gamma_6, \gamma_7, \gamma_8, \gamma_9$ for the decorated Demazure weave $\ww_1: \beta_1\rarrow \beta_1' = 212$ from Figure~\ref{fig: example of beta1}. The cycle $\gamma_5$ is shown in Figure~\ref{fig: Example of beta1 with cycles part2}.}
	\label{fig: Example of beta1 with cycles part1}
\end{figure}

\begin{figure}[H]
	\centering
	\includegraphics[trim= 0cm 2.5cm 5cm 0cm, clip = true, scale = 0.5]{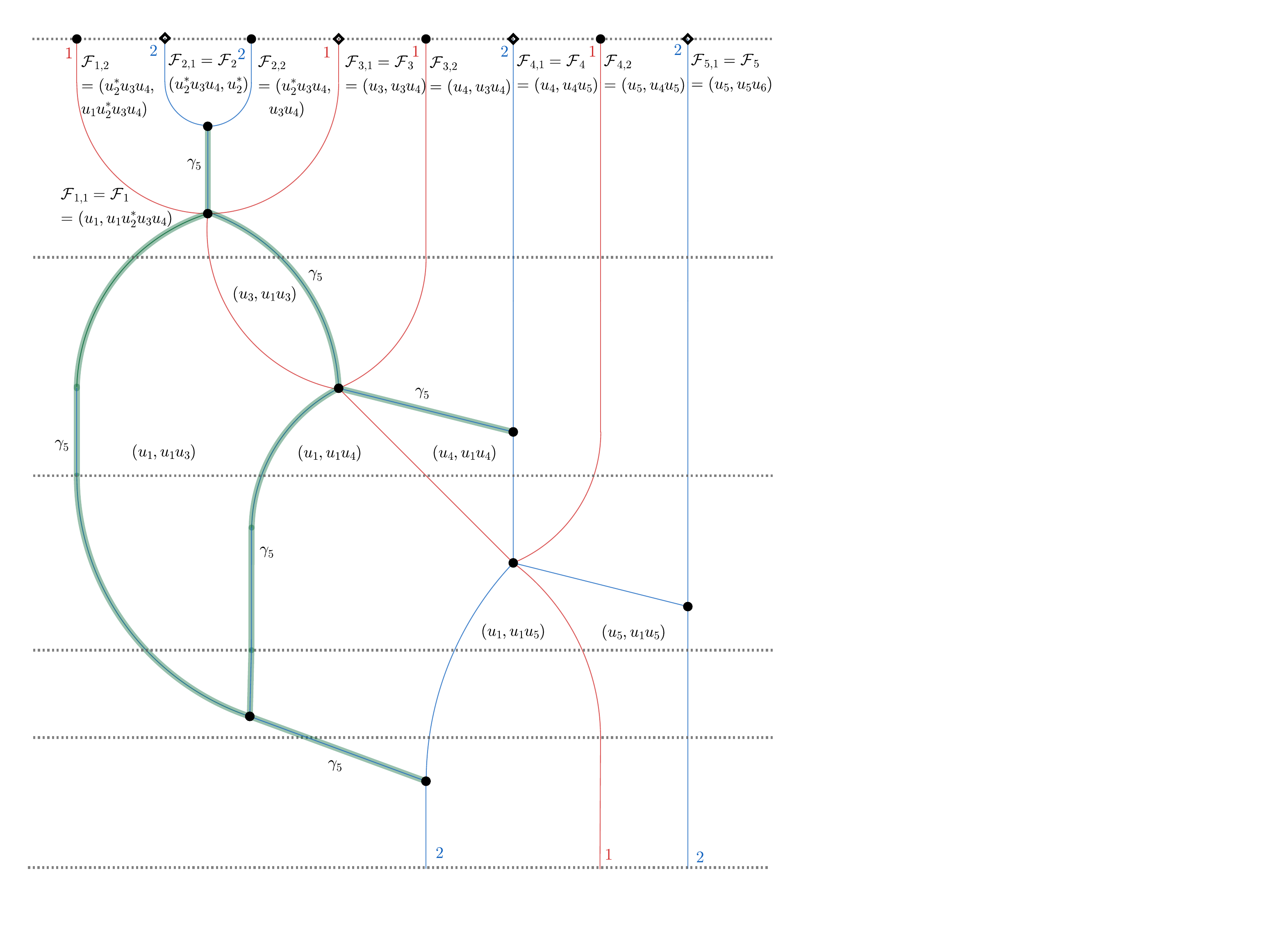}
	\caption{The cycle $\gamma_5$ in the decorated Demazure weave $\ww_1: \beta_1\rarrow \beta_1' = 212$ from Figures~\ref{fig: example of beta1} and~\ref{fig: Example of beta1 with cycles part1}.}
	\label{fig: Example of beta1 with cycles part2}
\end{figure}

We next construct the seed $\seed(\ww_1) = (Q_1, \zz_1)$ associated with $\ww_1$, as described in Definition~\ref{defn: seed from weave}. Let $i$ denote the vertex of the quiver corresponding to the cycle $\gamma_i$, and let $\Delta_i :=\Delta_{\gamma_i}$ be the cluster or frozen variable associated with the vertex $i$, $1\le i \le 9$. The quiver $Q_1 = Q(\ww_1)$ associated with the weave $\ww_1$ is obtained by drawing all the cycles of $\ww_1$, cf.\ Figure~\ref{fig: Example of beta1 with cycles part1} and~\ref{fig: Example of beta1 with cycles part2}; and computing the intersection pairings between the cycles, cf.\ Definition~~\ref{defn:local intersection pairing}. The result is shown in Figure~\ref{fig: quiver associated with m1}.

\begin{figure}
	\centering
	\begin{tikzcd}
		\boxed{1} \arrow[dr] & \boxed{2} \arrow[ddr]   & \boxed{3} \arrow[dl] & \boxed{4} \arrow[dl]\\
		&5 \arrow[drr]  & 6 \arrow[r] \arrow[u] \arrow[d] & \boxed{7} \arrow[u]\\
		&&8 \arrow[r] & \boxed{9} \arrow[ul]
	\end{tikzcd}
	\caption{The quiver $Q_1 = Q(\ww_1)$ associated with $\ww_1$.}
	\label{fig: quiver associated with m1}
\end{figure}
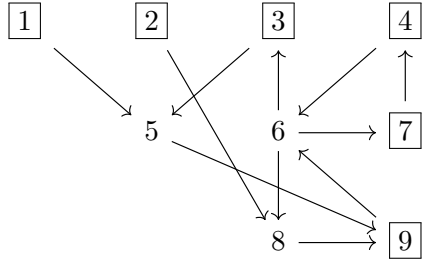

\noindent The extended cluster $\zz_1$ is computed by Definition~\ref{defn: decorated demazure weave}:
\begin{align*}
	\Delta_{1}&= u_2^*(u_1),  &\Delta_{2} &=  u_2^*(u_3), & \Delta_{3}&= \det(u_3, u_4, u_5),\\
 \Delta_{4} &= \det(u_4, u_5, u_6), &\Delta_{5} &= \det(u_1, u_3, u_4), &\Delta_{6} &= \det(u_1, u_4, u_5), \\
\Delta_{7}&= \det(u_1, u_5, u_6), &\Delta_{8} &= u_2^*(u_4),& \Delta_{9} &=  (u_1 \mixwed u_2^*\mixwed u_3\mixwed u_4)\mixwed u_5.\\
\end{align*}

\noindent The exchange relations are:
\begin{equation*}
	\begin{split}
		\Delta_5\Delta'_5 &= \Delta_1\Delta_3 + \Delta_9, \quad \text{where }   \Delta'_5 = u_2^*(u_5);\\
		\Delta_6\Delta'_6 &= \Delta_4\Delta_9 + \Delta_3\Delta_7\Delta_8, \quad \text{where }   \Delta'_6 = (u_2^*\mixwed u_3\mixwed u_4)\mixwed u_5 \mixwed u_6;\\
		\Delta_8\Delta'_8 &= \Delta_2\Delta_6 + \Delta_9, \quad \text{where }  \Delta'_8 = \det(u_1, u_3, u_5).
	\end{split}
\end{equation*}
Notice that the mutable part of $Q_1$ is of type $A_1\times A_2$, so the cluster algebra $\mca(\ww_1)$ is generated by the following extended clusters: $\zz_1, \mu_5(\zz_1), \mu_6(\zz_1), \mu_8(\zz_1), \mu_6(\mu_8(\zz_1))$. The new cluster variable in $\mu_6(\mu_8(\zz_1))$ is $\det(u_3, u_5, u_6)$. Finally, notice that 
\begin{align*}
\Delta_{9} &=  (u_1 \mixwed u_2^*\mixwed u_3\mixwed u_4)\mixwed u_5 \\
&= u_2^*(u_4)\det(u_1, u_3, u_5) - u_2^*(u_3)\det(u_1, u_4, u_5);
\end{align*}
and 
\begin{align*}
	\Delta'_6 &= (u_2^*\mixwed u_3\mixwed u_4)\mixwed u_5 \mixwed u_6\\
	&=u^*_2(u_4)\det(u_3, u_5, u_6) - u_2^*(u_3)\det(u_4, u_5, u_6).
\end{align*}
We conclude that $\mca(\ww_1)$ is a subalgebra of $\rsigma$ generated by
\begin{multline*}
\big\{u_2^*(u_1), u_2^*(u_3), u_2^*(u_4), u_2^*(u_5), 
\det(u_1, u_3, u_4),\det(u_1, u_3, u_5),
\det(u_1, u_4, u_5),\\
\det(u_1, u_5, u_6),\det(u_3, u_4, u_5),\det(u_3, u_5, u_6),\det(u_4, u_5, u_6)\big\}.
\end{multline*}
Similarly, for $\beta_2 = \rho \rho \rev \rho \rho = 1\overline{2}1\overline{2}2\overline{1}1\overline{2}$ and $\beta' = 212$, we pick a Demazure weave $\ww_2: \beta_2 \rarrow \beta_2' = 212$ (details are omitted, it is picked as the \emph{initial weave} as in Section~\ref{sec: initial weave}). The seed for $\ww_2$ is $\seed(\ww_2) = (Q_2, \zz_2)$ with the quiver $Q_2 = Q(\ww_2)$ as shown in Figure~\ref{fig: quiver associated with m2}.
\begin{figure}
	\centering
	\begin{tikzcd}
		\boxed{10} \arrow[dr] \arrow[bend right = 30]{ddrr} & \boxed{11} \arrow[l]   & \boxed{12} \arrow[bend right = 30]{ll} & \boxed{13} \arrow[dl]\\
		&14 \arrow[r] \arrow[u]   & 15 \arrow[llu] \arrow[r] & \boxed{16} \arrow[u]\\
		&&17 \arrow[r] \arrow[uu, bend right = 50] & \boxed{18} \arrow[ul]
	\end{tikzcd}
	\caption{The quiver $Q_2 = Q(\ww_2)$ associated with $\ww_2$.}
	\label{fig: quiver associated with m2}
\end{figure}
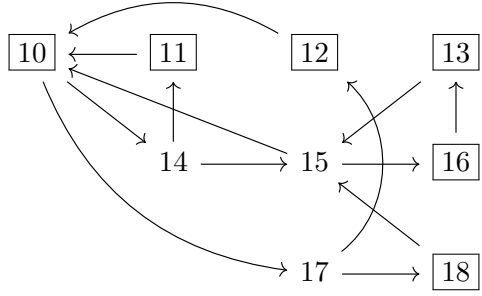
\noindent The corresponding extended cluster $\zz_2$ is given by
\begin{align*}
	\Delta_{10} & = u_5\mixwed u_6\mixwed u^*_7 \mixwed u_8\mixwed u_1, &\Delta_{11} & = u_7^*(u_6), &\Delta_{12} &= u^*_7(u_8),\\
	\Delta_{13} &=u_8\wed u_1 \wed u^*_2 \wed u_3\wed u_4, &\Delta_{14} & = u_7^*(u_5), & \Delta_{15} &= \det(u_5, u_8, u_1),\\
	\Delta_{16} &= u_5\wed u_1 \wed u^*_2 \wed u_3\wed u_4, &\Delta_{17} &= \det(u_5, u_6, u_8), &\Delta_{18} &= \det(u_5, u_6, u_1).
\end{align*}
The exchange relations are
\begin{align*}
	\Delta_{14}\Delta_{14}' &= \Delta_{10}+ \Delta_{11}\Delta_{15}, \quad \text{where } \Delta'_{14} = \det(u_6, u_8, u_1);\\
	\Delta_{15}\Delta'_{15} & = \Delta_{13}\Delta_{14}\Delta_{18} + \Delta_{10}\Delta_{16}, \quad \text{where } \Delta'_{15} = u_5\mixwed u_6\mixwed u_7^*\mixwed u_1\mixwed u_2^*\mixwed u_3\mixwed u_4;\\
	\Delta_{17} \Delta'_{17} &= \Delta_{10} + \Delta_{12}\Delta_{18}, \quad \text{where } \Delta'_{17} = u_7^*(u_1).
\end{align*}
Similar to the calculation of the generators for $\mca(\ww_1)$, the extra generator needed here is $\mu_{14}(\mu_{16}(\zz_2)) = u_6\wed u_1 \wed u^*_2 \wed u_3\wed u_4$; and  we can conclude that the cluster algebra $\mca(\ww_2)$ is a subalgebra of $\rsigma$ generated by 
\begin{multline*}
	\big\{
	u_7^*(u_5), u_7^*(u_6), u_7^*(u_8), u_7^*(u_1), \det(u_5, u_6, u_8), \det(u_5, u_6, u_1), \det(u_5, u_8, u_1), \\
	\det(u_6, u_8, u_1), u_5\wed u_1 \wed u^*_2 \wed u_3\wed u_4,  u_6\wed u_1 \wed u^*_2 \wed u_3\wed u_4, u_8\wed u_1 \wed u^*_2 \wed u_3\wed u_4
\big	\}.
\end{multline*}
Notice that 
\[
\Delta_{7} = \Delta_{18} = \det(u_5, u_6, u_1) = \mcf_1^1\wedge \mcf_5^2,
\]
and 
\[
\Delta_{9} = \Delta_{16} = u_5\wed u_1 \wed u^*_2 \wed u_3\wed u_4 = \mcf_1^2 \wedge \mcf_5^1.\]
We can now amalgamate $\seed(\ww_1)$ and $\seed(\ww_2)$ along 
\begin{align*}
\zz_0 &= \{\Delta_7, \Delta_9\} = \{\Delta_{16}, \Delta_{18}\} \\
&= \{\det(u_5, u_6, u_1), u_5\wed u_1 \wed u^*_2 \wed u_3\wed u_4\} = \{ \mcf_1^1\wedge \mcf_5^2, \mcf_1^2 \wedge \mcf_5^1\}.
\end{align*}
The resulting seed $\seed_\sigma(1, 5) = (Q, \zz)$ has the quiver $Q$ as shown in Figure~\ref{fig: amalgamated quiver from m1 and m2} and $\zz = \zz_1\cup \zz_2$. The quiver $Q$ is obtained by first rotating $Q_2$ in Figure~\ref{fig: quiver associated with m2} by $180^\circ$, and then identifying the vertex $18$ (resp., $16$) in $Q_2$ with vertex $7$ (resp., $9$) in $Q_1$, and defrosting the resulting vertices $7$ and $9$. Then $\mca_\sigma(1, 5)$ is the cluster algebra associated with $\seed_\sigma(1, 5)$. 

Notice that the mutable part of the quiver $Q$ has mutation type $T_{433}$, matching the result in \cite[Figure 20]{FominPylyavskyy}. 
\begin{figure}
	\centering
	\begin{tikzcd}
		\boxed{1} \arrow[dr] & \boxed{2} \arrow[ddr]   & \boxed{3} \arrow[dl] & \boxed{4} \arrow[dl] &&&\\
		&5 \arrow[drr]  & 6 \arrow[r] \arrow[u] \arrow[d] & {7} \arrow[u] \arrow[dr] &  17 \arrow[l] \arrow[dd, bend left = 50] & &\\
		&&8 \arrow[r] & {9} \arrow[ul] \arrow[d] & 15 \arrow[l] \arrow[drr] & 14 \arrow[l] \arrow[d] &\\
		&&& \boxed{13} \arrow[ur] &\boxed{12} \arrow[rr, bend right = 30] & \boxed{11} \arrow[r] & \boxed{10} \arrow[ul] \arrow[uull, bend right]
	\end{tikzcd}
	\caption{The quiver $Q$ amalgamated from $Q_1$ and $Q_2$ along the subset $\{7 = 18, 9 =16\}$.}
	\label{fig: amalgamated quiver from m1 and m2}
\end{figure}
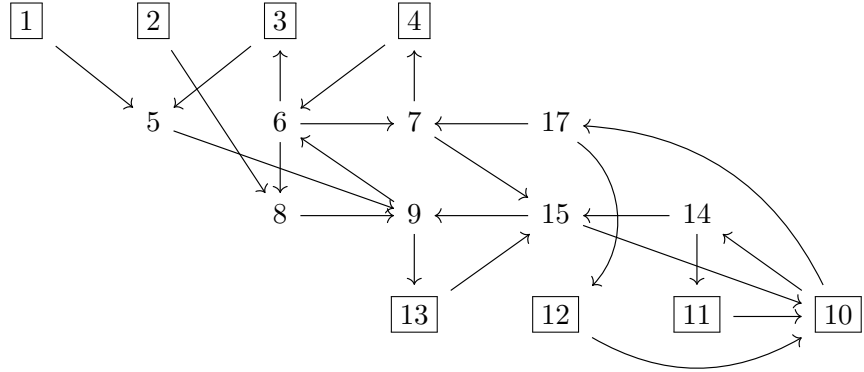

Lastly, let us show that $\mca_\sigma(1,5) \subseteq R_\sigma$.  In the seed $\seed_\sigma(1, 5)$, the exchange relations for the defrosted cluster variables are:
\begin{align*}
	\Delta_7\Delta'_7 &= \Delta_6\Delta_{17} + \Delta_4\Delta_{15}, \quad \text{where }\Delta'_7 = \det(u_4, u_5, u_8);\\
	\Delta_9\Delta'_9 &= \Delta_5\Delta_8\Delta_{15}+ \Delta_{6}\Delta_{13}, \quad \text{where }\Delta'_9 = \det(u_1, u_4, u_8).
\end{align*}
Notice that all the cluster/frozen variables and once-mutated cluster variables are irreducible invariants in $R_\sigma$. Therefore by Starfish Lemma, cf.\ Proposition~\ref{prop:cluster-criterion}, $\mca_\sigma(1, 5) \subseteq R_\sigma$. 

\end{example}

As anticipated, the cluster algebra $\mcas(p, q)$ exhibits independence from the choice of a valid cut $(p, q)$. This is a foundational result central to our framework. We rigorously establish this independence in Section~\ref{sec: ca does not depend of the choice of the cut}.

\begin{prop}\label{prop: cluster algebra does not depend on the cut}
	The cluster algebra $\mca_\sigma(p, q)$ does not depend on the choice of  the valid cut $(p, q)$.
\end{prop}

\begin{proof}
See Section~\ref{sec: ca does not depend of the choice of the cut}.
\end{proof}

\begin{defn}\label{defn: cluster algebra A sigma}
	\emph{The cluster algebra $\mcas$ associated with $\sigma$} is defined as the cluster algebra $\mca_\sigma(p, q)$ for any valid cut $(p, q)$. This is well-defined by Proposition~ \ref{prop: cluster algebra does not depend on the cut}.
\end{defn}

We now present the central result of this work, Theorem~\ref{thm: mixed grassmannian is a cluster algebra}, which establishes the cluster algebra structure on mixed Grassmannians. The proof, developed through two distinct methodological approaches, appears in Section~\ref{sec: proof of the main theorem}.

\begin{theorem}\label{thm: mixed grassmannian is a cluster algebra}
	Assume that $d$ is odd and $n>d^2$. Let $\sigma$ be a $d$-admissible signature. Then the mixed Pl\"ucker ring $R_\sigma$ is a cluster algebra. Indeed we have 
	\[
	R_\sigma  
	= \mcas.
	\]
\end{theorem}

Theorem \ref{thm: mixed grassmannian is a cluster algebra} is stated for any $d$-admissible signature with requirements that $d$ is odd and $n>d^2$. As explained in Remark \ref{remk: d odd and n>d squared explained}, the condition $d$ is odd is always required, while the condition $n>d^2$ can be relaxed when the signature is nice. 

Recall from Definition \ref{defn: separated signature} that a signature  is {separated} if it contains $a$ consecutive black vertices followed by $b$ consecutive white vertices. By Lemma \ref{lemma: separate signature is admissible}, separated signatures are always $d$-admissible provided that $n \ge 2d-2$. The following result will be proved in Section \ref{sec: separated signatures}.

\begin{theorem}\label{thm: mixed grassmannian is a cluster algebra in the case of separated signatures}
	Assume that $d$ is odd and $n \ge 2d$. Let $\sigma$ be a separated signature with $a, b \ge d-1$. Then the mixed Pl\"ucker ring $R_\sigma$ is a cluster algebra. 
\end{theorem}

\begin{remk}
	Indeed, in Example \ref{example: describing a seed from a decorated weave}, we have $\mcas(1, 5) = R_\sigma$. This result does not follow from Theorem \ref{thm: mixed grassmannian is a cluster algebra}, as $n = 8 \le d^2 = 9$. The fact that all Weyl generators for $R_\sigma$ can be obtained by mutating the seed $\seed(1, 5)$ similarly follows from the proofs of Proposition~\ref{prop: cluster algebra does not depend on the cut} (in Section \ref{sec: ca does not depend of the choice of the cut}). 
\end{remk}


Properties of the cluster structure on $R_\sigma$ and connections to related work will be thoroughly examined in Section~\ref{sec: properties and further results}. We emphasize that the cluster structure on $R_\sigma(V) \cong R_{a, b}(V)$ critically depends on the signature $\sigma$. As shown in \cite[Figure 20]{FominPylyavskyy}, fixing $a$ and $b$ does not determine the cluster type of $R_\sigma$, nor whether $R_\sigma$ is of finite of infinite type. 

\newpage 
\section{Proof of the Main Theorem}\label{chap: proof of the main theorem}

\subsection{The cluster algebra $\mcas$ is well defined}\label{sec: ca does not depend of the choice of the cut}

The section proves Proposition~\ref{prop: cluster algebra does not depend on the cut}, which establishes that the cluster structure on a  mixed Grassmannian remain invariant under different choices of the cut $(p, q)$.

Recall from Definition~\ref{defn: interval words} that an interval word for an interval $[i, j]$ is a word either of the form $I_i^j = i (i+1)\cdots j$ or  of the form $I_j^i = j (j-1)\cdots i$, for $1\le i \le j \le d-1$.

\begin{defn}\label{defn: nested words and complete nested words}
	A word $T$ is called \emph{nested} if it has the form
	\[
	T = T_sT_{s+1}\cdots T_{d-1}, \quad \text{for some } s\in [1, d-1]
	\]
	where $T_s, \dots, T_{d-1}$ are interval words such that
	\begin{itemize}[wide, labelwidth=!, labelindent=0pt]
		\item $T_{d-1}$ is an interval word for $[1,d-1]$;
		\item for $k\in [s, d-1]$, $T_k$ is an interval word for $[i, j]$, $1\le i \le j \le d-1$; 
		\item for $k\ge s+1$, if $T_k = I_i^j$, then $T_{k-1}$ is an interval word for $[i+1, j]$; if $T_k = I_j^i$, then $T_{k-1}$ is an interval word for $[i, j-1]$. 
	\end{itemize}
	We call a nested word $T$ \emph{complete} if $s = 1$. For $k\in [s, d-1]$, $T_k$ is said to be an interval word for $T$. 
\end{defn}

\begin{example}
	$T^+ := I_{d-1}^{d-1}I_{d-2}^{d-1}\cdots I_1^{d-1}$ is a complete nested word with all $T_k$ increasing; and $T^-:= I_{1}^{1}I_{2}^{1}\cdots I_{d-1}^1$ is a complete nested word with all $T_k$ decreasing. 
\end{example}

\begin{remk}
	Let $T$ be a nested word as in Definition~\ref{defn: nested words and complete nested words}. For $k\in [1, d-1]$, let $x$ (resp. $y$) be the number of increasing (resp. decreasing) words $T_i$ with $i>k$. Then $T_k$ is an interval word for $[i, j]$ with $i =  x +1$ and $j = d- 1 - y$.
\end{remk}

\begin{defn}
	A \emph{marked complete nested word} is a complete nested word $T = T_1T_2\cdots T_{d-1}$ in which the last character of each interval word $T_k$ is marked.  As an example, $T = \bar{4}3\bar423\bar 4 123 \bar 4$ is a marked complete nested word. This is the only marking we will  be considering for a complete nested word.
\end{defn}

It's well-known that all complete nested words are braid equivalent to the longest reduced word $w_0$ in $S_d$. Moreover, the marking is preserved under the equivalence in the sense as described in the next lemma.

We say that an edge of a Demazure weave is a \emph{bottom edge} (resp., \emph{top edge}) if it is incident to the bottom boundary (resp., top boundary) of the weave. 

\begin{lemma}\label{lemma: complete nested words are all related by braid moves}
	Let $T$ be a marked complete nested word. Then any Demazure weave $\ww: T\rarrow T^+$ has the property that a bottom edge has a cycle passing through it if and only if the edge corresponds to a marked character in $T^+$. 
\end{lemma}
\begin{proof}
	First we notice that we only need to prove the statement for a specific Demazure weave by Theorem~\ref{thm: demazure classification} and Lemma~\ref{lemma: cycles end behavior remain unchanged under mutation}.
	
	Let us start with $T = T^{-} = I_1^1 \cdots I_{d-2}^1 I_{d-1}^1$. 
	Notice that we have partial weaves $I_k^1 I_{d-1}^1 \rarrow I_{d-1}^1I_{k+1}^2$ for any $1\le k \le d-2$. Indeed, when $k = d-2$, it is of the form in Figure~\ref{fig: T- to T+ boundary case} with $i = 1$; and when $k< d-2$, it is of the form in Figure~\ref{fig: T- to T+ generic case}.
	\begin{figure}[H]
		\centering
		\includegraphics[trim= 0cm 5cm 5.5cm 2cm, clip = true, scale = 0.38]{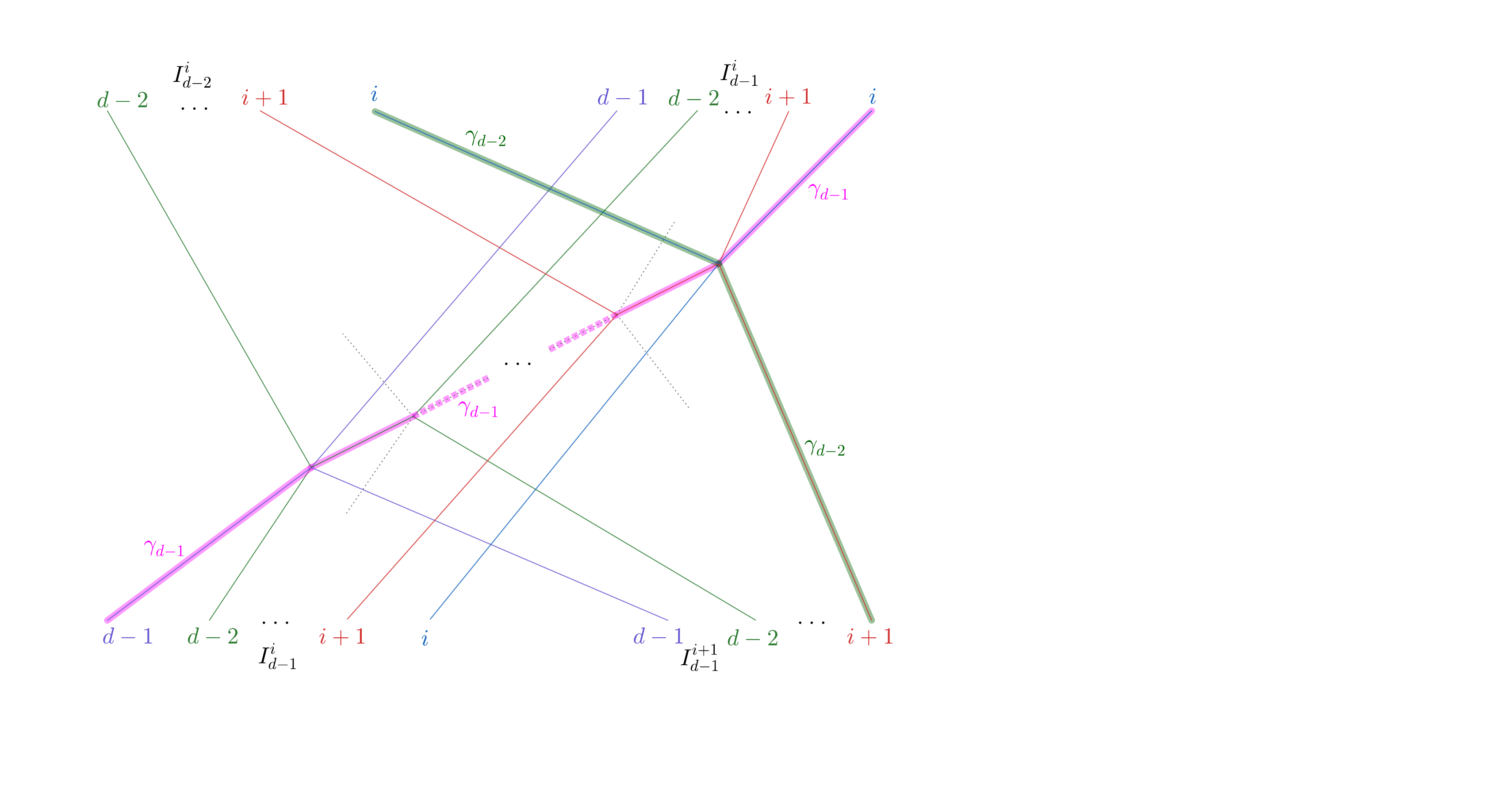}
		\caption{A (local) weave $\ww: I_{d-2}^i I_{d-1}^i \rarrow I_{d-1}^i I_{d-1}^{i+1}$, $1\le i \le d-2$.}
		\label{fig: T- to T+ boundary case}
	\end{figure}
	\begin{figure}[H]
		\centering
		\includegraphics[trim= 0cm 9cm 22cm 2cm, clip = true, scale = 0.38]{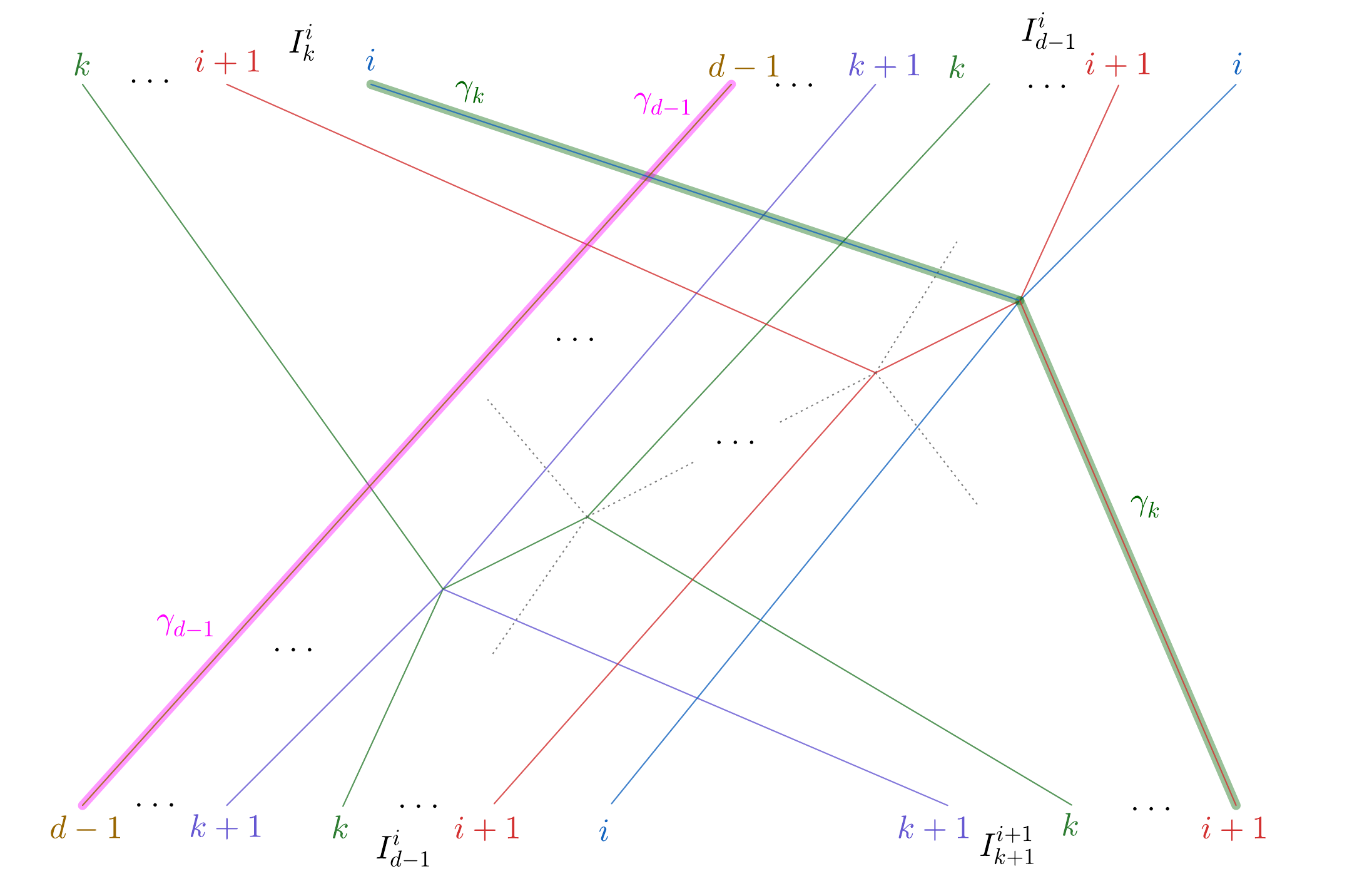}
		\caption{A (local) weave $\ww: I_{k}^i I_{d-1}^i \rarrow I_{d-1}^iI_{k+1}^{i+1}$, $1\le i\le k <d-2$.}
		\label{fig: T- to T+ generic case}
	\end{figure}
	
	Now for $k = d-2, d-3, \dots, 1$, by concatenating these $d-2$ partial weaves $I_k^1 I_{d-1}^1 \rarrow  I_{d-1}^1I_{k+1}^2$, we  obtain a Demazure weave
	\begin{multline*}
	T^- = I_{1}^{1}\cdots I_{d-2}^1 I_{d-1}^1 \rarrow  I_1^1  \cdots I_{d-3}^1 I_{d-1}^1 I_{d-1}^2\\ \rarrow
	  I_1^1 \cdots I_{d-4}^1 I_{d-1}^1 I_{d-2}^2 I_{d-1}^2 \rarrow \cdots \rarrow  I_{d-1}^1 I_{2}^2 I_{3}^2 \cdots I_{d-2}^2 I_{d-1}^2.
	\end{multline*}
	Let us also record the information of the cycles originated from the top marked boundary vertices. The first $d-2$ cycles end at the last character of the second to last interval words of the bottom word, and the last cycle ends at the first character of the first interval word $I_{d-1}^1$.
	
	We now extend this methodology to the subsequent interval using an analogous sequence of partial weaves. For example, we have partial weaves $I_k^2 I_{d-1}^2 \rarrow I_{d-1}^2I_{k+1}^3$ for $k \in [2, d-2]$ (cf.\ Figures~\ref{fig: T- to T+ boundary case} and~\ref{fig: T- to T+ generic case}). By concatenating these partial weaves for $k = d-2, d-3, \dots, 2$, we get
	\[
	T^-  \rarrow I_{d-1}^1 I_{2}^2 I_{3}^2 \cdots I_{d-2}^2I_{d-1}^2 \rarrow  I_{d-1}^1 I_{d-1}^2 I_3^3\cdots I_{d-2}^3 I_{d-1}^{3}.
	\]
	
	Iteratively applying the partial weave transformations described above, we recursively construct the full Demazure weave through successive concatenation of local braid moves:
	\[
		T^-  \rarrow I_{d-1}^1 I_{2}^2 I_{3}^2 \cdots I_{d-2}^2I_{d-1}^2 \rarrow  I_{d-1}^1 I_{d-1}^2 I_3^3\cdots I_{d-2}^3 I_{d-1}^{3} \rarrow\\
		\cdots \rarrow I_{d-1}^1 I_{d-1}^2 \cdots I_{d-1}^{d-1}.
	\]
	The cycles originated from the top boundary vertices will end at the first character of each interval word for the bottom word. 
	
	Finally we notice that by swapping (non-adjacent) weave lines, we have partial weaves 
	\[I_{d-k}^1I_{d-k+1}^2\cdots I_{d-1}^k \rarrow I_{d-k}^{d-1} I_{d-k-1}^1I_{d-k}^2 \cdots I_{d-1}^{k+1}, \text{ for } 1\le k \le d-1.
	\]
	Concatenate these partial weave for $k = 1, 2,\dots, d-1$, we can transform $I_{d-1}^1 I_{d-1}^2 \cdots I_{d-1}^{d-1}$ into $T^+$ as follows: 
	\begin{multline*}
			T^-  \rarrow  I_{d-1}^1 I_{d-1}^2 \cdots I_{d-1}^{d-1} = I_{d-1}^{d-1} I_{d-2}^1 I_{d-1}^2 I_{d-1}^3 \cdots I_{d-1}^{d-1}\\
			\rarrow  I_{d-1}^{d-1} I_{d-2}^{d-1}I_{d-3}^1I_{d-2}^2I_{d-1}^3 I_{d-1}^4\cdots I_{d-1}^{d-1}
			 \rarrow \cdots
			\rarrow  I_{d-1}^{d-1}I_{d-2}^{d-1}\cdots I_1^{d-1} = T^+.
	\end{multline*}

	We notice that the cycles originated from the top boundary vertices will end at the last character of each interval word for the bottom word $T^+$. This completes the proof of the statement for $T = T^-$.

	Let $T = T_1T_2\cdots T_{d-1}$ be a marked complete nested word.  If $T = T^+$ or $T= T^-$, then we are done. Otherwise let $T_{k-1}, T_{k}$ be two adjacent interval words such that one of them is increasing and the other is decreasing; without loss of generality, we assume that $T_{k-1}$ is increasing and $T_{k}$ is decreasing, i.e., $T_{k-1} = I_i^{j-1}$ and $T_{k} = I_j^i$. Let 
	\[
	T' = T_1T_2\cdots T_{k-2}T'_{k-1}T'_{k}T_{k+1}\cdots T_{d-1}
	\]
	such that $T'_{k-1} = I_j^{i+1}$ and $T'_{k} = I_i^j$. Then we have a Demazure weave 
	\[
	T = T_1T_2 \cdots T_{k-1}T_k \cdots T_{d-1} \rarrow T_1T_2\cdots T_{k-1}'T_{k}' \cdots T_{d-1} = T'
	\]
	with the part $T_{k-1}T_k \rarrow T'_{k-1}T'_k$ shown in Figure~\ref{fig: braid X-patch-bw with cycles}. Now notice that $T'$ is complete nested: $T'_{k} = I_i^j$ and $T'_{k-1}$ is an interval word for $[i+1, j]$. Also notice that the cycles originated from the top boundary vertices will end at the marked characters in $T'$.
	\begin{figure}
		\centering
		{\includegraphics[trim= 11cm 10cm 5cm 8cm, clip = true, scale = 0.4]{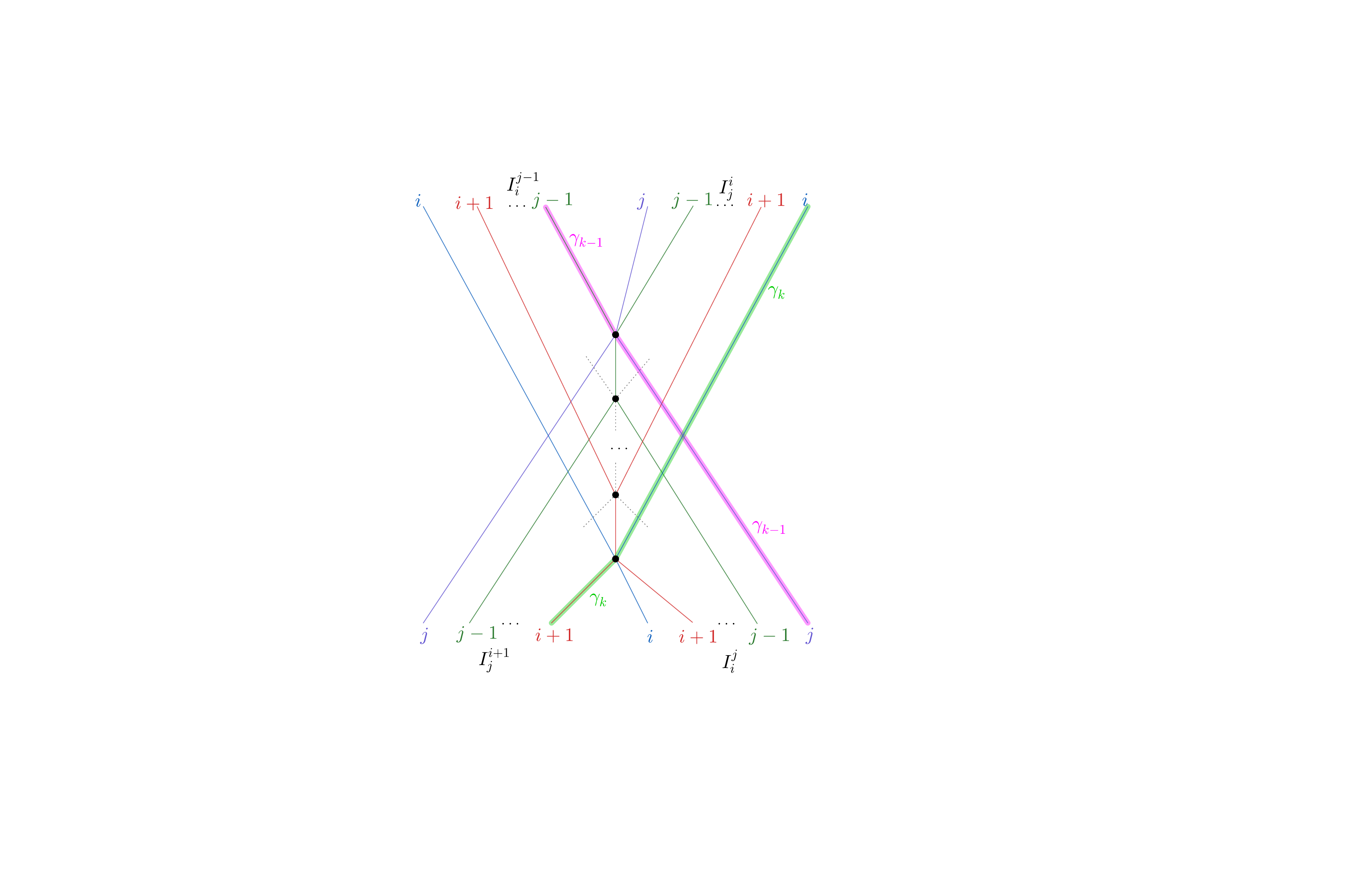}}
		\caption{A (local) weave $\ww: I_i^{j-1}I_j^i \rarrow I_j^{i+1}I_i^j$. }
		\label{fig: braid X-patch-bw with cycles}
	\end{figure}
	\noindent Repeat this process (and concatenate the Demazure weaves we obtained in the process) until no such pair exists. The (final) Demazure has bottom word $\tilde{T} = \tilde{T}_1\tilde{T}_2\cdots \tilde{T}_{d-1}$ such that $\tilde{T}_i$ is decreasing for $i\le k$ and increasing for $i>k$, for some $k\in [1, d-2]$, i.e., the (final) Demazure weave is of the form
	\[
	T \rarrow   I_{d-k}^{d-k} I_{d-k+1}^{d-k}\cdots I_{d-1}^{d-k} I_{d-k-1}^{d-1}I_{d-k-2}^{d-1} \cdots I_{1}^{d-1} = \tilde{T}.
	\]
	Notice that $I_{d-k}^{d-k} I_{d-k+1}^{d-k}\cdots I_{d-1}^{d-k}$ is a $T^{-}$ in $k$ characters $(d-k, d-k+1, \dots, d-1)$, so we can make a Demazure weave from it to the $T^+ (= I_{d-1}^{d-1}I_{d-2}^{d-1}\cdots I_{d-k}^{d-1})$ in the same characters by braid moves. Combining with the rest of the interval words, we get a Demazure weave
	\begin{multline*}
	T \rarrow I_{d-k}^{d-k} I_{d-k+1}^{d-k}\cdots I_{d-1}^{d-k} I_{d-k-1}^{d-1}I_{d-k-2}^{d-1} \cdots I_{1}^{d-1} \\
	\rarrow I_{d-1}^{d-1}I_{d-2}^{d-1}\cdots I_{d-k}^{d-1} I_{d-k-1}^{d-1}I_{d-k-2}^{d-1} \cdots I_{1}^{d-1} = T^+.
	\end{multline*}
	Lastly we simply notice that the cycles will end exactly at the last character of each interval word for $T^+$. 
\end{proof}

\begin{defn}\label{defn: left/right/double push by flags}
	Let $\mcg, \mch$ be decorated flags. For $0\le i \le d-1$, define $\lpush{\mcg}{i}{\mch}$ 
to be the decorated flag 
\[
\lpush{\mcg}{i}{\mch} = \big(\mcg^1, \cdots, \mcg^i, \mcg^i \wed \mch^1, \cdots, \mcg^i \wed \mch^{d-1-i}\big).
\]
By convention we have $\lpush{\mcg}{0}{\mch} = \mch$ and $\lpush{\mcg}{d-1}{\mch} = \mcg$. 
We can  interpret $\mcg^i\wed \mch$ as follows: first we shift $\mch$ to the right by $i$ steps, then we fill in the first $i$ spots with the first $i$ components of $\mcg$, and finally wedge the rest with $\mcg^i$. 

Similarly, we define $\rpush{\mch}{\mcg}{j}$ to be the decorated flag
\[
\rpush{\mch}{\mcg}{j} = \big(\mcg^{j}\wed\mch^{d-j+1} , \cdots, \mcg^{j}\wed\mch^{d-1} , \mcg^{j}, \cdots, \mcg^{d-1}\big)
\]
when $1\le j \le d$. By convention we have $\rpush{\mch}{\mcg}{d} = \mch$ and $\rpush{\mch}{\mcg}{1} = \mcg$. 

Also we define $\lrpush{\mcg}{i}{\mch}{j}$ (with $i<j$) to be the decorated flag 
\begin{align*}
\lrpush{\mcg}{i}{\mch}{j}=&\big(\mcg^1, \cdots, \mcg^i, \mcg^{j}\wed\mcg^i \wed \mch^{d-j+1}, \cdots, \mcg^{j}\wed\mcg^i \wed \mch^{d-i-1}, \mcg^{j}, \cdots, \mcg^{d-1}\big)\\
=&\big(\mcg^1, \cdots, \mcg^i, \mcg^{i}\wed\mcg^j \wed \mch^{d-j+1}, \cdots, \mcg^{i}\wed\mcg^j \wed \mch^{d-i-1}, \mcg^{j}, \cdots, \mcg^{d-1}\big).
\end{align*}
Here we used the fact that $\mcg^{j}\wed\big(\mcg^i \wed \mch^{k}\big) = \mcg^i \wed \big(\mcg^{j} \wed \mch^{k}\big)$ by Proposition~\ref{prop: wedge of any two within a flag is zero}. By convention we have $\lrpush{\mcg}{0}{\mch}{j} = \rpush{\mch}{\mcg}{j}$ and $\lrpush{\mcg}{i}{\mch}{d} = \lpush{\mcg}{i}{\mch}$. 

\begin{remk}
	Notice that in general, $\lrpush{\mcg}{i}{\mch}{j}$ is NOT equal to $\lpush{\mcg}{i}{\big(\rpush{\mch}{\mcg}{j}\big)}$, nor to $\rpush{\big(\lpush{\mcg}{i}{\mch}\big)}{\mcg}{j}$.
\end{remk}

\end{defn}

Recall from Lemma~\ref{lemma: decoration for interval word}, a decoration for an interval word is uniquely determined by the first and last decorated flags. And  recall from Definition~\ref{defn: decoration of marked word} that a normalized decoration for a marked word is a decoration such that the non-marked crossing values are equal to $1$.

\begin{lemma}\label{lemma: decoration for complete nested word is unique}
	Let $\mcg, \mch$ be two decorated flags and $T = T_1T_2\cdots T_{d-1}$ be a marked complete nested word. Then the word $T$ has a unique normalized decoration $\dec_{T}(\mcg, \mch)$ from $\mcg$ to $\mch$. This decoration is determined as follows. 
	Let $i, j \in [1, d-1]$ such that $T_k$ is an interval word for $[i, j]$; and let $\mck_{k+1} = \lrpush{\mcg}{i-1}{\mch}{j+1}$, for $k\in [1, d-1]$. 
	Then $\dec_T(\mcg, \mch) = (\mck_1, \mck_2, \cdots, \mck_d)$, denoted by
	\[
	\mcg = \mck_1 \rel{T_1} \mck_2 \rel{T_2} \cdots \rel{T_{k-1}} \mck_{k} \rel{T_k} \mck_{k+1} \rel{T_{k+1}} \cdots \rel{T_{d-1}} \mck_{d} = \mch.
	\]
\end{lemma}

\begin{proof}
	The proof is by induction on $k$. For $k = d-1$, we have $\mck_d = \mch$; recall from the definition of a complete nested word (cf.\ Definition~\ref{defn: nested words and complete nested words}), $T_{d-1}$ is an interval word for $[1, d-1]$, and we have $\lrpush{\mcg}{0}{\mch}{d} = \mch = \mck_d$. 
	
	Now assume that $T_k$ is an interval word for $[i,j]$ and $\mck_{k+1} = \lrpush{\mcg}{i-1}{\mch}{j+1}$. Without loss of generality, we assume that $T_k = I_i^j$. Then $T_{k-1}$ is an interval word for $[i+1, j]$,  let us show that $\mck_{k} = \lrpush{\mcg}{i}{\mch}{j+1}$. 
	
	We have $\mck_k \rel{T_k} \mck_{k+1}$. Recall that $T_k = I_i^j$ is marked only for the last character, namely $j$, so for any $r\in [i, j-1]$, the crossing value at $r$ is equal to $1$, i.e., 
	\[
	\frac{\qwed{\mck_k^r}{\mck_{k+1}^r}{\mck_{k+1}^{r-1}}}{\mck_k^{r+1}} = 1, \quad r\in [i, j-1];
	\]
	equivalently
	\[
	\mck_k^{r+1} = \qwed{\mck_k^r}{\mck_{k+1}^r}{\mck_{k+1}^{r-1}}, \quad r\in [i,j-1].
	\]
	Notice that for $t \le i$ or $t\ge j+1$, we have $\mck^t_k = \mck^t_{k-1}= \dots = \mck_1^t= \mcg^t$. Now let us prove $\mck^t_{k} = \mcg^i \mixwed\mcg^{j+1} \mixwed\mch^{d+t-i-j-1}$ for $t \in [i+1, j]$ by an induction on $t$. When $t = i+1$, notice that $\mck_{k+1}^{i-1} = \mcg^{i-1}$ and $\mck_{k+1}^{i} = \mcg^{i-1}\mixwed \mcg^{j+1}\mixwed\mch^{d-j}$, then we have
	\begin{align*}
		\mck_k^{i+1} &= \qwed{\mck_k^i}{\mck_{k+1}^i}{\mck_{k+1}^{i-1}} \\
		&=\qwed{\mcg^i}{\big(\mcg^{i-1}\mixwed \mcg^{j+1}\mixwed\mch^{d-j}\big)}{\mcg^{i-1}}\\
		&= \mcg^{i}\mixwed \mcg^{j+1}\mixwed\mch^{d-j},
	\end{align*}
	where the last equality follows from Lemma~\ref{lemma: cancellation law for quotient wedge}. Now assume that $\mck^t_{k} = \mcg^i \mixwed\mcg^{j+1} \mixwed\mch^{d+t-i-j-1}$. Notice that $\mck_{k+1}^{t-1} = \mcg^{i-1}\mixwed \mcg^{j+1}\mixwed\mch^{d+t-i-j-1} $ and $\mck_{k+1}^t = \mcg^{i-1}\mixwed \mcg^{j+1}\mixwed\mch^{d+t-i-j}$, then
	\begin{align*}
	\mck_k^{t+1} &= \qwed{\mck_k^t}{\mck_{k+1}^t}{\mck_{k+1}^{t-1}} \\
	&=\qwed{\big(\mcg^i \mixwed\mcg^{j+1} \mixwed\mch^{d+t-i-j-1}\big)}{\big(\mcg^{i-1}\mixwed \mcg^{j+1}\mixwed\mch^{d+t-i-j}\big)}{\mcg^{i-1}\mixwed \mcg^{j+1}\mixwed\mch^{d+t-i-j-1}}\\
	&= \mcg^i \mixwed\mcg^{j+1} \mixwed\mch^{d+t-i-j}
	\end{align*}
	where the last equality follows from Lemma~\ref{lemma: cancellation law for quotient wedge, version2}. This completes the proof that $\mck_{k} = \lrpush{\mcg}{i}{\mch}{j+1}$.
\end{proof}

\begin{lemma}\label{lemma: changing from T to T^+ does not affact the seed}
	Let $\mcg, \mch$ be two decorated flags and $T = T_1T_2\cdots T_{d-1}$ be a marked complete nested word. Let $\ww: T \rarrow T^+ := I_{d-1}^{d-1}I_{d-2}^{d-1}\cdots I_1^{d-1}$ be a decorated Demazure weave with the top word $T$ decorated by $\dec_T(\mcg, \mch)$. Then the cycles of $\ww$ can be described as follows. For $k\in [1, d-1]$, let $T_k$ be the $k$-th interval word for $T$. 
	\begin{itemize}[wide, labelwidth=!, labelindent=0pt]
		\item If $T_k = I_i^j$, then the cycle $\gamma_k$ start at the last character of $T_k$, of color $j$, and end at the last character of $I_i^{d-1}$, of color $d-1$; furthermore,  $\Delta_{\gamma_k} = \mcg^i \wedge \mch^{d-i}$.
		\item If $T_k = I_j^i$, then the cycle $\gamma_k$ start at the last character of $T_k$, of color  $i$, and end at the last character of $I_j^{d-1}$, of color $d-1$; furthermore,  $\Delta_{\gamma_k} = \mcg^j \wedge \mch^{d-j}$. 
	\end{itemize}
	 Consequently, the decoration for the bottom word $T^+$ is the normalized decoration $\dec_{T^+}(\mcg, \mch)$.
\end{lemma}
\begin{proof}
	By Lemma~\ref{lemma: mutation equivalent weave same bottom flags} and Theorem~\ref{thm: mutation equivalence demazure weaves yield mutation equivalent seeds}, we only need to prove the lemma for a specific Demazure weave.  The statement involving where cycles end follows from Lemma~\ref{lemma: complete nested words are all related by braid moves} and its proof (where a specific Demazure weave is described). The frozen variables associated with the cycles can be calculated using the decoration for the top word. We will show the case when $T_k = I_i^j$; the other case follows in a similar way.  Let $\dec_T(\mcg, \mch) = (\mck_1, \mck_2, \cdots, \mck_d)$. Then
	\begin{align*}
	\Delta_{\gamma_k} &= \frac{\qwed{\mck_k^j}{\mck_{k+1}^j}{\mck_{k+1}^{j-1}}}{\mck_{k}^{j+1}}=  \frac{\qwed{\big(\mcg^i\mixwed \mcg^{j+1}\mixwed \mch^{d-i-1}\big)}{\big(\mcg^{i-1}\mixwed \mcg^{j+1}\mixwed \mch^{d-i}\big)}{\mcg^{i-1}\mixwed\mcg^{j+1}\mixwed \mch^{d-i-1}}}{\mcg^{j+1}}\\
	&= \frac{{\mcg^{i}\mixwed \mcg^{j+1}\mixwed \mch^{d-i}}}{\mcg^{j+1}}= \frac{{\mcg^{j+1}\mixwed \mcg^{i}\mixwed \mch^{d-i}}}{\mcg^{j+1}} = \mcg^i \wedge \mch^{d-i}.
	\end{align*}
	Here the third equality follows from Lemma~\ref{lemma: cancellation law for quotient wedge, version2} and the fourth equality follows from Proposition~\ref{prop: wedge of any two within a flag is zero}. 
	
	Finally, the statement that the decoration for the bottom word is normalized follows from the description for the cycles. A bottom edge has a cycle passing through it if and only if the edge corresponds to the last character of an interval word for $T^+$; hence other than the last character of an interval word for $T^+$, the crossing values are equal to $1$, i.e., the decoration is normalized. 
\end{proof}

\begin{lemma} \label{lemma: w0 rho to w0 equals rho w0 w0}
	Let $\mcg, \mch, \mck$ be decorated flags such that $\mch = \lpush{\mch}{1}{\mck}$. Recall that $\rho = 12\cdots\overline{d-1}$. We have the following statements. 
	\begin{enumerate}[wide, labelwidth=!, labelindent=0pt]
		\item Let $\ww: T^+ \rho \rarrow T^+$ be a decorated Demazure weave with the top word $T^+\rho$ decorated from $\mcg$ to $\mck$ such that the decoration for $T^+$ is $\dec_{T^+}(\mcg, \mch)$. Then the decoration for the bottom word is the normalized decoration $\dec_{T}(\mcg, \mck)$. Moreover, the seed $\seed(\ww)$ does not depend on the choice of the Demazure weave $\ww$.
		\item Let $\ww': \rho \rev{{T^{+}}} \rarrow \rev{T^{+}}$ be a decorated Demazure weave with the top word $\rho \rev{T^{+}}$ decorated from $\mch$ to $\mcg$ such that the decoration for $\rev{T^+}$ is $\rev{\dec_T}(\mck, \mcg)$. Then the decoration for the bottom word is the reversed normalized decoration $\rev{\dec_T}(\mch, \mcg)$.Moreover, the seed $\seed(\ww')$ does not depend on the choice of the Demazure weave $\ww'$. 
		\item We have $\seed(\ww) = \seed(\ww')$. 
	\end{enumerate}
	See Figure~\ref{fig: A 4-weave w0 rho to w0} for an example when $d = 5$. Similar results hold for $\rev{\rho} = d-1\cdots 2 \bar{1}$.	
	\begin{figure}
		\centering
		\includegraphics[trim= 7.3cm 7.4cm 13.5cm 5.8cm, clip = true, scale = 0.34]{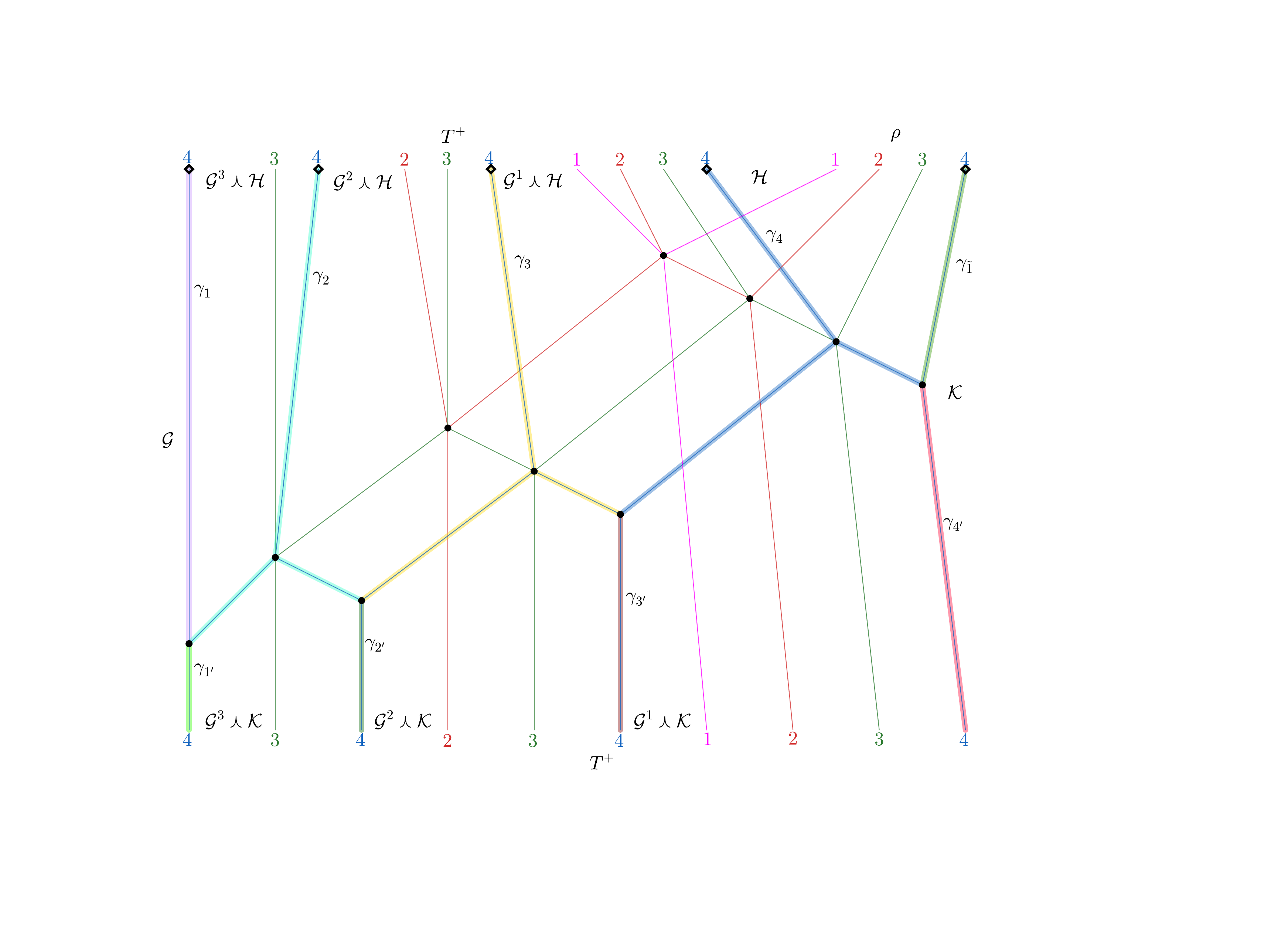}
		\caption{A decorate Demazure 4-weave $\ww: T^+\rho \rarrow T^+$. It can be interpreted as a decorated Demazure weave in two different ways, one is viewing $T^+\rho$ as the top and $T^+$ as the bottom; the other is viewing $\rho\protect\overleftarrow{T^+}$as the top and $\protect\rev{T^+}$ as the bottom. They correspond to the weave $\ww$ and $\ww'$, respectively, constructed in the proof of Lemma~\ref{lemma: w0 rho to w0 equals rho w0 w0}. The seeds associated with these two Demazure weaves coincide.}
		\label{fig: A 4-weave w0 rho to w0}
	\end{figure}
\end{lemma}

\begin{proof}
	First notice that the condition $\mch = \lpush{\mch}{1}{\mck}$ is necessary so that $\mch \rel{\rho} \mck$ is a normalized decoration. 
	
	We will show that both $\seed(\ww)$ and $\seed(\ww')$  contain only frozen variables, so different choices of Demazure weaves will result in the same seed (instead of mutation equivalent seeds). We also notice that different choices of Demazure will not change the decoration for the bottom word (cf.\ Lemma~\ref{lemma: mutation equivalent weave same bottom flags}), hence we only need to prove these results for specific Demazure weaves $\ww$ and $\ww'$. 
	
	Consider the partial weaves $I_{k}^{d-1}I_{k}^{d-1} \rarrow I_{k+1}^{d-1} I_{k}^{d-1}$ in Figure~\ref{fig: H-patch-bb-cycles}, for $1\le k \le d-1$. Let $\ww$ be the Demazure weave obtained by recursively concatenating these partial weaves for $k = 1, 2, \cdots, d-1$, i.e., 
	\begin{multline*}
		\ww: T^+ \rho = I_{d-1}^{d-1}I_{d-2}^{d-1}\cdots I_{2}^{d-1}(I_{1}^{d-1}I_{1}^{d-1}) \rarrow  I_{d-1}^{d-1}I_{d-2}^{d-1}\cdots I_3^{d-1}(I_{2}^{d-1}I_{2}^{d-1})I_{1}^{d-1}\\
		 \rarrow  I_{d-1}^{d-1}I_{d-2}^{d-1}\cdots( I_{3}^{d-1} I_{3}^{d-1})I_{2}^{d-1}I_{1}^{d-1} \rarrow \cdots \\
		  \rarrow (I_{d-1}^{d-1} I_{d-1}^{d-1})I_{d-2}^{d-1}\cdots I_{2}^{d-1}I_{1}^{d-1} 
		 \rarrow I_{d-1}^{d-1}I_{d-2}^{d-1}\cdots I_{2}^{d-1}I_{1}^{d-1} = T^+.
	\end{multline*}
\begin{figure}
	\centering
	\includegraphics[trim= 0cm 10cm 30cm 11cm, clip = true, scale = 0.45]{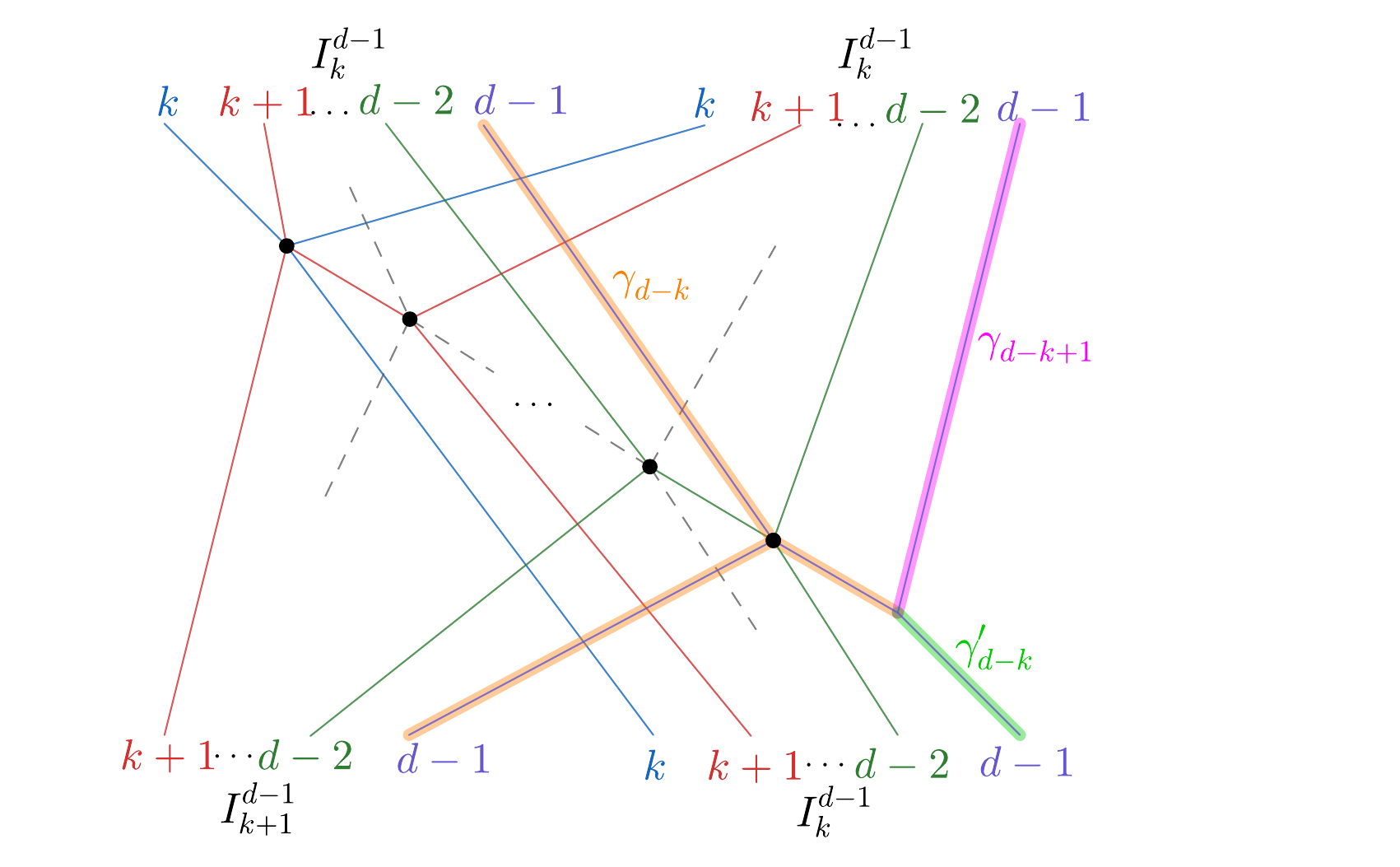}
	\caption{A partial weave $I_k^{d-1}I_k^{d-1} \rarrow I_{k+1}^{d-1} I_{k}^{d-1}$ with cycles indicated, for $k \in [1, d-1]$.}
	\label{fig: H-patch-bb-cycles}
\end{figure}

	Then by keeping track of the cycles from these partial weaves, we conclude that a bottom edge has a cycle passing through it if and only if the edge corresponds to the last character of an interval word for $T^+$. Consequently, other than the last character of an interval word for $T^+$, the crossing values are equal to $1$, i.e., the decoration for the bottom word is normalized.

	Moreover,  the seed $\seed(\ww) = (Q, \zz)$ can be described as follows: the quiver $Q$ is given by Figure~ \ref{fig: quiver associated with w0 rho to w0}; the frozen variable associated with vertex $\gamma_{d-k}$ (resp., $\gamma'_{d-k}$) is $\Delta_{\gamma_{d-k}} = \mcg^{k}\wedge \mch^{d-k}$ (resp., $\Delta_{\gamma'_{d-k} }= \mcg^{k}\wedge \mck^{d-k}$), $1\le k \le d-1$; the frozen variable for $\tilde{\gamma_1}$ is $\Delta_{\tilde{\gamma_{1}}} = \mch^1\wedge \mck^{d-1}$; these frozen variables are calculated by the same method as in the proof of Lemma~\ref{lemma: changing from T to T^+ does not affact the seed}, using the top and bottom decoration of $\ww$. 
	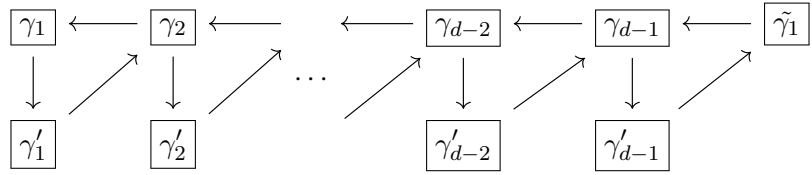
\begin{figure}
		\centering
		\begin{tikzcd}
			\boxed{\gamma_1} \arrow[d] & \boxed{\gamma_2} \arrow[l] \arrow[d] & \phantom{\boxed{k}} \arrow[l] \arrow[d, phantom,  "\cdots"] & \boxed{\gamma_{d-2}} \arrow[l]\arrow[d] & \boxed{\gamma_{d-1}} \arrow[l]\arrow[d] & \boxed{\tilde{\gamma_{1}}} \arrow[l]\\
			\boxed{\gamma_1'} \arrow[ur] & \boxed{\gamma_2'} \arrow[ur] & \phantom{\boxed{k'}} \arrow[ur] & \boxed{\gamma_{d-2}'} \arrow[ur] &\boxed{\gamma_{d-1}'} \arrow[ru] &
		\end{tikzcd}
	\caption{Quiver $Q$ associated with the weave $\ww: T^+ \rho \rarrow T^+$.}
	\label{fig: quiver associated with w0 rho to w0}
	\end{figure}
	Similarly, we have partial weaves $I_{k}^{d-1}I_{d-1}^{k} \rarrow I_{d-1}^{k} I_{k+1}^{d-1}$, for $1\le k \le d-1$. Let $\ww'$ be the Demazure weave obtained by recursively concatenating these partial weaves for $k = 1, 2, \cdots, d-1$. We end up getting the same weave (as a graph), the same decorations and the same set of cycles and intersection pairings as $\ww$. So we conclude that $\seed(\ww) = \seed(\ww')$. An example is given in Figure~\ref{fig: A 4-weave w0 rho to w0} with $d = 5$. 
\end{proof}

The following lemma is needed for Lemma~\ref{lemma: the cluster algebra does not depend on the choice of the cut}, and will be proved in Section~\ref{sec: initial weave}.

\begin{lemma}\label{lemma: bottom cycle description of ww(p, q) and bottom decoration is normalized}
	Let $(p, q)$ be a valid cut of $\sigma$ (cf.\ Definition~\ref{defn: valid cut in general}). Let $\ww(p, q): \beta(p, q) \rarrow T^+$ be a Demazure weave as in Definition~\ref{defn: cutting along the disk to get two weaves}. Then a bottom edge has a cycle passing through it if and only if the edge corresponds to the last character of an interval word for $T^+$. These frozen cycles correspond to the frozen variables 
	\[
	\zz_0 = \{\mcf_p^1\wedge \mcf_q^{d-1}, \mcf_p^2\wedge \mcf_q^{d-2}, \dots, \mcf_p^{d-1}\wedge \mcf_q^1\}
	\]
	in $\seed(\ww(p, q))$. Consequently, the decoration for the bottom word is the normalized decoration $\dec_{T^+}(\mcf_p, \mcf_q)$.
\end{lemma}

\begin{lemma}\label{lemma: the cluster algebra does not depend on the choice of the cut}
	Assume that both $(p, q)$ and $(p, q+1)$ are valid cuts of $\sigma$. Then $\seed(p, q)$ and $\seed(p, q+1)$ are mutation equivalent. Therefore $\mca(p, q) = \mca(p, q+1)$.
\end{lemma}

\begin{proof}
	Without loss of generality, we assume that $\sigma(q) = 1$. Let $\mcg = \mcf_p, \mch = \mcf_q, \mck = \mcf_{q+1}$. Then we have $\mch = \lpush{\mch}{1}{\mck}$, cf.\ Proposition~\ref{prop: recursive relation between flags associated with a signature}. Let $\ww(p, q, q+1)$ (resp., $\ww(q, q+1, p+n)$ be a decorated Demazure $\ww: T^+\rho \rarrow T^+$ (resp., $\ww': \rho \rev{T^+} \rarrow \rev{T^+}$) as in Lemma~\ref{lemma: w0 rho to w0 equals rho w0 w0}. 
	
	We claim that $\seed(\ww(p, q+1))$ is mutation equivalent to the amalgamation of $\seed(\ww(p, q))$ and $\seed(p, q, q+1)$ along the frozen variables 
	\[
	\zz_0 = \{\mcf_p^1\wedge \mcf_q^{d-1}, \mcf_p^2\wedge \mcf_q^{d-2}, \dots, \mcf_p^{d-1}\wedge \mcf_q^1\}, 
	\]
	i.e., we have 
\begin{equation}\label{equation: (p, q+1) = (p,q) + (p,q, q+1)}
	\seed(\ww(p, q+1)) \sim \glueseeds{\seed(\ww(p, q))}{\seed(p, q, q+1)}{\zz_0}.
\end{equation}
	Here $\seed(\ww(p, q))$ (resp. $\seed(\ww(p, q+1))$) is the seed associated with the Demazure weave $\ww(p,q): \beta(p, q) \rarrow T^+$ (resp. $\ww(p, q+1) : \beta(p, q+1) \rarrow T^+$) as in Definition~\ref{defn: cutting along the disk to get two weaves}; and $\seed(p, q, q+1)$ is the seed associated with $\ww(p, q, q+1)$. 
	
	\begin{proof}[Proof of the claim]
		Consider the decorated Demazure weave $\ww: \beta(p, q+1) \rarrow T^+$ constructed as follows. First notice that $\beta(p, q+1) = \beta(p, q) \rho$, we make a Demazure weave $\ww^*: \beta(p, q) \rho \rarrow T^+ \rho$, which is $\ww(p, q)$ accompanying with $d-1$ vertical strands representing $\rho$, see the top portion of Figure~\ref{fig: glue m(p, q) with m(p, q, q+1)}. Then we concatenate the bottom of the weave $\ww^*$ with the top of $\ww(p, q, q+1): T^+ \rho \rarrow T^+$. We have a valid concatenation: by Lemma~\ref{lemma: bottom cycle description of ww(p, q) and bottom decoration is normalized}, the decoration for the bottom word of the weave $\ww^*$ matches with the decoration for the top word of the weave $\ww(p, q, q+1)$; furthermore, the cycles ending at the bottom boundary of the weave $\ww^*$ concatenate exactly at the marked boundary vertices for the weave $\ww(p, q, q+1)$.
	\begin{figure}
	\centering
	\includegraphics[trim= 0.5cm 3.2cm 7cm 4.2cm, clip = true, scale = 0.33]{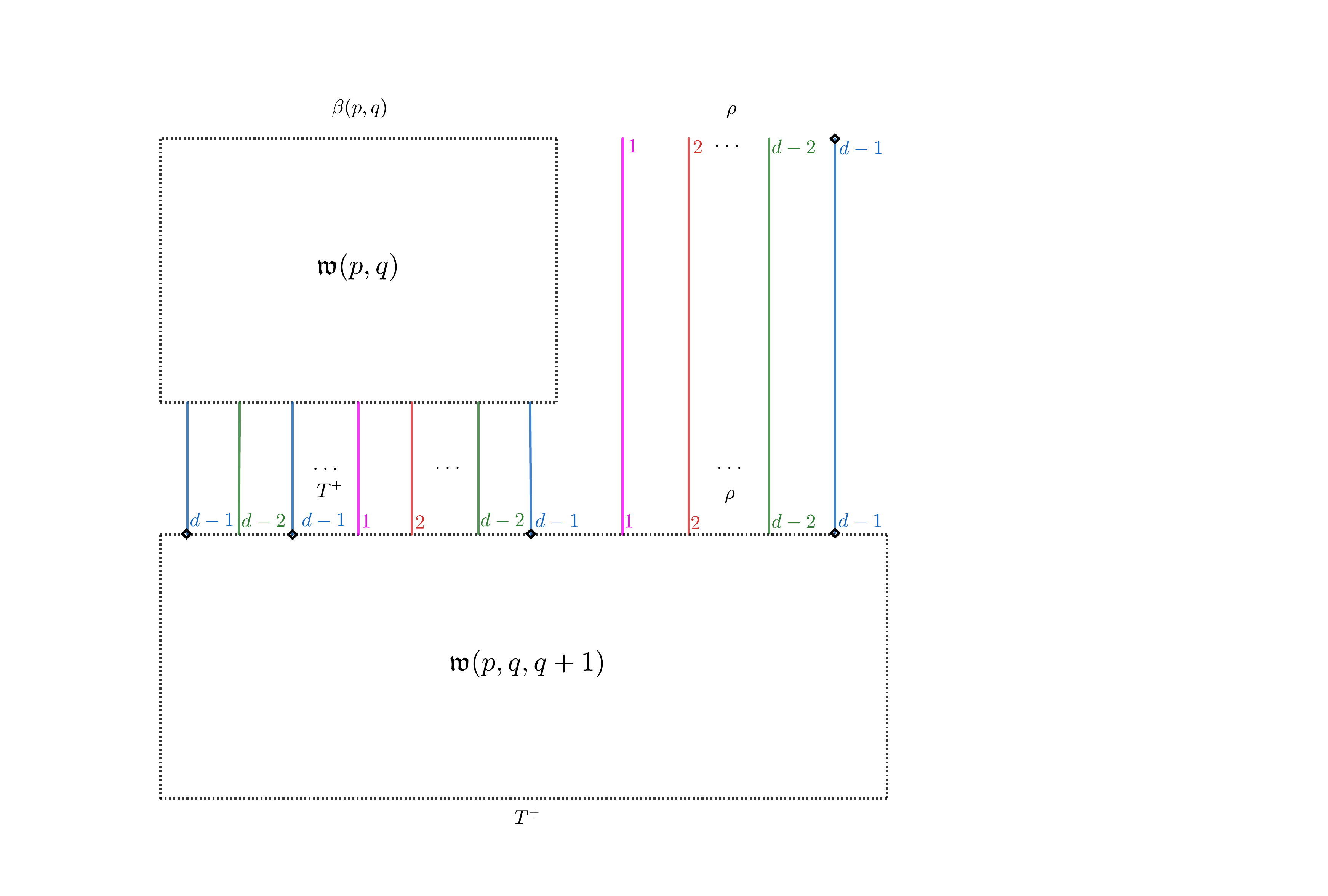}
	\caption{Concatenating $\ww(p, q)$ with $\ww(p, q, q+1)$.}
	\label{fig: glue m(p, q) with m(p, q, q+1)} 
\end{figure}
		
		Now notice that by construction, the seed associated with $\ww$ is the amalgamation of $\seed(\ww(p, q))$ and $\seed(p, q, q+1)$ along the frozen variables $
		\zz_0$:
		\[
		\seed(\ww) = \glueseeds{\seed(\ww(p, q))}{\seed(p, q, q+1)}{\zz_0}.
		\](the amalgamation conditions are satisfied by Lemma~\ref{lemma: amalgamation condition is satisfied}).
		Finally we notice that $\ww$ is a Demazure weave $\beta(p, q+1) \rarrow T^+$, we have that $\seed(\ww)\sim \seed(\ww(p, q+1))$ by Theorem~\ref{thm: mutation equivalence demazure weaves yield mutation equivalent seeds}. Hence we have 
		\[
		\seed(\ww(p, q+1)) \sim \seed(\ww) = \glueseeds{\seed(\ww(p, q))}{\seed(p, q, q+1)}{\zz_0}.
		\]
		This completes the proof of the claim. \qedhere 
\end{proof}

	Similarly we can show that 
\begin{equation}\label{equation: (q, p+n) = (q, q+1, p+n) + (q+1, p+n)}
	\seed(\ww(q, p+n)) \sim \glueseeds{\seed(\ww(q+1, p+n))}{\seed(q, q+1, p+n)}{\zz'_0},
\end{equation}
	where
	\[
	\zz'_0 = \{\mcf_p^1\wedge \mcf_{q+1}^{d-1}, \mcf_p^2\wedge \mcf_{q+1}^{d-2}, \dots, \mcf_p^{d-1}\wedge \mcf_{q+1}^1\}.
	\]

	Now by Definition~\ref{defn: cluster algebra cutting at p, q}, $\seed(p, q)$ is the seed obtained by amalgamating $\seed(\ww(p, q))$ and $\seed(\ww(q, p+n))$ along the frozen variables $\zz_0$, i.e., 
	\[
	\seed(p, q) = \glueseeds{\seed(\ww(p, q))}{\seed(\ww(q, p+n))}{\zz_0}.
	\]
	Then by Remark~\ref{remk: properties of amalgamation}, we have
	\begin{align*}
		\seed(p, q) =& \glueseeds{\seed(\ww(p, q))}{\seed(\ww(q, p+n))}{\zz_0}\\
		{\sim} &\glueseeds{\seed(\ww(p, q))}{\big(\glueseeds{\seed(\ww(q+1, p+n))}{\seed(q, q+1, p+n)}{\zz'_0}\big)}{\zz_0} \text{ by  } (\ref{equation: (q, p+n) = (q, q+1, p+n) + (q+1, p+n)})\\
		 =&\glueseeds{\big(\glueseeds{\seed(\ww(p, q))}{\seed(q, q+1, p+n)}{\zz_0}\big)}{\seed(\ww(q+1, p+n))}{\zz'_0}\\
		 {=}&\glueseeds{\big(\glueseeds{\seed(\ww(p, q))}{\seed(p, q, q+1)}{\zz_0}\big)}{\seed(\ww(q+1, p+n))}{\zz'_0}\text{ by Lemma } \ref{lemma: w0 rho to w0 equals rho w0 w0}\\
		{\sim}&\glueseeds{\seed(\ww(p, q+1))}{\seed(\ww(q+1, p+n))}{\zz'_0} \text{ by } (\ref{equation: (p, q+1) = (p,q) + (p,q, q+1)})\\
	= 	&\seed(p, q+1). \qedhere
	\end{align*}
\end{proof}
	\begin{figure}
		\centering
		\includegraphics[trim= 5cm 6cm 7cm 3.5cm, clip = true, scale = 0.23]{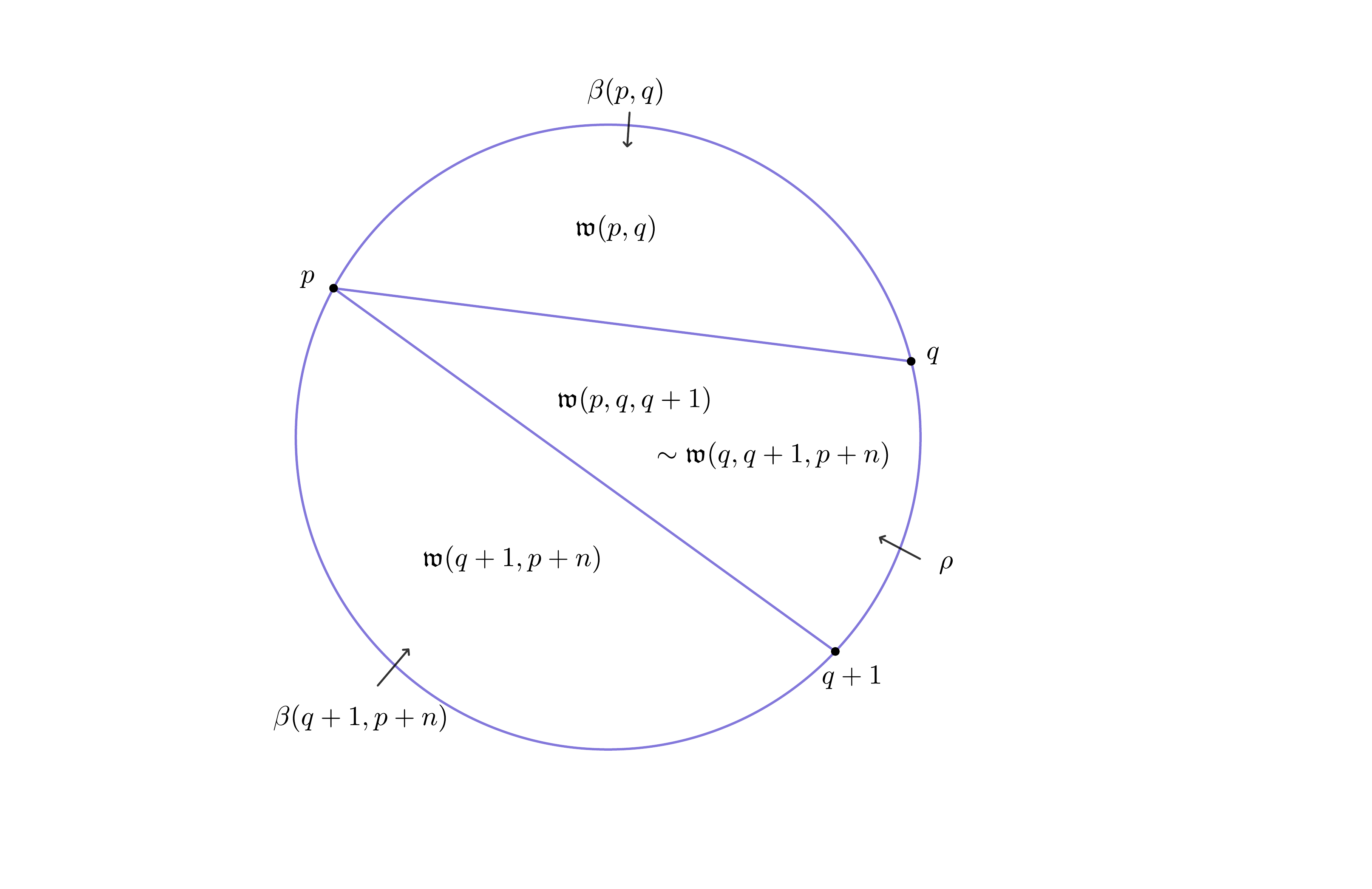}
		\caption{Amalgamating three decorated Demazure weaves in two different ways. They result in the same seeds.}
		\label{fig: Amalgamating three decorated Demazure weaves in two different ways, result in the same seeds}
	\end{figure}

Now Proposition~\ref{prop: cluster algebra does not depend on the cut} follows by repeatedly applying Lemma~\ref{lemma: the cluster algebra does not depend on the choice of the cut}.

\begin{proof}[Proof of Proposition~\ref{prop: cluster algebra does not depend on the cut}]
	We first show that 
	\[
	\mca(p, q) = \mca(p+1, q+1),
	\]
	where $1\le p \le q \le n$ and $||p-q||\ge d+1$. Since $n\ge 2(d+1) + 1$, we may assume that $\|q+1 - p\| \ge d+1$ (or else $\|p+1 - q\| \ge d+1$), then $\mca(p, q) = \mca(p, q+1)$ by Lemma~\ref{lemma: the cluster algebra does not depend on the choice of the cut}. Similarly, as $\|p+1 - (q+1)\| = \|p - q\| \ge d+1$,  we get $\mca(p, q+1) = \mca(p+1, q+1)$ by Lemma~\ref{lemma: the cluster algebra does not depend on the choice of the cut} again. Hence $\mca(p, q) = \mca(p+1, q+1)$. 
	
	We then show that 
	\[
	\mca(p, q) = \mca(p, q'),
	\]
	where $1\le p \le q \le q' \le n$, $\|p-q\| \ge d+1$ and $\|p - q'\| \ge d+1$. Notice that for any $r \in [q, q']$, we have $\|p-r\| \ge d+1$, hence by Lemma~\ref{lemma: the cluster algebra does not depend on the choice of the cut}, we have $\mca(p, q) = \mca(p, q+1) = \cdots = \mca(p, q')$. 
	
	Finally, notice that $\mca(p, q) = \mca(p+1, q+1)$ and $\mca(p, q) = \mca(p, q')$ implies that $\mca(p, q)$ does not depend on the choice of the cut $(p, q)$. 
\end{proof}

\subsection{An initial weave for $\mca_\sigma$}\label{sec: initial weave}
In this section, we will construct and study an \emph{initial weave} for the cluster algebra ~$\mca_\sigma$. 
We begin by defining local components of weaves called \emph{patches}, which will serve as modular units for larger constructions.

\begin{defn}\label{defn: patches}
	Let $1\le i \le j \le d-1$. A \emph{patch} is a particular kind of a partial weave of the form $\cdots TT' \cdots \rarrow \cdots \tilde{T}T \cdots$ where $T$ and $T'$ are two interval words for the same interval. Specifically:
	\begin{itemize}[wide, labelwidth=!, labelindent=0pt]
		\item An \emph{$X$-patch} is a  partial weave, either of the form $\cdots I_i^jI_j^i \cdots  \rarrow \cdots I_{j}^{i+1} I_{i}^j \cdots$ or of the form $\cdots I_j^i I_i^j \cdots \rarrow \cdots I_{i}^{j-1} I_{j}^i\cdots$, as shown in Figure~ \ref{fig: X patch small}.	
		\item An \emph{$H$-patch} is a  partial weave, either of the form $\cdots I_i^jI_i^j \cdots  \rarrow \cdots I_{i+1}^j I_{i}^j \cdots$ or of the form $\cdots I_j^i I_j^i \cdots \rarrow \cdots I_{j-1}^i I_{j}^i\cdots$, as shown in Figure~\ref{fig: H patch and Y patch small} (left) and Figure~\ref{fig: H patch and Y patch small dual} (left).
		\item A \emph{$Y$-patch} is a partial weave, either of the form $\cdots I_i^jI_j^i \cdots  \rarrow \cdots I_{i}^{j-1} I_{j}^i \cdots$ or of the form $\cdots I_j^i I_i^j \cdots \rarrow \cdots I_{j}^{i-1} I_{i}^j\cdots$, as shown in Figure~\ref{fig: H patch and Y patch small} (right) and Figure~\ref{fig: H patch and Y patch small dual} (right).
	\end{itemize}
	We say that a patch is of type $X$ (resp., $H$, $Y$) if it is an $X$-patch (resp., $H$-patch, $Y$-patch). 
\end{defn}

\begin{figure}
	\centering
	\subfloat{\includegraphics[trim= 4cm 6.5cm 38cm 6.5cm, clip = true, scale = 0.40]{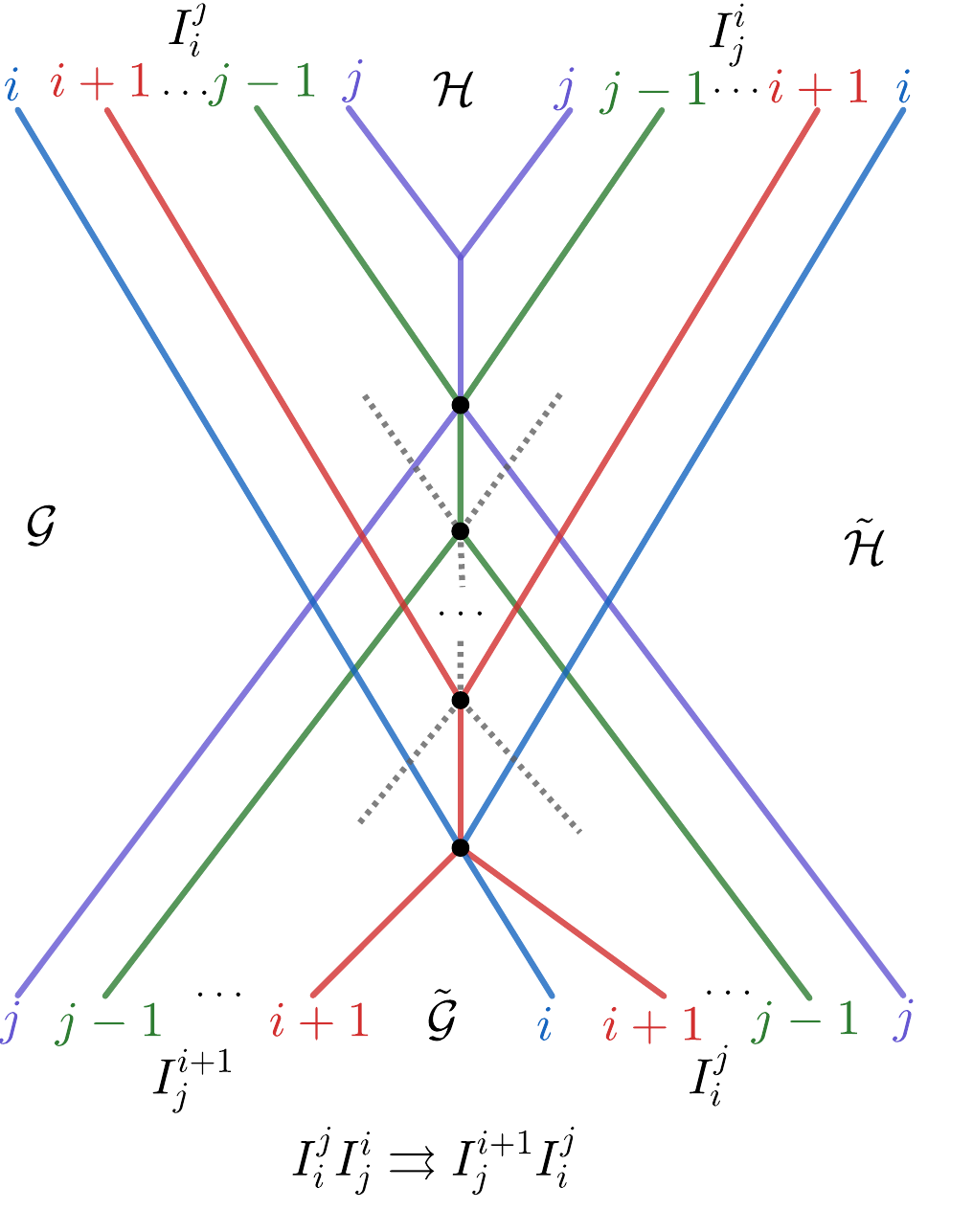}}
	\quad 
	\subfloat{\includegraphics[trim= 1cm 6.5cm 39cm 6.5cm, clip = true, scale = 0.40]{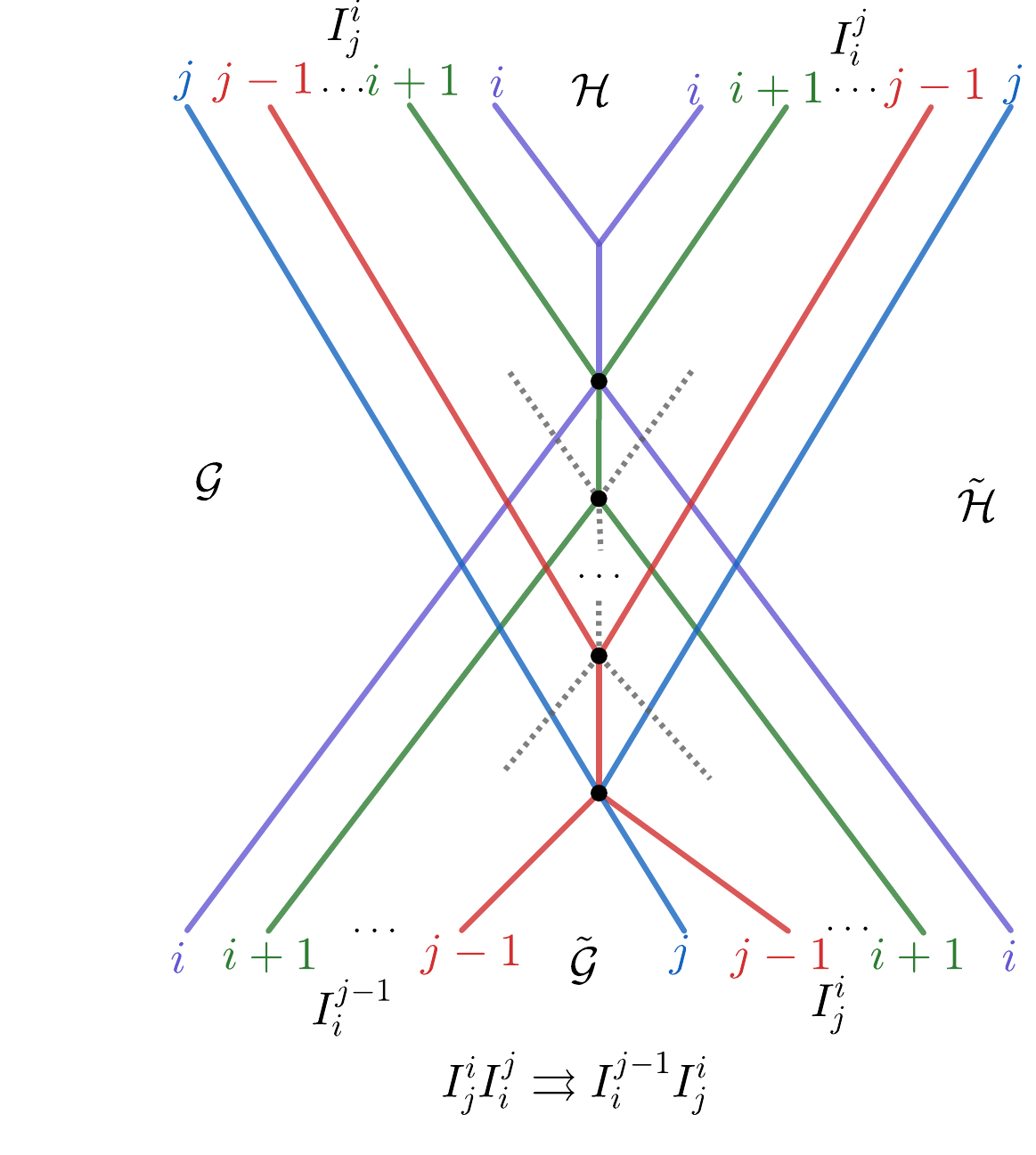}}
	\caption{$X$-patches $I_i^jI_j^i \rarrow I_j^{i+1}I_i^j$ (left) and $I_j^i I_i^j \rarrow I_i^{j-1}I_j^i$ (right).}
	\label{fig: X patch small}
\end{figure}
\begin{figure}
	\centering
	\subfloat{\includegraphics[trim= 2cm 8.5cm 37.85cm 10.55cm, clip = true, scale = 0.41]{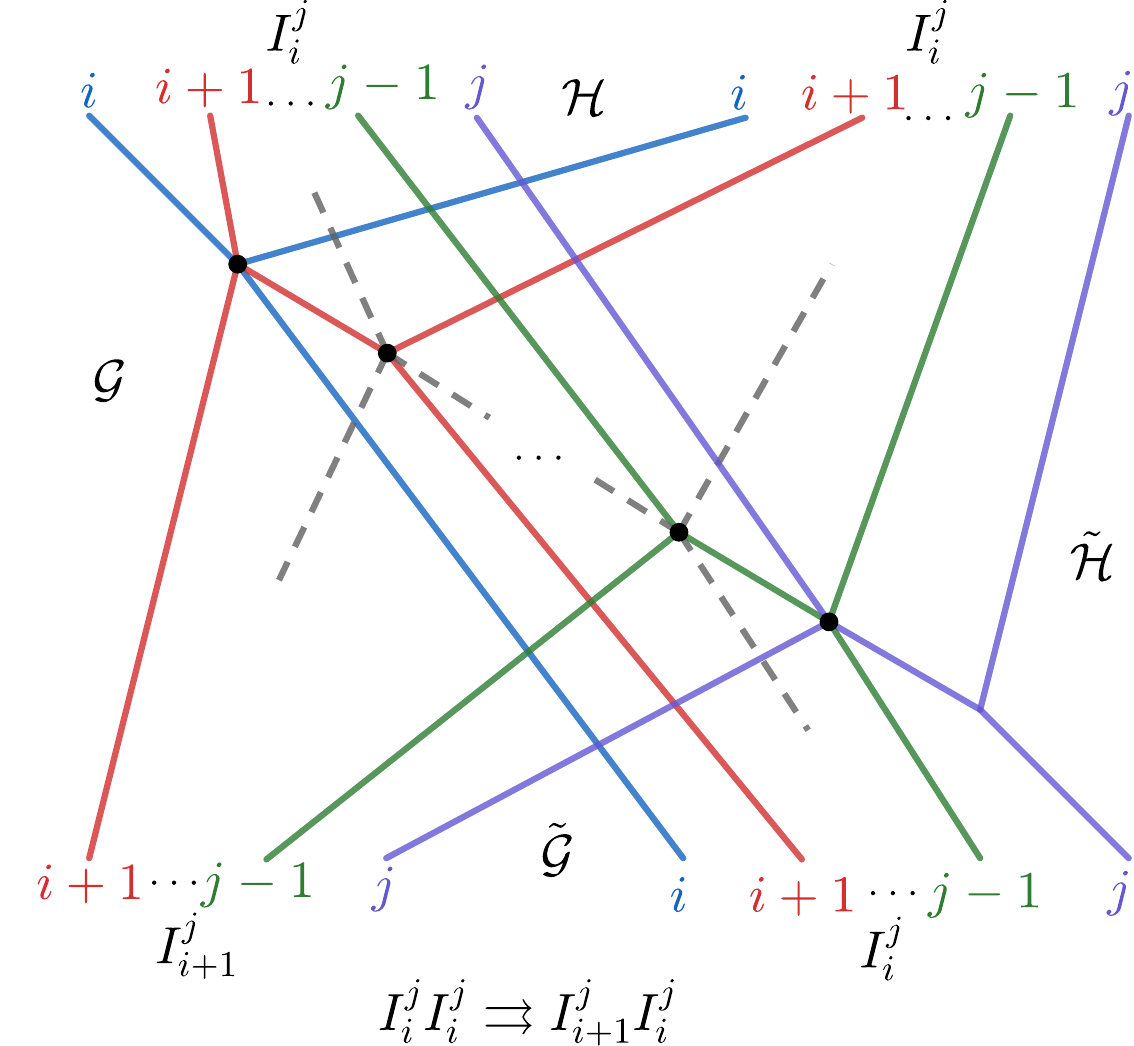}}
	\quad 
	\subfloat{\includegraphics[trim= 0.5cm 7.5cm 41cm 8.2cm, clip = true, scale = 0.41]{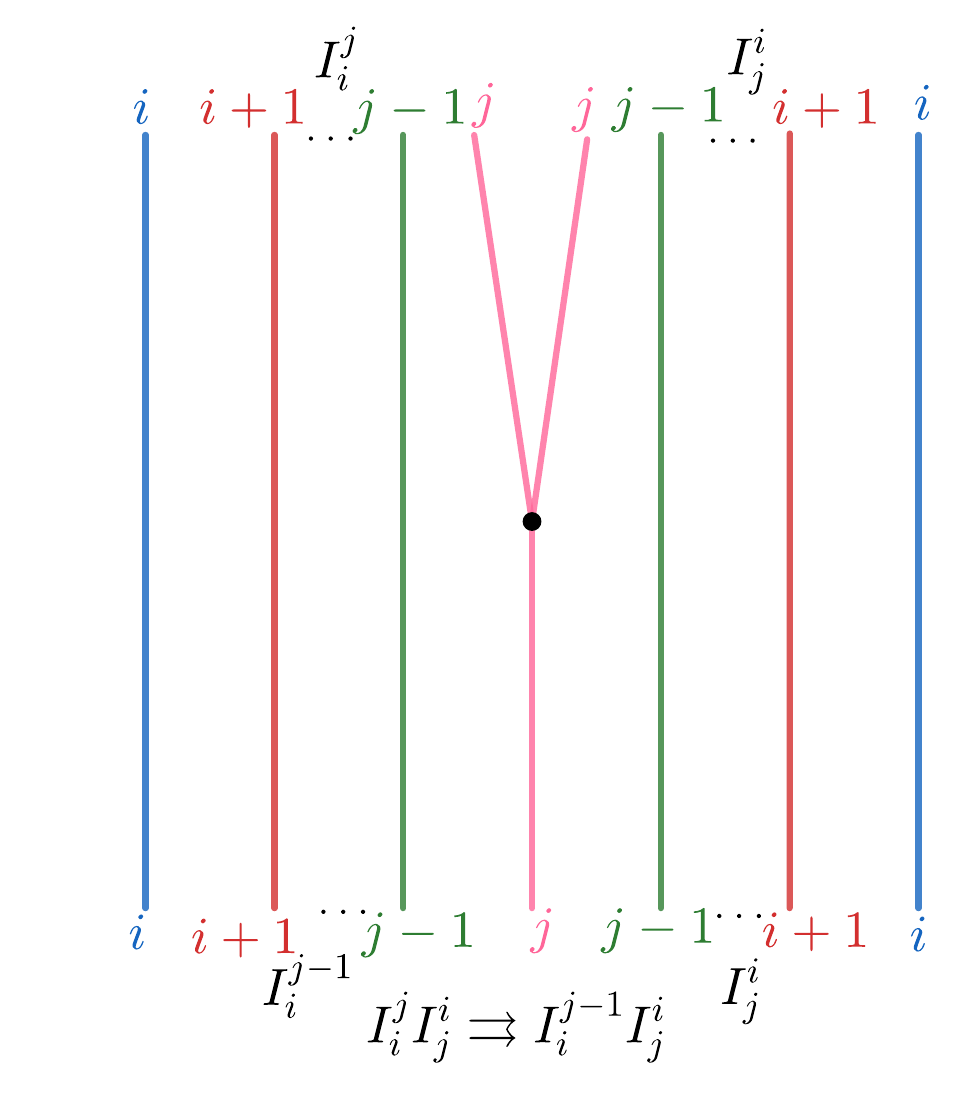}}
	\caption{An $H$-patch $I_i^jI_i^j \rarrow I_{i+1}^jI_i^j$ and $Y$-patch $I_j^i I_j^i\rarrow I_i^{j-1}I_j^i$.}
	\label{fig: H patch and Y patch small}
\end{figure}

\begin{figure}
	\centering
	\subfloat{\includegraphics[trim= 2cm 9cm 38cm 11cm, clip = true, scale = 0.41]{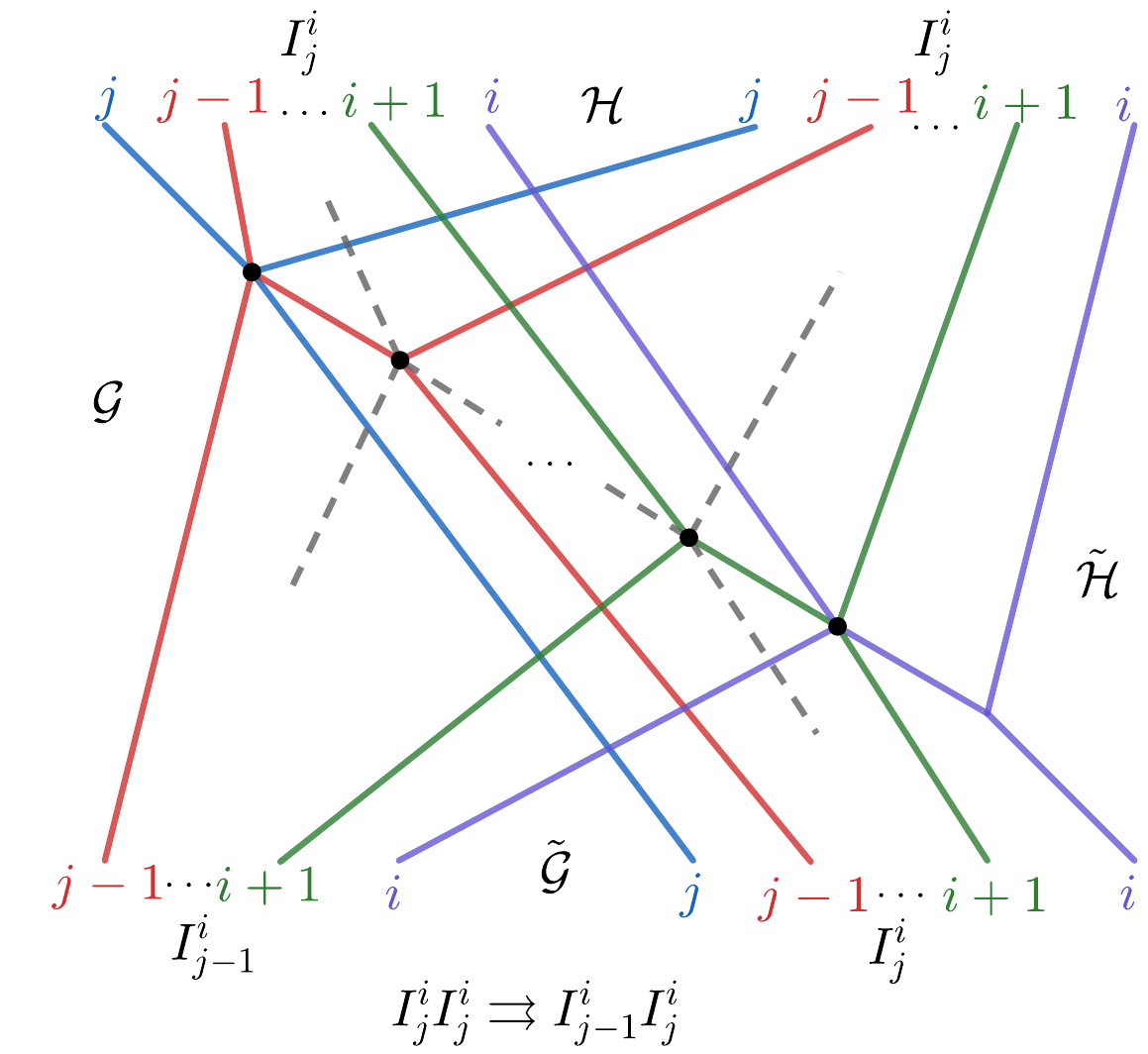}}
	\quad 
	\subfloat{\includegraphics[trim= 0.5cm 8cm 41cm 8.5cm, clip = true, scale = 0.41]{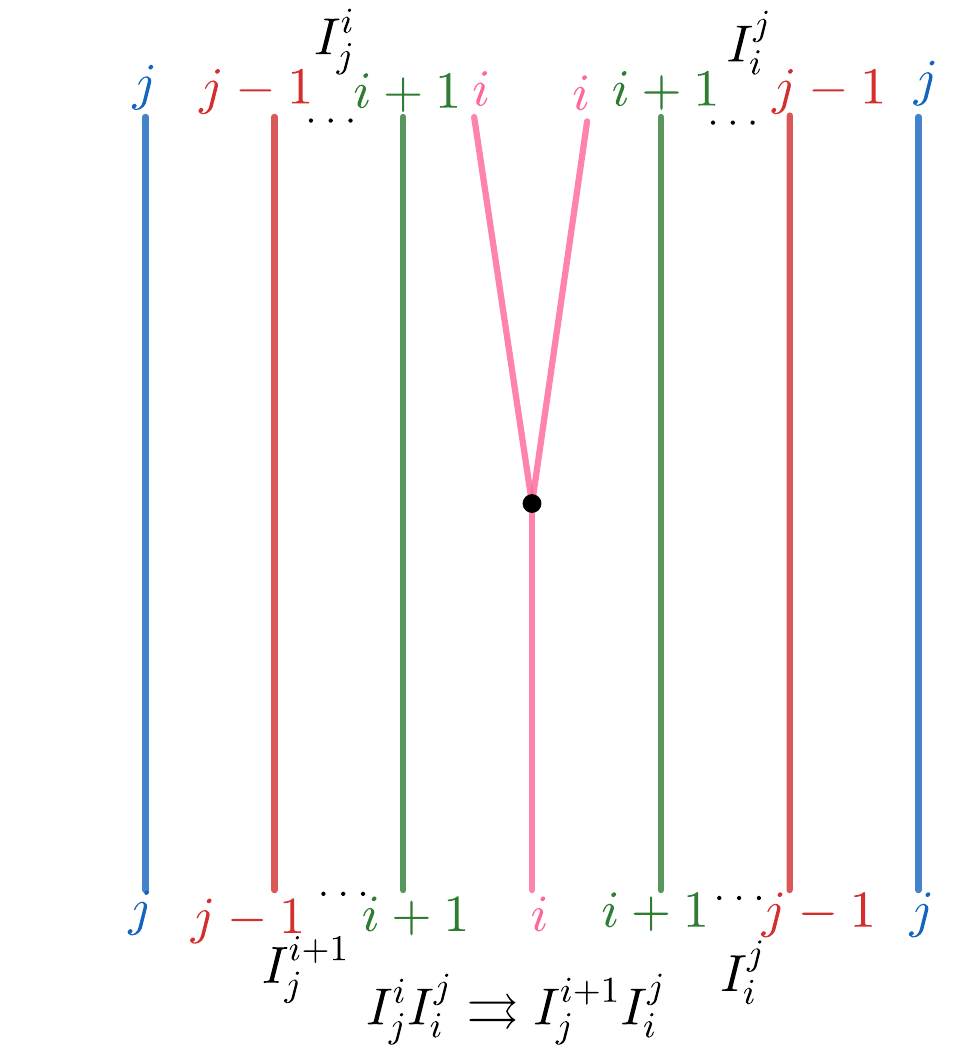}}
	\caption{An $H$-patch $I_j^iI_j^i \rarrow I_{j-1}^{i}I_j^i$ and $Y$-patch $I_j^i I_i^j \rarrow I_j^{i+1}I_i^j$.}
	\label{fig: H patch and Y patch small dual}
\end{figure}
\begin{remk}
	When $i = j$, all these patches ``degenerate" into the same partial weave of the form $\cdots i i \cdots \rarrow \cdots i \cdots$. We call it a \emph{degenerate patch}. 
\end{remk}

\begin{defn}\label{defn: weakly nested words}
	Let $m \ge d-1$. A word $T$ is called \emph{weakly nested (of rank $m$) at $s$} for some $s\in [1, m]$ if it has the form
	\[
	T = T_1T_2\cdots T_sT_{s+1}\cdots T_{m},
	\]
	where $T_1, \dots, T_m$ are interval words such that
	\begin{itemize}[wide, labelwidth=!, labelindent=0pt]
		\item $T_{1}, T_2, \dots, T_s$ are  interval words for the same interval;
		\item $T_{s}T_{s+1}\cdots T_{m}$ is a nested word (cf. Definition~\ref{defn: nested words and complete nested words}).
	\end{itemize}
	We note that the nested word $T_{s}T_{s+1}\cdots T_{m}$ is not necessarily complete (cf.\ Definition~\ref{defn: nested words and complete nested words}). 
\end{defn}

	\begin{defn}[\textbf{Construction of $\str{\beta}$}]   \label{defn: strips}
		Let $\beta = T_1T_2\cdots T_sT_{s+1}\cdots T_{m}$ be a weakly nested word at $s$.
		We denote by $\str{\beta}$ a partial weave, with $\beta$ as the top word, that constructed by concatenating patches as follows.

		Without loss of generality, we assume that $T_1 = I_i^j$ with $1\le i < j\le d-1$. 
		
		If $T_1\neq T_2$ and $T_2 = T_3$, then we apply a $Y$-patch to $T_1T_2$:
		\[
		T_1T_2\cdots T_sT_{s+1}\cdots T_{m} \rarrow T'_1T_2T_3\cdots T_sT_{s+1}\cdots T_m,
		\]
		where $T'_1 = I_{i}^{j-1}$. We then apply an $H$-patch (resp., $X$-patch) to $T_2T_3$, provided $T_2 = T_3$ (resp., $T_2 \neq T_3$):
		\[
		T'_1T_2T_3T_4\cdots T_sT_{s+1}\cdots T_m \rarrow T'_1 T'_3 T_2 T_4\cdots T_sT_{s+1}\cdots T_m,
		\]
		with $T'_3 = I_{j-1}^i$ (resp., $T'_3 = I_{i}^{j-1}$). Recursively applying patches to $T_2T_k$, $3\le k \le s$, we end up getting a Demazure weave
		\[
		\str{\beta}: \beta = T_1T_2\cdots T_sT_{s+1}\cdots T_{m} \rarrow \cdots \rarrow \beta' = T'_1T'_3\cdots T'_s T_2 T_{s+1}T_{s+2}\cdots T_m.
		\]
		Notice that the bottom word for $\str{\beta}$, $\beta' = T'_1T'_3\cdots T'_s T_2 T_{s+1}T_{s+2}\cdots T_m$, is weakly nested at $s-1$: the first $s-1$ interval words are on the same interval, and the rest $T'_s T_2 T_{s+1}T_{s+2}\cdots T_m$ forms a nested word.

		Otherwise we apply a $H$-patch (resp., $X$-patch) to $T_1T_2$ if $T_1 = T_2$ (resp., $T_1 \neq T_2$):
		\[
		T_1T_2\cdots T_sT_{s+1}\cdots T_{m} \rarrow T'_2T_1T_3\cdots T_sT_{s+1}\cdots T_m,
		\]
		where $T'_2 = I_{i-1}^{j}$ (resp., $T'_2 = I_j^{i-1}$). Recursively applying patches to $T_1T_k$, $2\le k \le s$, we end up getting a Demazure weave
		\[
		\str{\beta}: \beta = T_1T_2\cdots T_sT_{s+1}\cdots T_{m} \rarrow \cdots \rarrow \beta' = T'_2T'_3\cdots T'_s T_1 T_{s+1}T_{s+2}\cdots T_m.
		\]
		Similarly the bottom word for $\str{\beta}$, $\beta' = T'_2T'_3\cdots T'_s T_1 T_{s+1}T_{s+2}\cdots T_m$, is weakly nested at $s-1$: the first $s-1$ interval words are all on the same interval, and $T'_s T_1 T_{s+1}T_{s+2}\cdots T_m$ forms a nested word. 
		
		The \emph{type} of $\str{\beta}$ is defined to be the type of the first patch we used in the construction of the strip. 
		
		If $i = j$, then $T_1 = T_2 = \cdots = T_s = i$, and $\str\beta$ is formed by merging these weave lines (all of color $i$) from left to right. The bottom word is a complete nested word $iT_{s+1}\cdots T_m$. In this case, we say $\str \beta$ is of \emph{degenerate type}.
	\end{defn}

	\begin{remk}\label{remk: nested part unchanged}
		The \emph{nested part} $T_{s+1}\cdots T_m$ of $\beta$ remains unchanged during the transformation process from $\beta$ to $\beta'$.
	\end{remk}
	Following the convention in Definition~\ref{defn: cutting along the disk to get two weaves}, we next construct two specific Demazure weaves $\ww_1 = \ww(p, q)$ and $\ww_2 = \ww(q, p+n)$, called \emph{the initial weaves}, as a concatenation of strips. 

\begin{defn}[\textbf{Construction of the initial weaves}]\label{defn: construction of the initial weaves as concatenations of strips}
	Let $m_1 = q-p$ and $m_2 = p+n-q$. For $i = 1, 2$, the Demazure weave $\ww_i$ is constructed by the following recursive procedure, with $\beta_i$ as the input.
	
	\begin{itemize}[wide, labelwidth = 0pt, labelindent=0pt]
		\item[\textbf{Step 0:}] Set $\beta = \beta_i$. Set $\ww_i: \beta \rarrow \beta$ to be the trivial Demazure weave. 
		
		\item[\textbf{Step 1:}] If $\beta$ is a complete nested word, output $\ww_i$ and end the algorithm. Otherwise, let $\str{\beta}: \beta \rarrow \beta'$ be the strip defined in Definition~\ref{defn: strips}; notice that this can be done since $\beta$ is weakly nested of rank $m_i$ at $s$ for some $s\in [1, m_i]$.
		
		\item[\textbf{Step 2:}] Concatenate the bottom of $\ww_i$ with the top of $\str{\beta}$. Assign the new Demazure weave to $\ww_i$. Set $\beta := \beta'$. Return to Step 1.
	\end{itemize}
	
	Notice that Steps 1 \& 2 of the algorithm above will run exactly $d-1$ times. To be precise, Step 1 will be executed $d$ times, while during the last time, we simply output the weave and end the algorithm.
\end{defn}

\begin{example}
	Consider the weave $\ww_1$ in Figure~ \ref{fig: example of beta1} of Example~\ref{example: describing a seed from a decorated weave}. It's obtained by a concatenation of two strips. The portion between any two consecutive dashed lines is a single patch. The concatenation of the first three patches forms the first strip. The concatenation of the fourth and fifth patches forms the second strip.
\end{example}

We next describe the cycles and the decorations for the initial weaves $\ww_1$ and $\ww_2$. For this purpose, we focus on $\ww_1$; the result for $\ww_2$ is obtained in a similar way. Let us fix some notations. 

Recall from Definition~\ref{defn: cutting along the disk to get two weaves} that $\beta_1 = \beta(p, q) = \prod_{j = p}^{q-1} \rho_j$, where $\rho_j$ is either $\rho$ or $\rho^*$ depending on the signature $\sigma(j)$, is the top word for $\ww_1$. By Theorem~\ref{thm: bijection between configuration spaces of signature and decorated flag moduli spaces} and Remark~\ref{remk: restriction of decorations to a subword}, the normalized cyclic decoration $\dec_{\beta_\sigma} = \big(\mcf_1, \mcf_2, \dots, \mcf_n, \mcf_{n+1}=\mcf_1\big)$ for $\beta_\sigma$ restricts to a normalized decoration $\dec_{\beta_1} = \big(\mcf_p, \mcf_{p+1}, \dots, \mcf_{q}\big)$ for $\beta_1$:
\[
\mcf_p \rel{\rho_p} \mcf_{p+1}\rel{\rho_{p+1}} \dots \rel{\rho_{q-1}} \mcf_q.
\]
Let $(\ww_1, \dec)$ be the decorated Demazure weave with $\dec_{\beta_1}$ as the top decoration, cf.\ Definition~\ref{defn: decorated demazure weave} and Theorem~\ref{thm: unique extension of decorations for a weave}. We will simply refer to $(\ww_1, \dec)$ as $\ww_1$, since $\dec$ is fixed throughout this section. 

Now recall that $\ww_1$ is constructed via concatenations of strips, cf.\ Definition~\ref{defn: construction of the initial weaves as concatenations of strips}. There will be $d-1$ strips in total; let us label them so that the $r$-th strip is labeled as $\str{r}: \beta(r-1) \rarrow \beta(r)$, for $r\in [1, d-1]$. Hence $\beta(0)$ is the top word $\beta(p, q)$ and $\beta(d-1)$ is the bottom word for $\ww_1$. Notice that $\beta(d-1)$ is a complete nested word; in particular, it is braid equivalent to the longest word $w_0\in S_d$.

\begin{proof}[Proof of Lemma~\ref{lemma: bottom word is w0}]
	By the discussion above, the bottom word for a specific Demazure weave $\ww_1$ is a complete nested word; hence it is braid equivalent to $w_0$. As different choices of a reduced bottom word for a Demazure weave are  braid equivalent, we conclude that $\beta_1'$ is braid equivalent to $w_0$. Similarly $\beta_2'$ is braid equivalent to $w_0$. 
\end{proof}

Before moving on to the descriptions of the cycles and decorations of $\ww$, let us look at an example.

\begin{example}	
	Let $d = \dim V= 5, \sigma = [\bullet\, \circ\, \bullet\, \bullet\, \circ\, \bullet\, \bullet\,\bullet\, \cdots]$, $p = 1$, $q = 8$, $n\ge 13$. Then $\beta_1 = \beta(p, q) = \rho\rho^*\rho\rho\rho^*\rho\rho$, and the initial weave $\ww_1$ is described in Figure~\ref{fig: initial weave example}. 
\end{example}

\begin{figure}
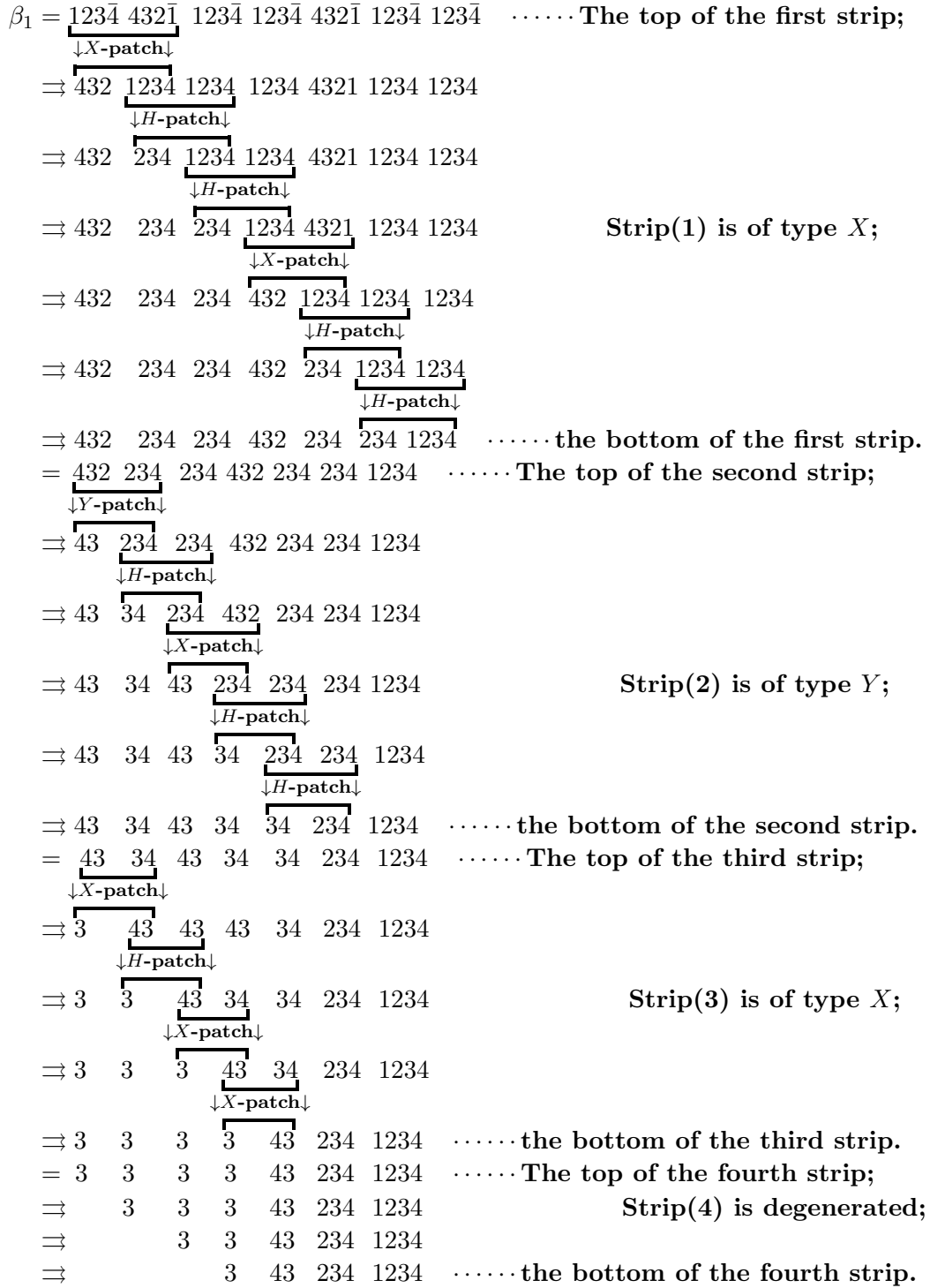

	\centering
	\begin{align*}
			\beta_1&= \underbracket{123\bar4 \ 432\bar1}_{\downarrow\textbf{$X$-patch}\downarrow} \ 123\bar4 \ 123\bar4 \ 432\bar1 \ 123\bar4 \ 123\bar4  \quad \cdots \cdots   \textbf{The top of the first strip;}\\[-0.2em]
			& \rarrow \lefteqn{\overbracket{\phantom{432\ 1234}}} 432 \ \underbracket{1234 \ 1234}_{\downarrow \textbf{$H$-patch}\downarrow} \ 1234 \ 4321 \ 1234\ 1234 \\[-0.2em]
			& \rarrow 432 \ \  \,\lefteqn{\overbracket{\phantom{234 \ 1234}}} 234 \ \underbracket{1234 \ 1234}_{\downarrow \textbf{$H$-patch} \downarrow} \ 4321 \ 1234 \ 1234   \\[-0.2em]
			& \rarrow 432 \ \ \  234 \ \ \lefteqn{\overbracket{\phantom{234 \ 1234}}} 234 \ \underbracket{1234 \ 4321}_{\downarrow \textbf{$X$-patch} \downarrow} \ 1234 \ 1234 \quad \quad\quad\quad \quad  \textbf{\str{1} is of type $X$;}\\[-0.2em]
			 &\rarrow 432 \ \ \  234 \ \ 234 \ \  \lefteqn{\overbracket{\phantom{432 \ 1234}}} 432 \ \underbracket{1234 \ 1234}_{\downarrow \textbf{$H$-patch} \downarrow}  \ 1234\\[-0.2em]
			&\rarrow 432 \ \ \  234 \ \ 234 \ \  432 \ \ \lefteqn{\overbracket{\phantom{234 \ 1234}}} 234 \ \underbracket{1234 \ 1234}_{\downarrow \textbf{$H$-patch} \downarrow}\\[-0.2em]
			&\rarrow 432 \ \ \  234 \ \ 234 \ \  432 \ \  234 \ \ \lefteqn{\overbracket{\phantom{234 \ 1234}}} 234 \ {1234} \quad \cdots \cdots   \textbf{the bottom of the first strip.}\\[-0.2em]
			&= \underbracket{432 \ \,234}_{\downarrow\textbf{$Y$-patch}\downarrow} \ 234 \ 432 \ 234 \ 234 \ 1234  \quad \cdots \cdots   \textbf{The top of the second strip;}\\[-0.2em]
			& \rarrow \lefteqn{\overbracket{\phantom{43\ \ 234}}} 43 \ \, \underbracket{234 \ \ 234}_{\downarrow \textbf{$H$-patch}\downarrow} \ 432 \ 234 \ 234\ 1234 \\[-0.2em]
			& \rarrow 43 \ \  \,\lefteqn{\overbracket{\phantom{34 \ \ 234}}} 34 \ \, \underbracket{234 \ \ 432}_{\downarrow \textbf{$X$-patch} \downarrow} \ 234 \ 234 \ 1234   \\[-0.2em]
			& \rarrow 43 \ \ \  34 \ \ \lefteqn{\overbracket{\phantom{43 \ \ 234}}} 43 \,\ \underbracket{234 \ \ 234}_{\downarrow \textbf{$H$-patch} \downarrow} \ 234 \ 1234 \hspace{3cm} \textbf{\str{2} is of type $Y$;}\\[-0.2em]
			&\rarrow 43 \ \ \  34 \ \ 43 \ \ \, \lefteqn{\overbracket{\phantom{34 \ \ 234}}} 34 \ \ \underbracket{234\  \ 234}_{\downarrow \textbf{$H$-patch} \downarrow}  \ 1234\\[-0.2em]
			&\rarrow 43 \ \ \  34 \ \ 43 \ \ \, 34 \ \ \ \lefteqn{\overbracket{\phantom{34 \, \ \ 234}}} 34\, \ \ {234\ \ 1234} \quad \cdots \cdots   \textbf{the bottom of the second strip.}\\[-0.2em]
			&= \underbracket{43 \quad 34}_{\downarrow\textbf{$X$-patch}\downarrow} \, 43 \ \ \, 34 \ \ \ 34 \, \ \ 234 \ \ 1234  \quad \cdots \cdots   \textbf{The top of the third strip;}\\[-0.2em]
			& \rarrow \lefteqn{\overbracket{\phantom{3\quad \ \ 43}}} 3 \quad \underbracket{43 \ \ \ 43}_{\downarrow \textbf{$H$-patch}\downarrow} \, 43 \ \ \ 34 \, \ \ 234 \ \ 1234 \\[-0.2em]
			& \rarrow 3 \quad \ \lefteqn{\overbracket{\phantom{3 \quad \ \  43}}} 3 \quad \underbracket{43 \ \ \, 34}_{\downarrow \textbf{$X$-patch} \downarrow}\ 34 \, \ \ 234 \ \ 1234 \hspace{3cm} \textbf{\str{3} is of type $X$;} \\[-0.2em]
			& \rarrow 3 \quad \ 3 \quad \  \ \lefteqn{\overbracket{\phantom{3 \quad \  43}}} 3 \ \ \underbracket{43 \ \ \ 34}_{\downarrow \textbf{$X$-patch} \downarrow} \ 234 \ \ 1234 \\[-0.2em]
			&\rarrow 3 \ \ \ \ 3 \ \ \ \ \  3  \ \ \ \, \overbracket{{3 \quad \  43}}\ \  234  \ \ 1234 \quad \cdots \cdots   \textbf{the bottom of the third  strip.}\\[-0.2em]
			&=\; 3 \ \ \ \ 3 \ \ \ \ \  3  \ \ \ \ {{3 \quad \;  43}}\ \ \,  234  \ \ 1234 \quad \cdots \cdots   \textbf{The top of the fourth  strip;}\\[-0.2em]
			&\rarrow \quad \ \ \: 3 \ \ \ \ \  3  \ \ \ \ {{3 \quad \;  43}}\ \ \,  234  \ \ 1234  \hspace{3cm} \textbf{\str{4} is degenerated;}\\[-0.2em]
			&\rarrow \quad \ \ \: \phantom{3} \ \ \ \ \  3  \ \ \ \ {{3 \quad \;  43}}\ \ \,  234  \ \ 1234\\[-0.2em]
			&\rarrow \quad \ \ \: \phantom{3} \ \ \ \ \  \phantom{3}  \ \ \ \ {{3 \quad \;  43}}\ \ \,  234  \ \ 1234 \quad \cdots\cdots    \textbf{the bottom of the fourth strip.}
	\end{align*}
	\caption{An example of the initial weave $\ww_1$.}
	\label{fig: initial weave example}
\end{figure}

\begin{remk}
	Assume that $\sigma$ is of the following form on $[p, q-1]$: 
	\[
	\underbrace{\bullet\,\bullet\, \cdots \,\bullet}_{b_1+1}\,\underbrace{\circ\,\circ\,\cdots\,\circ}_{w_1+1}\,\underbrace{\bullet\,\bullet\, \cdots \,\bullet}_{b_2+1}\,\underbrace{\circ\,\circ\,\cdots\,\circ}_{w_2+1}\,\bullet\, \cdots ;
	\]
	then the types of the strips in $\ww_1$ are of the following form:
	\[
	\underbrace{H\, H\, \cdots \, H}_{b_1}\,\underbrace{Y\,Y\,\cdots\, Y}_{w_1}\, X\, \underbrace{Y\,Y\,\cdots\, Y}_{b_2}\,X\,\underbrace{Y\,Y\,\cdots\, Y}_{w_2}\, X\, \cdots.
	\]
\end{remk}

Now let us describe the cycles for $\ww_1$. Let $m_1 = q - p$ and $r\in [1, d-1]$. First notice that we have $m_1$ frozen cycles that originate from the marked characters of $\beta$, i.e., the last character for each factor $\rho$ or $\rho^*$. Notice that $\str{r}$ is a concatenation of $m_1-r$ patches; and there are $m_1-r$ cycles originated from $\str{r}$, one for each patch. We will refer to each of them as the \emph{cycle associated with the corresponding patch}. 

\begin{prop}\label{prop: descriptions of cycles of the initial weave}
	All the cycles in $\ww_1$ are  (unweighted) subgraphs of $\ww$. Moreover, all the cycles are trees except for the cycles associated with the first patch (of type $X$) of a type $X$ strip. Each exceptional cycle that starts at the first patch of an $X$-strip will have a unique merging vertex. This vertex is either the last degree $6$ vertex in the first patch (of type $X$) of the next $X$-strip; or if such an $X$-strip does not exist, it is the unique trivalent vertex in the first patch of the last strip of $\ww_1$. 
\end{prop}
\begin{proof}
 	As an example of a non-tree cycle, we refer the reader to the cycle $\gamma_5$ in Figure~ \ref{fig: Example of beta1 with cycles part2}. 
 	
 	We will describe all these cycles in detail; as a result, we will verify the above statement.
 	\begin{enumerate}[wide, labelwidth=!, labelindent=0pt]
 	\item
 	Let us start with the frozen cycles coming from the last character of each $\rho$ or $\rho^*$. 
 	\begin{itemize}
 		\item[$\bullet$] Consider the first frozen cycle, i.e., the cycle that originate from (the last character of) the first interval word. If $\str{1}$ is of type $X$ or $Y$, then the cycle will end immediately at the trivalent vertex in the first patch. Otherwise the first patch is of type $H$; then the cycle will branch at the last degree 6 vertex; the right branch will end immediately at the trivalent vertex of this patch, while the left branch will enter the next strip, as the last character of the first interval word. So the argument above can be iterated, i.e., it will enter into the next strip if the current strip is of type $H$ and end otherwise. In the case where we have $d-2$ consecutive $H$-strips, the cycle will enter into the $(d-1)$-st strip, which is of degenerate type, and the cycle will end at the trivalent vertex in the first patch. To conclude, the cycle is a tree. 
 		\item[$\bullet$] Now consider the frozen cycle from (the last character of) the $k$-th interval word, for $k\in [2, m_1]$. It will enter into the $(k-1)$-st patch of the first strip, as the last weave line. If the patch is of type $H$, then the cycle ends immediately at the trivalent vertex in that patch. If the patch is of type $X$, then the cycle passes through the last degree 6 vertex, and enters the next strip: if $k = 2$, then it enters the first patch as the last character of the first interval word, repeating the same pattern as the first frozen cycle; if $k>2$, then it enters the $(k-2)$-rd patch as the last weave line, repeating the same pattern as the $(k-1)$-st frozen cycle. Finally if the patch is of type $Y$ (then in particular, $k = 2$), and it will enter the next patch, which is of type $H$, then repeating the same pattern as the first frozen cycle in a $H$-patch. To conclude, all the cycles are  trees. 
 	\end{itemize}

	\item Now consider a cycle associated with a patch of a strip, assuming the patch is neither the first patch nor the last one. Then the patch is either of type $H$ or of type $X$.
	\begin{itemize}
		\item[$\bullet$] If the cycle originates from a patch of type $H$, then it will enter into the next patch as the last weave line for the first interval word. If the patch is of type ~$X$, it will end immediately at the trivalent vertex of that patch. Otherwise, the patch is of type $H$, and it will branch at the last degree 6 vertex; the right branch will end immediately at the trivalent vertex of the patch, and the left branch will enter a patch of the next strip, as the last weave line; repeating the same pattern as a $k$-th frozen cycle, for some $k\in [2, m_1]$. To conclude, the cycle is a tree. 
		\item[$\bullet$] If the cycle originates from a patch of type $X$, then it will branch at the first degree 6 vertex of the patch. The left branch will enter a patch in the next strip, as the first weave line of the second interval word. The right branch will enter the next patch, as the last weave line of the first interval word. If the next patch is of type $X$, then it ends at the trivalent vertex of that patch. If the next patch is of type $H$, then it branches again at the last degree 6 vertex; the (second) right branch ends immediately at the trivalent vertex of the patch, while the (second) left branch enters a patch in the next strip (as the last weave line), right next to the patch that its left branch enters. Notice that these two adjacent patches are of different type. If we have an $X$ patch followed by an $H$ patch, then both of these two branches will end at the trivalent vertex of the corresponding patch. If we have an $H$ patch followed by an $X$ patch, then each of them will pass through a degree 6 vertex and enter the next strip, repeating the same pattern again. To conclude, the cycle is a tree.
	\end{itemize}
	
	\item Next we consider the cycle associated with the last patch of a strip. The discussion is very similar to the above, except that when a weave line enters the next strip as a weave line for the nested part of the nested word (cf.\ Remark~\ref{remk: nested part unchanged}), they continue all the way to the bottom of the weave. 
	\begin{itemize}
		\item[$\bullet$] If the patch is of type $H$, then the cycle will continue all the way to the bottom of the weave along the nested part of the row word. 
		\item[$\bullet$] If the patch is of type $X$, then the cycle will branch at the first degree 6 vertex of the patch. The left branch follows the same pattern as above, while the right branch will continue all the way to the bottom of the weave along the nested part of the row word. 
	\end{itemize}

	\item Finally consider a cycle associated with the first patch of a strip. 
	\begin{itemize}
		\item[$\bullet$] If the patch is of type $Y$, then the cycle will enter the next patch, as the first weave line. The next patch is of type $H$, so the cycle will pass through all degree 6 vertices in the second patch, and end at the trivalent vertex. It's a tree.
		
		\item[$\bullet$] If the patch is of type $H$, it follows the same pattern as a cycle associated with a type $H$-patch (not necessary the first patch of a strip) as above. 
		
		\item[$\bullet$] If the patch is of type $X$, then similar to the case where the cycle is not from the first patch, it will branch at the first degree 6 vertex of the patch. The left branch will enter the first patch in the next strip, as the first weave line; the right branch will enter the next patch, as the last weave line of the first interval word; the next patch is of type $H$, and it branches again at the last degree 6 vertex; the (second) right branch ends immediately at the trivalent vertex of the patch, while the (second) left branch enter sthe first patch in the next strip, as the last weave line. If that strip is of type $Y$, then these two branches enter the first patch of the next strip, as the first and last weave line, repeating the same pattern as before. Eventually we will enter a strip of type $X$. Then these two branches will merge at the last degree 6 vertex of the first patch, and enter the second patch, as the first weave line. The second patch is of type $H$, so the cycle will pass through all degree 6 vertices and end at the trivalent vertex. In the case that the final strip we entered is degenerated, these two branches simply merge at the trivalent vertex of the first patch, enters the next patch as the first weave line, and end at the trivalent vertex of the second patch. \qedhere
	\end{itemize}
\end{enumerate}
\end{proof}

Now let us describe the frozen cycles ending at the bottom of the weave $\ww_1$. 
\begin{cor}\label{cor: description of the frozen cycles to the bottom for the initial weave}
	The frozen cycles in $\ww_1$ ending at the bottom of the weave are exactly the cycles that originate from the last patch of each strip. In particular, there are $d-1$ of them. Moreover, for an edge for the bottom nested word $T= T_{m_1-d+2}T_{m_1-d+3}\cdots T_{m_1}$, there is a cycle passing through it if and only if the edge corresponds to the last character of an interval word $T_k$ for some $k\in [m_1-d+2, m_1]$.
\end{cor}
\begin{proof}
	The first part of the statement follows from the description of the cycles in $\ww_1$ in the proof of Proposition~\ref{prop: descriptions of cycles of the initial weave}; here we simply point out that for the last non-tree cycle, it will merge in the first patch of $\str{d-1}$ and end at the second patch of $\str{d-1}$. Notice that the requirement $m_1= q-p \ge d+1$ guarantees that the number of patches in $\str{d-1}$ is $m_1-(d-1) \ge 2$. 
	
	Now for the second statement, notice that for a cycle associated with the last patch of a strip, the branch that does not end will enter the next strip as the last character of an interval word (of the nested part); this property will be preserved to the bottom word. The claim follows.
\end{proof}

\begin{remk}\label{remk: separated case no X strips so relaxed condition on n}
	Notice that if there is no $X$-strip in $\ww_1$, then $m_1 \ge d$ (instead of $m_1 \ge d+1$) would be suffice for Corollary \ref{cor: description of the frozen cycles to the bottom for the initial weave}. This fact will be used in Definition~\ref{defn: valid cut for separated signatures}. 
\end{remk}

\begin{proof}[Proof of Lemma~\ref{lemma: bottom cycle description of ww(p, q) and bottom decoration is normalized}]
	By Lemma~\ref{lemma: changing from T to T^+ does not affact the seed}, changing the bottom word from $T$ to $T^+$ preserves the statement in the lemma. The latter then follows from Corollary~\ref{cor: description of the frozen cycles to the bottom for the initial weave}.
\end{proof}

We can now count both  frozen and mutable cycles.

\begin{cor}\label{cor: size of the cluster algebra for half of the weave}
	The weave $\ww_1$ contains:
	\begin{itemize}[wide, labelwidth=!, labelindent=0pt]
		\item $m_1$ frozen cycles originating from the top boundary (they do not end at the bottom boundary);
		\item $d-1$ frozen cycles ending at the bottom boundary;
		\item $\sum_{r=1}^{d-1} (m_1-r-1)= (m_1-1)(d-1) - \frac{d(d-1)}{2}$ mutable cycles.
	\end{itemize} 
	 Consequently the cluster (resp., the extended cluster) for the seed $\seed(\ww_1)$ contains $(m_1-1)(d-1) - \frac{d(d-1)}{2}$ (resp., $m_1d - \frac{d(d-1)}{2}$) elements. 
\end{cor}
\begin{proof}
	This follows from Proposition~\ref{prop: descriptions of cycles of the initial weave} and Corollary~\ref{cor: description of the frozen cycles to the bottom for the initial weave}. 
\end{proof}
We next describe the decoration of $\ww_1$ and the cluster variable $\Delta_\gamma$ associated with any cycle $\gamma$. Let us fix some notations and prove a few identities involving the decorated flags in $\dec_{\beta_1}$.

\begin{defn}\label{defn: consective wedges of the first k blacks and whites}
	Define $\mcb{j}{k} := u_{i_1}\mixwed u_{i_2} \mixwed \cdots \mixwed u_{i_k}$ (resp., $\mcw{j}{d-k} := u^*_{i_1}\mixwed u^*_{i_2} \mixwed \cdots \mixwed u^*_{i_k}$) where $i_1, i_2, \cdots, i_k$ are the first $k$ black (resp., white) vertices starting from $j$ (including $j$ itself). 
	
	Let $\mcg, \mch$ be decorated flags. Recall from Definition~\ref{defn: left/right/double push by flags} the notion of the decorated flag 
	\begin{align*}
		\lrpush{\mcg}{i}{\mch}{j}=&\big(\mcg^1, \cdots, \mcg^i, \mcg^{j}\wed\mcg^i \wed \mch^{d-j+1}, \cdots, \mcg^{j}\wed\mcg^i \wed \mch^{d-i-1}, \mcg^{j}, \cdots, \mcg^{d-1}\big)\\
		=&\big(\mcg^1, \cdots, \mcg^i, \mcg^{i}\wed\mcg^j \wed \mch^{d-j+1}, \cdots, \mcg^{i}\wed\mcg^j \wed \mch^{d-i-1}, \mcg^{j}, \cdots, \mcg^{d-1}\big),
	\end{align*}
	for $i<j$. Recall that $\mcg^{j}\wed\big(\mcg^i \wed \mch^{k}\big) = \mcg^i \wed \big(\mcg^{j} \wed \mch^{k}\big)$ by Proposition~\ref{prop: wedge of any two within a flag is zero}. By convention, we have $\lrpush{\mcg}{0}{\mch}{j} = \rpush{\mch}{\mcg}{j}$ and $\lrpush{\mcg}{i}{\mch}{d} = \lpush{\mcg}{i}{\mch}$. 

	For $\lambda \in \C-\{0\}$, define the decorated flag
	\[
	\frac{\mch}{\lambda} = \Big(\frac {\mch^1} \lambda, \frac{\mch^2} \lambda, \cdots, \frac{\mch^{d-1}} \lambda\Big),
	\]
	so that
	\[
	 \lrpush{\mcg}{i}{\frac{\mch}{\lambda}}{j} =
	\Big(\mcg^1, \cdots, \mcg^i, \frac{\mcg^{j}\wed\mcg^i \wed \mch^{d-j+1}}{\lambda}, \cdots, \frac{\mcg^{j}\wed\mcg^i \wed \mch^{d-i-1}}{\lambda}, \mcg^{j}, \cdots, \mcg^{d-1}\Big).
	\]
\end{defn}

\begin{lemma}\label{lemma: mcf p+r double push gives mcf p up to a constant}
	Let $r < d$ and $x$ (resp., $y$) be the number of black (resp., white) vertices in $[p, p+r-1]$. Notice that $x + y = r <d$. Let $\lambda = \mcb p x \wedge \mcf_{p+r}^{d-x}$ and $\lambda^* = \mcw{p}{d-y}\wedge \mcf_{p+r}^{y}$. Then we have the following identities of decorated flags. 
	\begin{enumerate}[wide, labelwidth=!, labelindent=0pt]
		\item If $\mcb p x = \mcf_p^x$ and $\mcw{p}{d-y} \neq \mcf_p^{d-y}$, or equivalently $\sigma(p) = \sigma(p+1) = \cdots  = \newline \sigma(p+x-1) = 1$ and $\sigma(p+x) = \sigma(p+x+1) = \dots = \sigma(p+r-1) = -1$, then 
		\begin{equation}\label{equation: factorization formula, BW case}
		 \mcf_p^x \wedge \mcf_{p+r}^{d-x} = \lambda, \quad 	\mcf_p^{d-y} \wedge \mcf_{p+r}^y = \lambda\lambda^*, \quad \text{and} \quad \mcf_p = \lrpush{\mcf_p}{x} {\frac{\mcf_{p+r}}{\lambda}}{d-y}.
		\end{equation}
	\item If $\mcb p x \neq \mcf_p^x$ and $\mcw{p}{d-y} = \mcf_p^{d-y}$, or equivalently, $\sigma(p) = \sigma(p+1) = \cdots = \sigma(p+y-1) = -1$ and $\sigma(p+y) = \sigma(p+y+1) = \dots = \sigma(p+r-1) = 1$, then 
	\begin{equation}\label{equation: factorization formula, WB case}
		\mcf_p^{x} \wedge \mcf_{p+r}^{d-x} = \lambda\lambda^*, \quad \mcf_p^{d-y} \wedge \mcf_{p+r}^y = \lambda^*, \quad  \text{and} \quad \mcf_p = \lrpush{\mcf_p}{x} {\frac{\mcf_{p+r}}{\lambda^*}}{d-y}.
	\end{equation}
	\item If $\mcb p x \neq \mcf_p^x$ and $\mcw{p}{d-y} \neq \mcf_p^{d-y}$, then 
	\begin{equation}\label{equation: factorization formula, generic case}
		\mcf_p^x \wedge \mcf_{p+r}^{d-x} = \mcf_p^{d-y} \wedge \mcf_{p+r}^y = \lambda\lambda^*, \quad \text{and} \quad\mcf_p = \lrpush{\mcf_p}{x} {\frac{\mcf_{p+r}}{\lambda\lambda^*}}{d-y}.
	\end{equation}
	\end{enumerate}
\end{lemma}
\begin{proof}
	We will prove the last statement. The first two statements follow in a similar (and easier) way. 
	
	Let us first show that $\mcf_p^{d-y}\wedge \mcf_{p+r}^{y} = \lambda \lambda^*$. (The other identity $\mcf_p^x\wedge \mcf_{p+r}^{d-x} = \lambda\lambda^*$ will follow in a similar way.) Write $\mcf_{p}^{d-y} = u_p\mixwed u_{p+1}\mixwed \cdots \mixwed u_{p+r-1}\mixwed \mcf_{p+r}^{d-x}$; view both $\mcf_{p+r}^{y}$ and $\mcf_{p+r}^{d-x}$ as extensors in $\bigwedge V^*$ and notice that $\mcf_{p+r}^{d-x} = \mcf_{p+r}^{y}\wedge v^*$ for some $v^*\in \bigwedge^{d-x-y} V^*$. Since there are exactly $x$ vectors in between $u_p, u_{p+1}, \dots, u_{p+r-1}$, these vectors have to pair with $\mcf_{p+r}^{d-x}$ in the expansion of the dual wedge $\mcf_p^{d-y}\wedge^*\mcf_{p+r}^y = \mcf_p^{d-y}\wedge \mcf_{p+r}^y$ (otherwise, the extra extensors in $\mcf_{p+r}^{d-x}$ will wedge with $\mcf_{p+r}^{y}$, resulting in $0$). Hence we have 
	\begin{align*}
	\mcf_p^{d-y}\wedge^* \mcf_{p+r}^{y}  &= \big(u_p\mixwed u_{p+1}\mixwed \cdots \mixwed u_{p+r-1}\mixwed \mcf_{p+r}^{d-x}\big) \wedge^* \mcf_{p+r}^y\\
	 &= \big(\mcb{p}{x}\wedge \mcf_{p+r}^{d-x}\big)\cdot \big(\mcw{p}{d-y} \wedge \mcf_{p+r}^y\big) = \lambda \lambda^*.
	\end{align*}
	
	Now we prove the last identity in (\ref{equation: factorization formula, generic case}). We need to show that for $x+1 \le s \le d-y-1$, we have 
	\begin{equation}
			\mcf_p^s = \frac{\mcf_p^x \mixwed \mcf_p^{d-y} \mixwed\mcf_{p+r}^{s-x + y}  }{\big(\mcb p x \wedge \mcf_{p+r}^{d-x}\big)\big(\mcw{p}{d-y}\wedge \mcf_{p+r}^{y}\big)}.
	\end{equation}

	First we notice that $\mcf_p^{d-y} \mixwed \mcf_{p+r}^{s-x + y}  = \big(\mcb p x \wedge \mcf_{p+r}^{d-x}\big)\cdot \mcw{p}{d-y}\mixwed \mcf_{p+r}^{s-x+y}$, similar to the arguments above.

	Now let us compute $\mcf_p^x \mixwed \big(\mcw{p}{d-y}\mixwed \mcf_{p+r}^{s-x+y}\big)$.
	Notice that $\mcf_p^x = u_p\mixwed u_{p+1} \mixwed \cdots \mixwed u_{p+r-1} \mixwed \mcf_{p+r}^{y}$; since $s-x+y \ge y$, there exists $w^* \in \bigwedge V^*$ such that $\mcf_{p+r}^y = \mcf_{p+r}^{s-x+y} \mixwed w^*$. Then we have $\mcf_p^x = u_p\mixwed u_{p+1} \mixwed \cdots \mixwed u_{p+r-1} \mixwed \big(\mcf_{p+r}^{s-x+y} \mixwed w^*\big)$. Now as we wedge $\mcf_p^x$ with $\mcw{p}{d-y}\mixwed \mcf_{p+r}^{s-x+y}$, treating the latter as dual wedges of covectors, then we are applying $\cap^*$, i.e., $\mcf_p^x \mixwed \big(\mcw{p}{d-y}\mixwed \mcf_{p+r}^{s-x+y}\big) = \mcf_p^x \cap^*\big (\mcw{p}{d-y}\mixwed \mcf_{p+r}^{s-x+y}\big)$. The only covectors in the expansions of $\mcf_p^x$ that can be used to pair with the latter term are from $w^*$ (the remaining covectors are either from $\mcw{p}{d-y}$ or from $\mcf_{p+r}^{s-x+y}$, resulting in $0$ when paired with the second term). Therefore we have
	\begin{align*}
		&\mcf_p^x \mixwed \big(\mcw{p}{d-y}\mixwed \mcf_{p+r}^{s-x+y}\big) \\
		=& \Big(u_p\mixwed u_{p+1} \mixwed \cdots \mixwed u_{p+r-1} \mixwed \big(\mcf_{p+r}^{s-x+y} \mixwed w^*\big)\Big) \mixwed \big(\mcw{p}{d-y}\mixwed \mcf_{p+r}^y\big)\\
		=&(-1)^{(s-x)(d-s+x)} \big(u_p\mixwed u_{p+1} \mixwed \cdots \mixwed u_{p+r-1} \mixwed \mcf_{p+r}^{s-x+y}\big) \mixwed \big(w^*\mixwed \mcw{p}{d-y}\mixwed \mcf_{p+r}^{s-x+y}\big)\\
		=& \big(\mcw{p}{d-y}\mixwed \mcf_{p+r}^y\big)\cdot \mcf_p^s.
	\end{align*}
	Putting everything together, we conclude that
	\begin{align*}
		\mcf_p^x\mixwed \mcf_p^{d-y} \mixwed \mcf_{p+r}^{s-x+y} &= \mcf_p^x \mixwed \Big(\big(\mcb p x \wedge \mcf_{p+r}^{d-x}\big)\cdot \mcw{p}{d-y}\mixwed \mcf_{p+r}^{s-x+y}\Big)\\
		&=\big(\mcb p x \wedge \mcf_{p+r}^{d-x}\big)\cdot\Big(\mcf_p^x \mixwed \big(\mcw{p}{d-y}\mixwed \mcf_{p+r}^{s-x+y}\big)\Big)\\
		&=\big(\mcb p x \wedge \mcf_{p+r}^{d-x}\big)\big(\mcw{p}{d-y}\mixwed \mcf_{p+r}^y\big)\cdot\mcf_p^s. \qedhere
		\end{align*}
\end{proof}

\begin{remk}
	There is a single identity unifying all the formulas for $\mcf_p$ in Lemma~ \ref{lemma: mcf p+r double push gives mcf p up to a constant}:
	\[
	\quad\mcf_p = \lrpush{\mcf_p}{x} {\frac{\mcf_{p+r}}{\mu}}{d-y}, \quad \text{where } \mu = \frac{\big(\mcf_p^x \wedge \mcf_{p+r}^{d-x}\big)\big(\mcf_p^{d-y} \wedge \mcf_{p+r}^y\big)}{\big(\mcb p x \wedge \mcf_{p+r}^{d-x}\big)\big(\mcw{p}{d-y}\wedge \mcf_{p+r}^{y}\big)}.
	\]
\end{remk}

\begin{defn}\label{defn: normal vertex and normal patches}
	Let $\ww$ be a Demazure weave. Let $v$ (resp., $w$) be a trivalent (resp., 6-valent) vertex in $\ww$. Let $e_v$ (resp., $e_w$) be the unique edge (resp., middle edge) incident to $v$ (resp., $w$) that is below $v$ (resp., $w$). Then we say that $v$ is \emph{normal} if $\gamma_v$ is the only cycle passing through $v$; and we say that $w$ is \emph{normal} if there is no cycle passing through $e_w$. As a consequence, if $v$ (resp., $w$) is normal, then $\Delta_{e_v} = \Delta_{\gamma_v}$ (resp., $\Delta_{e_w} = 1$). 
	
	Let $P$ be a patch of $\ww$. Then we say that $P$ is \emph{normal} if all trivalent and 6-valent vertices in $P$ are normal. 
\end{defn}

\begin{defn}\label{defn: decoration for a patch notation}
	Let $P$ be a patch: $\cdots T T' \cdots  \rarrow \cdots  \tilde{T} T\cdots$, where $T, T'$ are interval words for the same interval $[i,j]$, with $1\le i \le j \le d-1$. We will use the following notation to represent a decoration for the patch $P$:
	\[
	\cdots \mcg \rel{T} \mch \rel{T'} \tilde\mch \cdots \rarrow \cdots \mcg \rel{\tilde{T}} \tilde\mcg \rel{T} \tilde\mch\cdots 
	\] where 
	\[
	\cdots \mcg \rel{T} \mch \rel{T'} \tilde\mch \cdots  
	\]
	is the decoration for the top word and 
	\[
	\cdots \mcg \rel{\tilde{T}} \tilde\mcg \rel{T} \tilde\mch\cdots 
	\]
	is the decoration for the bottom word, see Figure~\ref{fig: X patch small}~\ref{fig: H patch and Y patch small}~\ref{fig: H patch and Y patch small dual}. 
\end{defn}


\begin{lemma}\label{lemma: decorations for normal H and X-patches}
	Let $P$ be a normal patch of type $H$ or $X$ with the decoration
	\[
	\cdots \mcg \rel{T} \mch \rel{T'} \tilde\mch \cdots \rarrow \cdots \mcg \rel{\tilde{T}} \tilde\mcg \rel{T} \tilde\mch\cdots
	\]
	These decorated flags satisfy the following recursive formulas, depending on the type of the patch; and the cluster or frozen variables $\Delta_\gamma$ associated with the cycle $\gamma$ originating in patch $P$ is as follows:
	\begin{enumerate}[wide, labelwidth=!, labelindent=0pt]
		\item if $T = T' = I_i^j$, then $P$ is of type $H$ and we have
		\begin{equation}\label{equation: H-patch formula bb}
			\tilde{\mcg}^{k+1} = \qwed{\mcg^i}{\tilde{\mch}^k}{\mcg^{i-1}}, \quad \text{for } k \in [i, j-1]; \quad \text{and}\quad  \Delta_\gamma = \frac{\qwed{\mcg^i}{\tilde{\mch}^j}{\mcg^{i-1}}}{\mcg^{j+1}};
		\end{equation}
	\item if $T = T' = I_j^i$, then $P$ is of type $H$ and we have
	\begin{equation}\label{equation: H-patch formula ww}
		\tilde{\mcg}^{k-1} = \qwedd{\mcg^j}{\tilde{\mch}^k}{\mcg^{j+1}}, \quad \text{for } k \in [i+1, j]; \quad \text{and}\quad  \Delta_\gamma = \frac{\qwedd{\mcg^j}{\tilde{\mch}^i}{\mcg^{j+1}}}{\mcg^{i-1}};
	\end{equation}
\item if $T = I_i^j$ and $T' = I_j^i$, then $P$ is of type $X$ and we have
\begin{equation}\label{equation: X-patch formula bw}
	\tilde{\mcg}^{k+1}  = \qwed{\mcg^k}{\tilde{\mch}^k}{\mch^{k-1}}, \quad \text{for } k\in [i, j-1]; \quad \text{and}\quad  \Delta_\gamma = \frac{\qwed{\mcg^j}{\tilde{\mch}^j}{\mch^{j-1}}}{\mcg^{j+1}};
\end{equation}
furthermore, if $\mcg$ and $\mch$ satisfy the condition $\mcg^k = \qwed{\mcg^i}{\mch^{k-1}}{\mcg^{i-1}}$ for $k\in [i, j-1]$, then equations (\ref{equation: X-patch formula bw}) coincide with equations (\ref{equation: H-patch formula bb});
\item if $T = I_j^i$ and $T' = I_i^j$, then $P$ is of type $X$ and we have
\begin{equation}\label{equation: X-patch formula wb}
	\tilde{\mcg}^{k-1}  = \qwedd{\mcg^k}{\tilde{\mch}^{k}}{\mch^{k+1}}, \quad \text{for } k\in [i+1, j]; \quad \text{and}\quad  \Delta_\gamma = \frac{\qwedd{\mcg^i}{\tilde{\mch}^i}{\mch^{i+1}}}{\mcg^{i-1}};
\end{equation}
furthermore, if $\mcg$ and $\mch$ satisfy the condition $\mcg^k = \qwedd{\mcg^j}{\mch^{k+1}}{\mcg^{j+1}}$ for $k\in [i, j-1]$, then equations (\ref{equation: X-patch formula wb}) coincide with equations (\ref{equation: H-patch formula ww}).
	\end{enumerate}
\end{lemma}
\begin{proof}
	We explain the details for the first case, the other cases follow in a similar way.   Assume  that $T = T' = I_i^j$. For $k\in [i, j-1]$, consider the 6-valent vertex $w$ such that $e_w$ is of color $k$. Since the patch is normal, we have $\Delta_{e_w} = 1$, cf.\ Definition~\ref{defn: normal vertex and normal patches}. Hence 
	\[
	\Delta_{e_w} = \frac{\qwed{\tilde{\mcg}^k}{\tilde{\mch}^k}{\tilde{\mch}^{k-1}}}{\tilde{\mcg}^{k+1}} = 1,
	\]
	and therefore 
	\begin{equation}\label{equation: H-patch formula, for proofs}
	\tilde{\mcg}^{k+1} = \qwed{\tilde{\mcg}^k}{\tilde{\mch}^k}{\tilde{\mch}^{k-1}}.
	\end{equation}
	Now we apply the induction on $k$ to formula (\ref{equation: H-patch formula bb}). When $k = i$, equation (\ref{equation: H-patch formula bb}) reads exactly the same as equation (\ref{equation: H-patch formula, for proofs}). assume that equation (\ref{equation: H-patch formula bb}) is true for $k-1$. Then by equation (\ref{equation: H-patch formula, for proofs}) and the induction hypothesis, we have
	\[
	\tilde{\mcg}^{k+1} = \qwed{\tilde{\mcg}^k}{\tilde{\mch}^k}{\tilde{\mch}^{k-1}} = \qwed{\big(\qwed{\mcg^i}{\tilde{\mch}^{k-1}}{\mcg^{i-1}}\big)}{\tilde{\mch}^k}{\tilde{\mch}^{k-1}} = \qwed{\mcg^i}{\tilde{\mch}^k}{\mcg^{i-1}}.
	\]
	Here the last equality follows from Lemma~\ref{lemma: cancellation law for quotient wedge}. Finally, let $v$ be the unique trivalent vertex in $P$, then $\gamma = \gamma_v$; since $P$ is normal, we conclude that
	\[
	\Delta_\gamma= \Delta_{e_v} = \frac{\qwed{\tilde{\mcg}^j}{\tilde{\mch}^j}{\tilde{\mch}^{j-1}}}{\mcg^{j+1}} =  \frac{\qwed{\big(\qwed{\mcg^i}{\tilde{\mch}^{j-1}}{\mcg^{i-1}}\big)}{\tilde{\mch}^j}{\tilde{\mch}^{j-1}}}{\mcg^{j+1}}=\frac{\qwed{\mcg^i}{\tilde{\mch}^j}{\mcg^{i-1}}}{\mcg^{j+1}}.\qedhere
	\]
\end{proof}

\begin{remk}
	Notice that by Proposition~\ref{prop: descriptions of cycles of the initial weave}, except for the first patch of each strip, all patches in the initial weave $\ww_1$ are normal. For the first patch of a strip, there is only one 6-valent vertex that is not normal. In this sense, we can adapt Lemma~\ref{lemma: decorations for normal H and X-patches} to compute the decoration for $\ww_1$ and the cluster and frozen variables associated with the cycles. 
\end{remk}

Let us fix some notations for the next Proposition, which describes the decoration for $\ww_1$ and presents formulas for the cluster and frozen variables for $\ww_1$. Let $r\in [1, d-1]$ and let $\str{r}: \beta(r-1) \rarrow \beta(r)$ be the $r$-th strip. Then $\beta(r-1) = T_1T_2\cdots T_{m_1-r}T_{m_1-r+1}\cdots T_m$ is a weakly nested word at $m-r+1$. Let $x$ (resp., $y$) be the number of black (resp., white) vertices in $[p, p+r-1]$. Then $x + y = r$. Let $\gamma(r, i)$ be the cycle associated with the $i$-th patch in $\str{r}$, for $i\in [1, m_1-r]$; and let $\gamma(0, i)$ be the $i$-th frozen cycles associated with the marked boundary vertices, for $i\in [1, m_1]$. 

Notice that the normalized decoration for a complete nested word is uniquely determined by Lemma~\ref{lemma: decoration for complete nested word is unique}, and in our case, the decoration for the nested part of the word $\beta(r-1)$ is normalized by Corollary~\ref{cor: description of the frozen cycles to the bottom for the initial weave} and the arguments in the proof of Lemma~\ref{lemma: changing from T to T^+ does not affact the seed}. So we will not explicitly describe them here. 

\begin{prop} \label{prop: decoration flags of the initial weave}
	The weakly nested word $\beta(r-1)$ and its decoration can be described as follows. This description in particular provides a full list of the cluster and frozen variables associated with the cycles in $\ww_1$.
	\begin{enumerate}[wide, labelwidth=!, labelindent=0pt]
		\item[\textbf{Case 1:}] Assume that $\str{r}$ is of type $H$.
		\begin{enumerate}[wide, labelwidth=!, labelindent=0pt]
		\item Suppose $\sigma(p+r) = 1$. Then we have $\sigma(p) = \sigma(p+1) = \dots = \sigma(p+r) = 1$. And $\beta(r-1) = T_1T_2\cdots T_{m_1-r}T_{m_1-r+1}\cdots T_m$ where for $j\in [1, m_1-r+1]$, $T_j = I_{r}^{d-1}$ (resp., $T_j = I_{d-1}^r$) if $\sigma(p+r-2+j) = 1$ (resp., $\sigma(p+r-2+j) = -1$); and for $j \in [m_1-r+1, m_1]$, $T_j = I_{m_1+1-j}^{d-1}$. The decoration for $\beta(r-1)$ is
		\begin{multline*}
			\mcf_p \rel{I_{r}^{d-1}} \lpush{\mcf_p}{r-1}{\mcf_{p+r}} \rel{T_2} \lpush{\mcf_p}{r-1}{\mcf_{p+r+1}} \rel{T_3} \lpush{\mcf_p}{r-1}{\mcf_{p+r+2}} \rel{T_4}\\ \cdots \rel{T_{m_1-r+1}} \lpush{\mcf_p}{r-1}{\mcf_{q}} \stackrel{I_{r-1}^{d-1}\cdots I_{2}^{d-1}I_1^{d-1}}{\leftarrow \joinrel \xrightarrow{\hspace*{2cm}}}  \mcf_q.
		\end{multline*}
		 The cluster or frozen variable associated with the cycle $\gamma(r, i)$ is 
		\[
		\Delta_{\gamma(r, i)} = \mcf_p^r \wedge \mcf_{p+r+i}^{d-r}, \quad \text{for } i\in [1, m_1-r].
		\]
		
		\item Suppose $\sigma(p+r) = -1$. Then we have $\sigma(p) = \sigma(p+1) = \dots = \sigma(p+r) = -1$, and $\beta(r-1) = T_1T_2\cdots T_{m_1-r}T_{m_1-r+1}\cdots T_m$ where for $j\in [1, m_1-r+1]$, $T_j = I_{r}^{d-1}$ (resp., $T_j = I_{d-1}^r$) if $\sigma(p+r-2+j) = 1$ (resp., $\sigma(p+r-2+j) = -1$); and for $j \in [m_1-r+1, m_1]$, $T_j = I^{m_1+1-j}_{d-1}$. The decoration for $\beta(r-1)$ is
		\begin{multline*}
			\mcf_p \rel{I^{1}_{d-r}} \rpush{\mcf_{p+r}}{\mcf_p} {d-(r-1)}\rel{T_2} \rpush{\mcf_{p+r+1}}{\mcf_p}{d-(r-1)} \rel{T_3} \rpush{\mcf_{p+r+2}}{\mcf_p}{d-(r-1)} \rel{T_4}\\ \cdots \rel{T_{m_1-r+1}} \rpush{\mcf_{q}}{\mcf_p}{d-(r-1)} \stackrel{I_{d-r+1}^{1}\cdots I_{d-2}^{1}I_{d-1}^{1}}{\leftarrow \joinrel \xrightarrow{\hspace*{2cm}}}  \mcf_q.
		\end{multline*}
		The cluster or frozen variable associated with the cycle $\gamma(r, i)$ is 
		\[
		\Delta_{\gamma(r, i)} = \mcf_p^{d-r} \wedge \mcf_{p+r+i}^{r}, \quad \text{for } i\in [1, m_1-r].
		\]
	\end{enumerate}

		\item[\textbf{Case 2:}] Assume that $\str r$ is of type $Y$ or $X$. 
		\begin{enumerate}[wide, labelwidth=!, labelindent=0pt]
			\item Suppose $\sigma(p+r) = -1$. 
			Let $s_1 \in [0, r-1]$ so that $p+s_1$ is the closest black vertex to the left of $p+r$; let $r_1\in[-1, s_1-1]$ so that $p+r_1$ is the closest white vertex to the left of $p+s_1$ (if all vertices in $[p, p+s_1]$ are black, then we define $r_1 = -1$). 
		If $\sigma(p+r+1) = 1$ (resp., $\sigma(p+r+1) = -1$), then $\str r$ is of type $X$ (resp., $Y$). Then $\beta(r-1) = T_1T_2\cdots T_{m_1-r}T_{m_1-r+1}\cdots T_m$ where $T_1 = I_{x}^{d-(y+1)}$; for $j\in [2, m_1-r+1]$, $T_j = I_{x}^{d-(y+1)}$ (resp., $T_j = I_{d-(y+1)}^{x}$) if $\sigma(p+r-2+j) = 1$ (resp., $\sigma(p+r-2+j) = -1$); and $T_{m-r+1}\cdots T_m$ is a nested word. The decoration for $\beta(r-1)$ is  
		\begin{multline*}
			\mcf_p \rel{I_{x}^{d-(y+1)}} \mch \rel{I_{d-(y+1)}^{x}} \lrpush{\mcf_p}{x-1}{\mcf_{p+r+1}}{d-y} \rel{T_3} \lrpush{\mcf_p}{x-1}{\mcf_{p+r+2}}{d-y} \rel{T_4} \\\cdots
			\rel{T_{m_1-r+1}} \lrpush{\mcf_p}{x-1}{\mcf_{q}}{d-y} \stackrel{T_{m_1-r+2}\cdots T_{m_1}}{\leftarrow \joinrel \xrightarrow{\hspace*{1.5cm}}} \mcf_q,
		\end{multline*}
		where $\mch$ is the decorated flag 
		\begin{multline*}
			(\mcf_p^1, \dots, \mcf_p^{x-1}, \\
			 \frac{\mcf_p^{x-1}\mixwed \mcf_p^{d-(y-(r-s_1-1))} \mixwed \mcf_{p+s_1+1}^{y-(r-s_1-1)+1}}{\mu}, \dots, 
			\frac{\mcf_p^{x-1}\mixwed \mcf_p^{d-(y-(r-s_1-1))}\mixwed \mcf_{p+s_1+1}^{d-x-(r-s_1)+1}}{\mu},\\
			 \mcf_p^{d-y}, \dots, \mcf_{p}^{d-1}),
		\end{multline*}
	and $\mu = 1$ if $x = s_1 - r_1$ and $\mu = \mcb{p}{x}\wedge \mcf_{p+s_1+1}^{d-x}$ otherwise. 
	The cluster or frozen variables associated with the cycles are: $\Delta_{\gamma(r, 1)} = \mcb{p}{x} \wedge \mcf_{p+r+1}^{d-x}$ and for $i\in [2, m_1-r]$, 
	\[
	\Delta_{\gamma(r, i)} = \begin{cases}
		\mcf_p^x \wedge \mcf_{p+r+i}^{d-x}, &\quad \text{if } \sigma(p+r+1) = 1 \ (\text{type X}\,),\\
		\mcf_p^{d-(y+1)} \wedge \mcf_{p+r+i}^{y+1}, &\quad \text{if } \sigma(p+r+1) = -1 \ (\text{type Y}\,).
	\end{cases}
	\]
	
	\item Suppose $\sigma(p+r) = 1$. 
	Let $r_1 \in [0, r-1]$ so that $p+r_1$ is the closest white vertex to the left of $p+r$; let $s_1\in[-1, r_1-1]$ so that $p+s_1$ is the closest black vertex to the left of $p+r_1$ (if all vertices in $[p, p+r_1]$ are  white, then we define $s_1 = -1$). 
	If $\sigma(p+r+1) = -1$ (resp., $\sigma(p+r+1) = 1$), then $\str r$ is of type $X$ (resp., $Y$). Then $\beta(r-1) = T_1T_2\cdots T_{m_1-r}T_{m_1-r+1}\cdots T_{m_1}$ where $T_1 = I_{d-y}^{x+1}$; for $j\in [2, m_1-r+1]$, $T_j = I_{x+1}^{d-y}$ (resp., $T_j = I_{d-y}^{x+1}$) if $\sigma(p+r-2+j) = 1$ (resp., $\sigma(p+r-2+j) = -1$); and $T_{m_1-r+1}\cdots T_{m_1}$ is a nested word. The decoration for $\beta(r-1)$ is  
	\begin{multline*}
		\mcf_p \rel{I_{d-y}^{x+1}} \mch \rel{I_{x+1}^{d-y}} \lrpush{\mcf_p}{x}{\mcf_{p+r+1}}{d-(y-1)} \rel{T_3} \lrpush{\mcf_p}{x}{\mcf_{p+r+2}}{d-(y-1)} \rel{T_4} \\\cdots
		\rel{T_{m_1-r+1}} \lrpush{\mcf_p}{x}{\mcf_{q}}{d-(y-1)} \stackrel{T_{m_1-r+2}\cdots T_{m_1}}{\leftarrow \joinrel \xrightarrow{\hspace*{1.5cm}}} \mcf_q,
	\end{multline*}
	where $\mch$ is the decorated flag 
	\begin{multline*}
		(\mcf_p^1, \dots, \mcf_p^{x}, \\
		\frac{\mcf_p^{x-(r-r_1-1)}\mixwed \mcf_p^{d-(y-1)} \mixwed \mcf_{p+r_1+1}^{y+r-r_1-1}}{\mu}, \dots, 
		\frac{\mcf_p^{x-(r-r_1-1)}\mixwed \mcf_p^{d-(y-1)}\mixwed \mcf_{p+r_1+1}^{d-x+(r-r_1-1)-1}}{\mu},\\
		\mcf_p^{d-(y-1)}, \dots, \mcf_{p}^{d-1}),
	\end{multline*}
	and $\mu = 1$ if $y = r_1- s_1$ and $\mu = \mcw{p}{d-y}\wedge \mcf_{p+r_1+1}^{y}$ otherwise. 
	The cluster or frozen variables associated with the cycles are: $\Delta_{\gamma(r, 1)} = \mcw{p}{d-y} \wedge \mcf_{p+r+1}^{y}$ and for $i\in [2, m_1-r]$, 
	\[
	\Delta_{\gamma(r, i)} = \begin{cases}
		\mcf_p^{x+1} \wedge \mcf_{p+r+i}^{d-(x+1)}, &\quad \text{if } \sigma(p+r+1) = 1 \ (\text{type Y}\,),\\
		\mcf_p^{d-y} \wedge \mcf_{p+r+i}^{y}, &\quad \text{if } \sigma(p+r+1) = -1 \ (\text{type X}\,).
	\end{cases}
	\]
\end{enumerate}
	\end{enumerate}
Finally, the frozen variables associated with the marked boundary vertices are
\[
\Delta_{\gamma(0, i)} = \begin{cases}
	\mcf_{p+i-1}^{1} \wedge \mcf_{p+i}^{d-1}, &\quad \text{if } \sigma(p+i-1) = 1;\\
	\mcf_{p+i-1}^{d-1} \wedge \mcf_{p+i}^{1}, &\quad \text{if } \sigma(p+i-1) = -1,
\end{cases}
\] for $i\in [1, m_1]$.

\end{prop}

\begin{proof}
	In order to prove the Proposition, we are required to examine many cases. We will always show details for at least one case that is general enough; the other cases follow in a similar way. 
	
	The proof is by induction on $r$. We are going to work out the base case in details, as it is non-trivial.  The general case will rely on certain arguments explained in the base case. 
	
	\textbf{Base case:} when $r = 1$, we have $\beta(r-1) = \beta(0) = \beta(p, q)$ and the decoration is given by $\mcf_p, \mcf_{p+1}, \cdots, \mcf_{q}$. Let us check this case when $\str r$ is of type $Y$ and $\sigma(p) = -\sigma(p+1)= -\sigma(p+2) = -1$. Then we are in Case 2(b) and we have $x = 0, y = 1, r_1 = 0, s_1 = -1$. Then notice that $\mch = \mcf_{p+1}$ and $\lrpush{\mcf_p}{x}{\mcf_{p+r+i}}{d-(y-1)} = \mcf_{p+r+i}$ for $i\in [1, m_1-r]$. So the decoration of $\beta(0)$ is correct. Now let us compute the cluster and frozen variables The first patch is of type $Y$, and 
	\[
	\Delta_{\gamma(1, 1)} = \frac{{\mcf_p^{1}}\wedge{\mcf_{p+2}^{1}}}{\mcf_{p+1}^{2}} =\frac{\mcf_p^{d-1}\mixwed {\mcf_p^{1}}\wed{\mcf_{p+2}^{1}}}{\mcf_p^{d-1}\mixwed \mcf_{p+1}^{2}}=\frac{\big(\mcw{p}{d-1}\wedge \mcf_{p+2}^1\big)\cdot \mcf_p^1}{\mcf_p^{1}} = \mcf_p^{d-1}\wedge \mcf_{p+2}^{1}.
	\]
	where we used Lemma~\ref{lemma: mcf p+r double push gives mcf p up to a constant} for the third equality. 
	
	For the rest of the cluster and frozen variables, we need to find the decoration for each patch. Let $\dec_{\beta(1)} = \big(\mcf_p = \mcg_{p+1}, \mcg_{p+2}, \cdots, \mcg_{q}, \mcf_q\big)$ be the decoration for the bottom word $\beta(1)$ of the first strip. Then the decoration for the $i$-th patch is 
	\[
	\cdots \mcg_{p+i} \rel{I_1^{d-1}} \mcf_{p+i} \rel{T_{p+i}} \mcf_{p+i+1}, \cdots \rarrow \cdots \mcg_{p+i} \rel{T'_{p+i}} \mcg_{p+i+1} \rel{I_1^{d-1}} \mcf_{p+i+1} \cdots
	\]
	cf.\ Definition~\ref{defn: decoration for a patch notation}, where $T_{p+i} = I_1^{d-1}$ (resp., $T_{p+i} = I_{d-1}^1$) if $\sigma(p+i) =1$ (resp., $\sigma(p+i) = -1$). Then we can show that $\mcg_{p+i+1} = \lpush{\mcf_p}{1}{\mcf_{p+i+1}}$ for $i\ge 2$ by an induction on $i$. Indeed, when $i=2$, patch $2$ is of type $H$, notice that $\mcg_{p+i}^1 = \mcf_p^1$, then by Lemma~\ref{lemma: decorations for normal H and X-patches}, we have 
	\[
	\mcg_{p+i+1}^{k+1} = \mcg_{p+i}^1\wedge \mcf_{p+i+1}^k = \mcf_p^1\wedge \mcf_{p+i+1}^k\quad \text{for } k\in [1, d-2], \ i\ge 2.
	\] 
	Now assume that the formula is true for $i-1$. In particular, we notice that $\mcg_{p+i}^{k} = \mcf_p^1\wedge \mcf_{p+i}^{k-1}$ for $k \in [1, d-2]$, this implies that the formulas for $X$-patch and $H$-patches coincide, as in Lemma~\ref{lemma: decorations for normal H and X-patches}. Therefore, the same calculation as in the base case $i = 2$ yields the result that $\mcg_{p+i+1} = \lpush{\mcf_p}{1}{\mcf_{p+i+1}}$ for $i\ge 2$.
	
	Now for any $i\in [2, m_1-r]$, if the $i$-th patch is of type $H$, then
	\[
	\Delta_{\gamma(1, i)} = \frac{\qwed{\mcg_{p+i+1}^{d-1}}{\mcf_{p+i+1}^{d-1}}{\mcf_{p+i+1}^{d-2}}}{1} =  \qwed{\big(\mcf_p^1\wedge \mcf_{p+i+1}^{d-2}\big)}{\mcf_{p+i+1}^{d-1}}{\mcf_{p+i+1}^{d-2}}  = \mcf_p^1 \wedge \mcf_{p+i+1}^{d-1};
	\]
	if the $i$-th patch is of type $X$ (then $i\ge 3$), then 
	\[
	\Delta_{\gamma(1, i)} = \frac{\qwed{\mcg_{p+i}^{d-1}}{\mcf_{p+i+1}^{d-1}}{\mcf_{p+i}^{d-2}}}{1} =  \qwed{\big(\mcf_p^1\wedge \mcf_{p+i}^{d-2}\big)}{\mcf_{p+i+1}^{d-1}}{\mcf_{p+i}^{d-2}}  = \mcf_p^1 \wedge \mcf_{p+i+1}^{d-1}.
	\]
	All these expressions match the description as in the Proposition. 
	
	Now assume that the proposition is true for the first $(r-1)$-th strips; and the decoration for $\beta(r-1)$ (top word of the $\str r$) is also true. We will show that the decoration for the bottom word $\beta(r)$ of the $r$-th strip matches with our descriptions in the proposition; and the cluster  and frozen variables in the $r$-th strip also match with our description in the Proposition. We will illustrate the proof when $\str r$ is of type $X$ and assume that $\sigma(p+r)  = -\sigma(p+r+1) = -1$. We are in Case 2(a), and the decoration for $\beta(r-1) = T_1T_2\cdots T_{m_1-r}T_{m_1-r+1}\cdots T_{m_1}$ is given by 
	\begin{multline*}
		\mcf_p \rel{I_{x}^{d-(y+1)}} \mch \rel{I_{d-(y+1)}^{x}} \lrpush{\mcf_p}{x-1}{\mcf_{p+r+1}}{d-y} \rel{T_3} \lrpush{\mcf_p}{x-1}{\mcf_{p+r+2}}{d-y} \rel{T_4} \\\cdots
		\rel{T_{m_1-r+1}} \lrpush{\mcf_p}{x-1}{\mcf_{q}}{d-y} \stackrel{T_{m_1-r+2}\cdots T_{m_1}}{\leftarrow \joinrel \xrightarrow{\hspace*{1.5cm}}} \mcf_q,
	\end{multline*}
	where $\mch$ is the decorated flag 
	\begin{multline*}
		(\mcf_p^1, \dots, \mcf_p^{x-1}, \\
		\frac{\mcf_p^{x-1}\mixwed \mcf_p^{d-(y-(r-s_1-1))} \mixwed \mcf_{p+s_1+1}^{y-(r-s_1-1)+1}}{\mu}, \dots, 
		\frac{\mcf_p^{x-1}\mixwed \mcf_p^{d-(y-(r-s_1-1))}\mixwed \mcf_{p+s_1+1}^{d-x-(r-s_1)+1}}{\mu},\\
		\mcf_p^{d-y}, \dots, \mcf_{p}^{d-1}),
	\end{multline*}
	and $\mu = 1$ when $x = s_1 - r_1$ and $\mu = \mcb{p}{x}\wedge \mcf_{p+s_1+1}^{d-x}$ otherwise. Let us examine the case $x > s_1 - r_1$, i.e., there are more than one type $X$ strip before $\str r$. Then $\mu = \mcb{p}{x}\wedge \mcf_{p+s_1+1}^{d-x}$. Let $\big(\mcf_p = \mcf_{p+r}, \mcg_{p+r+1}, \cdots, \mcg_{q}, \cdots\big)$ be the decoration for~ $\beta(r)$. 
	
	Let us first calculate $\mcg_{p+r+1}$. Consider the first patch of $\str r$, it is of type $X$. All of its 6-valent vertices are normal except for the last one.
	Then by Lemma~\ref{lemma: decorations for normal H and X-patches} (3), we get 
	\[
	\mcg_{p+r+1}^{k+1} = \qwed{\mcf_p^k}{\big(\lrpush{\mcf_p}{x-1}{\mcf_{p+r+1}}{d-y}\big)^{k}}{\mch^{k-1}}, \quad \text{for } k \in [x+1, d-(y+1)-1].
	\]
	As for $k = x$, we notice that by Proposition~\ref{prop: descriptions of cycles of the initial weave}, there is one cycle, namely the cycle associated with the first patch of the previous $X$ strip, will merge at the last 6-valent vertex. Since the previous $X$ strip is $\str{s_1}$, then by induction, the cluster variable associated with the cycle $\gamma(s_1, 1)$ is (Case 2(b) for $\str{s_1}$):
	\[
	\Delta_{\gamma(s_1, 1)} = \mcw{p}{d-(y-(r-s_1-1))} \wedge \mcf_{p+s_1+1}^{y-(r-s_1-1)} = \mcw{p}{d-(y+1)} \wedge \mcf_{p+r+1}^{y+1}.
	\]
	Here the second equality follows from the fact that there are exactly $y-(r-s_1-1)$ white vertices from $p$ to $p+s_1$; and $p+s_1+1, \cdots, p+r$ are all white vertices. Then by adjusting the formula in~\ref{lemma: decorations for normal H and X-patches} (3) by  (dividing) a factor of $\Delta_{\gamma(s_1, 1)}$, we get
	\begin{align*}
	\mcg_{p+r+1}^{x+1} &=\frac{\qwed{\mcf_p^x}{\big(\lrpush{\mcf_p}{x-1}{\mcf_{p+r+1}}{d-y}\big)^{x}}{\mch^{x-1}}}{\mcw{p}{d-(y+1)} \wedge \mcf_{p+r+1}^{y+1}}\\
	&= \frac{\qwed{\mcf_p^x}{\big(\mcf_p^{x-1}\mixwed \mcf_p^{d-y}\mixwed \mcf_{p+r+1}^{y+1}\big)}{\mcf_p^{x-1}}}{\mcw{p}{d-(y+1)} \wedge \mcf_{p+r+1}^{y+1}} =\frac{{\mcf_p^{x}\mixwed \mcf_p^{d-y}\mixwed \mcf_{p+r+1}^{y+1}}}{\mcw{p}{d-(y+1)} \wedge \mcf_{p+r+1}^{y+1}}.
	\end{align*}
	Now let us simplify the formula for $\mcg_{p+r+1}^{k+1}$, for $k\in [x+1, d-(y+1)-1]$. By definition, we have 
	\[
	\mch^{k-1}  = 	\frac{\mcf_p^{x-1}\mixwed \mcf_p^{d-(y-(r-s_1-1))} \mixwed \mcf_{p+s_1+1}^{k-x+y-r+s_1+1}}{\mcb{p}{x}\wedge \mcf_{p+s_1+1}^{d-x}}.
	\]
	Notice that
	\begin{align*}
		\qwed{\mcf_p^x}{\mch^{k-1}}{\mcf_p^{x-1}} &= \qwed{\mcf_p^x}{\frac{\mcf_p^{x-1}\mixwed \mcf_p^{d-(y-(r-s_1-1))} \mixwed \mcf_{p+s_1+1}^{k-x+y-r+s_1+1}}{\mcb{p}{x}\wedge \mcf_{p+s_1+1}^{d-x}}}{\mcf_p^{x-1}} \\
		&= \frac{\mcf_p^{x}\mixwed \mcf_p^{d-(y-(r-s_1-1))} \mixwed \mcf_{p+s_1+1}^{k-x+y-r+s_1+1}}{\mcb{p}{x}\wedge \mcf_{p+s_1+1}^{d-x}}\\
		&= \frac{\big(\mcb{p}{x}\wedge \mcf_{p+s_1+1}^{d-x}\big)\big(\mcw{p}{d-(y-(r-s_1-1))} \wedge \mcf_{p+s_1+1}^{y-(r-s_1-1)}\big)\cdot \mcf_p^k}{\mcb{p}{x}\wedge \mcf_{p+s_1+1}^{d-x}}\\
		& = \big(\mcw{p}{d-(y+1)} \wedge \mcf_{p+r+1}^{y+1}\big)\cdot \mcf_p^k;
	\end{align*}
	where the third equality follows from Lemma~\ref{lemma: mcf p+r double push gives mcf p up to a constant} (3) (notice that $x>r-s_1$ guarantees that we are in this case). Therefore
	\begin{align*}
		\mcg_{p+r+1}^{k+1} &= \qwed{\mcf_p^k}{\big(\lrpush{\mcf_p}{x-1}{\mcf_{p+r+1}}{d-y}\big)^{k}}{\mch^{k-1}}\\
		&= \qwed{\frac{\qwed{\mcf_p^x}{\mch^{k-1}}{\mcf_p^{x-1}}}{\mcw{p}{d-(y+1)} \wedge \mcf_{p+r+1}^{y+1}}}{\big(\mcf_p^{x-1}\mixwed \mcf_{p}^{d-y}\mixwed \mcf_{p+r+1}^{k-x+y+1}\big)}{\mch^{k-1}}\\
			&= \frac{\mcf_p^{x}\mixwed \mcf_{p}^{d-y}\mixwed \mcf_{p+r+1}^{k-x+y+1}}{\mcw{p}{d-(y+1)} \wedge \mcf_{p+r+1}^{y+1}}.
	\end{align*}
	Therefore, we can conclude that
	\[
	\mcg_{p+r+1} = \lrpush{\mcf_p}{x}{\frac{\mcf_{p+r+1}}{\mu'}}{d-y}, \quad \text{where } \mu' = \mcw{p}{d-(y+1)} \wedge \mcf_{p+r+1}^{y+1}.
	\]
	Then similar to the base case, we can show that 
	\[
	\mcg_{p+r+i} = \lrpush{\mcf_p}{x}{\mcf_{p+r+i}}{d-y}, \quad \text{for } i \in [2, m_1-r],
	\] by an induction on $i$. It's easy to check that the decoration we just computed matches with the description for $\beta(r)$ in Case 2(b). 
	
	Lastly, let us calculate the cluster and frozen variables. For $\gamma(r, 1)$, we have 
	\begin{align*}
	\Delta_{\gamma(r, 1)} &= \frac{\qwedd{\mcf_p^{d-(y+1)}}{\big(\lrpush{\mcf_p}{x-1}{\mcf_{p+r+1}}{d-y}\big)^{d-(y+1)}}{\mcf_p^{d-y}}}{\mch^{d-y-2}}\\
	&= \frac{\mcf_{p}^{x-1}\mixwed \mcf_p^{d-(y+1)} \mixwed \mcf_{p+r+1}^{d-x}}{\mcf_p^{x-1}\mixwed \mcf_p^{d-(y-(r-s_1-1))} \mixwed \mcf_{p+s_1+1}^{d-x-r+s_1}}\cdot \big(\mcb{p}{x}\wedge \mcf_{p+s_1+1}^{d-x}\big)\\
	&= \frac{\qwedd{\mcf_p^x}{\big(\mcf_{p}^{x-1}\mixwed \mcf_p^{d-(y+1)} \mixwed \mcf_{p+r+1}^{d-x}\big)}{\mcf_p^{x-1}}}{\qwedd{\mcf_p^x}{\big(\mcf_p^{x-1}\mixwed \mcf_p^{d-(y-(r-s_1-1))} \mixwed \mcf_{p+s_1+1}^{d-x-r+s_1}\big)}{\mcf_p^{x-1}}}\cdot\big(\mcb{p}{x}\wedge \mcf_{p+s_1+1}^{d-x}\big)\\
	&= \frac{\mcf_{p}^{x}\mixwed \mcf_p^{d-(y+1)} \mixwed \mcf_{p+r+1}^{d-x}}{\mcf_p^{x}\mixwed \mcf_p^{d-(y-(r-s_1-1))} \mixwed \mcf_{p+s_1+1}^{d-x-r+s_1}}\cdot\big(\mcb{p}{x}\wedge \mcf_{p+s_1+1}^{d-x}\big)\\
	&= \frac{\big(\mcb{p}{x}\wedge \mcf_{p+r+1}^{d-x}\big)\big(\mcw{p}{d-(y+1)}\wedge \mcf_{p+r+1}^{y+1}\big)\cdot \mcf_p^{d-y-2}}{\big(\mcb{p}{x}\wedge \mcf_{p+s_1+1}^{d-x}\big)\big(\mcw{p}{y-(r-s_1-1)}\wedge \mcf_{p+s_1+1}^{r-s_1-1}\big)\cdot \mcf_p^{d-y-2}}\cdot\big(\mcb{p}{x}\wedge \mcf_{p+s_1+1}^{d-x}\big)\\
	&= \mcb{p}{x}\wedge \mcf_{p+r+1}^{d-x}.
	\end{align*}
	Here in the fifth equality, we applied Lemma~\ref{lemma: mcf p+r double push gives mcf p up to a constant} (3) to both the numerator and denominator. 
	
	For $i\in [2, m_1-r]$, if the $i$-th patch is of type $H$, then we have
	\begin{align*}
	\Delta_{\gamma(r, i)} &= \frac{\qwed{\mcf_p^x}{\big(\lrpush{\mcf_p}{x-1}{\mcf_{p+r+i}}{d-y}\big)^{d-(y+1)}}{\mcf_p^{x-1}}}{\mcf_p^{d-y}}\\
	&= \frac{\mcf_p^x \mixwed \mcf_p^{d-y}\mixwed \mcf_{p+r+i}^{d-x}}{\mcf_p^{d-y}}= \frac{\mcf_p^{d-y}\mixwed \big(\mcf_p^x \mixwed \mcf_{p+r+i}^{d-x}\big)}{\mcf_p^{d-y}} = \mcf_p^x \mixwed \mcf_{p+r+i}^{d-x}.
	\end{align*}
	If the $i$-th patch is of type $X$, then $i\ge 3$. Notice that for $k\in [x, d-y]$, we have
	\[
	\mcg_{p+r+i-1}^k =  \qwed{\mcg_{p+r+i-1}^x}{\big(\lrpush{\mcf_p}{x-1}{\mcf_{p+r+i-1}}{d-y}\big)^{k-1}}{\mcg_{p+r+i-1}^{x-1}},
	\]
	so the conditions in Lemma~\ref{lemma: decorations for normal H and X-patches} (3) are satisfied, hence the expression for $\Delta_{\gamma(r, i)}$ will be the same as the case of a $H$-patch. To conclude, we get the same formulas for the cluster and frozen variables as in the proposition. 
	
	Finally, we notice that the frozen variables for the cycles originating from the marked boundary vertices can be calculated directly using the decoration for the top word $\beta_1 = \beta(p, q)$. They are clearly of the form described in the proposition. 
\end{proof}
	
	\begin{cor}\label{cor: seed lies in R sigma}
		The cluster and frozen variables for $\ww_1$ belong to $R_\sigma$. 
	\end{cor}
	\begin{proof}
		By Proposition~\ref{prop: decoration flags of the initial weave}, every cluster or frozen variable is of one the following three forms:\[
		\mcf_{p'}^x\wedge \mcf_{q'}^{d-x}; \quad \mcb{p}{x}\wedge \mcf_{q'}^{d-x}; \quad \text{or} \quad \mcw{p}{d-y}\wedge \mcf_{q'}^y,
		\]
		where $p\le p'\le q' \le q$ and $x, y \in[1, d-1]$. Recall from Definition~\ref{defn: tuple of flags mcf associated with sigma and u} that $\mcf_{p'}^x$ is defined as a consecutive mixed wedge of $u_i$'s. As a consequence, it will be a linear combination of wedges of $u_i$'s or dual wedges of $u^*_i$'s, with coefficients being polynomials in Weyl generators (determinants, dual determinants, and pairings). In particular, these coefficients lie in $R_\sigma$. Wedging two or more objects of these kinds will result in an extensor again in this form. If the result extensor is in $\bigwedge^d V = \bigwedge^0 V$, then we get a $R_\sigma$-linear combinations of determinants in $u_i$'s, dual determinants in $u_i^*$'s or pairings; in particular, it lies $R_\sigma$. 
	\end{proof}

	Similar results hold for the other initial weave $\ww_2 = \ww(q, p+n)$. We next apply the construction from Definition~\ref{defn: cluster algebra cutting at p, q} to the initial Demazure weaves $\ww_1$ and $\ww_2$. As before, we use the notation $\seed_i = \seed(\ww_i)$ for $i = 1, 2$, and denote by  $\seed_\sigma(p, q)$ the seed obtained by amalgamating $\seed_1$ and $\seed_2$ along their common frozen variables 
	\[\zz_0 = \{\mcf_p^1\wedge \mcf_q^{d-1}, \mcf_p^2\wedge \mcf_q^{d-2}, \dots, \mcf_p^{d-1}\wedge \mcf_q^1\};\]
	and $\mca_\sigma(p, q)$ the cluster algebra associated with $\seed_\sigma(p, q)$. We will call this special seed $\seed_\sigma(p, q)$ \emph{the initial seed} for $\mca_\sigma(p, q)$. We have the following results.

	\begin{cor}\label{cor: dimension for mca sigma}
		The cardinalities of the cluster and the extended cluster for $\seed_\sigma(p, q)$ are $(d-1)(n-d-1)$ and  $d(n-d)+1$ respectively. 
	\end{cor}

\begin{proof}
	By Corollary~\ref{cor: size of the cluster algebra for half of the weave}, the sizes of the cluster and the extended cluster for $\seed_1$ are  $(m_1 - 1)(d-1) - \frac{d(d-1)}{2}$ and $m_1d-\frac{d(d-1)}{2}$ respectively. Similarly, the size of the cluster and the extended cluster for $\seed_2$ are $(m_2 - 1)(d-1) - \frac{d(d-1)}{2}$ and $m_2d-\frac{d(d-1)}{2}$ respectively. Since $\seed_\sigma(p, q)$ is the seed obtained by amalgamating $\seed_1$ and $\seed_2$ along $d-1$ frozen variables, we conclude that the size of the cluster and the extended cluster for $\seed_\sigma(p, q)$ are $(d-1)(n-d-1)$ and $d(n-d)+1$ respectively (notice that $m_1 + m_2 = n$).
\end{proof}

	\begin{cor}
		The cluster and frozen variables in the initial seed $\seed_\sigma(p, q)$ for $\mca_\sigma(p, q)$ belong to $R_\sigma$.
	\end{cor}
	\begin{proof}
		This follows from Corollary~\ref{cor: seed lies in R sigma}.
	\end{proof}

\subsection{Quivers and mutations}\label{sec: quivers and mutations}

In this section, we study the quiver associated with an initial weave and derive all exchange relations for the initial seed. Furthermore, we reduce the once-mutated cluster variables to their ``simplest" form, demonstrating that they all belong to $R_\sigma$. These results will play a critical role in the second proof of the main theorem, where we invoke the ``Starfish Lemma".


\begin{defn}
	To further characterize strip types, we classify them based on whether they involve a black vertex or a white vertex. Formally, we define the \emph{$BW$ type} of a strip as follows. Recall from Definition \ref{defn: strips} that a strip is a partial weave with both top and bottom word weakly nested. 
	\begin{itemize}[wide, labelwidth=!, labelindent=0pt]
		\item An $X$-strip is said to be \emph{of type $XB$} (resp., $XW$) if the first interval word for the top weakly nested word is increasing (resp., decreasing);
		\item an $H$-strip is said to be \emph{of type $HB$} (resp., $HW$) if the first interval word for the top weakly nested word is increasing (resp., decreasing);
		\item an $Y$-strip is said to be \emph{of type $YB$} (resp., $YW$) if the second interval word for the top weakly nested word is increasing (resp., decreasing). 
	\end{itemize}
	Furthermore, we say that a strip is of $BW$-type  $B$ (resp., $W$) if it is of type $XB$, $HB$ or $YB$ (resp., $XW$, $HW$, $YW$). 
	As a convention, we will treat a degenerate strip (cf.\ Definition~\ref{defn: strips}) as of the same $BW$-type as any strip. 
\end{defn}

This classification of $BW$-types has a concrete interpretation for the initial weave, as the following remark demonstrates. Our analysis will primarily focus on the initial Demazure weave $\ww_1$, with analogous results holding for $\ww_2$ due to their structural symmetry. 

\begin{remk}
	For an initial weave $\ww_1$, the $BW$-type of the $r$-th strip ($r\in[1, d]$)   is determined by the values of $\sigma(p+r)$ and $\sigma(p+r+1)$ as follows:
	\begin{itemize}[wide, labelwidth=!, labelindent=0pt]
		\item suppose that the $r$-th strip is a $Y$-strip or an $X$-strip. Then it is of $BW$-type $B$ (resp., $W$) if $\sigma(p+r+1) = 1$ (resp., $\sigma(p+r+1) = -1$); 
		\item suppose that the $r$-th strip is an $H$-strip. Then it is of $BW$-type $B$ (resp., $W$) if $\sigma(p+r) = 1$ (resp., $\sigma(p+r) = -1$). 
	\end{itemize}
\end{remk}

We will use the same notation as in Section~\ref{sec: initial weave}. Recall that for $r\in [1, d-1]$ and $i\in [1, m_1-r]$, we denote by $\gamma(r, i)$ the cycle associated with the $i$-th patch in $\str{r}$. For $i\in [1, m_1]$,  we denote by $\gamma(0, i)$  the $i$-th frozen cycles associated with the marked boundary vertices.

Our goal is to derive all exchange relations. To achieve this, we will systematically characterize the \emph{local quiver structure} around each mutable vertex. By a \emph{local picture} at a vertex $v$, we mean the subquiver containing $v$ and all vertices directly connected to it. We first address the generic case, where cycles do not originate from the first patch of a strip, as these configurations exhibit simpler combinatorial behavior.


\begin{prop}\label{prop: description of local pictures, i not 1}
	Let $r\in[1, d-2]$ and $i\in [2, m_1-r-1]$. Then the local picture of $Q(\ww_1)$ at $\gamma(r, i)$ is described as follows. 
	\begin{enumerate}[wide, labelwidth=!, labelindent=0pt]
		\item[({HX})] Suppose that the patches $i$ and $i+1$ of $\str r$ are of types $H$ and $X$ respectively. Then the local picture at $\gamma(r, i)$ is as shown in Figure~\ref{fig: local picture HX i not 1}.
		\begin{figure}[H]
			\centering
			\begin{tikzcd}
				\gamma(r'', 1) \arrow[ddrr, dashed, bend left = 30] &&&\\
				&\gamma(r',i'-1) \arrow[d] & &\gamma(r', i'+1) \arrow[dl]\\
				&\gamma(r, i-1) \arrow[r] & \gamma(r, i) \arrow[ul] \arrow[r] & \gamma(r, i+1) \arrow[u].
			\end{tikzcd}
			\caption{Local picture at $\gamma(r,i)$, case $(HX)$. Here either $r'\in[1, r-1]$ is such that $\str{r'}$ is the closest strip above $\str r$ that is of the same $BW$-type as $\str r$; or $r' = 0$ if no such strip exists. In the latter case, we need to remove $\gamma(r', i'+1)$ from the local picture. In either case, $i'$ is chosen such that $r'+i' = r+i$. The dashed arrow appears only when $i = 2$ and $\str r$ is of type $Y$; here $r''\in [0, r-1]$ is such that $\str{r''}$ is the closest $X$-strip above $\str r$.}
			\label{fig: local picture HX i not 1}
		\end{figure}
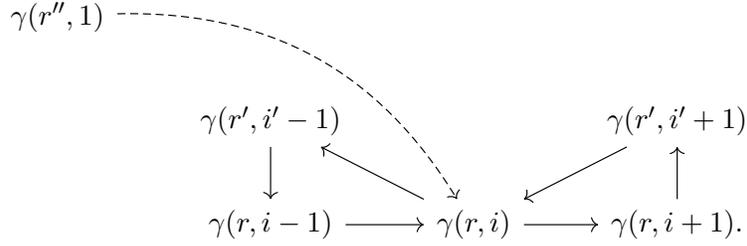
		
		\item[(XH)] Suppose that the patches $i$ and $i+1$ of $\str r$ are of types $X$ and $H$ respectively. Then the local picture at $\gamma(r, i)$ is as shown in Figure~\ref{fig: local picture XH i not 1}.
		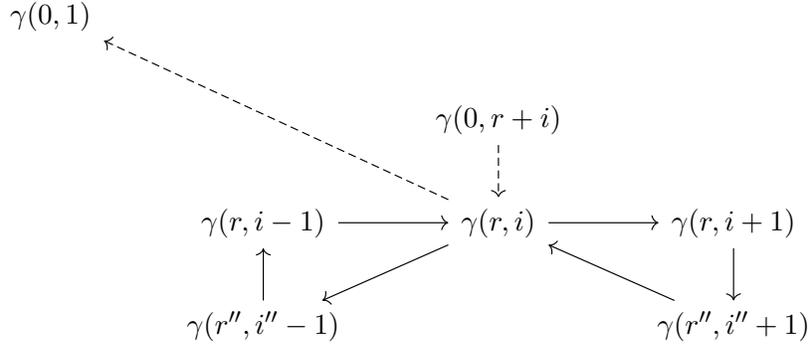
\begin{figure}[H]
			\centering
			\begin{tikzcd}
			\gamma(0, 1) &&&\\
			&	&  \gamma(0, r+i) \arrow[d, dashed]&\\
			&	\gamma(r, i-1) \arrow[r] & \gamma(r, i) \arrow[r] \arrow[dl] \arrow[uull, dashed] & \gamma(r, i+1) \arrow[d]\\
			&	\gamma(r'', i''-1)\arrow[u]  & & \gamma(r'', i''+1) \arrow[ul]
			\end{tikzcd}
		\caption{Local picture at $\gamma(r, i)$, case $(XH)$. Here $r''\in[r+1, d-1]$ is such that $\str{r''}$ is the closest strip below $\str r$ that is of the same $BW$-type as $\str r$; and $i''$ is chosen so that $r''+i'' = r+i$. The dashed arrow from $\gamma(0, r+i)$ appears only when $\str r$ is the first strip of its $BW$ type. The dashed arrow to $\gamma(0, 1)$ appears only when $r = 2$ and $\str r$ is of type $H$; in that case, $\gamma(r'', i''-1)$ should be removed from the picture.}
		\label{fig: local picture XH i not 1}
		\end{figure}

		\item[(HH)] Suppose that both patches $i$ and $i+1$ of $\str r$ are of type $H$. Then the local picture at $\gamma(r, i)$ is shown in Figure~\ref{fig: local picture HH i not 1}.
		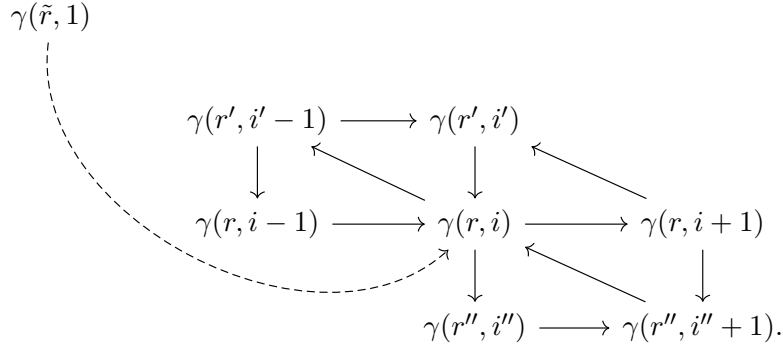
\begin{figure}[H]
			\centering
			\begin{tikzcd}
				\gamma(\tilde{r}, 1) \arrow[ddrr, dashed, bend right = 70] &&&\\
			&	\gamma(r',i'-1) \arrow[d] \arrow[r] & \gamma(r', i') \arrow[d] &\\
			&	\gamma(r, i-1) \arrow[r] & \gamma(r, i) \arrow[ul] \arrow[d] \arrow[r] & \gamma(r, i+1) \arrow[ul] \arrow[d]\\
			&	& \gamma(r'', i'')\arrow[r] & \gamma(r'',i''+1)\arrow[ul].
			\end{tikzcd}
			\caption{Local picture at $\gamma(r,i)$, case $(HH)$. Here $r'\in[1, r-1]$ (resp., $r''\in [r+1, d-1]$) is such that $\str{r'}$ is the closest strip above (resp., below) $\str r$ that is of the same $BW$-type as $\str r$; and $i', i''$ are chosen so that $r'+i' = r'' + i'' = r+i$. (Similarly, we set $r' = 0$ if $\str r$ is the first strip of its $BW$-type.) The dashed arrow appears only when $i = 2$, $\str r$ is of type $X$, and $\str{\tilde{r}}$ is the closest $X$-strip above $\str r$; in that case, $\gamma(r, i-1)$ should be removed from the picture.}
			\label{fig: local picture HH i not 1}
		\end{figure}

		\item[(XX)] Suppose that both patches $i$ and $i+1$ of $\str r$ are of type $X$. Then the local picture at $\gamma(r, i)$ is as shown in Figure~\ref{fig: local picture XX i not 1}.
		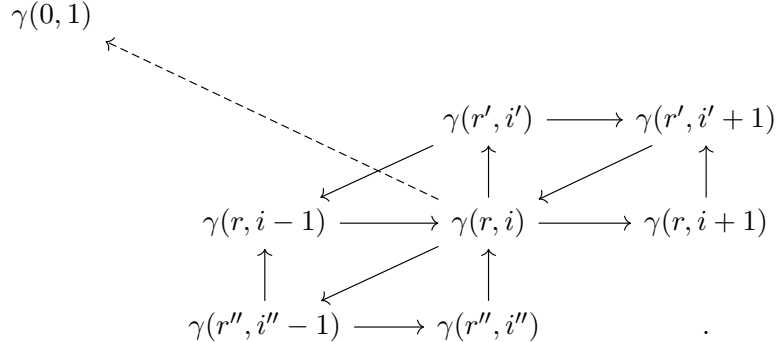
\begin{figure}[H]
			\centering
			\begin{tikzcd}
				\gamma(0, 1)  &&&\\
			&	& \gamma(r',i') \arrow[dl] \arrow[r] & \gamma(r', i'+1) \arrow[dl] \\
			&	\gamma(r, i-1) \arrow[r] & \gamma(r, i) \arrow[uull, dashed] \arrow[u] \arrow[dl] \arrow[r] & \gamma(r, i+1) \arrow[u]\\
			&	 \gamma(r'', i''-1)\arrow[r] \arrow[u] & \gamma(r'',i'')\arrow[u] &.
			\end{tikzcd}
			\caption{Local picture around $\gamma(r,i)$, case $(XX)$. Here $r'\in[1, r-1]$ (resp., $r''\in [r+1, d-1]$) is such that $\str{r'}$ is the closest strip above (resp., below) $\str r$ that is of the same $BW$-type as $\str r$; and $i', i''$ is chosen so that $r'+i' = r'' + i'' = r+i$. (Similarly, we set $r' = 0$ if $\str r$ is the first strip of its $BW$-type.) The dashed arrow appears only when $i = 2$ and $\str{r}$ is of type $H$; in that case, $\gamma(r'', i''-1)$ should be removed from the picture.}
			\label{fig: local picture XX i not 1}
		\end{figure}
\end{enumerate}
\end{prop}

\begin{proof}
	This follows from the proof of Proposition~\ref{prop: descriptions of cycles of the initial weave}, where we have described how these cycles behave; in particular, we know how they intersect with each other. 
\end{proof}

Once we have the local pictures of quivers around the cycles $\gamma(r, i)$ appearing in Proposition~\ref{prop: description of local pictures, i not 1}, we can write down the exchange relations for these cycles. Then using the exchange relations, we can calculate the once-mutated cluster variables. 

\begin{prop}\label{prop: exchange relations, generic case}
	Let $r\in [1, d-1]$ and $i\in [2, m_1-r-1]$. Then 
	\[
	\Delta_{\gamma(r, i)} = \mcf_p^k\wedge \mcf_{t}^{d-k}
	\]
	where $t = p+r+i$ and $k\in [1, d-1]$ (cf.\ Proposition~\ref{prop: decoration flags of the initial weave}). Without loss of generality, we assume that $\str r$ is of $BW$-type $B$. 
	\begin{enumerate}[wide, labelwidth=!, labelindent=0pt]
		\item[(HX)] Suppose that $\sigma(t-1) = -\sigma(t) =1$. Then the exchange relation for $\gamma(r, i)$ is
		\begin{equation}\label{equation: degree 4 exchange relations HX}
		\Delta_{\gamma(r, i)}\Delta'_{\gamma(r, i)} = \big(\mcf_p^k\wedge \mcf_{t-1}^{d-k}\big)\big(\mcf_p^{k-1}\wedge \mcf_{t+1}^{d-(k-1)}\big) + \big(\mcf_p^k\wedge \mcf_{t+1}^{d-k}\big)\big(\mcf_p^{k-1}\wedge \mcf_{t-1}^{d-(k-1)}\big),
		\end{equation}
		where $\Delta'_{\gamma(r, i)}=
			\mcf_p^{k-1}\mixwed \mcf_{t-1}^{1}\mixwed \mcf_{t+1}^{d-k}.$
		\item[(XH)] Suppose that $\sigma(t-1) = -\sigma(t) =-1$. Then the exchange relation for $\gamma(r, i)$ is
		\begin{equation}\label{equation: degree 4 exchange relations XH}
			\Delta_{\gamma(r, i)}\Delta'_{\gamma(r, i)} = \big(\mcf_p^k\wedge \mcf_{t-1}^{d-k}\big)\big(\mcf_p^{k+1}\wedge \mcf_{t+1}^{d-(k+1)}\big) + \big(\mcf_p^k\wedge \mcf_{t+1}^{d-k}\big)\big(\mcf_p^{k+1}\wedge \mcf_{t-1}^{d-(k+1)}\big),
		\end{equation}
		where $\Delta'_{\gamma(r, i)}=
			\mcf_p^{k+1}\mixwed \mcf_{t-1}^{d-1}\mixwed \mcf_{t+1}^{d-k}.$
	\item[(XX)] Suppose that $\sigma(t-1) = \sigma(t) =-1$ ($XX$). 
	Then the exchange relation for $\gamma(r, i)$ is 
	\begin{multline}\label{equation: degree 6 exchange relations XX}
		\Delta_{\gamma(r, i)}\Delta'_{\gamma(r, i)} = \big(\mcf_p^{k}\wedge \mcf_{t-1}^{d-k}\big)\big(\mcf_p^{k+1}\wedge \mcf_t^{d-(k+1)}\big) \big(\mcf_p^{k-1}\wedge \mcf_{t+1}^{d-(k-1)}\big) \\
		+\big(\mcf_p^{k}\wedge \mcf_{t+1}^{d-k}\big)\big(\mcf_p^{k+1}\wedge \mcf_{t-1}^{d-(k+1)}\big) \big(\mcf_p^{k-1}\wedge \mcf_{t}^{d-(k-1)}\big),
	\end{multline}
	where $\Delta'_{\gamma(r, i)} =
	\mcf_p^{k+1} \mixwed \mcf_{t-1}^{d-k} \mixwed \mcf_{t+1}^{d-k} \mixwed\mcf_p^{k-1}$.
	\item[(HH)] Suppose that $\sigma(t-1) = \sigma(t)=1$. If $\str r$ is of type $H$ or $i\ge 3$, then the exchange relation for $\gamma(r, i)$ is 
	\begin{multline}\label{equation: degree 6 exchange relations HH}
		\Delta_{\gamma(r, i)}\Delta'_{\gamma(r, i)} = \big(\mcf_p^{k}\wedge \mcf_{t-1}^{d-k}\big)\big(\mcf_p^{k-1}\wedge \mcf_t^{d-(k-1)}\big) \big(\mcf_p^{k+1}\wedge \mcf_{t+1}^{d-(k+1)}\big) \\
		+\big(\mcf_p^{k}\wedge \mcf_{t+1}^{d-k}\big)\big(\mcf_p^{k-1}\wedge \mcf_{t-1}^{d-(k-1)}\big) \big(\mcf_p^{k+1}\wedge \mcf_{t}^{d-(k+1)}\big),
	\end{multline}
	where $\Delta'_{\gamma(r, i)} =
	\mcf_p^{k-1} \mixwed \mcf_{t-1}^{d-k} \mixwed \mcf_{t+1}^{d-k} \mixwed\mcf_p^{k+1}$.
	\item[(XHH)] Suppose that $i = 2$, $\sigma(t-1) = \sigma(t)$ and $\str r$ is of type $X$.  Then $k = x$ and the exchange relation for $\gamma(r, i)$ is 
		\begin{multline}\label{equation: degree 6 exchange relations, second cycle XB}
			\Delta_{\gamma(r, 2)}\Delta'_{\gamma(r, 2)} = \big(\mcw{p}{d-(y+1)}\wedge \mcf_{t-1}^{y+1}\big)\big(\mcf_p^{x-1}\wedge \mcf_t^{d-(x-1)}\big) \big(\mcf_p^{x+1}\wedge \mcf_{t+1}^{d-(x+1)}\big) \\
			+\big(\mcf_p^{x}\wedge \mcf_{t+1}^{d-x}\big)\big(\mcf_p^{x-1}\wedge \mcf_{t-1}^{d-(x-1)}\big) \big(\mcw{p}{d-(y+1)}\wedge \mcf_t^{y+1}\big),
		\end{multline}
		where 
		\begin{align}
		\Delta'_{\gamma(r, 2)} &=
		\frac{\mcf_p^{x-1} \mixwed \mcf_{t-1}^{d-x} \mixwed  \mcf_{t+1}^{d-x}\mixwed \mcf_p^{x+1}}{\mcb{p}{x}\wedge \mcf_{t-1}^{d-x}}
		\\ 
		&= 
		\begin{cases}
			\mcw{p}{d-(y+1)}\mixwed \mcf_{t-1}^y \mixwed  \mcf_{t+1}^{d-x}\mixwed\mcf_p^{x+1} &  \text{if } \, \mcb{p}{x-1} \neq \mcf_p^{x-1};\\
			\mcw{p}{d-(y+1)}\mixwed \mcf_{t-1}^{y+2} \mixwed\mcf_{t+1}^{d-x} \mixwed \mcb{p}{x-1} &  \text{if } \, \mcb{p}{x-1} = \mcf_p^{x-1}.\\
		\end{cases}
		\end{align}
		\item[(YHH)] Suppose that $i = 2$, $\sigma(t-1) = \sigma(t)$ and $\str r$ is of type $Y$.  Then $k = x+1$ and the exchange relation for $\gamma(r, i)$ is 
		\begin{multline}\label{equation: degree 6 exchange relations, second cycle YB}
			\Delta_{\gamma(r, 2)}\Delta'_{\gamma(r, 2)} = \big(\mcw{p}{d-y}\wedge \mcf_{t-1}^y\big)\big(\mcf_p^{x}\wedge \mcf_t^{d-x}\big) \big(\mcf_p^{x+2}\wedge \mcf_{t+1}^{d-(x+2)}\big) \\
			+\big(\mcf_p^{x+1}\wedge \mcf_{t+1}^{d-(x+1)}\big)\big(\mcf_p^{x}\wedge \mcf_{t-1}^{d-x}\big) \big(\mcw{p}{d-y}\wedge \mcf_{t}^y\big),
		\end{multline}
		where 
		\begin{align}
		\Delta'_{\gamma(r, 2)} &=
		\frac{\mcf_p^{x} \mixwed \mcf_{t-1}^{d-(x+1)} \mixwed\mcf_{t+1}^{d-(x+1)} \mixwed \mcf_p^{x+2} }{\mcb{p}{x+1}\wedge \mcf_{t-1}^{d-(x+1)}}\\
		&= 
		\begin{cases}
			\mcw{p}{d-y}\mixwed \mcf_{t-1}^{y-1} \mixwed \mcf_{t+1}^{d-(x+1)}\mixwed\mcf_p^{x+2} &  \text{if } \, \mcb{p}{x} \neq \mcf_p^{x};\\
			\mcw{p}{d-y}\mixwed \mcf_{t-1}^{y+1} \mixwed\mcf_{t+1}^{d-(x+1)} \mixwed \mcb{p}{x} &  \text{if } \, \mcb{p}{x} = \mcf_p^{x}.\\
		\end{cases}
		\end{align}
	\end{enumerate} 
\end{prop}

\begin{remk}
When $k = 1$ or $k = d-1$, certain degeneracies arise in \crefrange{equation: degree 4 exchange relations HX}{equation: degree 6 exchange relations XX}. These cases occur when either $\str r$ is the first strip of its $BW$-type, or $\str r$ is the last strip. To address these degeneracies, the exchange relations must be adjusted according to the quiver structure described in Proposition \ref{prop: description of local pictures, i not 1}. 

For example, when $k = 1$, $\sigma(t-1) = -\sigma(t) = 1$, equation (\ref{equation: degree 4 exchange relations HX}) simplifies to:
\[
\big(\mcf_p^1\wedge \mcf_t^{d-1}\big)\big(\mcf_{t-1}^1\wedge \mcf_{t+1}^{d-1}\big) = \mcf_p^1\wedge \mcf_{t-1}^{d-1} + \big(\mcf_p^1\wedge \mcf_{t+1}^{d-1}\big)\big(\mcf_{t-1}^{1}\wedge \mcf_{t}^{d-1}\big).
\]

This adjustment reflects near-boundary behavior in the quiver structure.
Consequently, \crefrange{equation: degree 4 exchange relations HX}{equation: degree 6 exchange relations XX} are called the \emph{generic exchange relations} for $\gamma(r, i)$, as they hold uniformly outside these degenerate boundary scenarios.
\end{remk}

\begin{proof}[Proof of Proposition~\ref{prop: exchange relations, generic case}]
	The exchange relations arise directly from the combinatorial structure of the local pictures (cf.\ Proposition~\ref{prop: description of local pictures, i not 1}) together with Proposition~\ref{prop: decoration flags of the initial weave}. The expressions for the once mutated cluster variables are obtained by a direct calculation. Here we present the calculation for the expression of $\Delta'_{\gamma(r, i)}$ for equation~(\ref{equation: degree 6 exchange relations HH}); other cases are obtained in a similar way. 
	
	Assume that $\sigma(t-1) = \sigma(t) = 1$. With the same notation as in Proposition \ref{prop: decoration flags of the initial weave}, we have $k = x+1$. 
	
	We first simplify the first term on the right-hand side of the equation (\ref{equation: degree 6 exchange relations HH}):
	\[
	\big(\mcf_p^{x+1}\wedge \mcf_{t-1}^{d-x-1}\big)\big(\mcf_p^{x}\wedge \mcf_t^{d-x}\big) \big(\mcf_p^{x+2}\wedge \mcf_{t+1}^{d-x-2}\big).
	\]
	Rearranging terms within the wedge product, we get:
	\begin{align*}
		\big(\mcf_p^{x+1}\wedge \mcf_{t-1}^{d-x-1}\big)\big(\mcf_p^{x+2}\wedge \mcf_{t+1}^{d-x-2}\big) &= \mcf_{t-1}^{d-x-1}\wed \big( \mcf_p^{x+1}\wed \big(\mcf_p^{x+2} \wed \mcf_{t+1}^{d-x-2}\big)\big) \\
		&= \mcf_{t-1}^{d-x-1}\wed \big( \mcf_p^{x+2}\wed \big(\mcf_p^{x+1} \wed \mcf_{t+1}^{d-x-2}\big)\big)\\
		&= \big(\mcf_{t-1}^{d-x-1}\wed  \mcf_p^{x+2} \big) \wed \big(\mcf_p^{x+1} \wed \mcf_{t+1}^{d-x-2}\big)\\
		&= \mcf_{t-1}^2 \wed\big(\mcf_p^{x+2} \wed  \mcf_{t+1}^{d-x-3} \big) \wed \big(\mcf_p^{x+1} \wed \mcf_{t+1}^{d-x-2}\big)\\
		&= -\big(\mcf_{t-1}^2 \wed \big(\mcf_p^{x+1} \wed \mcf_{t+1}^{d-x-2}\big) \big) \wed \big(\mcf_p^{x+2} \wed  \mcf_{t+1}^{d-x-3} \big)
	\end{align*}
	where the second equality follows from Proposition~\ref{prop: wedge of any two within a flag is zero}, and the fourth equality follows from Proposition~\ref{prop: intersection production expansion formula} (by expanding out the intersection, all but two terms are left, and they can be combined into a single term). Hence the first term simplifies to
	\begin{multline*}
		\big(\mcf_p^{x+1}\wedge \mcf_{t-1}^{d-x-1}\big)\big(\mcf_p^{x}\wedge \mcf_t^{d-x}\big) \big(\mcf_p^{x+2}\wedge \mcf_{t+1}^{d-x-2}\big) = \\
		-\big(\mcf_{t-1}^2 \wed \big(\mcf_p^{x+1} \wed \mcf_{t+1}^{d-x-2}\big) \big)\wed \big(\big(\mcf_p^{x}\wedge \mcf_t^{d-x}\big)\cdot \big(\mcf_p^{x+2} \wed  \mcf_{t+1}^{d-x-3} \big)\big).
	\end{multline*}
	Similarly, the second term simplifies to
	\begin{multline*}
		\big(\mcf_p^{x+1}\wedge \mcf_{t+1}^{d-x-1}\big)\big(\mcf_p^{x}\wedge \mcf_{t-1}^{d-x}\big) \big(\mcf_p^{x+2}\wedge \mcf_{t}^{d-x-2}\big) = \\
		\big(\mcf_{t-1}^2 \wed \big(\mcf_p^{x+1} \wed \mcf_{t+1}^{d-x-2}\big) \big)\wed \big(\big(\mcf_p^{x+2}\wedge \mcf_t^{d-x-2}\big)\cdot \big(\mcf_p^{x} \wed  \mcf_{t+1}^{d-x-1} \big)\big).
	\end{multline*}
	Combining these two simplified terms, we get
	\begin{multline*}
		\big(\mcf_p^{x+1}\wedge \mcf_{t}^{d-(x+1)}\big)\Delta'_{\gamma(r, i)} = 
		\big(\mcf_{t-1}^2 \wed \big(\mcf_p^{x+1} \wed \mcf_{t+1}^{d-x-2}\big) \big) \\ 
		\wed \Big(\big(\mcf_p^{x+2}\wedge \mcf_t^{d-x-2}\big)\cdot \big(\mcf_p^{x} \wed  \mcf_{t+1}^{d-x-1} \big) - \big(\mcf_p^{x}\wedge \mcf_t^{d-x}\big)\cdot \big(\mcf_p^{x+2} \wed  \mcf_{t+1}^{d-x-3} \big) \Big).
	\end{multline*}
	Now we notice that 
	\begin{align*}
	&\mcf_{t-1}^2 \wed \big(\mcf_p^{x+1} \wed \mcf_{t+1}^{d-x-2}\big)\\
	=&\big( \mcf_{t-1}^1 \wed \mcf_p^{x+1} \wed \mcf_{t+1}^{d-x-2}\big)\cdot \mcf_t^1 - \big( \mcf_{t}^1 \wed \mcf_p^{x+1} \wed \mcf_{t+1}^{d-x-2}\big)\cdot \mcf_{t-1}^1\\
	=&\big( \mcf_{t-1}^1 \wed \mcf_p^{x+1} \wed \mcf_{t+1}^{d-x-2}\big)\cdot \mcf_t^1 + (-1)^x \big( \mcf_p^{x+1} \wed \mcf_{t}^{d-x-1}\big)\cdot \mcf_{t-1}^1
	\end{align*}
and
\begin{align*}
	&\mcf_t^1 \wedge \Big(\big(\mcf_p^{x+2}\wedge \mcf_t^{d-x-2}\big)\cdot \big(\mcf_p^{x} \wed  \mcf_{t+1}^{d-x-1} \big) - \big(\mcf_p^{x}\wedge \mcf_t^{d-x}\big)\cdot \big(\mcf_p^{x+2} \wed  \mcf_{t+1}^{d-x-3} \big) \Big) \\
	=& (-1)^x \big(\mcf_p^{x+2}\wedge \mcf_t^{d-x-2}\big)\cdot \big(\mcf_p^{x} \wed  \mcf_{t}^{d-x} \big) - (-1)^x\big(\mcf_p^{x}\wedge \mcf_t^{d-x}\big)\cdot \big(\mcf_p^{x+2} \wed  \mcf_{t}^{d-x-2} \big)\\
	=& 0.
\end{align*}
Therefore we conclude that 
\[
\Delta'_{\gamma(r, i)} = 
(-1)^x \mcf_{t-1}^1
\wed \Big(\big(\mcf_p^{x+2}\wedge \mcf_t^{d-x-2}\big)\cdot \big(\mcf_p^{x} \wed  \mcf_{t+1}^{d-x-1} \big) - \big(\mcf_p^{x}\wedge \mcf_t^{d-x}\big)\cdot \big(\mcf_p^{x+2} \wed  \mcf_{t+1}^{d-x-3} \big) \Big).
\]
Finally by Lemma~\ref{lemma: derivative formula for mixwed wedge} and  Propositions~\ref{prop: wedge of any two within a flag is zero} and \ref{prop: intersection production expansion formula}, we have
\begin{align*}
	&\big(\mcf_p^x \wed \mcf_{t+1}^{d-x-1}\big) \big(\mcf_p^{x+2}\wed \mcf_t^{d-x-2}\big)  \\
	=& \big(\mcf_p^{x+2} \mixwed \mcf_p^x \mixwed \mcf_{t+1}^{d-x-1}\big) \mixwed \mcf_t^{d-x-2} + \big(\mcf_p^{x+2}\wed \big(\mcf_p^x\wed \mcf_{t+1}^{d-x-1}\big)\wed \mcf_t^{d-x-2}\big)\\
	=& \big(\mcf_p^{x} \mixwed \mcf_p^{x+2} \mixwed \mcf_{t+1}^{d-x-1}\big) \mixwed \mcf_t^{d-x-2} + \big(\mcf_p^{x+2}\wed \big(\mcf_p^x\wed \mcf_{t}^{d-x}\big)\wed \mcf_{t+1}^{d-x-3}\big)\\
	=& \big(\mcf_p^{x} \mixwed \mcf_p^{x+2} \mixwed \mcf_{t+1}^{d-x-1}\big) \mixwed \mcf_t^{d-x-2} + \big(\mcf_p^{x}\wedge \mcf_t^{d-x}\big)\cdot \big(\mcf_p^{x+2} \wed  \mcf_{t+1}^{d-x-3} \big).
\end{align*}
Therefore 
\begin{align*}
	\Delta'_{\gamma(r, i)} &= (-1)^x \mcf_{t-1}^1 \mixwed \Big(\big(\mcf_p^{x} \mixwed \mcf_p^{x+2} \mixwed \mcf_{t+1}^{d-x-1}\big) \mixwed \mcf_t^{d-x-2} \Big)\\
	&=\mcf_p^{x} \mixwed \mcf_{t-1}^{d-(x+1)} \mixwed \mcf_{t+1}^{d-(x+1)}\mixwed \mcf_p^{x+2}.
\end{align*}
The derived expression for $\Delta'_{\gamma(r, i)}$ matches the proposed formula, confirming that the exchange relation holds under the given assumptions.
\end{proof}

For the case  $r = d-1$ and $i\in [2, m_1-r-1]$, the local pictures are constructed analogously. Specifically, the cycle $\gamma(r, i)$ connects to vertices in two strips above, $\str{r'}$ and $\str{r''}$, where $\str{r'}$ (resp., $\str{r''}$) is the closest strip above such that its $i' = (i+r-r')$-th (resp., $i'' = (i+r-r'')$-th) patch is of type $X$ (resp. $H$). To fully characterize the quiver structure, we analyze four distinct configurations depending on the types of the subsequent patches $i'+1$ in $\str{r'}$ and $i''+1$ in $\str{r''}$. These cases follow patterns analogous to those in Propositions~\ref{prop: description of local pictures, i not 1} and~\ref{prop: exchange relations, generic case}, and we omit their explicit examination here.

Having analyzed generic cycles, we now turn to cycles originating from the first patch of a strip. These require special attention due to their boundary behavior and connections to adjacent strips

\begin{prop}\label{prop: description of local pictures, i = 1}
	 Let $r\in [1, d-1]$ and $i = 1$. The the local picture of $Q(\ww_1)$ at $\gamma(r, 1)$ is described as follows, depending on the type of $\str r$. 
	\begin{enumerate}[wide, labelwidth=!, labelindent=0pt]
		\item[(Y)] Suppose that $\str r$ is of type $Y$. 
		Let $r'$ (resp., $r''$) be the closest strip above $\str r$ that is of different (resp., the same) $BW$-type, and $i', i''$ are chosen so that $r' + i' = r''+i'' = r+1$. Notice that $r'' = r -1$ and $i'' = 2$. Then the local picture of the quiver at $\gamma(r, 1)$ is as shown in Figure~\ref{fig: local picture Y i = 1}.
		\begin{figure}[H]
			\centering
			\begin{tikzcd}
				\gamma(0, x+1) \arrow[dddr, dashed] &&\\
				&\gamma(r', i') \arrow[r] &\gamma(r', i'+1) \arrow[ddl, bend left = 15] \\
				& \gamma(r'', i'') \arrow[d] &\\
				&\gamma(r, 1) \arrow[uu, bend right = 50] \arrow[r] & \gamma(r, 2)
			\end{tikzcd}
			\caption{Local picture at $\gamma(r, 1)$, type $(Y)$. The dashed arrow appears only when $\str{r-1}$ is the first $X$-strip.}
			\label{fig: local picture Y i = 1}
		\end{figure}
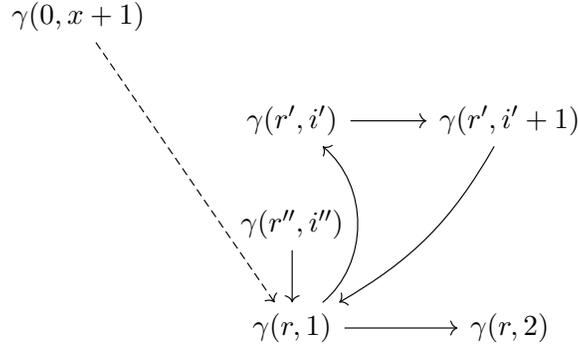
		
		\item[(X)] Suppose that $\str r$ is of type $X$. 
		Let $r'\in [r+1, d-1]$ be such that $\str{r'}$ is the next $X$-strip. Then the local picture at $\gamma(r, 1)$ is as shown in Figure~\ref{fig: local picture X i = 1}.
		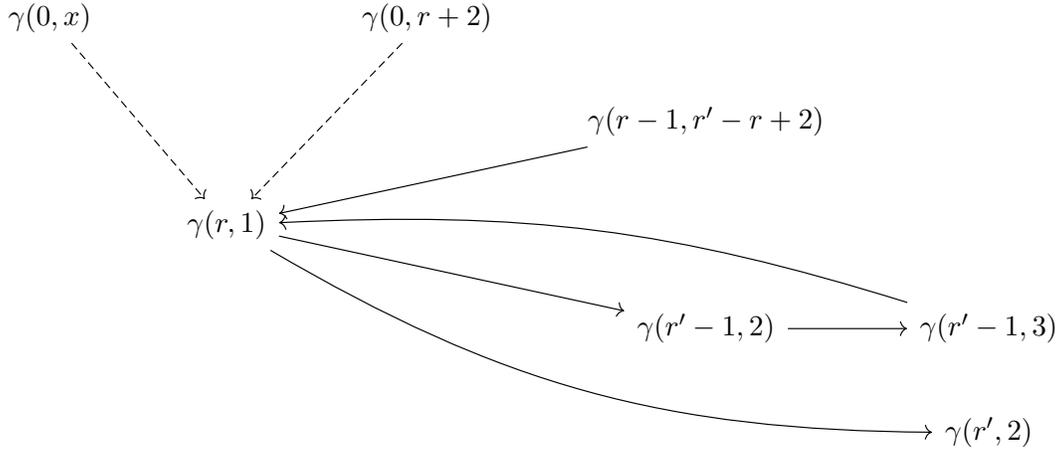
\begin{figure}[H]
			\centering
			\begin{tikzcd}
				\gamma(0, x) \arrow[ddr, dashed]  & & \gamma(0, r+2) \arrow[ddl, dashed] &&\\
				&&& \gamma(r-1, r'-r+2) \arrow[dll] & \\
				& \gamma(r, 1) \arrow[drr] \arrow[ddrrr, bend right = 15] &&&\\
				&&&\gamma(r'-1, 2) \arrow[r] & \gamma(r'-1, 3) \arrow[ulll, bend right = 10] \\
				&&&&\gamma(r', 2) 
			\end{tikzcd}
		\caption{Local picture at $\gamma(r, 1)$, type $(X)$. The dashed arrow from $\gamma(0, x)$ appears only when $\str{r}$ is the first $X$-strip. The dashed arrow from $\gamma(0, r+2)$ appears only when $\str r$ is the first strip of its $BW$-type.}
		\label{fig: local picture X i = 1}
		\end{figure}
		\item[(HH)] Suppose that $\str r$ is of type $H$ and the second patch of $\str r$ is of type $H$. Then the local picture at $\gamma(r, 1)$ is as shown in Figure~\ref{fig: local picture HH i = 1}.
		\begin{figure}[H]
			\centering
			\begin{tikzcd}
				\gamma(r-1, 1) \arrow[r] & \gamma(r-1, 2) \arrow[d] & \\
				& \gamma(r, 1) \arrow[r]\arrow[ul] & \gamma(r, 2) \arrow[ul] \arrow[d]\\
				&&\gamma(r+1, 1) \arrow[ul]
			\end{tikzcd}
			\caption{Local picture at $\gamma(r, 1)$, type $(HH)$.}
			\label{fig: local picture HH i = 1}
		\end{figure}
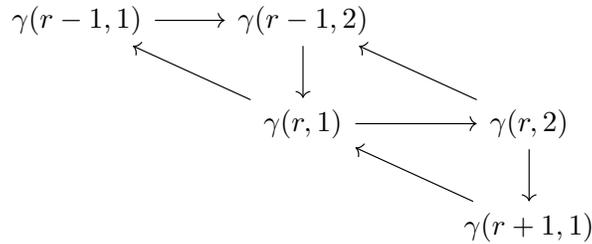
	\item[(HX)] Suppose that $\str r$ is of type $H$ and the second patch of $\str r$ is of type $X$. Then the local picture at $\gamma(r, 1)$ is as shown in Figure~\ref{fig: local picture HX i = 1}.
	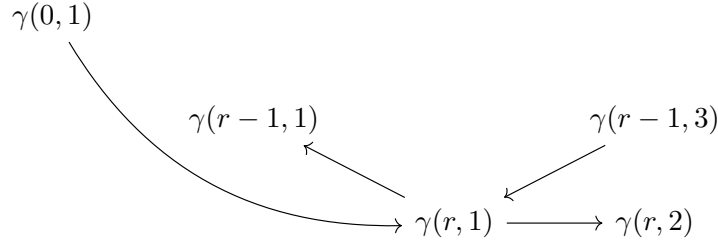
\begin{figure}[H]
		\centering
		\begin{tikzcd}
			\gamma(0, 1) \arrow[ddrr, bend right = 30] &&&\\
			 & \gamma(r-1, 1) & & \gamma(r-1, 3) \arrow[dl] \\
			&& \gamma(r, 1) \arrow[r]\arrow[ul] & \gamma(r, 2)
		\end{tikzcd}
		\caption{Local picture at $\gamma(r, 1)$, type $(HX)$.}
		\label{fig: local picture HX i = 1}
	\end{figure}
	\end{enumerate}
\end{prop}
\begin{proof}
	These statements follow from the description of the cycles in the proof of Proposition~\ref{prop: descriptions of cycles of the initial weave}.
\end{proof}

\begin{prop}\label{prop: exchange relations, i = 1}
	Let $r\in [1, d-1]$ and $i = 1$. Then the exchange relation for $\gamma(r, 1)$ is described as follows, depending the type of $\str r$. 
	\begin{enumerate}[wide, labelwidth=!, labelindent=0pt]
		\item[(Y)]
		Suppose that $\str r$ is of type $Y$. Without loss of generality, we assume that $\str r$ is of $BW$-type $B$. 
		With notation as in Case 2(b) of Proposition~ \ref{prop: decoration flags of the initial weave}, we have 
		\[
		\Delta_{\gamma(r, 1)} = \mcw{p}{d-y}\wedge \mcf_{t}^y, \quad \text{with }\, t = p+r+1.
		\]
		The exchange relation for $\gamma(r, 1)$ is given by
		\begin{multline}\label{equation: degree 4 exchange relations for the first cycle Y-strip}
			\Delta_{\gamma(r, 1)}\Delta'_{\gamma(r, 1)} = \big(\mcf_p^x\wedge \mcf_{t}^{d-x}\big)\big(\mcf_p^{d-(y-1)}\wedge \mcf_{t+1}^{y-1}\big)\\ + \big(\mcf_p^{x+1}\wedge \mcf_{t+1}^{d-(x+1)}\big)\big(\mcf_p^{d-(y-1)}\wedge \mcf_{t}^{y-1}\big),
		\end{multline}
		where $\Delta'_{\gamma(r, 1)} = 
		\mcb{p}{x+1}\mixwed \mcf_{t}^{y-1}\mixwed\mcf_{t+1}^{d-(x+1)}\mixwed 	\mcf_p^{d-(y-1)}$.
		In the case when $\mcf_p^x = \mcb{p}{x}$, we need to adjust  equation (\ref{equation: degree 4 exchange relations for the first cycle Y-strip}) by inserting an extra factor $\mcf_{p+x}^{d-1}\wedge \mcf_{p+x+1}^1$ into the first term of the right-hand side. 
		
		\item[(X)]
		Suppose that $\str r$ is of type $X$. Without loss of generality, assume that $\str r$ is of $BW$-type $W$. 
		Let $r'\in [r+1, d-1]$ be such that $\str{r'}$ is the next $X$-strip. With notation (for $\str{r'}$) as in Case 2(b) of Proposition~\ref{prop: decoration flags of the initial weave}, we have 
		\[
		\Delta_{\gamma(r, 1)} = \mcw{p}{d-(y+1)}\wedge \mcf_{t}^{y+1}, \quad \text{with }\, t = p+r'+1.
		\]
		The exchange relation for $\gamma(r, 1)$ is given by
		\begin{multline}\label{equation: degree 4 exchange relations for the first cycle X strip}
			\Delta_{\gamma(r, 1)}\Delta'_{\gamma(r, 1)} = \big(\mcf_p^{x-1}\wedge \mcf_{t}^{d-{(x-1)}}\big)\big(\mcf_p^{d-y}\wedge \mcf_{t+1}^{y}\big)\\
			 + \big(\mcf_p^{x}\wedge \mcf_{t+1}^{d-x}\big)\big(\mcf_p^{d-y}\wedge \mcf_{t}^{y}\big),
		\end{multline}
		where $\Delta'_{\gamma(r, 1)} = 
		\mcb{p}{x}\mixwed \mcf_{t}^{y}\mixwed\mcf_{t+1}^{d-x}	\mixwed \mcf_p^{d-y}$.
		In the case when $\mcf_p^{x-1} = \mcb{p}{x-1}$ (and $x>1$), we need to adjust  equation (\ref{equation: degree 4 exchange relations for the first cycle X strip}) by inserting an extra factor $\mcf_{p+x-1}^{d-1}\wedge \mcf_{p+x}^1$ into the first term of the right-hand side. In the case when $\str{r}$ is the first strip of its $BW$-type, we need to adjust  equation (\ref{equation: degree 4 exchange relations for the first cycle X strip}) by inserting an extra factor $\mcf_{p+r+1}^{d-1}\wedge \mcf_{p+r+2}^1$ into the first term of the right-hand side.
		
		\item[(HH)] Suppose that $\str r$ is of type $H$ and the second patch of $\str r$ is of type $H$. Without loss of generality, assume that $\str r$ is of $BW$-type $B$. 
		With notation as in Case 1(a) of Proposition \ref{prop: decoration flags of the initial weave}, we have
		\[
		\Delta_{\gamma(r, 1)} = \mcf_p^r \wedge \mcf_t^{d-r}, \quad \text{with } \, t = p+r+1.
		\]
		The exchange relation for $\gamma(r, 1)$ is given by 
		\begin{multline}\label{equation: degree 4 exchange relations for the first cycle HH}
			\Delta_{\gamma(r, 1)}\Delta'_{\gamma(r, 1)} = \big(\mcf_p^{r-1}\wedge \mcf_{t}^{d-(r-1)}\big)\big(\mcf_p^{r+1}\wedge \mcf_{t+1}^{d-(r+1)}\big)\\ + \big(\mcf_p^{r-1}\wedge \mcf_{t-1}^{d-(r-1)}\big)\big(\mcf_p^{r}\wedge \mcf_{t+1}^{d-r}\big),
		\end{multline}
		where $\Delta'_{\gamma(r, i)} =
			\mcf_p^{r-1}\mixwed \mcf_{t-1}^{1}\mixwed \mcf_{t+1}^{d-r}$.
		\item[(HX)] Suppose that $\str r$ is of type $H$ and the second patch of $\str r$ is of type $X$. Without loss of generality, assume that $\str r$ is of $BW$-type $B$. 
		With notation as in Case 1(a) of Proposition \ref{prop: decoration flags of the initial weave}, we have
		\[
		\Delta_{\gamma(r, 1)} = \mcf_p^r \wedge \mcf_t^{d-r}, \quad \text{with } \, t = p+r+1.
		\]
		The exchange relation for $\gamma(r, 1)$ is given by 
		\begin{multline}\label{equation: degree 4 exchange relations for the first cycle HX}
			\Delta_{\gamma(r, 1)}\Delta'_{\gamma(r, 1)} = \big(\mcf_p^{r-1}\wedge \mcf_{t+1}^{d-(r-1)}\big)\big(\mcf_p^{r+1}\wedge \mcf_{t}^{d-(r+1)}\big)\\ + \big(\mcf_p^{r-1}\wedge \mcf_{t-1}^{d-(r-1)}\big)\big(\mcf_p^{r}\wedge \mcf_{t+1}^{d-r}\big),
		\end{multline}
		where $\Delta'_{\gamma(r, i)} =
		\mcf_p^{r-1}\mixwed \mcf_{t-1}^{1}\mixwed \mcf_{t+1}^{d-r}$.
	\end{enumerate}
\end{prop}
\begin{proof}
	The exchange relations follow from Propositions~\ref{prop: description of local pictures, i = 1} and~\ref{prop: decoration flags of the initial weave}. The expressions for the once mutated cluster variables are derived by a direct calculation. 
\end{proof}

To complete the description of the initial seed, we examine the exchange relations for the defrosted cluster variables. These relations emerge from the amalgamation process; they encode interactions between the boundary frozen variables.
To simplify notation and results, we require that $\sigma(p-1) = \sigma(q-1)$. Without loss of generality, we assume that $\sigma(p-1) = \sigma(q-1) = 1$. 

Consider the following two Demazure weaves:
\begin{itemize}[wide, labelwidth=!, labelindent=0pt]
	\item the initial weave $\ww(p, q) : \beta(p, q) \rarrow T^+$; 
	\item the Demazure weave $\ww_1^*$ obtained by a concatenation of the initial weave \newline ${\beta(p, q-1)\rho \rarrow T^+ \rho}$ and any Demazure weave $T^+\rho \rarrow T^+$.
\end{itemize}
Notice that these two Demazure weaves $\ww_1^*$ and $\ww(p, q)$ are equivalent, by the nature of our construction of the initial weave: the last patch of each strip ``lives" in the weakly nested part of the row word, and will continue all the way down to the bottom; therefore we can move all these patches to the bottom, and they will form a Demazure weave $w_0 \rho \rarrow w_0$; we can choose any representation of $w_0$ by Lemma~\ref{lemma: complete nested words are all related by braid moves} or by Lemma~\ref{lemma: changing from T to T^+ does not affact the seed}, in particular, we pick $T^+$ in this case. Lastly we notice that all Demazure weaves $T^+\rho \rarrow T^+$ are equivalent by Lemma \ref{lemma: w0 rho to w0 equals rho w0 w0}.

Similarly, the initial weave $\ww(q, p+n): \beta(q, p+n) \rarrow \rev{T^+}$ is equivalent to the Demazure weave $\ww_2^*$ obtained by a concatenation of the initial weave $\beta(q, p+n-1) \rho \rarrow  \rev{T^+}\rho$ and any Demazure weave $\rev{T^+}\rho \rarrow \rev{T^+}$.

As a result, we conclude that the initial seed $\seed_\sigma(p, q)$ is equal to the amalgamation of seed $\seed(\ww_1^*)$ and $\seed(\ww_2^*)$ along their common frozen variables
\[
\zz_0 = \{\mcf_p^1\wedge \mcf_q^{d-1}, \mcf_p^2\wedge \mcf_q^{d-2}, \dots, \mcf_p^{d-1}\wedge \mcf_q^1\}.
\]

\begin{prop}\label{prop: local picture around the defrosted vertices}
	The local picture at the defrosted cluster variables (middle line) is as shown in Figure \ref{fig: local picture along the middle line}. 
	\begin{figure}[H]
		\centering
		\begin{tikzcd}
			&\gamma'_1 \arrow[d] & \gamma'_2 \arrow[l] \arrow[d] & \cdots \arrow[l] \arrow[d, phantom,  "\cdots"] & {\gamma'_{d-2}} \arrow[l]\arrow[d] & {\gamma_{d-1}'} \arrow[l]\arrow[d] & \boxed{\tilde{\gamma'}} \arrow[l]\\
			&{\gamma_1} \arrow[ur] \arrow[dl]& {\gamma_2} \arrow[ur] \arrow[dl]& \cdots\arrow[ur] \arrow[dl] \arrow[d, phantom,  "\cdots"]& {\gamma_{d-2}} \arrow[ur] \arrow[dl]&{\gamma_{d-1}} \arrow[ru] \arrow[dl]&\\
			\boxed{\tilde{\gamma''}} \arrow[r]& \gamma''_1 \arrow[u] \arrow[r] & \gamma''_2 \arrow[u] \arrow[r] & \cdots \arrow[r] & \gamma''_{d-2} \arrow[r] \arrow[u] & \gamma''_{d-1} \arrow[u]
		\end{tikzcd}
		\caption{Local picture of the quiver at the defrosted cluster variables. The frozen variables are $\Delta_{\tilde{\gamma'}} = \mcf_{q-1}^1\wedge \mcf_q^{d-1}$ and $\Delta_{\tilde{\gamma''}} = \mcf_{p-1}^1\wedge \mcf_{p}^{d-1}$. The cluster variables are $\Delta_{\gamma_k} = \mcf_{p}^{d-k} \wedge \mcf_{q}^{k}$,  $\Delta_{\gamma'_k} = \mcf_{p}^{d-k} \wedge \mcf_{q-1}^{k}$ and   $\Delta_{\gamma''_k} = \mcf_{p-1}^{d-k} \wedge \mcf_{q}^{k}$.
			for $k\in [1, d-1]$.}
		\label{fig: local picture along the middle line}
	\end{figure}
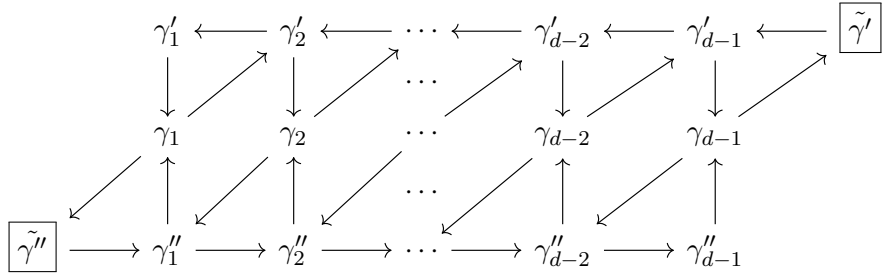
	
\end{prop}
\begin{proof}
	This follows from the discussion above, Proposition \ref{prop: decoration flags of the initial weave} and the proof of Lemma~\ref{lemma: w0 rho to w0 equals rho w0 w0} (cf.\ Figure~\ref{fig: quiver associated with w0 rho to w0}).
\end{proof}

\begin{prop}\label{prop: exchange relations for defrosted variables}
	The exchange relations for the defrosted cluster variables are 
	\begin{multline}
		\big(\mcf_p^{d-k}\wedge \mcf_q^{k}\big)\big(\mcf_{p-1}^{d-k}\wedge \mcf_{q-1}^k\big) = \big(\mcf_p^{d-k}\wedge \mcf_{q-1}^k \big)\big(\mcf_{p-1}^{d-k}\wedge \mcf_{q}^k\big)\\
		 + \big(\mcf_{p}^{d-(k+1)}\wedge \mcf_{q-1}^{k+1}\big)\big(\mcf_{p-1}^{d-(k-1)}\wedge \mcf_q^{k-1}\big),
	\end{multline}
	for $k\in [1, d-1]$. When $k = 1$ or $k = d-1$, the degenerated expressions should be adjusted following Proposition \ref{prop: local picture around the defrosted vertices}. 
\end{prop}

\begin{proof}
	This follows from Proposition \ref{prop: local picture around the defrosted vertices} and a direct calculation. 
\end{proof}

\begin{cor}\label{cor: once mutated c.v. are in Rsigma}
	Let $(p, q)$ be a valid cut of $\sigma$ with $\sigma(p-1) = \sigma(q-1)$. Consider the initial seed $\seed_\sigma(p, q)$. Then all the once mutated cluster variables lie in the ring  $R_\sigma$. 
\end{cor}
\begin{proof}
	Combining the exchange relations from Propositions~\ref{prop: exchange relations, generic case},~\ref{prop: exchange relations, i = 1} and~\ref{prop: exchange relations for defrosted variables}, we observe that all once-mutated cluster variables are expressible as mixed wedges of~$u_i$'s. Hence all of them lie in $R_\sigma$, by the same argument as in the proof of Corollary~\ref{cor: seed lies in R sigma}.
\end{proof}

\subsection{$\mca_\sigma$ contains the Weyl generators}\label{section: weyl generators}

In this section, we will show that all Weyl generators for $R_\sigma$ (cf. Definition~\ref{defn: Weyl generators}) belong to $\mca_\sigma$. This result will have two consequences. First, it will imply that our ``seeds" $\seed_\sigma(p, q)$ for $\mca_\sigma$ are indeed seeds, i.e., the elements of a extended cluster are algebraically independent. Second, it will imply that  $R_\sigma$ is a subalgebra of $\mca_\sigma$, by the First Fundamental Theorem of Invariant Theory, cf.\ Theorem~\ref{thm: the first fundamental theorem of invariant theory}.

Let $1\le i \le j \le n$ be such that $\sigma(i) = -\sigma(j) = 1$. We begin by showing that  $\langle u_i , u_j^* \rangle = u_i\mixwed u^*_j = \mcf_i^1\wedge \mcf_j^{d-1}\in \mca_\sigma$. 

\begin{lemma}\label{lemma: pairings are in A sigma}
	Any pairing $\langle u_i , u_j^* \rangle$ lies in the extended cluster for the initial seed $\seed_\sigma(p, q)$, for some $1\le p \le q \le n$ with $\|p-q\| \ge d+1$.
\end{lemma}
\begin{proof}
If $\| i - j\| \ge d+1$, then we take $p = i, q = j$. Then $\langle u_i , u_j^* \rangle = \mcf_p^1\wedge \mcf_q^{d-1}$ is a cluster variable for $\seed_\sigma(p, q)$, cf. Proposition~\ref{prop: decoration flags of the initial weave}. 

Otherwise, without loss of generality, we assume that $|j - i| \le d$.  Let $p = i$ and $q>j$ such that $\|p - q\| \ge d+1$. 

Consider the initial weave $\ww_1 = \ww(p, q)$. If $j = i+1$, then $\mcf_p^1\wedge \mcf_j^{d-1}$ is the frozen variable associated with the  cycle originating from the first marked boundary vertex. Thus, from now on we assume that $j>i+1$. 

If $\sigma(p+1) = \sigma(p+2) = \cdots = \sigma(j) = -1$, then $\mcf_p^1\wedge \mcf_j^{d-1}$ is the first cluster or frozen variable in the $(j-i-1)$-th strip.  Otherwise, $\mcf_p^1\wedge \mcf_j^{d-1}$ is the $(j-i-1)$-th cluster or frozen variable in the first strip. Both of the statements  follow directly from Proposition~\ref{prop: decoration flags of the initial weave}.
\end{proof}

Now let $p= i_1<i_2<\cdots <i_d\le q$ be such that all $i_j$'s are of the same color. We are going to show that $\det(u_{i_1}, u_{i_2}, \dots, u_{i_d})\in \mca_\sigma$.

\begin{defn}\label{defn: generalized patches and strips}
	A \emph{generalized $X$-patch} is a partial weave, either of the form \newline $\cdots I_i^j I_j^k\cdots \rarrow \cdots I_{j+1}^k I_i^j \cdots$ with $k\le i$ or of the form $\cdots I_j^i I_i^k \cdots \rarrow \cdots I_{i+1}^k I_j^i\cdots$ with $k\le j$. Similarly a \emph{generalized $H$-patch} is a partial weave, either of the form $\cdots I_i^j I_k^j\cdots \rarrow \cdots I_{k+1}^j I_i^j \cdots$ with $k\le i$ or of the form $\cdots I_j^i I_k^i \cdots \rarrow \cdots I_{k+1}^i I_j^i\cdots$ with $k\le j$. We say that a generalized patch is \emph{of type} $H$ (resp., $X$) if it is a $H$-patch (resp., $X$-patch). The ``degenerated" patch $\cdots i i \cdots \rarrow \cdots i \cdots$ is called a \emph{degenerate patch}. 
	
	A \emph{generalized strip} is a partial weave formed by concatenations of generalized patches. The type of a generalized strip is defined to be the type of its first patch. 
\end{defn}

\begin{prop}\label{prop: determinants are in A sigma}
	The determinant $\det(u_{i_1}, u_{i_2}, \cdots, u_{i_d})$ lies in the extended cluster for a seed $\seed_\sigma(p, q)$.
\end{prop}
\begin{proof}
Without loss of generality, we assume that all vertices $i_j$ are black.
Consider the decorated Demazure weave $\ww: \beta(p, q) \rarrow \beta'$ constructed as follows. 

Let $\beta(p, q) = T_1T_2\cdots T_m$ where $m = q-p$. The weave lines on and after the interval word $T_{i_d}$ will remain unchanged in the construction. Start with the word $T_{i_{d-1}}T_{i_{d-1}+1}\cdots T_{i_{d}-1}$. Construct an $H$-strip (resp., $X$-strip) if the first two interval words are the same (resp., different), cf.\ Definition~\ref{defn: strips}. Denote the bottom word by $T(d-1)T_{i_{d-1}}$. Then by Proposition~\ref{prop: decoration flags of the initial weave}, the decoration for the bottom word is of the form
\[
\mcf_{i_{d-1}} \rel{T(d-1)} \lpush{\mcf_{i_{d-1}}}{1}{\mcf_{i_d}} \rel{T_{i_{d-1}}} \mcf_{i_d}.
\]

Then consider the word $T_{i_{d-2}}T_{i_{d-2}+1}\cdots T_{i_{d-1}-1} T(d-1)T_{i_{d-1}}$. Construct a generalized $H$-strip (resp., $X$-strip) if the first two interval words are the same (resp., different), following a similar procedure as in Definition~\ref{defn: strips}; denote the bottom word by $T(d-2)T'(d-1)T_{i_{d-2}}T_{i_{d-1}}$. Following the proof of Proposition~\ref{prop: decoration flags of the initial weave}, we conclude that the decoration for the bottom word is of the form
\[
\mcf_{i_{d-2}} \rel{T(d-2)} \lpush{\mcf_{i_{d-2}}}{1}{\mcf_{i_{d-1}}} \rel{T'(d-1)} \lpush{\mcf_{i_{d-1}}}{1}{(\lpush{\mcf_{i_{d-2}}}{1}{\mcf_{i_d}})} \stackrel{T_{i_{d-2}}T_{i_{d-1}}}{\leftarrow \joinrel \xrightarrow{\hspace*{1cm}}}\mcf_{i_d}.
\]
Repeat this construction until we have made $d-1$ strips. Consider the last strip, denote the bottom word by $ \tilde{T}(1)\tilde{T}(2)\cdots \tilde{T}(d-2)T_{i_1}T_{i_2}\cdots T_{i_{d-1}}$. The decoration for the bottom word is 
\begin{multline*}
	\mcf_{i_{1}} \rel{\tilde{T}(1)} \lpush{\mcf_{i_{1}}}{1}{\mcf_{i_{2}}} \rel{\tilde{T}(2)} \cdots \rel{\tilde{T}(d-2)} \lpush{\mcf_{i_1}}{1}{\bigg(\lpush{\mcf_{i_2}}{1}{\big(\cdots(\lpush{\mcf_{i_{d-2}}}{1}{\mcf_{i_{d-1}}})\big)}\bigg)}\\
	\rel{T_{i_1}} \lpush{\mcf_{i_2}}{1}{\bigg(\lpush{\mcf_{i_3}}{1}{\big(\cdots(\lpush{\mcf_{i_{d-1}}}{1}{\mcf_{i_{d}}})\big)}\bigg)}  \stackrel{T_{i_2}T_{i_3}\cdots T_{i_{d-1}}}{\leftarrow \joinrel \xrightarrow{\hspace*{1.7cm}}}  \mcf_{i_d}.
\end{multline*}

Now we simply notice that for the last patch (degenerated) of the last generalized strip, the cycle $\gamma$ associated with that patch originates at the last character of $T_{i_1}$. As a result, the cluster or frozen variable associated with that cycle is 
\begin{align*}
\Delta_\gamma &= \qwed{\mcf_{i_1}\wedge \mcf_{i_2}\wedge \cdots \wedge \mcf_{i_{d-1}}}{\mcf_{i_2}\wedge \mcf_{i_3} \wedge \cdots \wedge  \mcf_{i_d}}{\mcf_{i_2}\wedge \mcf_{i_3}\wedge \cdots \wedge \mcf_{i_{d-1}}}\\
& = \mcf_{i_1}\wedge \mcf_{i_2}\wedge \cdots \wedge \mcf_{i_{d}} \\
&= \det(u_{i_1}, u_{i_2}, \cdots, u_{i_d}).
\end{align*}

Although $\beta' = \tilde{T}(1)\tilde{T}(2)\cdots \tilde{T}(d-2)T_{i_1}T_{i_2}\cdots T_{i_{d-1}} T_{i_{d}}T_{i_{d}+1}\cdots T_m$ is not reduced, we simply make any partial weave $\ww': \beta' \rarrow \beta''$ such that $\beta''$ is reduced, then concatenate $\ww: \beta(p, q) \rarrow \beta'$ with $\ww'$. The resulting Demazure weave is denoted by $\ww_1$.  We have shown that  $\det(u_{i_1}, u_{i_2}, \cdots, u_{i_d})$ is a cluster or frozen variable for the seed $\seed(\ww_1)$. This implies that $\det(u_{i_1}, u_{i_2}, \cdots, u_{i_d})$ lies in the extended cluster for the amalgamated seed $\seed_\sigma(p, q)$.
\end{proof}

\begin{cor}\label{cor: weyl generators are in A sigma}
	Assume that $n > d^2$ and let $I\in R_\sigma$ be a Weyl generator. Then there exist $1\le p\le q\le n$ with $\|p-q\| \ge d+1$ such that $I$ lies in the extended cluster for the seed $\seed_\sigma(p, q)$. 
\end{cor}
\begin{proof}
	The condition $n>d^2$ guarantees that there exists a cut $(p, q)$ satisfying $\|p-q\|>d$ such that all the vertices used for the Weyl generator $I$ lie on one side of the cut. The result then follows from Lemma~\ref{lemma: pairings are in A sigma} and Proposition~\ref{prop: determinants are in A sigma}. 
\end{proof}

\begin{remk}
	The condition $n>d^2$ seems to be stronger than necessary. We conjecture that it can be relaxed to $n \ge 2d$. 
\end{remk}

\begin{lemma}\label{lemma: alg indep of the cluster}
	Assume that $n > d^2$. Let $1\le p\le q\le n$ with $\| p - q\| \ge d+1$. Then the elements in the extended cluster for a seed $\seed_\sigma(p, q)$ are algebraically independent. 
\end{lemma}
\begin{proof}
	By Proposition~\ref{prop: cluster algebra does not depend on the cut}, any two seeds $\seed_\sigma(p, q)$ and $\seed_\sigma(p', q')$ are mutation equivalent. As a consequence, the corresponding extended seeds are birationally related to each other, i.e., any cluster variable in $\seed_\sigma(p', q')$ can be written as a rational function in the extended cluster for $\seed_\sigma(p, q)$. By Corollary~\ref{cor: weyl generators are in A sigma}, these extended clusters will collectively generate the fraction field $K_\sigma$ of $R_\sigma$. 
	
	By Corollary~\ref{cor: dimension for mca sigma}, the size of an extended cluster is $d(n-d)+1$, which coincides with the (Krull) dimension of $R_\sigma$. The transcendent degree of $K_\sigma$ over $\C$ is equal to the (Krull) dimension of $R_\sigma$, cf.\ \cite[Theorem 5.9]{Kemper}. It follows that the elements in an extended cluster for a seed $\seed_\sigma(p, q)$ are algebraically independent. 
\end{proof}

\begin{proof}[Proof of Lemma~\ref{lemma: amalgamation condition is satisfied}]
	By Proposition~\ref{prop: decoration flags of the initial weave}, the two seeds $\seed_1 = \seed(\ww(p, q))$ and $\seed_2 = \seed(\ww(q, p+n))$ have common frozen variables \[\zz_0 = \{\mcf_p^1\wedge \mcf_q^{d-1}, \mcf_p^2\wedge \mcf_q^{d-2}, \dots, \mcf_p^{d-1}\wedge \mcf_q^1\}.\]
	By Lemma~\ref{lemma: alg indep of the cluster}, we know that the elements in the union of the extended clusters for $\seed_1$ and $\seed_2$ are algebraically independent. Therefore, these two seeds can be amalgamated along $\zz_0$ (cf.\ Definition~\ref{defn: amalgamating two seeds}). 
\end{proof}

\begin{prop}\label{prop: rsigma is in mcasigma}
	Assume that $n>d^2$. Then $R_\sigma$ is a subalgebra of $\mca_\sigma$, i.e., we have the inclusions
	\[
	R_\sigma \subseteq \mca_\sigma \subseteq \K.
	\]
\end{prop}
\begin{proof}
	This follows from Corollary~\ref{cor: weyl generators are in A sigma} and Theorem~\ref{thm: the first fundamental theorem of invariant theory}.
\end{proof}

\subsection{Proof of the main theorem}\label{sec: proof of the main theorem}

In this section, we present two different proofs for Theorem~\ref{thm: mixed grassmannian is a cluster algebra}, both relying on the results in \cite{GLS}. 

Assume that $d$ is odd and $n>d^2$. Recall that by Proposition~\ref{prop: rsigma is in mcasigma}, we have $R_\sigma \subseteq \mcas$. So we only need to show the other inclusion. 

The first proof uses \cite[Theorem 1.4]{GLS}. By finding two disjoint seeds lying in $R_\sigma\subseteq \mcas$, we deduce that $R_\sigma = \mcas$.

The second proof uses \cite[Theorem 1.3]{GLS} to show that all cluster and frozen variables in $\mca_\sigma$ that lies in $R_\sigma$ are irreducible in $R_\sigma$; and then uses Proposition~\ref{prop:cluster-criterion} to conclude that $R_\sigma = \mcas$ by showing that the cluster and frozen variables for the initial seed and it is neighbors are all in $R_\sigma$.

We start by presenting the first proof. The key result we will need is \cite[Theorem~ 1.4]{GLS}, restated as Theorem~\ref{thm: disjoint clustes lies in a subalgebra ufd implies equality} below.

\begin{defn}
	Let $\mca$ be a cluster algebra of rank $r$ and dimension $s$. Denote the set of frozen variables of $\mca$ by $\{z_{r+1}, \cdots, z_{s}\}$.  Two clusters $\xx = (x_1, \dots, x_r)$ and $\yy = (y_1, \dots, y_r)$ of a cluster algebra $\mca$ are \emph{disjoint} if $\{x_1, \cdots, x_r\} \cap \{y_1, \dots, y_r\} = \emptyset$. 
\end{defn}

\begin{theorem}[{\cite[Theorem 1.4]{GLS}}]\label{thm: disjoint clustes lies in a subalgebra ufd implies equality}
	Let $\xx$ and $\yy$ be disjoint clusters of $\mca$. Let $R\subseteq \mca$ be a factorial subalgebra of $\mca$ such that 
	\[
	\{x_1, \cdots, x_r, y_1, \cdots, y_r, z_{r+1}, \cdots, z_{s}\} \subseteq R \subseteq \mca. 
	\]
	Then $R = \mca$.
\end{theorem}

Our goal is to identify two disjoint seeds for $\mca_\sigma$. They will both be initial seeds for some cuts $(p, q)$. 

\begin{lemma}\label{lemma: a tuple for disjoint seeds} Assume that $n\ge 4d-2$.
	Then there exists a tuple $(p, p', q, q')$ with $1\le p < p' < q < q' \le n$ (up to a cyclic shift) that satisfies the following conditions:
	\begin{enumerate}[wide, labelwidth=!, labelindent=0pt]
		\item there exists $\tilde{p}$ with $q'\le \tilde{p} < p+n$ such that $\sum_{i = \tilde{p}}^{p-1} \sigma(i) \equiv 0 \mod d$;
		\item there exists $\tilde{q}$ with $p'\le  \tilde{q} < q$ such that $\sum_{i = \tilde{q}}^{q-1} \sigma(i) \equiv 0 \mod d$;
		\item $p'-p\ge d-1$ and $q'-q\ge d-1$;
		\item $\| p-q\|\ge d+1$ and $\|p'-q'\| \ge d+1$. 
	\end{enumerate}
\end{lemma}
\begin{proof}
	If $\sigma$ is monochromatic, then we pick $p = 1, p' = d, q= 2d, q' = 3d-1$; it is easy to check that all conditions are satisfied. 
	
	Now assume that $\sigma$ is not monochromatic. Up to a cyclic shift, we may assume that $p = 1$ and $\sigma(n) = -\sigma(n-1)$. Let $p' = d$ and $q' = n-2$. Now consider the restriction of $\sigma$ to $[d, n-d-2]$. 
	\begin{itemize}[wide, labelwidth=!, labelindent=0pt]
		\item If $\sigma$ is monochromatic on $[d, n-d-2]$, then we take $q= n-d-1$. Let us check that all conditions are satisfied. 
		\begin{itemize}
			\item Take $\tilde{p} = n-2$, then $\sigma(n-2) + \sigma(n-1) = 0$;
			\item take $\tilde{q} = n-2d-1$, notice that $\tilde{q}= n-2d-1\ge 4d-2 - 2d -1 = 2d - 3 \ge d = p'$, and $\sum_{i = n-2d-1}^{n-d-2} \sigma(i) = d \equiv 0\mod d$;
			\item we have $p'-p = d-1$ and $q' - q = n-2 - (n-d-1) = d-1$;
			\item $\| p-q \| = \min\{n-d-2, d+2\} \ge d+1$ and $\|p' -q'\| = \min\{n-2-d, d+2\} \ge d+1$. 
		\end{itemize}
		\item Otherwise, let $j\in [d, n-d-3]$ such that $\sigma(j) = -\sigma(j+1)$, we take $q = j+1$. Similarly we can check that all conditions are satisfied.  \qedhere
	\end{itemize}
\end{proof}

\begin{defn}
	Recall from Definition~\ref{defn: tuple of flags mcf associated with sigma and u} that $\mcf_j^k = u_j\mixwed u_{j+1}\mixwed \cdots \mixwed u_{j'}$ where $j' \ge j$  is the smallest integer such that $\sum_{i = j}^{j'}\sigma(i) \equiv k \mod d$. Let $j\le \bar j < j+n$ such that $\bar j \equiv j' \mod n$; and let $h_j^k = \frac{j' - \bar j}{n} \in \Z$. We say the extensor $\mcf_j^k$ \emph{starts at $j$, ends at $\bar{j}$, and has $h_j^k$ revolutions}. 
	
	Recall from Definition~\ref{defn: multidegree} that for an invariant $f\in R_\sigma$, the multidegree of $f$ is defined to be 
	\[
	\multideg(f) = (d_1, \cdots, d_n)
	\]
	where $d_i= \deg_i(f)$ is the homogeneous degree of $f$ in $u_i$, for $i\in [1, n]$. We extend $d_i$ to all $i\in  \Z$ by periodicity, i.e., $d_{i+n} = d_i$ for $i\in \Z$. This notion can be naturally adapted to an extensor $\mcf_j^k$. To be precise, we define the  \emph{multidegree} of $\mcf_j^k$ to be
	\[
	\multideg(\mcf_j^k) = (\dots, d_1, \dots, d_n, \dots)
	\]
	such that $d_{i+n} = d_i$ for all $i\in \Z$ and 
	\[
	d_i = \deg_i(\mcf_j^k) = \begin{cases}
		h_j^k, & \text{if }\, \bar j < i <j+n;\\
		h_j^k + 1, & \text{if }\, j\le i \le \bar j.
	\end{cases}
	\]
\end{defn}

\begin{lemma}\label{lemma: necessary conditions for two wedge product to be the same}
	Let $1\le j_1 < j_2 \le n$, $1\le j_1'< j_2'\le n$, $k, k'\in [1, d-1]$ such that $j_1 \neq j_1'$, $j_1\neq j_2'$, $j_2\neq j'_1$. Then we have $\mcf_{j_1}^k\wedge \mcf_{j_2}^{d-k} \neq  \mcf_{j_1'}^{k'}\wedge \mcf_{j_2'}^{d-k'}$ provided one of the following is true. 
	\begin{enumerate}[wide, labelwidth=!, labelindent=0pt]
		\item Both $\mcf_{j_1}^k$ and $\mcf_{j_2}^{d-k}$ do not end at $j_1-1$.
		\item $\mcf_{j_1}^k$ does not end at $j_1-1$ or $j_2-1$; $\mcf_{j_1'}^{k'}$ does not end at $j_1'-1$ or $j_2'-1$;  $\mcf_{j_1}^k$ and $\mcf_{j'_1}^{k'}$ do not end at the same vertex. 
	\end{enumerate}
\end{lemma}
\begin{proof}
	Assume that $\mcf_{j_1}^k\wedge \mcf_{j_2}^{d-k} = \mcf_{j_1'}^{k'}\wedge \mcf_{j_2'}^{d-k'}$, then
	\[
	\multideg(\mcf_{j_1}^k)+\multideg(\mcf_{j_2}^{d-k}) = \multideg(\mcf_{j'_1}^{k'})+\multideg(\mcf_{j'_2}^{d-k'}).
	\]
	Then the lemma follows from a case by case analysis, depending on where $\mcf_{j_1}^k$, $\mcf_{j_2}^{d-k}$, $\mcf_{j'_1}^{k'}$, $ \mcf_{j'_2}^{d-k'}$ ends. 
\end{proof}

\begin{lemma}\label{lemma: a disjoint pair of seeds}
	Let $(p, p', q, q')$ be a tuple satisfying the conditions in Lemma~\ref{lemma: a tuple for disjoint seeds}. The the clusters for the initial seeds $\seed_\sigma(p, q)$ and $\seed_\sigma(p', q')$ are disjoint. 
\end{lemma}
\begin{proof}
	By Proposition~\ref{prop: decoration flags of the initial weave}, cluster variables in $\seed_\sigma(p, q)$ are of the form
	\[
	\mcf_p^x\wedge \mcf_{p_1}^{d-x}, \quad \mcb{p}{x} \wedge \mcf_{p_1}^{d-x}, \quad \mcw{p}{x}\wedge \mcf_{p_1}^{d-x}, \quad \mcf_q^x\wedge \mcf_{q_1}^{d-x}, \quad \mcb{q}{x} \wedge \mcf_{q_1}^{d-x}, \quad \mcw{q}{x}\wedge \mcf_{q_1}^{d-x},
	\] 
	where $p+2\le p_1\le q$, $q+2\le q_1\le p+n$ and $x\in [1, d-1]$; and cluster variables in $\seed_\sigma(p', q')$ are of the form
	\[
	\mcf_{p'}^{x'}\wedge \mcf_{p'_1}^{d-x'}, \quad \mcb{p'}{x'} \wedge \mcf_{p'_1}^{d-x'}, \quad \mcw{p'}{x'}\wedge \mcf_{p'_1}^{d-x'}, \quad \mcf_{q'}^{x'}\wedge \mcf_{q'_1}^{d-x'}, \quad \mcb{q'}{x'} \wedge \mcf_{q'_1}^{d-x'}, \quad \mcw{q'}{x'}\wedge \mcf_{q'_1}^{d-x'},
	\] 
	where $p'+2\le p'_1\le q'$, $q'+2\le q'_1\le p'+n$ and $x'\in [1, d-1]$. 
	
	First let us consider the cluster variable $\mcf_p^x\wedge \mcf_{p_1}^{d-x}$. Let us show that it is not a cluster variable in $\seed_\sigma(p', q')$. If $\mcf_p^x\wedge \mcf_{p_1}^{d-x} = \mcf_{p'}^{x'}\wedge \mcf_{p'_1}^{d-x'}$, then by Lemma~\ref{lemma: a tuple for disjoint seeds} (1), we know that both $\mcf_p^x$ and $\mcf_{p_1}^{d-x}$ cannot end at $p-1$, hence by Lemma~\ref{lemma: necessary conditions for two wedge product to be the same} (1), we get a contradiction. Similarly we can show $\mcf_p^x\wedge \mcf_{p_1}^{d-x}\neq \mcf_{q'}^{x'}\wedge \mcf_{q'_1}^{d-x'}$. It's also clear that $\mcf_p^x\wedge \mcf_{p_1}^{d-x}$ is not equal to any of $\mcb{p'}{x'} \wedge \mcf_{p'_1}^{d-x'}$, $ \mcw{p'}{x'}\wedge \mcf_{p'_1}^{d-x'}$, $\mcb{q'}{x'} \wedge \mcf_{q'_1}^{d-x'}$,  $\mcw{q'}{x'}\wedge \mcf_{q'_1}^{d-x'}$ by an similar argument involving multidegree. 
	
	Now consider the cluster variable $\mcb{p}{x} \wedge \mcf_{p_1}^{d-x}$. If $\mcb{p}{x} \wedge \mcf_{p_1}^{d-x} = \mcb{p'}{x'} \wedge \mcf_{p'_1}^{d-x'}$, then by Lemma~\ref{lemma: a tuple for disjoint seeds} (3), we know that $\mcb{p}{x}$ only uses the black vertices in $[p, p'-1]$ and $\mcb{p'}{x'}$ only uses black vertices in $[p', q']$. If either $\mcb{p}{x}$ or $\mcb{p'}{x'}$ is not consecutive wedges of black vertices, then it is clear that the multidegree will not match. Now we assume that they are both consecutive wedges of black vertices, hence $\mcb{p}{x} =\mcf_{\bar{p}}^x$ and $\mcb{p'}{x'} = \mcf_{\bar{p'}}^x$ for some $\bar p \in [p, p']$ and $\bar{p'}\in [p', q']$. Notice that $\mcf_{\bar{p}}^x$ can not end at $\bar p-1$; if $\mcf_{\bar{p}}^x$ end  at $p_1-1$, then $\mcf_{\bar{p}}^x\wedge \mcf_{p_1}^{d-x}$ is a frozen variable, a contradiction. Similarly we can show that $\mcf_{\bar{p'}}^x$ can not end at $\bar{p'}-1$ or $p_1'-1$. Finally notice that $\mcf_{\bar{p}}^x$ and $\mcf_{\bar{p'}}^x$ can not end at the same vertex. Therefore, by Lemma~\ref{lemma: necessary conditions for two wedge product to be the same} (2), we get a contradiction. Similarly we can show that $\mcb{p}{x} \wedge \mcf_{p_1}^{d-x}$ is not equal to any of the cluster variables in $\seed_\sigma(p', q')$.
	
	Similar arguments can be applied to the other types of cluster variables in $\seed_\sigma(p, q)$, and we conclude that the two clusters for $\seed_\sigma(p, q)$ and $\seed_\sigma(p, q)$ are disjoint. 
\end{proof}

Now Theorem~\ref{thm: mixed grassmannian is a cluster algebra} follows.

\begin{proof}[The first proof of Theorem~\ref{thm: mixed grassmannian is a cluster algebra}]
	By Proposition~\ref{prop: rsigma is in mcasigma}, we have $R_\sigma \subseteq \mca_\sigma$. By Lemma \ref{lemma: R(V) is a UFD}, $R_\sigma$ is factorial. Then by Lemmas~\ref{lemma: a tuple for disjoint seeds} and~\ref{lemma: a disjoint pair of seeds}, there exists a pair of disjoint clusters such that all their cluster variables are in $R_\sigma$ (cf.\ Corollary~\ref{cor: seed lies in R sigma}). The frozen variables are clearly in $R_\sigma$. Therefore $R_\sigma = \mcas$ by Theorem~\ref{thm: disjoint clustes lies in a subalgebra ufd implies equality}. 
\end{proof}

Now let us present the second proof. The key result needed from \cite{GLS} is restated as Theorem~\ref{thm: units of a cluster algebra are trivial and cluster varibles are irreducible} below. 

\begin{defn}
	Let $R$ be an integral domain and $R^\times$ be the set of invertible elements in $R$.  A nonzero element $f\in R$ is \emph{irreducible} if it cannot be written as a product $f = gh$ with non-invertible $g, h\in R$.
\end{defn}

\begin{theorem}[{\cite[Theorem 1.3]{GLS}}]\label{thm: units of a cluster algebra are trivial and cluster varibles are irreducible}
	Let $\mca$ be a cluster algebra. Then $\mca^\times = \C$ and any cluster or frozen variable in $\mca$ is irreducible. 
\end{theorem}

We remind the reader that frozen variables are not invertible by our convention. 

\begin{cor}\label{cor: cluster variables irreducible in Rsigma}
	A cluster or frozen variable in $\mcas$ that lies in $R_\sigma$ is irreducible in ~$R_\sigma$.
\end{cor}
\begin{proof}
	Let $x\in R_\sigma \subseteq \mca_\sigma$ be a cluster or frozen variable. Then $x$ is irreducible in $\mca_\sigma$ by Theorem~\ref{thm: units of a cluster algebra are trivial and cluster varibles are irreducible}. Now assume that $x$ is not irreducible in $R_\sigma$, then $x = gh$ with $g,h\in R_\sigma$ non-invertible in $R_\sigma$. Notice that $R_\sigma^\times = \mcas^\times = \C$, this implies $x = gh$ with $g, h\in \mcas$ and $g, h$ non-invertible in $\mcas$. Hence $x$ is not irreducible in $\mca_\sigma$, a contradiction.
\end{proof}

\begin{lemma}\label{lemma: requirements for starfish lemma}
	There exists a seed $\seed = (Q, \zz)$ for $\mcas$ satisfying the following conditions:
	\begin{enumerate}[wide, labelwidth=!, labelindent=0pt]
		\item all elements of $\zz$ belong to $R_\sigma$;
		\item the cluster variables are all irreducible;
		\item for each cluster variable $z\in \zz$, the seed mutation $\mu_z$ replace $z$ with an element $z'$ that lies in $R_\sigma$ and is irreducible in $R_\sigma$. 
	\end{enumerate}
\end{lemma}
\begin{proof}
	Let $(p, q)$ be a valid cut with $\sigma(p-1) = \sigma(q-1)$ and consider the initial seed $\seed_\sigma(p, q) = (Q, \zz)$. By Corollary~\ref{cor: seed lies in R sigma}, all elements of $\zz$ belong to $R_\sigma$. By Corollary~\ref{cor: once mutated c.v. are in Rsigma}, the once mutated cluster variables are also in $R_\sigma$. The lemma then follows from Corollary~\ref{cor: cluster variables irreducible in Rsigma}. 
\end{proof}

\begin{proof}[The second proof of Theorem~\ref{thm: mixed grassmannian is a cluster algebra}]
	By Proposition~\ref{prop: rsigma is in mcasigma}, we have $R_\sigma \subseteq \mcas$. Now we show the other inclusion. By Lemma~\ref{lemma: requirements for starfish lemma}, we can find a seed $\seed = (Q, \zz)$ such that all the frozen variables, cluster variables, and once mutated cluster variables are irreducible in $R_\sigma$. Notice that two elements in $\zz$ cannot differ by a scalar as they are algebraically independent; similarly $z$ and $z'$ cannot differ by a scalar since otherwise the exchange relation will give a non-trivial algebraic relation for elements in $\zz'$. We conclude that $\mcas \subseteq R_\sigma$ by Proposition~\ref{prop:cluster-criterion}, as desired. 
\end{proof}

\newpage

\section{Beyond the Main Theorem}\label{chap: properties and generalizations}

\subsection{Separated signatures}\label{sec: separated signatures}

In this section, we consider the case of separated signatures. Recall from Definition \ref{defn: separated signature} that a signature of type $(a, b)$ is {separated} if it contains $a$ consecutive black vertices followed by $b$ consecutive white vertices.  


\begin{defn}\label{defn: monochormatic and separated restrictions}
	Let $i \le j$. We say that  the restriction of $\sigma$ to $[i, j]$ is \emph{monochromatic} if $\sigma$ is constant on $[i, j]$. We say that the restriction of $\sigma$ to $[i, j]$ is \emph{separated} if there exists $k\in [i, j-1]$ such that the restrictions of $\sigma$ to $[i,k]$ and $[k+1, j]$ are both monochromatic. We say that $\sigma$ is \emph{monochromatic} if its restriction on $[1, n]$ is monochromatic. Note that $\sigma$ is {separated} if and only if there exists $j\in \Z$ such that the restriction of $\sigma$ to $[j, j+n-1]$ is separated. 
\end{defn}

\begin{example}
	The signature $\sigma = [\bullet\, \bullet \, \bullet\, \circ\, \circ\, \bullet\, \bullet\ \bullet]$ is separated. Its restriction on $[3, 7]$ is not separated. 
\end{example}

\begin{defn}\label{defn: valid cut for separated signatures}
	Let $\sigma$ be a separated signature. Let $1\le p \le q \le n$. The pair $(p, q)$ is called a \emph{feasible cut} for $\sigma$ (or $\beta_\sigma$) if one of the following conditions is satisfied:
	\begin{itemize}[wide, labelwidth=!, labelindent=0pt]	
		\item $n = 2d$, $q - p = d$, and each of the restrictions of $\sigma$ to $[p, q-1]$ and $[q, p+n-1]$ is separated;
		\item $n \ge 2d + 1$, $q - p = d$, and the restriction of $\sigma$ to $[p, q-1]$ is separated; 
		\item $n \ge 2d + 1$,  $p+n - q = d$, and the restriction of $\sigma$ to $[q, p+n-1]$ is separated;
		\item $\| p - q\| \ge d+1$.
	\end{itemize}
\end{defn}

\begin{defn}
	Let $(p, q)$ be a feasible cut for a separated signature $\sigma$ (note that $(p, q)$ is not necessarily a valid cut, cf.\ Definition \ref{defn: valid cut in general}). The cluster algebra $\mcas(p, q)$ is then defined in exactly the same way as in Definitions~\ref{defn: cutting along the disk to get two weaves} and~\ref{defn: cluster algebra cutting at p, q}. Notice that with the conditions for a feasible cut, we either have $m_1 = q-p \ge d+1$, or $\ww(p, q)$ does not contain an $X$-strip; in either case, Corollary~\ref{cor: description of the frozen cycles to the bottom for the initial weave} works, hence so does the amalgamation process, cf.\ Remark~\ref{remk: separated case no X strips so relaxed condition on n}. 
\end{defn}

\begin{prop}\label{prop: cluster algebra does not depend on the cut separated case}
	The cluster algebra $\mcas(p, q)$ does not depend on the choice of a feasible cut $(p, q)$.
\end{prop}
\begin{proof}
	This follows similarly to the proof of Proposition~\ref{prop: cluster algebra does not depend on the cut}.
\end{proof}

\begin{defn}
	In view of Proposition~\ref{prop: cluster algebra does not depend on the cut separated case}, we can defined the cluster algebra $\mcas$ associated with a separated signature $\sigma$ as the cluster algebra $\mcas(p, q)$, for any feasible cut $(p, q)$. 
\end{defn}

We are now prepared to state and prove a more precise version of Theorem \ref{thm: mixed grassmannian is a cluster algebra in the case of separated signatures}:

\begin{theorem}\label{thm: cluster structure for separated signatures}
	Assume that $d$ is odd and $n \ge 2d$. Let $\sigma$ be a separated signature with $a, b \ge d-1$. Then the cluster algebra $\mcas$ coincides (as a $\C$-algebra) with the mixed Pl\"ucker ring $R_\sigma$.
\end{theorem}
\begin{proof}
	
	The fact that all Weyl generators for $R_\sigma$ are cluster or frozen variables in $\mcas$ follows from Lemma~\ref{lemma: pairings are in A sigma} and Proposition \ref{prop: determinants are in A sigma}, together with the existence of a valid cut such that all the black (resp., white) vertices lie on one side of the cut (notice that $a, b\ge d-1$). 
	
	The result then follows from the second proof of the main theorem. 
\end{proof}

We next show that when $\sigma$ is separated (and $d$ is odd), the cluster structure  constructed above coincides with the one given by Carde in \cite{Carde}.
\begin{theorem}\label{thm: separated coincide with carde}
	Assume that $d$ is odd and $\sigma$ is a separated signature with $a, b \ge d$. Then the cluster algebra structure on $R_\sigma$ described in Theorem \ref{thm: cluster structure for separated signatures} coincides with the one described in \cite[Theorem 4.7]{Carde}.
\end{theorem}
\begin{proof}
	Without loss of generality, we may assume that $\sigma$ is black on $[1, a]$ and white on $[a+1, n]$. Take the cut $(1, a+1)$. 
	We will now define the ``zig-zag" Demazure weave $\ww_1 = \ww(1, a+1): \beta(1, a+1) \rarrow T(1, a+1)$ (here we use $T(i, i')$ to represent any nested word) using   the following recursive procedure (illustrated in the top half of the Figure \ref{fig: The zig zag weave}).

	\begin{enumerate}[wide, labelwidth=!, labelindent=0pt]	
		\item The ``zig-zag" Demazure weave $\ww(j, j+2): \beta(j, j+2) \rarrow T(j, j+2)$, where $j = \lfloor \frac{a}{2} \rfloor$. It is formed by an $H$-patch.
		
		For $1\le k \le \lfloor \frac{a}{2} \rfloor - 1$, we define the ``zig-zag" Demazure weave 
			\[
		\ww(k, a+1-k): \beta(k, a+1-k)\rarrow T(k, a+1-k)
		\]
		and 
		\[\ww(k+1, a+1-k): \beta(k+1, a+1-k) \rarrow T(k+1, a+1-k)
		\]
		via the following joint recursion. 
		
		\item To define the ``zig-zag" Demazure weave 
		\[
		\ww(k, a+1-k): \beta(k, a+1-k)\rarrow T(k, a+1-k),
		\]
		we take the
		``zig-zag" Demazure weave (see step (3))
		\[\ww(k+1, a+1-k): \beta(k+1, a+1-k) \rarrow T(k+1, a+1-k).
		\]
		We then prefix it with the word $\rho$ to get the weave
		\[\rho\ww(k+1, a+1-k): \rho \beta(k+1, a+1-k) \rarrow \rho T(k+1, a+1-k).
		\]
		Concatenating it  with any reduced Demazure weave $\rho T(k+1, a+1-k) \rarrow T(k, a+1-k)$, we obtain the Demazure weave $\ww(k, a+1-k)$. 
		(Here we used  that $\beta(k, a+1-k) = \rho \beta(k+1, a+1-k)$.)
		\item To define the ``zig-zag" Demazure weave 
		\[\ww(k+1, a+1-k): \beta(k+1, a+1-k) \rarrow T(k+1, a+1-k),
		\]
		we start with the 
		``zig-zag" Demazure weave 
		\[
		\ww(k+1, a-k): \beta(k+1, a-k)\rarrow T(k+1, a-k).
		\]
		(if $k < \lfloor \frac{a}{2}\rfloor - 1$, execute step (2), with $k: = k + 1$; otherwise execute step (1)).
		We then append the word $\rho$ on the right to get the weave
		\[
		\ww(k+1, a-k)\rho: \beta(k+1, a-k)\rho \rarrow T(k+1, a-k)\rho.
		\]
		Finally, we concatenate $\ww(k+1, a-k) \rho$ with any reduced Demazure weave $T(k+1, a-k) \rho \rarrow T(k+1, a+1-k)$ to obtain the  Demazure weave  $\ww(k+1, a+1-k)$. (Here we used that $\beta(k+1, a+1-k) = \beta(k+1, a-k) \rho$.)
	\end{enumerate}
	
	We next define the ``dual zig-zag" Demazure weave 
	\[
	\ww_2 = \ww(a+1, n+1): \beta(a+1, n+1) \rarrow T(a+1, n+1)
	\]
	as follows. We start with the ``zig-zag" Demazure weave 
	\[
	\ww(a+2, n+1): \beta(a+2, n+1) \rarrow T(a+2, n+1)
	\]
	similarly defined to the recursive procedure as above. We then prefix the word $\rho^*$ to get the weave 
	\[
	\rho^*\ww(a+2, n+1): \rho^* \beta(a+2, n+1) \rarrow \rho^*T(a+2, n+1).
	\]
	Finally, we concatenate $\rho^*\ww(a+2, n+1)$ with any reduced Demazure weave  $\rho^* T(a+2, n+1) \rarrow \rev{T(1, a+1)}$ to get the Demazure weave $\ww_2 = \ww(a+1, n+1)$. 
	
	It is straightforward to check that the seed obtained by amalgamating the seeds $\seed(\ww(1, a+1))$ and $\seed(\ww(a+1, n+1))$ coincides with the initial seed described in \cite{Carde}, up to a relabeling of the coordinates. In order to match their labelings, we need to label $a$ vectors and $b$ covectors (cyclically) as $a, a-1, \dots, 2, 1, b, b-1, \dots, 2, 1$, see Figure \ref{fig: The zig zag weave}.
\end{proof}	

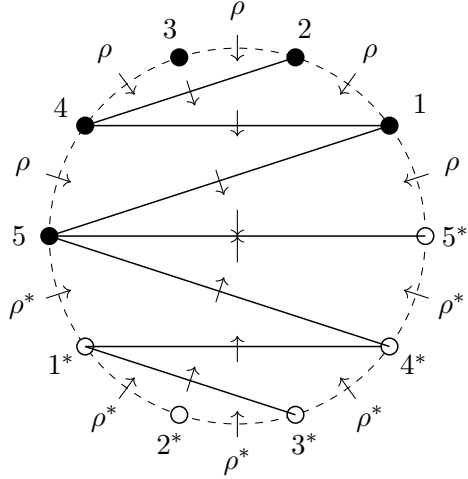
\begin{figure}[htp]\centering
	\begin{tikzpicture}[scale=.5]
		\draw [black, dashed] let \n1 = 5 in
		(0,0) circle (\n1 cm);
		
		\draw[->] (18: 5+0.35) -- (18: 5 - 0.35); \node at (18: 5+1) {$\rho$}; 
		
		\draw[line width=.2mm, fill = black] (36:5) circle (6.3pt);  
		\node at (36:5+1) {$1$};	
		
		\draw[line width=0.2mm] (36:5) -- (36*4:5); \draw[->] (36*2.5:5-1.65) -- (36*2.5:5-2.35);
		\draw[->] (54: 5+0.35) -- (54: 5 - 0.35); \node at (54: 5+1) {$\rho$};
		
		\draw[line width=.2mm, fill = black] (72:5) circle (6.3pt);  
		\node at (72:5+0.8) {$2$};

		\draw[->] (90: 5+0.35) -- (90: 5 - 0.35); \node at (90: 5+1) {$\rho$};
		
		\draw[line width=.2mm, fill = black] (108:5) circle (6.3pt);  
		\node at (108:5+0.8) {$3$};
		
		\draw[->] (126: 5+0.35) -- (126: 5 - 0.35); \node at (126: 5+1) {$\rho$};
		
		\draw[line width=.2mm, fill = black] (144:5) circle (6.3pt);  
		\node at (144:5+0.8) {$4$};
		
		\draw[line width=0.2mm] (36*4:5) -- (36*2:5); \draw[->] (36*3:5-0.65) -- (36*3:5-1.35);
		
		\draw[->] (162: 5+0.35) -- (162: 5 - 0.35); \node at (162: 5+1) {$\rho$};
		
		\draw[line width=.2mm, fill = black] (180:5) circle (6.3pt);  
		\node at (180:5+0.8) {$5$};
		
		\draw[line width= .2mm] (180:5) -- (36:5); \draw[->] (36*3:5-3.15) -- (36*3:5-3.85);
		
		\draw[->] (198: 5+0.35) -- (198: 5 - 0.35); \node at (198: 5+1) {$\rho^*$};
		
		\draw[line width=.2mm, fill = white] (216:5) circle (6.3pt);   
		\node at (216:5+0.8) {$1^*$};
		
		\draw[->] (234: 5+0.35) -- (234: 5 - 0.35); \node at (234: 5+1) {$\rho^*$};
		
		\draw[line width=0.2mm] (36*6:5) -- (36*9:5); \draw[->] (36*7.5:5-1.65) -- (36*7.5:5-2.35);
		
		\draw[line width=.2mm, fill = white] (252:5) circle (6.3pt);  
		\node at (252:5+0.8) {$2^*$};	
		
		\draw[->] (270: 5+0.35) -- (270: 5 - 0.35); \node at (270: 5+1) {$\rho^*$};
		
		\draw[line width=.2mm, fill = white] (288:5) circle (6.3pt);  
		\node at (288:5+0.8) {$3^*$};
		
		\draw[line width=0.2mm] (36*6:5) -- (36*8:5); \draw[->] (36*7:5-0.65) -- (36*7:5-1.35);
		
		\draw[->] (306: 5+0.35) -- (306: 5 - 0.35); \node at (306: 5+1) {$\rho^*$};
		
		\draw[line width=.2mm, fill = white] (324:5) circle (6.3pt);  
		\node at (324:5+0.8) {$4^*$};
		
		\draw[line width= .2mm] (36*5:5) -- (36*9:5); \draw[->] (36*7:5-3.15) -- (36*7:5-3.85);
		\draw[->] (342: 5+0.35) -- (342: 5 - 0.35); \node at (342: 5+1) {$\rho^*$};
		
		\draw[line width=.2mm, fill = white] (0:5) circle (6.3pt);  
		\node at (0:5+0.8) {$5^*$};
		
		\draw[line width = .2mm] (0:5) -- (180:5); \draw[->] (90:0.7) -- (90:0); \draw[->] (90:-0.7) -- (90:0);
	\end{tikzpicture}\caption{The ``zig-zag" weave for $a = b = 5$.}\label{fig: The zig zag weave}
\end{figure}

\subsection{Properties of cluster algebra $\mcas$}\label{sec: properties and further results}

This section establishes the ``naturality" of the cluster structures on mixed Grassmannians constructed in Section~\ref{chap: cluster structures on mixed grassmannians}. Specifically, we show that 
\begin{itemize}[wide, labelwidth=!, labelindent=0pt]
	\item all generators of $\mcas$ (the cluster and frozen variables) are multi-homogeneous elements of $\rab$;
	\item the set of these generators contains all Weyl generators;
	\item the construction recovers 
	\begin{itemize}
		\item the standard cluster structure on Grassmannians (cf.\ \cite{Scott}) when $b = 0$;
		\item cluster structures built from tensor diagrams for $d = 3$, cf.\ \cite{FominPylyavskyy};
		\item cluster structures built from mixed plabic graphs of separated signature with $a, b \ge d$, cf.\ \cite{Carde}.
	\end{itemize} 
\end{itemize}

We have already established the last claim (for separated signatures) in Section~\ref{sec: separated signatures}. Let us tend to the case where $\sigma$ is monochromatic. In that case,  $R_\sigma = \C[\text{Gr}_{d, n}]$, the homogeneous coordinate ring of the Grassmannian $\text{Gr}_{d,n}$. The cluster structure for $ \C[\text{Gr}_{d, n}]$ was first described by Scott \cite{Scott} (cf.\ also  \cite{GSV}) and extensively studied thereafter.

\begin{theorem}\label{thm: monochromatic case}
	Assume that $\sigma$ is monochromatic. The condition $n>d^2$ in Theorem~\ref{thm: mixed grassmannian is a cluster algebra} can be relaxed to $n\ge 2d$.  The cluster structure in the cluster algebra $\mcas$ coincides with the standard cluster structure in the Pl\"ucker ring $R_\sigma$.
\end{theorem}
\begin{proof}
	The fact that $\mcas$ is well-defined follows via the same arguments as in the separated case. The main difference is that we are not able to prove that $\mcas$ contains all the Weyl generators for $R_\sigma$ simply by choosing different cuts or different Demazure weaves for $\ww_1$ and $\ww_2$. However, we will show that we can find a seed that is the same as the rectangle seed (cf.\ \cite[Section 6.7]{FominWilliams6}); as a result, we will conclude that $\mcas = R_\sigma$. 
	
	Let $(p, q )$ with $q = n$ be any valid cut. Let $\ww_2 = \ww(n, p+n): \beta(n, p+n)\rarrow T^+$ be the usual initial weave. Construct the Demazure weave $\tilde{\ww}_1 = \ww(p, n): \beta(p, n) \rarrow \rev{T^+}$ via the the algorithm used in Definition \ref{defn: construction of the initial weaves as concatenations of strips} for constructing $\ww_1$, modified by replacing $\str{\beta}$ by $\rev{\str{\rev{\beta}}}$. That is,  we reverse the weakly nested word, apply the Strip construction, then reverse the resulting weave (mirror image). Notice that we can get $T^{-}$ as the bottom word after this algorithm, so we need to concatenate the resulting weave with a Demazure weave $T^{-} \rarrow \rev{T^+}$ (cf.\ Lemma~\ref{lemma: complete nested words are all related by braid moves}). It is routine to verify that the seed obtained by amalgamating the seeds $\seed(\tilde{\ww}_1)$ and $\seed(\ww_2)$ matches the rectangle seed described in \cite[Section~6.7]{FominWilliams6}.
\end{proof}

We next discuss the case $d= 3$. The fact that $R_\sigma$ is a cluster algebra is proved in \cite[Theorem 8.1]{FominPylyavskyy}. We claim that the cluster structures constructed in \cite{FominPylyavskyy} coincide with ours:

\begin{theorem}
	Assume that $d = 3$, $n\ge 6$ and $\sigma$ is non-alternating. Then the cluster algebra structure on $R_\sigma$ described in Theorem \ref{thm: mixed grassmannian is a cluster algebra in the case of separated signatures} coincides with the one described in \cite[Theorem 8.1]{FominPylyavskyy}.
\end{theorem}

\begin{proof}
	Let us explicitly establish the correspondence between our framework and the ``special invariants" construction in \cite{FominPylyavskyy}. The trees $\Lambda_p$ and $\Lambda^p$ correspond to the extensors $\mcf_p^1$ and $\mcf_p^2$, respectively. The special invariants $J_p^q$, $J_{pqr}$, $J^{pqr}$ are given by $J_p^q = \mcf_p^1 \wedge \mcf_q^{2}$, $J_{pqr} = \mcf_p^1\wedge \mcf_q^1\wedge \mcf_r^1$, and $J^{pqr} = \mcf_p^2 \mixwed \mcf_q^2\mixwed \mcf_r^2$. 
	
	For any valid cut $(p, q)$, the initial seed $\seed_\sigma(p, q)$ coincide with the  seed in \cite[Theorem 8.1]{FominPylyavskyy} arising from the triangulation $T$ with the diagonals $(p, k)$ for $k\in [p+2, q]$ and $(q, k')$ for $k'\in [q+2, p+n-1]$. 
	This equivalence is verified using the exchange relations in Section~\ref{sec: quivers and mutations}. Notably, each of $\ww_1$ and $\ww_2$ contains only two strips. This leads to significant simplifications in the mutation dynamics, mirroring the tensor diagram model's combinatorial rules.
	
	We note that under the assumptions of the theorem, the condition $n > d^2 = 9$
	in Theorem~\ref{thm: mixed grassmannian is a cluster algebra} can be relaxed to $n \ge 2d = 6$. 
\end{proof}

We next show that all generators of $\mcas$ are multi-homogeneous, cf.\ Definition~\ref{defn: multidegree}.

\begin{prop}
	All  cluster and frozen variables in $\mcas$ are multi-homogeneous. 
\end{prop}
\begin{proof}
As in Section \ref{sec: quivers and mutations}, consider the initial seed $\seed_\sigma(p, q)$ for a valid cut $(p, q)$ with $\sigma(p-1) =\sigma(q-1)$. In Proposition~\ref{prop: decoration flags of the initial weave}, we calculated all cluster and frozen variables; they are all multi-homogeneous.  In Section \ref{sec: quivers and mutations}, we derived all exchange relations for the cluster variables in $\seed_\sigma(p, q)$. Critically, both terms on the right-hand side of every exchange relation have the same multidegree. By \cite[Proposition 6.2]{FominZelevinsky4}, this ensures that mutations preserve multi-homogeneity.
\end{proof}

While our construction of $\mcas$ initially depends on the choice of a valid cut $(p, q)$, Proposition~\ref{prop: cluster algebra does not depend on the cut} demonstrates that $\mcas$ is independent of this choice. Furthermore, our construction ensures that the cluster structure is preserved under cyclic shifts of the signature. It would be interesting to explore whether the cluster structure in $\mcas$ is invariant under the reversal of $\sigma$ and/or under the duality involution $\bullet \leftrightarrow \circ$.

\subsection{Generalizations and conjectures}\label{sec: generalizations and conjectures}

We note that vectors can be viewed as $1$-extensors and covectors as $(d-1)$-extensors. It is natural to generalize the configuration space of vectors and covectors to more general configuration space of extensors of arbitrary level. 

Let $n, d\ge 3$ be integers. Let $V$ be a $d$-dimensional vector space. 

\begin{defn}
	A \emph{generalized signature} (of size $n$ and dimension $d$) is a map:
	\begin{equation}
		\sigma: \Z \rightarrow \{1, 2, \dots, d-1\}
	\end{equation}
such that $\sigma(j+n) = \sigma(j)$ for all $j\in \Z$. 
\end{defn}

\begin{defn}
	Let $\sigma$ be a generalized signature. 
	The \emph{configuration space of $\sigma$-tuples of extensors over $V$} is the ordered product
	\[
	V^\sigma = \prod_{j = 1}^n {\bigwedge}^{\sigma(j)} V.
	\]
	The invariant ring $R_\sigma$ is defined as the ring of $\slv$-invariant polynomial functions on~$V^\sigma$:
	\[
	R_\sigma = R_\sigma(V) = \C[V^\sigma]^{\slv}.
	\]
\end{defn}

The  ring $R_\sigma$ can be embedded into the Pl\"ucker ring as a subring. Explicitly, there is an inclusion:
\begin{equation}
	\iota: R_\sigma \hookrightarrow \C[\text{Gr}_{d, m}] = \C[V^m]^{\slv},
\end{equation}
where $m = \sum_{j = 1}^n \sigma(j)$.

\begin{theorem}[{\cite[Theorem 3.17]{VinbergPopov}}]
	The invariant ring $R_\sigma$ is factorial. 
\end{theorem}

By \cite[Theorem 3.5]{VinbergPopov}, the invariant ring $R_\sigma$ is finitely generated. We conjecture that the generators are of the following form:
\begin{conj}
	The invariant ring $R_\sigma$ is generated by the invariants of the form
	\[
	u_{j_1}\mixwed u_{j_2}\mixwed \cdots \mixwed u_{j_k}
	\]
	where $j_1, j_2, \dots, j_k$ are distinct elements in $[1, n]$ with $\sum_{i = 1}^k \sigma(j_i) \equiv 0 \bmod d$, and $u_{j_i}\in \bigwedge^{\sigma(j_i)}$ are coordinates. 
\end{conj} 

The notion of {$d$-admissibility} extends to  generalized signatures. 
\begin{defn}
	Let $\sigma$ be a generalized signature. We say that $\sigma$ is \emph{$d$-admissible} if there exists $j\in \Z$ such that for each $k\in [0, d-1]$, there exists $j'\ge j$ such that $\sum_{i = j}^{j'} \sigma(i) \equiv k \mod d$. 
\end{defn}

\begin{defn}
For $k\in [1, d-1]$, define the marked word
\[
\rho_k = I_k^1I_{k+1}^2\cdots I_{d-1}^{d-k}
\]
with the last character marked. 
Similarly we can define 
\[
\rho'_k = I_{k}^{d-1}I_{k-1}^{d-2}\cdots I_{1}^{d-k}
\]
with the last character marked. 
The (cyclic) marked word associated with a generalized signature $\sigma$ is 
\[
\beta_\sigma: = \prod_{j = 1}^n \rho_{\sigma(j)}.
\]
\end{defn}

\begin{defn}
	Let $\sigma$ be a generalized signature. The \emph{cyclic decorated flag moduli space} associated with $\sigma$ is defined to be the moduli space $\mfm(\beta_\sigma)$. 
\end{defn}

Analogously, define the distinguished invariant $f\in R_\sigma$
as in Definition~\ref{defn: generic point}.

\begin{conj}
	Assume that a generalized signature $\sigma$ is admissible. Then there is a $\emph{\text{SL}}(V)$-equivariant isomorphism of varieties:
	\begin{equation}
		\mfm(\beta_\sigma) \cong  V^\sigma \setminus \{f = 0\}.
	\end{equation}
	Consequently, we have 
	\begin{equation}
		\C[\mfm(\beta_\sigma)]^{\emph{\text{SL}}(V)} \cong R_\sigma[\frac{1}{f}].
	\end{equation}
\end{conj}


\begin{conj}
	Let $\sigma$ be a $d$-admissible generalized signature and $n \gg d$. Then the invariant ring $R_\sigma$ carries a natural cluster algebra structure that can be described using a suitable adaptation of the construction of the cluster algebra $\mcas$ in Section~\ref{sec: cluster structure on mixed Grassmannians}.
\end{conj}
\newpage


\section{Combinatorics of Signatures}\label{appendix: signature}
\subsection{Signatures and affine permutations}\label{sec: signatures and affine permutations}

Let $n, d\in \Z$ be positive integers such that $n = a + b \ge d \ge 3$.

\begin{defn}\label{defn: biased affine permutation and bounded affine permutation}
	A \emph{biased affine permutation of size $n$} is a bijection
	\begin{equation}
		\pi: \Z \rightarrow \Z
	\end{equation}
	such that $\pi(j) \ge j$ and $\pi(j + n) = \pi(j) + n$ for all $j\in \Z$. 
	The \emph{bias} of $\pi$, denoted by $\bb(\pi)$, is defined by
	\begin{equation}
		\bb(\pi) = \frac{1}{n}\sum_{j = 1}^n (\pi(j) - j).
	\end{equation}
	Notice that $\bb(\pi)\in \Z$ for any such $\pi$. 
	A biased affine permutation $\pi$ is called a \emph{bounded affine permutation} if  $\pi(j) \le j + n$ for all $j\in \Z$.
\end{defn}

\begin{remk}
	In \cite[\S2.1]{GalashinLam}, an \emph{affine permutation} is defined as a bijection $\pi: \Z\rightarrow \Z$ such that $\pi(j+n) = \pi(j) + n$ for all $j\in \Z$. A \emph{bounded affine permutation} further requires that $j \le \pi(j) \le j+n$ for all $j\in \Z$.   In contrast, \cite[Definition 1]{MadrasNeal} imposes the additional condition $\sum_{j=1}^n(\pi(j) - j) = 0$, or equivalently $\bb(\pi) = 0$m which is neither we nor \cite{GalashinLam} assume. 
\end{remk}

\begin{defn}\label{defn: biased affine permutation associated with a signature}
	Let $\sigma$ be a signature (cf.\ Definition~\ref{defn: signature}). For every $j\in \Z$, we have $\ell(\sigma, j, 0, d) < \infty$, cf.\ Definition~\ref{defn: admissible signature} and Remark~\ref{remk: ell(sigma, j, 0) < infty}. Define
	\begin{equation}
		\pi_{\sigma, d}: \Z \rightarrow \Z, \quad j \mapsto j + \ell(\sigma, j, 0, d).
	\end{equation}
	Then $\pi_{\sigma, d}$ is a biased affine permutation of size $n$. In what follows, we refer to $\pi_{\sigma, d}$ as \emph{the biased affine permutation associated with $\sigma$}. An affine permutation $\pi$ is called \emph{$d$-representable} if there exists a signature $\sigma$ such that $\pi = \pi_{\sigma, d}$. 
	Define the \emph{$d$-length of $\sigma$} to be the bias of $\pi_{\sigma, d}$:
	\begin{equation}
		\ell_d(\sigma) := \bb(\pi_{\sigma, d}) = \frac{1}{n}\sum_{j = 1}^n (\pi_{\sigma, d}(j) - j) = \frac{1}{n} \sum_{j = 1}^n \ell(\sigma, j , 0, d) \in \Z.
	\end{equation}
	The index $d$ will be dropped for $\pi_{\sigma, d}$, $\ell_d(\sigma)$, etc., if $d$ is clear from the context. 
\end{defn}

\begin{lemma}\label{lemma same pi implies same signature up to sign}
	Let $\sigma_1, \sigma_2$ be two signatures of the same size. Then $\pi_{\sigma_{1},d} = \pi_{\sigma_{2}, d}$ if and only if $\sigma_1 = \pm \sigma_2$. 
\end{lemma}

\begin{proof}	
	If $\sigma_1 = \pm \sigma_2$, then it is clear that $\pi_{\sigma_1, d} = \pi_{\sigma_2, d}$. 
	
	Conversely, assume that $\pi_{\sigma_{1}, d} = \pi_{\sigma_{2}, d} = \pi$. For any $j\in \Z$, we have $\pi(j) \ge j+2$ as $d \ge 2$. If $\pi(j) = j+2$, then $j$ and $j+1$ must be of different color in both $\sigma_1$ and $\sigma_2$. Otherwise $\pi(j) > j+2$, which implies that $j$ and $j+1$ are of the same color in both $\sigma_1$ and $\sigma_2$. The claim follows. 
\end{proof}

Recall from Definition~\ref{defn: admissible signature} that a signature $\sigma$ is \emph{$d$-admissible} if $\ell(\sigma, j, k, d) < \infty$  for all $j\in \Z$ and all $k\in [0, d-1]$. Equivalently, by Lemma~\ref{lemma: admissible equivalent def}, 
$\sigma$ is $d$-admissible if and only if $\ell(\sigma, j_0,k, d) < \infty$ for some $j_0\in \Z$ and all $k \in [0, d-1]$.


By Definition~\ref{defn: biased affine permutation associated with a signature} and Lemma~\ref{lemma same pi implies same signature up to sign}, we have an injective map
\begin{align*}
	\varphi_d:	\{\text{Signatures}\}/\{\pm 1\} &\longrightarrow \{\text{Biased affine permutations}\} \\
	\sigma & \longmapsto \pi_{\sigma, d}.
\end{align*}
We next show that the image of $\varphi_d$ consists of all $d$-representable biased affine permutations with bias $\le d$. Furthermore, the image of $d$-admissible signatures under $\varphi_d$ consists of  all $d$-representable biased affine permutation with bias~$d$ (cf.\ Theorem~\ref{thm: length and admissibility theorem} and Remark~\ref{remk: correspondence between signatures and affine permutations}).

\begin{defn}\label{defn: shifting of a signature}
	Let $\sigma$ be a signature. For $m\in \Z$, let $\tau_m(\sigma)$ be the signature (of the same type as $\sigma$) obtained from $\sigma$ by \emph{shifting to the left by $m$}, i.e., 
	\begin{equation}
		\tau_m(\sigma)(j) := \sigma(j+m), \quad j\in \Z.
	\end{equation}
	It's clear that 
	\begin{equation}
		\ell(\tau_m(\sigma), j, k, d) = \ell(\sigma, j+m, k, d).
	\end{equation}
	In particular, shifting a signature does not affect its admissibility. 
\end{defn}

\begin{defn}\label{defn: reversing of a signature}
	Let $\sigma$ be a signature. The \emph{reversal of $\sigma$}, denoted by $\rev{\sigma}$, is defined by
	\begin{equation}
		\rev{\sigma}(j) := \sigma(-j), \quad j\in \Z.
	\end{equation}
\end{defn}

\begin{lemma}\label{admissibility of reversing}
	$\sigma$ is $d$-admissible if and only if $\rev{\sigma}$ is $d$-admissible. 
\end{lemma}

\begin{proof}
	Note that $\rev{\rev{\sigma}} = \sigma$, so we only need to show that if $\sigma$ is $d$-admissible, then $\rev{\sigma}$ is $d$-admissible. Furthermore, we may assume that $\sigma$ is of type $(a, a)$ by Lemma~\ref{admissibility for signature of type a, b}.
	
	Assume that $\sigma$ is $d$-admissible. For $j\in \Z$ and $1\le k \le d-1$, there exists $j' \in \Z$ with $j'\ge -j + 1$ such that 
	\[
	\sum_{i = -j + 1}^{j'} \sigma(i) \equiv d - k \mod d.
	\]
	Now as $\sigma$ is of type $(a, a)$, we have 
	\[
	\sum_{i = -j+1}^{-j+tn} \sigma(i) = 0
	\]
	for any $t\in \Z$; take $t$ such that $-j + tn > j'$, then
	\[
	\sum_{i  = j' + 1 }^{-j + tn} \sigma(i) \equiv k \mod d.
	\]
	Therefore
	\[
	\sum_{i = j}^{-j'-1+tn} \rev{\sigma}(i) = \sum_{i = j - tn}^{-j'-1} \rev{\sigma}(i) = \sum_{i = j' + 1}^{-j + tn}\sigma(i) \equiv k \mod d.
	\]
	Note that $-j'-1+tn \ge j$, implying that $\rev{\sigma}$ is $d$-admissible. 
\end{proof}

\begin{lemma}\label{shifting/reversing preseves length}
	Let $\sigma$ be a signature. Then 
	\[
	\pi_{\rev{\sigma}}(j) = -\pi^{-1}_{\sigma}(-j) \quad \text{and} \quad \pi_{\tau_m(\sigma)}(j) = \pi_\sigma(j+m) - m
	\]
	for $j\in \Z$. As a consequence, $\ell_d(\sigma) = \ell_d({\rev{\sigma}})$ and $\ell_d(\sigma) = \ell_d({\tau_m(\sigma)})$ for all $m\in \Z$. In other words, shifts and reversals preserve the length of a signature. 
\end{lemma}

\begin{proof}
	For $1\le j\le n$, let $g_j = \pi_\sigma(j)$. Then 
	\[
	\pi_{\rev{\sigma}}(-g_j) = -g_j + \ell({\rev{\sigma}},{-g_j},0, d) = -g_j + (g_j - j) = -j.
	\]
	In other words, we have 
	\begin{equation}\label{pi of sigma and sigma reverse}
		\pi_{\rev{\sigma}}(j) = -\pi^{-1}_{\sigma}(-j), \quad j\in \Z.
	\end{equation}
	So 
	\[
	\ell_d({\rev{\sigma}}) = \frac{1}{n}\sum_{j=1}^n \left(\pi_{\rev{\sigma}}(-g_j) - (-g_j)\right) = \frac{1}{n}\sum_{j  = 1}^n (g_j - j) = \frac{1}{n}\sum_{j= 1}^n (\pi_\sigma(j) - j) = \ell_d(\sigma).
	\]
	For shifting, we have, for $j\in \Z$:
	\begin{equation}\label{pi of sigma and sigma shifts}
		\pi_{\tau_m(\sigma)}(j)  = j + \ell(\tau_m(\sigma),j,0, d) = j + \ell(\sigma,{j+m},0, d) = \pi_\sigma(j+m) - m.
	\end{equation}
	Hence 
	\[
	\ell_d({\tau_m(\sigma)}) = \frac{1}{n}\sum_{j=1}^n (\pi_{\tau_m(\sigma)}(j) - j) = \frac{1}{n}\sum_{j=1}^n (\pi_\sigma(j+m) - (j+m)) = \ell_d(\sigma).\qedhere
	\]
	
\end{proof}

\begin{defn}
	Let $\sigma$ be a signature. Then $j\in \Z$ is called a \emph{$d$-fixed point of $\sigma$} (or a \emph{fixed point of $\pi_{\sigma, d}$})  if 
	\begin{equation}
		\pi_{\sigma, d}(j) \equiv j \mod n.
	\end{equation}
	We use $\fix(\sigma, d)$ (or $\fix(\pi_{\sigma, d}))$ to denote the number of $d$-fixed points of $\sigma$ (or the number of fixed points of $\pi_{\sigma, d}$) in $\Z/n\Z$. 
\end{defn}

\begin{example}
	Let $n = 5, a = 4, b = 1, \sigma = [\bullet\, \bullet\, \bullet\, \bullet\, \circ]$ and $d = 6$. Then we have $(\pi_\sigma(j) - j)_{1\le j\le 5} = (8, 8, 10, 2, 2)$. Hence $\fix(\sigma) = 1$ and $j = 3$ is a fixed point of $\sigma$. Note that $\pi_\sigma(3) = 3 + 2n$.
\end{example}

\begin{lemma}\label{lemma: fixed points change color equivalent to a = b and must of length n}
	Let $\sigma$ be a signature of type $(a, b)$ and $j\in \Z$ a fixed point of $\pi_{\sigma}$. Then $\sigma(j-1)\neq \sigma(j)$ if and only if $a = b$. Moreover, if these conditions hold, then $\pi_\sigma(j) = j + n$.
\end{lemma}

\begin{proof}
	As $j$ is a fixed point of $\sigma$, we may assume that $\pi_\sigma(j) = j + sn$ for some $s\in \Z$ and $s\ge 1$. Without loss of generality, assume that $j$ is black. 
	
	Consider the sequence
	\[
	m_t = \sum_{i = j}^t \sigma(i), \text{ for } t \ge j.
	\]
	We have $m_j = 1$, $m_{j+n-1} = a - b$, note that $0< m_t < d$ for $j \le t < j + sn -1$ as $j + sn -1$ is the first time we have  $m_t \equiv 0 \mod d$. 
	
	Now if $\sigma(j-1) \neq \sigma(j)$, then $j-1$ is white, so $\sigma(j-sn-1) = \sigma(j-1) = -1$, hence $m_{j+sn-1} = m_{j+sn-2} + \sigma(j+sn-1) = m_{j+sn-2} - 1$, therefore $0\le m_{j+sn-1} <d-1$, as we know $m_{j+sn-1} \equiv 0 \mod d$, we must have $m_{j+sn-1} = 0$. Therefore $m_{j+sn-1} = s(a-b) = 0$, hence $a = b$. Moreover, as we have $m_{j+n-1} = a-b = 0$, we conclude that $s = 1$ and $\pi_\sigma(j) = j+n$.
	
	Conversely, if $a = b$, then we have $m_{j+n-1} = a -b = 0$. So $s =1$ and $\pi_\sigma(j) = j+n$. Then $m_{j+n-1} = 0$ implies $m_{j+n-2} = 1$, hence $\sigma(j-1) = -1$.\qedhere
\end{proof}

\begin{prop}
	For any signature $\sigma$ and any $m\in \Z$, we have $\fix(\sigma) = \fix(\tau_m(\sigma)) = \fix(\overline{\sigma})$.
\end{prop}

\begin{proof}
	We have 
	\[
	\pi_{\tau_m(\sigma)}(j)  =  \pi_\sigma(j+m) - m \quad \text{and} \quad \pi_{\overline{\sigma}}(j) = -\pi^{-1}_{\sigma}(-j), \text{ for } j\in \Z.
	\]
	So $j$ is a fixed point of $\sigma$ if and only if $j -m$ is a fixed points of $\tau_m(\sigma)$ if and only if $-j +1$ is a fixed point of $\overline{\sigma}$. Hence $\fix(\sigma) = \fix(\tau_m(\sigma)) = \fix(\overline{\sigma})$.
\end{proof}

\begin{prop}\label{the number of fixed points is at most 2 for type (a, a)}
	Let $\sigma$ be a signature of type $(a, a)$. Then any two fixed points of $\sigma$ must be of different colors. In particular,  $\fix(\sigma) \le 2$. 
\end{prop}
\begin{proof}
	Let $i < j$ be two fixed points of $\sigma$ such that $\sigma(i) = \sigma(j)$.  We may assume that $|j - i| < n$ by periodicity of $\sigma$.

	Without loss of generality, we may assume that $\sigma(i) = \sigma(j) = 1$. Then $\sigma(i-1) = -1$ and $\sigma(j-1) = -1$ by Lemma~\ref{lemma: fixed points change color equivalent to a = b and must of length n}. Let $x_1$ (resp. $y_1$) denote the number of blacks (resp. white) between $i$ and $j-1$, and let $x_2$ (resp., $y_2$) be the number of black (resp., white) vertices between $j$ and $i-1 + n$. 
	Consider the following sequence 
	\[
	m_t = \sum_{c = i}^{t} \sigma(c), \text{ for } t\ge i.
	\]
	We have $m_i = 1$, so $m_{j-1} > 0$, which implies $x_1 - y_1 > 0$. Similarly, $x_2 - y_2 >0$. However, as $\sigma$ is of type $(a, a)$, we must have $x_1 + x_2 = y_1 + y_2$, a contradiction.
\end{proof}

\begin{example}
	Let $\sigma_1 = [\bullet\, \bullet\, \circ\, \circ\, \bullet\, \bullet\, \circ\, \circ]$, $\sigma_2 = [\bullet\, \bullet\, \circ\, \bullet\, \circ\, \bullet\, \circ\, \circ]$ and $\sigma_3 = [\bullet, \bullet, \circ, \circ]$. Each of these signatures is $3$-admissible but not $4$-admissible. Furthermore,  $\fix(\sigma_1, 3) = \fix(\sigma_1, 4) = 0, \fix(\sigma_2, 3) =\fix(\sigma_2, 4) = 1, \fix(\sigma_3, 3)= \fix(\sigma_3, 4) = 2$.  
\end{example}

\begin{example}
	If $\sigma$ is monochromatic, then
	\begin{equation}
		\fix(\sigma, d) = \begin{cases}
			n, & \text{if} \  n \mid d\\
			0, & \text{if}\   n \nmid d
		\end{cases}
	\end{equation}
	If $\sigma$ is a signature of type $(n-1, 1)$ or $(1, n-1)$, then
	\begin{equation}
		\fix(\sigma, d) = \begin{cases}
			n, & \text{if} \ d = 0\\
			n - 4, & \text{if} \  d \ge 1, n \ge 4 \text{ and } n-2 \mid d\\
			0, & \text{otherwise}.
		\end{cases}
	\end{equation}
\end{example}

\begin{prop}\label{admissibility for switching an adjacent pair of different colors}
	Let $\sigma$ be a non-monochromatic $d$-admissible signature.  Let ${(i, i+1)}$ be such that $\sigma(i)\sigma(i+1) = -1$. Let $\sigma'$ be a signature obtained from $\sigma$ by switching the colors of $i$ and $i+1$, i.e., $\sigma'(i) = \sigma(i+1)$, $\sigma'(i+1) = \sigma(i)$ and $\sigma'(j) = \sigma(j)$ for $j\not\equiv i, i+1 \mod n$. Then the following are equivalent:
	\begin{enumerate}[wide, labelwidth=!, labelindent=0pt]
		\item[\rm\textbf{(a)}] $\sigma'$ is $d$-admissible;
		\item[\rm\textbf{(b)}] $i+1$ is not a fixed point of $\pi_{\sigma, d}$;
		\item[\rm\textbf{(c)}] $\ell_d({\sigma}) = \ell_d({\sigma'})$.
	\end{enumerate}
\end{prop}

\begin{proof}
	Without loss of generality, we may assume that $\sigma(i) = 1$ and $\sigma(i+1) = -1$. Let $j\in \Z$ such that $\pi_\sigma(j) = i + 1 $ and $j'\in \Z$ such that $\pi_{\sigma'}(j') = i + 1$. Then we have $\pi_{\sigma}(m) = \pi_{\sigma'}(m)$ for all $m\in \Z$ such that $m\notin \{i+1, j, j'\} \mod n$. 
	\begin{itemize}[wide, labelwidth=!, labelindent=0pt]
	\item[\textbf{(a)} $\Rightarrow$ \textbf{(b)}:]
	Assume that $\sigma'$ is $d$-admissible. If $i+1$ is a fixed point of $\pi_{\sigma}$, then $\sigma$ is of type $(a, a)$ and $\pi_\sigma(i+1) = i+1 + n$ by Lemma~\ref{lemma: fixed points change color equivalent to a = b and must of length n}. Consider the sequence 	
	\[
	m_t = \sum_{k = i+2}^{t} \sigma'(k), \ t\ge i+2.
	\]
	We claim that $m_{t} \not\equiv 1 \bmod n$ for $i+2 \le t < i+n -1$.  Otherwise, assume that $m_t \equiv 1 \bmod n$ for some $i+2 \le t < i+n-1$. Then $\pi_{\sigma}(i+1) = t + 1< i + 1 +n$, a contradiction. Notice that $m_{i+n-1} = m_{i+n+1} = 0$ and $m_{i+n} = -1$. We conclude that $m_t \not\equiv 1 \mod n$ for all $t\ge i + 2$ (as $\sigma$ is of type $(a, a)$). This implies that $\ell({\sigma'},{i+2},1, d) = \infty$, contradicting the admissibility of  $\sigma'$. Therefore $i+1$ is not a fixed point of $\pi_{\sigma}$. 
	
	\item[\textbf{(b)} $\Rightarrow$ \textbf{(c)}:]
	Assume that $i+1$ is not a fixed point of $\pi_{\sigma}$. Since $\sigma$ is admissible, the same argument as above implies that $i+1$ is not a fixed point of $\pi_{\sigma'}$, thus $j \not\equiv i+1\mod n$ and $j' \not\equiv i+1 \mod n$. Therefore
	\begin{equation*}
		\begin{split}
			n(\ell({\sigma}) - \ell({\sigma'})) &= (\pi_{\sigma}(i+1) - \pi_{\sigma'}(i+1)) +(\pi_{\sigma}(j) - \pi_{\sigma'}(j)) +(\pi_{\sigma}(j') - \pi_{\sigma'}(j'))\\
			&= \left(\left(1+\ell(\sigma,{i+2},{1}, d)\right) -\left(1+\ell(\sigma,{i+2},{d-1}, d)\right)\right) \\
			&\quad + \left(-\ell(\sigma,{i+2},{1}, d) - 1\right) + \left(\ell(\sigma,{i+2},{d-1}, d) + 1\right) =0.
		\end{split}
	\end{equation*}
	
	\item[\textbf{(c)} $\Rightarrow$ \textbf{(a)}:]
	Assume that $\ell(\sigma) = \ell({\sigma'})$. Suppose that $\sigma'$ is not admissible. Then there exist some $r\in \Z$ and $1\le k \le d-1$ such that $\ell({\sigma'},r,k, d) = \infty$. Consider $x = r + \ell(\sigma,r,k, d)-1$. Suppose that $x \not\equiv i \mod n$, then $\ell({\sigma'},r,k, d) \le \ell(\sigma,r,k, d) < \infty$, a contradiction. Hence $x \equiv i \mod n$. We claim that $i+1$ is a fixed point of $\pi_{\sigma}$. Otherwise, we will have $\ell({\sigma'},{i+2},1, d) < \infty$. As a result,  $\ell({\sigma'},r,k, d) \le \ell(\sigma,r,k, d) -1 + 2 + \ell({\sigma'},{i+2},1, d) < \infty$, a contradiction. 
	
	Now as $i+1$ is a fixed point of $\pi_{\sigma}$ and is not a fixed point of $\pi_{\sigma'}$, we conclude that $j\equiv i+1$ and $j' \not\equiv i+1$. Therefore 
	\begin{equation*}
		\begin{split}
			n(\ell({\sigma}) - \ell({\sigma'})) &= (\pi_{\sigma}(i+1) - \pi_{\sigma'}(i+1)) +(\pi_{\sigma}(j') - \pi_{\sigma'}(j'))\\
			&= \left(\left(1+\ell(\sigma,{i+2},{1}, d)\right) -\left(1+\ell(\sigma,{i+2},{d-1}, d)\right)\right) \\
			&+ \left(\ell(\sigma,{i+2},{d-1}, d) + 1\right)= 1 + \ell(\sigma,{i+2},1,d) =n.
		\end{split}
	\end{equation*}
	So $\ell({\sigma}) = \ell({\sigma'}) + 1$, a contradiction. \qedhere
\end{itemize}
\end{proof}

\begin{cor}
	Under the assumptions of Proposition~\ref{admissibility for switching an adjacent pair of different colors}, the following are equivalent:
	\begin{enumerate}[wide, labelwidth=!, labelindent=0pt]
		\item[\rm\textbf{(a)}] $\sigma'$ is not $d$-admissible,
		\item[\rm\textbf{(b)}] $i+1$ is a fixed point of $\pi_{\sigma, d}$,
		\item[\rm\textbf{(c)}] $\ell_d({\sigma}) = \ell_d({\sigma'}) + 1$.
	\end{enumerate}
\end{cor}
\begin{proof}
	This follows from Proposition~\ref{admissibility for switching an adjacent pair of different colors} and its proof. 
\end{proof}

\begin{prop}\label{length and admissibility of signature of type a b}
	Let $\sigma$ be a signature of type $(a, b)$ with $a \neq b$. Then $\sigma$ is $d$-admissible and $\ell_d({\sigma}) = d$. 
\end{prop}
\begin{proof}		
	The first claim was established in Lemma~\ref{admissibility for signature of type a, b}. Note that any two of such ($d$-admissible) signatures are related to each other by a sequence of switches of adjacent pairs. Hence $\ell_d(\sigma)$ is constant (for any signature $\sigma$ of type $(a, b)$ with $a \neq b$) by Proposition~\ref{admissibility for switching an adjacent pair of different colors}. Consider the separated signature $\sigma$ with $\sigma(1) = 1$ and $\sigma(n) = -1$ (cf.\ Definition~\ref{defn: separated signature}). Let us calculate $\ell(\sigma)$. 
	
	Without loss of generality, we may assume that $a> b$. Then 
	\begin{equation}
		\pi_{\sigma}(j) - j= \begin{cases}
			\left(\left[ \frac{d-a-2+j}{a-b} \right] + 1\right)(n-a+b) + d, &\text{if}\ 1\le j \le a - b;\\
			2(a + 1 - j), & \text{if}\  a-b + 1 \le j \le a;\\
			2(a + b + 1 -j), & \text{if} \ a+1 \le j \le a+b = n.
		\end{cases} 
	\end{equation}
	Hence
	\begin{equation*}
		\begin{split}
			\ell(\sigma) &= \frac{1}{n}\sum_{j = 1}^n (\pi_{\sigma}(j) - j) \\
			&= \frac{1}{n}\left(\sum_{j = 1}^{a-b}\left[ \frac{d-a-2+j}{a-b} \right] (n-a+b) + (n-a+b+d)(a-b) + \sum_{j = 1}^b 4j\right).
		\end{split}
	\end{equation*}
	Note that if we write $d-a-2 = u(a-b) + v$ with $0 \le v < a-b$, then 
	\[
	\left[ \frac{d-a-2+j}{a-b} \right] = \left[ \frac{u(a-b)+v+j}{a-b} \right] = u + \left[ \frac{v+j}{a-b} \right].
	\]
	Therefore
	\begin{equation}
		\begin{split}
			\ell(\sigma) &= \frac{1}{n}\sum_{j = 1}^n (\pi_{\sigma}(j) - j)\\
			& = \frac{1}{n}\left(\sum_{j = 1}^{a-b}\left[ \frac{d-a-2+j}{a-b} \right] (n-a+b) + (n-a+b+d)(a-b) + \sum_{j = 1}^b 4j\right)\\
			&= \frac{1}{n}\left(\sum_{j = 1}^{a-b}(u + \left[ \frac{v+j}{a-b} \right]) (n-a+b) + (n-a+b+d)(a-b) + 2b(b+1)\right)\\
			&= \frac{1}{n}\left((u(a-b) +v+1) (n-a+b) + (n-a+b+d)(a-b) + 2b(b+1)\right)\\
			& = \frac{1}{n}\left((d - a-1)(2b) + (2b+d)(a-b) + 2b(b+1)\right)= \frac{(a+b)d}{n} = d.\qedhere
		\end{split}
	\end{equation}
\end{proof}

\begin{remk}\label{remk: admissiblility for separated sigma}
	Let $\sigma$ be a separated signature of type $(a, b)$. Then $\sigma$ is $d$-admissible if and only if $a \neq b$ or $ a = b \ge d- 1$. 
\end{remk}

\begin{prop}\label{length and admissibility of type (a, a)}
	Let $\sigma$ be a signature of type $(a, a)$. Then $\ell_d({\sigma}) \le d$. Moreover, $\sigma$ is $d$-admissible if and only if $\ell_d({\sigma}) = d$.
\end{prop}
\begin{proof}
	We prove the statement by induction on $a$. Recall that we require $n\ge d \ge 3$, hence $a\ge 2$. When $a = 2$, then either $\sigma = \sigma_1  =  [\bullet\, \bullet\, \circ\ \circ]$ or $\sigma = \sigma_2=  [\bullet\, \circ\, \bullet\ \circ]$; and $\ell_d({\sigma_1}) = 3, \ell_d(\sigma_2) = 2$ for any $d\ge 3$. Notice that $\sigma_2$ is not $d$-admissible for any $d\ge 3$ and $\sigma_1$ is $d$-admissible only if $d = 3$. 
	
	Now let $\sigma$ be a signature of type $(a, a)$ and assume that the statement is true for signatures of type $(a-1,a-1)$.  
	Consider pairs $(i, i+1)$ with $\sigma(i) = -\sigma(i+1)$. If there are only two such pairs mod $n$, then $\sigma$ is separated. As shifting preserves admissibility and length by Lemma~\ref{shifting/reversing preseves length}, we may assume that $\sigma$ is separated. Note that $\sigma$ is admissible if and only if $a \ge d- 1$ (cf.\ Remark~\ref{remk: admissiblility for separated sigma}). Now let us calculate $\ell(\sigma)$. Suppose that $a < d -1$, then
	\[
	\ell(\sigma) = \frac{1}{n}\sum_{j=1}^a 4j = \frac{2a(a+1)}{n} = a+1 < d.
	\]
	Suppose that $a \ge d - 1$, then
	\[
	\ell(\sigma) = \frac{1}{n}\left(2d(a-d+1) + \sum_{j=1}^{d-1} 4j\right) = \frac{1}{n}\left(2ad - 2d^2 + 2d + 2d(d-1)\right) = d.
	\]
	Therefore we always have $\ell(\sigma) \le d$. Moreover,  $\ell(\sigma) = d$ if and only if $a \ge d - 1$ if and only if $\sigma$ is admissible.
	
	Now let us assume that there are at least three pairs of $(i, i+1)$ (mod $n$) with $\sigma(i) = -\sigma(i+1)$. Then there exists one such pair such that $i+1$ is not a fixed point as we can have at most two fixed points by Proposition~\ref{the number of fixed points is at most 2 for type (a, a)}. 
	
	First we show that if $\sigma$ is admissible, then $\ell(\sigma) = d$. As shifting preserves admissibility and length by Lemma~\ref{shifting/reversing preseves length}, we may assume that $i = 2a - 1$. Without loss of generality, we also assume that $\sigma(i) = 1$ and $\sigma(i+1) = -1$. Now consider the signature $\hat{\sigma}$ of type $(a-1, a-1)$ define by $\hat{\sigma}(j) = \sigma(j)$ for $1\le j \le 2a-2$. In other words, $\hat{\sigma}$ is obtained from $\sigma$ by removing $i$ and $i+1$. Let us show that $\hat{\sigma}$ is admissible. 
	Similar to the proof of Proposition~\ref{admissibility for switching an adjacent pair of different colors}, (c) $\Rightarrow$ (a), if $\hat{\sigma}$ is not admissible, then $i+1$ will be a fixed point of $\sigma$, a contradiction. Now $\hat{\sigma}$ is admissible implies $\ell({\hat{\sigma}}) = d$ by induction hypotheses. 
	
	As $\sigma$ is admissible, so is $\rev{\sigma}$ by Lemma~\ref{admissibility of reversing}. For $0\le k \le d-1$, let $j_k =2a - \ell(\rev{\sigma},1,k, d)$. Then $j_1 = 2a -1 = i$ since $\sigma(i) = 1$;  and $i+1$ is not a fixed point of $\sigma$ implies that $i$ is not a fixed point of $\rev{\sigma}$. Hence $\ell(\rev{\sigma},1,0, d) < 2a$, and therefore $1\le j_k \le 2a -2$ for $k\neq 1$. Now let $1\le j \le 2a-2$ such that $j \neq j_k$ for any $0\le k \le d-1$. Then we must have $\pi_{\hat{\sigma}}(j) = \pi_\sigma(j)$. Notice that $\pi_{\hat{\sigma}}(j_k) = \pi_\sigma(j_k) - 2$ for $1< k \le d -1$, $\pi_\sigma(j_1) - j_1 = \pi_\sigma(i) - i = 2$, and $\pi_\sigma(j_0) - \pi_{\hat{\sigma}}(j_0) = 1 - \ell(\sigma,{1},1, d) $. Lastly, we have $\pi_{\sigma}(i+1) - (i+1) = 1 + \ell(\sigma,{1},1, d)$. Put all these together, we get
	\begin{equation*}
		\begin{split}
			2a\ell(\sigma) - 2(a-1)\ell({\hat{\sigma}}) & = \sum_{0\le k \le d-1, \ k \neq 1} \bigg( \left( \pi_\sigma(j_k) - \pi_{\hat{\sigma}}(j_k)\right) \\
			&\quad\quad \quad \quad \quad\quad+ (\pi_\sigma(i) - i) + (\pi_\sigma(i+1) - (i+1))\bigg)\\
			& = 2(d-2) + (1- \ell(\sigma,1,1, d)) + 2 + (1 + \ell(\sigma,1,1, d)) = 2d.
		\end{split}
	\end{equation*}
	Since $\ell({\hat{\sigma}}) = d$, we conclude that 
	\[
	\ell(\sigma) = \frac{2(a-1)\ell({\hat{\sigma}})  + 2d}{2a} = \frac{2(a-1)d + 2d}{2a} = d,
	\]
	which completes the proof that if $\sigma$ is admissible, then $\ell(\sigma) = d$. 
	
	Now we show that if $\sigma$ is not admissible, then $\ell(\sigma) < d$. We follow the  same procedure as above,  it is clear that $\hat{\sigma}$ is not admissible, so $\ell({\hat{\sigma}}) < d$. As $\sigma'$ is not admissible, some of the $j_k$'s might not be defined (as $\ell(\rev{\sigma}, 1, k, d)$ can be $\infty$), and whenever that happens, the term $\left( \pi_\sigma(j_k) - \pi_{\hat{\sigma}}(j_k)\right)$ will not appear in $2a\ell(\sigma) - 2(a-1)\ell({\hat{\sigma}})$, as a result, we have an inequality:
	\begin{equation*}
		\begin{split}
			2a\ell(\sigma) - 2(a-1)\ell({\hat{\sigma}}) & \le \sum_{0\le k \le d-1, \ k \neq 1} \bigg(\left( \pi_\sigma(j_k) - \pi_{\hat{\sigma}}(j_k)\right) \\
			&\quad \quad \quad\quad\quad \quad\quad+ (\pi_\sigma(i) - i) + (\pi_\sigma(i+1) - (i+1))\bigg)= 2d.
		\end{split}
	\end{equation*}
	Therefore
	\[
	\ell(\sigma) \le \frac{2(a-1)\ell({\hat{\sigma}})  + 2d}{2a} < \frac{2(a-1)d + 2d}{2a} = d.
	\]
	This completes the proof that if $\sigma$ is not admissible, then $\ell(\sigma) < d$. 	
\end{proof}

\begin{theorem}\label{thm: length and admissibility theorem}
	For any signature $\sigma$, we have $\ell_d(\sigma) \le d$.  Moreover, $\sigma$ is $d$-admissible if and only if $\ell_d({\sigma}) = d$.
\end{theorem}

\begin{proof}
	This follows from Proposition~\ref{length and admissibility of signature of type a b} and Proposition~\ref{length and admissibility of type (a, a)}.
\end{proof}

\begin{cor}\label{remk: correspondence between signatures and affine permutations}
	The map $\sigma \mapsto \pi_{\sigma, d}$ gives a bijection
\[
		\{\text{signatures}\}/\{\pm 1\} \longleftrightarrow \{\text{$d$-representable biased affine permutations of bias} \le d\}.
\]
	This bijection restricts to a bijection 
	\[
	\{\text{$d$-admissible signatures}\}/\{\pm 1\} \longleftrightarrow \{\text{$d$-representable affine permutations of bias } d\}.
	\]
\end{cor}
\begin{proof}
	Combine Theorem~\ref{thm: length and admissibility theorem} and Lemma~\ref{lemma same pi implies same signature up to sign}.
\end{proof}

In Section~\ref{subsec: adundant signatures}, we will provide a counterpart of Corollary \ref{remk: correspondence between signatures and affine permutations} for bounded affine permutations.

\begin{cor}\label{Not admissible when d > a + 1}
	Let $\sigma$ be a signature of type $(a, a)$ with  $a < d-1$. Then $\sigma$ is not $d$-admissible.
\end{cor}
\begin{proof}
	Let $1 \le j \le n = 2a$, define $g(j) = \pi_{\sigma}(j)$, if $j$ is not a fixed point, then 
	\[
	\sum_{i = g(j)}^{j-1 +n} \sigma(i) = - \sum_{i = j}^{g(j) -1}\sigma(i) = 0,
	\]
	hence
	\[
	(\pi_\sigma(j) - j ) + (\pi_\sigma(g(j))-g(j)) = \pi_\sigma(g(j)) - j \le n.
	\]
	If $j$ is a fixed point, then $g(j) = j + n$ by Lemma~\ref{lemma: fixed points change color equivalent to a = b and must of length n}, hence
	\[
	(\pi_\sigma(j) - j ) + (\pi_\sigma(g(j))-g(j)) = 2n.
	\]
	Note that if $j \neq j'$, then $g(j) \not\equiv g(j') \mod n$, so we have 
	\begin{align*}
		2\ell(\sigma) &= \frac{1}{n}\sum_{j=1}^n (\pi_\sigma(j) - j ) + (\pi_\sigma(g(j))-g(j)) \\
		&\le  \frac{1}{n}\left(2n\fix(\sigma) + n(n-\fix(\sigma))\right) = n + \fix(\sigma).
	\end{align*} 
	Hence if $a < d- 1$, then $n + 2 = 2a + 2 < 2d$, so
	\begin{equation}
		\ell(\sigma) \le \frac{n + \fix(\sigma)}{2} \le \frac{n + 2}{2} < d,
	\end{equation}
	here we used the fact that $\fix(\sigma) \le 2$ by Proposition~\ref{the number of fixed points is at most 2 for type (a, a)}. Finally by Theorem~\ref{thm: length and admissibility theorem} we conclude that $\sigma$ is not $d$-admissible as $\ell(\sigma) < d$.
\end{proof}

\begin{cor}\label{cor: admissible = separated when d = a+ 1}
	Let $\sigma$ be a signature of type $(a, a)$ with $a = d - 1$. Then $\sigma$ is $d$-admissible if and only if $\sigma$ is separated. 
\end{cor}
\begin{proof}
	Assume that $\sigma$ is separated. Then $\sigma$ is $d$-admissible by Remark~\ref{remk: admissiblility for separated sigma}.
	
	Assume that $\sigma$ is $d$-admissible. If there exists a pair $(i, i+1)$ of opposite colors such that $i+1$ is not a fixed point of $\sigma$, then the signature $\hat{\sigma}$ constructed as in the proof of Proposition~\ref{length and admissibility of type (a, a)} (by removing $i$ and $i+1$) is also $d$-admissible. Notice that $\hat{\sigma}$ is of type $(a-1, a-1)$ and $a - 1  = d - 2 < d - 1$. Hence $\hat{\sigma}$  is not $d$-admissible by Corollary~\ref{Not admissible when d > a + 1}, a contradiction. This implies that if $i$ and $i+1$ are of different colors, then $i+1$ is a fixed point of $\sigma$. By Proposition~\ref{the number of fixed points is at most 2 for type (a, a)}, there are at most $2$ fixed points of $\sigma$ (mod $n$). Therefore, there are at most two pairs of adjacent opposite colors (mod $n$), which implies that $\sigma$ is separated. 
\end{proof}

\subsection{Abundant signatures}\label{subsec: adundant signatures}

In this section, we study the counterparts of bounded affine permutations under the correspondence described in Remark~\ref{remk: correspondence between signatures and affine permutations}. 

\begin{defn}\label{definition of abundant }
	A signature $\sigma$  of size $n$ is called \emph{$d$-bounded} if $\ell(\sigma,j,0, d) \le n$ for all $j\in \Z$. Equivalently, $\sigma$ is $d$-bounded if $\pi_{\sigma, d}$ is a bounded affine permutation, cf.\ Definition \ref{defn: biased affine permutation and bounded affine permutation}.
\end{defn}

\begin{prop}\label{shifting/reversing preseves boundness}
	Let $\sigma$ be a signature. The following are equivalent:
	\begin{itemize}[wide, labelwidth=!, labelindent=0pt]
		\item $\sigma$ is $d$-bounded;
		\item $\rev{\sigma}$ is $d$-bounded;
		\item $\tau_m(\sigma)$ is  $d$-bounded for some $m\in \Z$;
		\item $\tau_m(\sigma)$ is  $d$-bounded for every $m\in \Z$.
	\end{itemize}
\end{prop}
\begin{proof}
	For any $m\in \Z$, we have
	\[
	\pi_{\tau_m(\sigma)}(j) - j  =  \pi_\sigma(j+m) - (j+m),  \text{ for } j\in \Z.
	\]
	Hence $\sigma$ is bounded if and only if ${\tau_m(\sigma)}$ is $d$-bounded.
	
	Similarly, we have 
	\[
	\pi_{\rev{\sigma}}(j) = -\pi^{-1}_{\sigma}(-j), \text{ for } j\in \Z.
	\]
	Assume that $\sigma$ is $d$-bounded. Then for any $j\in \Z$, we have
	\[
	-j = \pi_{\sigma}(\pi^{-1}_{\sigma}(-j)) \le \pi^{-1}_{\sigma}(-j) + n.
	\]
	Hence for all $j\in \Z$, we have
	\[
	\pi_{\rev{\sigma}}(j) = -\pi^{-1}_{\sigma}(-j) \le j + n.
	\]
	Therefore ${\rev{\sigma}}$ is $d$-bounded. The converse follows from the fact that $\rev{\rev{\sigma}} = \sigma$. 
\end{proof}

\begin{lemma}\label{lemma: type (a,a) is bounded}
	Any signature $\sigma$ of type $(a, a)$ is $d$-bounded. 
\end{lemma}
\begin{proof}
	For any $j\in \Z$, we have 
	\begin{equation}
		\sum_{i = j}^{j+n-1} \sigma(i) = a - a = 0.
	\end{equation}
	Hence $\ell(\sigma,j,0,d) \le n$, so $\sigma$ is bounded. 
\end{proof}

\begin{defn}\label{def: abundant}
	A signature $\sigma$ is called \emph{$d$-abundant} if $\ell(\sigma,j,k, d) \le n$ for all $j\in \Z$ and $0\le k \le d-1$. 
\end{defn}

Note that if $\sigma$ is $d$-abundant, then it is both $d$-admissible and $d$-bounded. The converse is also true:

\begin{prop}\label{proposition abundant  = bounded + admissible}
	A signature $\sigma$ is $d$-abundant  if and only if $\sigma$ is $d$-bounded and $d$-admissible. 
\end{prop}

In other words, $\ell(\sigma, j, 0, d) \le n$ and $\ell(\sigma, j, k, d) < \infty$ for all $j\in \Z$ and $1\le k \le d -1$ forces $\ell(\sigma, j, k, d) \le n$ for $j\in \Z$, $0\le k \le d-1$. 

\begin{proof}
	As mentioned above, if $\sigma$ is abundant, then it is bounded and admissible by definition. 
	
	Now assume that $\sigma$ is bounded and admissible. Then $\ell(\sigma,j,k, d) < \infty$ and $\ell(\sigma,j,0, d) \le n$ for all $j\in \Z$ and $1\le k \le d-1$.
	Consider the set \[
	S = \{\ell(\sigma,{j},k, d)\}_{1\le j \le n, \, 0\le k \le d-1}.
	\]
	It is a finite set of integers, so $\max{S}$ exists. Take $j\in \Z,\  0\le k \le d-1$ such that $\ell(\sigma,{j},k, d) = \max{S}$. 
	
	Suppose that $\sigma$ is not abundant. Then $\ell(\sigma,{j},k, d) = \max{S} > n$. Since $\ell(\sigma,j,0, d) \le n$ by assumption, we must have $k \neq 0$. Without loss of generality, we may assume that $\sigma(j) = 1$. If $\sigma(j-1) = 1$, then
	\[
	\ell(\sigma,{j-1},{k+1}, d) =	\ell(\sigma,{j},k, d) + 1 > \max{S},
	\]
	a contradiction. If $\sigma(j-1) = -1$, then 
	\[
	\ell(\sigma,{j-1},{k-1}, d) =	\ell(\sigma,{j},k, d) +1 > \max{S},
	\]
	a contradiction. 
\end{proof}

\begin{cor}\label{cor: abundant = admissible for type (a, a)}
	Let $\sigma$ be a signature of type $(a, a)$. Then $\sigma$ is $d$-abundant if and only if $\sigma$ is $d$-admissible. 
\end{cor}
\begin{proof}
	This follows from Lemma~\ref{lemma: type (a,a) is bounded} and Proposition~\ref{proposition abundant  = bounded + admissible}.
\end{proof}

\begin{theorem}\label{abundant  = bounded affine permutation theorem}
	A signature $\sigma$ is $d$-abundant  if and only if $\pi_{\sigma, d}$ is a bounded affine permutation of bias $d$.
\end{theorem}

\begin{proof}
	Suppose that $\sigma$ is $d$-abundant. Then $\sigma$ is $d$-bounded, so $\pi_{\sigma, d}$ is a bounded affine permutation. By Theorem~\ref{thm: length and admissibility theorem}, we have $\bb(\pi_{\sigma, d}) = \ell_d(\sigma) = d$ as $\sigma$ is $d$-admissible. 
	
	Conversely, assume that $\pi_{\sigma, d}$ is a bounded affine permutation of bias $d$. Then $\sigma$ is $d$-bounded and $\ell_d({\sigma}) = d$. By Theorem~\ref{thm: length and admissibility theorem}, $\ell_d({\sigma}) = d$ implies that $\sigma$ is $d$-admissible. So $\sigma$ is $d$-bounded and $d$-admissible, hence $d$-abundant  by Proposition~\ref{proposition abundant  = bounded + admissible}.
\end{proof}

\begin{remk}
	Following Corollary~\ref{remk: correspondence between signatures and affine permutations},  the bijection
	\begin{align*}
		\{\text{signatures}\}/\{\pm 1\} &\longrightarrow \{\text{$d$-representable biased affine permutations of bias} \le d\}, \\
		\sigma & \longmapsto \pi_{\sigma, d}
	\end{align*}
	restricts to bijections
	\begin{equation*}
		\begin{split}
			\{\text{$d$-admissible signatures}\} &\leftrightarrow \{\text{$d$-representable affine permutations of bias} = d\},\\
			\{\text{$d$-bounded signatures}\} &\leftrightarrow \{\text{$d$-representable bounded affine permutations of bias} \le d\},\\
			\{\text{$d$-abundant  signatures}\} &\leftrightarrow \{\text{$d$-representable bounded affine permutations of bias} = d\}.
		\end{split}
	\end{equation*}
\end{remk}

\begin{cor}\label{shifting/reversing preserves strongly admissibility corollary}
	For a signature $\sigma$, the following are equivalent:
	\begin{itemize}[wide, labelwidth=!, labelindent=0pt]
		\item $\sigma$ is $d$-abundant;
		\item $\rev{\sigma}$ is $d$-abundant;
		\item $\tau_m(\sigma)$ is  $d$-abundant for some $m\in \Z$;
		\item $\tau_m(\sigma)$ is  $d$-abundant for every $m\in \Z$.
	\end{itemize}
\end{cor}

\begin{proof}
	This follows from Proposition~\ref{shifting/reversing preseves boundness}, Proposition~\ref{shifting/reversing preseves length} and Theorem~\ref{abundant  = bounded affine permutation theorem}. 
\end{proof}

\begin{prop}\label{separated signature abundant  condition}
	Let $\sigma$ be a separated signature of type $(a, b)$. Then $\sigma$ is $d$-abundant  if and only if 
	\begin{equation}
		d \le \max\{\min\{a, b\} + 1,  |a - b|\}.
	\end{equation}
\end{prop}

\begin{proof}
	Suppose that $a = b$. Then $\sigma$ is $d$-abundant  if and only if $\sigma$ is $d$-admissible if and only if $d \le a + 1$ by Corollary~\ref{cor: abundant = admissible for type (a, a)}  and Proposition~\ref{length and admissibility of signature of type a b}. 
	
	Suppose that $a\neq b$. Without loss of generality, we may assume that $a>b$,  $\sigma(1) = 1$ and $\sigma(n) = -1$. By Theorem~\ref{abundant  = bounded affine permutation theorem}, $\sigma$ is $d$-abundant  if and only if 
	\[\max_{j\in \Z}\{\pi_\sigma(j) - j\} \le n.\] As in the proof of Proposition~\ref{length and admissibility of signature of type a b}, we have 
	\begin{equation}
		\pi_{\sigma}(j) - j= \begin{cases}
			\left(\left[ \frac{d-a-2+j}{a-b} \right] + 1\right)(n-a+b) + d, &\text{if}\ 1\le j \le a - b;\\
			2(a + 1 - j), & \text{if}\  a-b + 1 \le j \le a;\\
			2(a + b + 1 -j), & \text{if} \ a+1 \le j \le a+b = n.
		\end{cases} 
	\end{equation}
	So the maximum of $\pi_{\sigma}(j) - j$ can only happen at $j = a-b$, $j = a-b + 1$ or $j = a + 1$. Hence
	\begin{align*}
		\max_{j\in \Z}\{\pi_\sigma(j) - j\} &= \max\{\left(\left[ \frac{d-a-2+a-b}{a-b} \right] + 1\right)(n-a+b) + d, 2b\} \\
		&= \max\{2b\left(\left[ \frac{d-b-2}{a-b} \right] + 1\right) + d, 2b\}.
	\end{align*}
	Notice that $2b < a + b = n$, so $\sigma$ is $d$-abundant  if and only if 
	\[2b\left(\left[ \frac{d-b-2}{a-b} \right] + 1\right) + d \le n.\]
	This is equivalent to $d \le \max\{b + 1, a -b\}$. \qedhere
\end{proof}

\newpage
\printbibliography
\end{document}